\documentclass[12pt,graybox,envcountchap,sectrefs]{svmono}

\renewcommand\thesection{\thechapter.\arabic{section}}

\usepackage{mathptmx}
\usepackage{helvet}
\usepackage{courier}
\usepackage{amssymb}
\usepackage{amsmath}
\usepackage{type1cm}

\usepackage{makeidx}         
\usepackage{graphicx}        
\usepackage{multicol}        

\makeindex             

\def\pp{\partial}
\def\C{\mathbb{C}}
\def\G{\mathbb{G}}
\def\D{\mathbb{D}}
\def\R{\mathbb{R}}
\def\F{\mathbb{F}}
\def\Z{\mathbb{Z}}
\def\T{\mathbb{T}}
\def\P{\mathbb{P}}
\def\dd{\displaystyle}
\def\9{\infty}
\def\vare{\varepsilon}
\def\ov{\overline}
\def\ce{\overset{\circ}{E}}
\def\wt{\widetilde}
\def\de{\delta}
\def\oo{\omega}
\def\delin{{^{-}}\!\!\!\!\!\!\de}
\def\O{\Omega}
\def\a{\alpha}
\def\b{\beta}
\def\g{\gamma}
\def\({\left(}
\def\){\right)}
\def\[{\left[}
\def\]{\right]}
\def\be{\begin{equation}}
\def\ee{\end{equation}}
\def\overunu{\overset{1}}
\def\overdoi{\overset{2}}
\def\overk{\overset{k}}
\def\slk{\Sigma_{L^k}}
\def\sfk{\Sigma_{F^k}}
\def\shk{\Sigma_{H^k}}
\def\circg{{\overset{\circ}g}}
\def\oci{\overset{\circ}}
\def\sil{\Sigma_L}

\input{tcilatex}

\begin{document}

\author{Radu MIRON}
\title{Lagrangian and Hamiltonian geometries. Applications to Analytical Mechanics}
\date{}

\maketitle

\frontmatter


%
%

\foreword
The aim of the present monograph is twofold: 1$^\circ$ to provide a Compendium of Lagrangian and Hamiltonian geometries and 2$^\circ$ to introduce and investigate new analytical Mechanics: Finslerian, Lagrangian and Hamiltonian.

One knows (R. Abraham, J. Klein, R. Miron et al.) that the geometrical theory of nonconservative mechanical systems can not be rigourously constructed without the use of the geometry of the tangent bundle of the configuration space.

The solution of this problem is based on the Lagrangian and Hamiltonian geometries. In fact, the construction of these geometries relies on the mechanical principles and on the notion of Legendre transformation.

The whole edifice has as support the sequence of inclusions: $$\{{\cal R}^n\}\subset\{F^n\}\subset\{L^n\}\subset\{GL^n\}$$ formed by Riemannian, Finslerian, Lagrangian, and generalized Lagrangian spaces. The ${\cal L}-$duality transforms this sequence into a similar one formed by Hamiltonian spaces.

Of course, these sequences suggest the introduction of the correspondent Mechanics: Riemannian, Finslerian, Lagrangian, Hamiltonian etc.

The fundamental equations (or evolution equations) of these Mechanics are derived from the variational calculus applied to the integral of action and these can be studied by using the methods of Lagrangian or Hamiltonian geometries.

More general, the notions of higher order Lagrange or Hamilton spaces have been introduced by the present author \cite{mir11}, \cite{mir12}, \cite{mir13} and developed is realized by means of two sequences of inclusions similarly with those of the geometry of order 1. The problems raised by the geometrical theory of Lagrange and Hamilton spaces of order $k\geq 1$ have been investigated by Ch. Ehresmann, W. M. Tulczyjew, A. Kawaguchi, K. Yano, M. Crampin, Manuel de L\'{e}on, R. Miron, M. Anastasiei, I. Buc\u{a}taru et al. \cite{Miron}. The applications lead to the notions of Lagrangian or Hamiltonian Analytical Mechanics of order $k$.

For short, in this monograph we aim to solve some difficult problems:

- The problem of geometrization of classical non conservative mechanical systems;

- The foundations of geometrical theory of new mechanics: Finslerian, Lagrangian and Hamiltonian;

- To determine the evolution equations of the classical mechanical systems for whose external forces depend on the higher order accelerations.

This monograph is based on the terminology and important results taken from the books: Abraham, R., Marsden, J., {\it Foundation of Mechanics}, Benjamin, New York, 1978; Arnold, V.I., {\it Mathematical Methods in Classical Mechanics}, Graduate Texts in Math, Springer Verlag, 1989; Asanov, G.S., {\it Finsler Geometry Relativity and Gauge theories}, D. Reidel Publ. Co, Dordrecht, 1985; Bao, D., Chern, S.S., Shen, Z., {\it An Introduction to Riemann-Finsler Geometry}, Springer Verlag, Grad. Text in Math, 2000; Bucataru, I., Miron. R., {\it Finsler-Lagrange geometry. Applications to dynamical systems}, Ed. Academiei Romane, 2007; Miron, R., Anastasiei, M., {\it The geometry of Lagrange spaces. Theory and applications to Relativity}, Kluwer Acad. Publ. FTPH no. 59, 1994; Miron, R., {\it The Geometry of Higher-Order Lagrange Spaces. Applications to Mechanics and Physics}, Kluwer Acad. Publ. FTPH no. 82, 1997; Miron, R., {\it The Geometry of Higher-Order Finsler Spaces}, Hadronic Press Inc., SUA, 1998; Miron, R., Hrimiuc, D., Shimada, H., Sabau, S., {\it The geometry of Hamilton and Lagrange Spaces}, Kluwer Acad. Publ., FTPH no. 118, 2001.

This book has been received a great support from  Prof. Ovidiu C\^{a}rj\u{a}, the dean of the Faculty of Mathematics from ``Alexandru Ioan Cuza'' University of Iasi.

My colleagues Professors M. Anastasiei, I. Buc\u{a}taru, I. Mihai, M. Postolache, K. Stepanovici made important remarks and suggestions and Mrs. Carmen Savin prepared an excellent print-form of the hand-written text. Many thanks to all of them.

\vspace{\baselineskip}
\begin{flushright}\noindent
Iasi, June 2011\hfill {\it Radu  Miron}\\
\end{flushright}

\newpage


The purpose of the book is a short presentation of the geometrical theory of Lagrange and Hamilton spaces of order 1 or of order $k\geq 1$, as well as the definition and investigation of some new Analytical Mechanics of Lagrangian and Hamiltonian type of order $k\geq 1$.

In the last thirty five years, geometers, mechanicians and physicists from all over the world worked in the field of Lagrange or Hamilton geometries and their applications. We mention only some important names: P.L. Antonelli \cite{Ant2}, M. Anastasiei \cite{Anastas2}, \cite{Anastas3}. G.~S. Asanov \cite{As}, A. Bejancu \cite{Bj1}, I. Buc\u{a}taru \cite{Buc}, M. Crampin \cite{Cr}, R. S. Ingarden \cite{ingarden}, S. Ikeda \cite{IS1}, M. de Leon \cite{leonrodri}, M. Matsumoto \cite{Ma1}, R. Miron \cite{mirana}, \cite{mirana2}, \cite{miranabuc}, H. Rund \cite{rund}, H. Shimada \cite{shimada}, P. Stavrinos \cite{StavBal}, L. Tamassy \cite{Tam1} and S. Vacaru \cite{VS}.

The book is divided in three parts: I. Lagrange and Hamilton spaces; II. Lagrange and Hamilton spaces of higher order; III. Ana\-ly\-ti\-cal Mechanics of Lagrangian and Hamiltonian mechanical systems.

The part I starts with the geometry of tangent bundle $(TM,\pi,M)$ of a differentiable, real, $n-$dimensional manifold $M$. The main geometrical objects on $TM$, as Liouville vector field $\C$, tangent structure $J$, semispray $S$, nonlinear connection $N$ determined by $S$, $N-$metrical structure $D$ are pointed out. It is continued with the notion of Lagrange space, defined by a pair $L^n=(M,L(x,y))$ with $L:TM\to \R$ as a regular Lagrangian and $g_{ij}=\dfrac12\dot{\pp}_i\dot{\pp}_j L$ as fundamental tensor field. Of course, $L(x,y)$ is a regular Lagrangian if the Hessian matrix $\|g_{ij}(x,y)\|$ is nonsingular. In the definition of Lagrange space $L^n$ we assume that the fundamental tensor $g_{ij}(x,y)$ has a constant signature. The known Lagrangian from Electrodynamics assures the existence of Lagrange spaces.

The variational problem associated to the integral of action $$I(c)=\dint^1_0 L(x,\dot{x})dt$$ allows us to determine the Euler--Lagrange equations, conservation law of energy, ${\cal E}_L$, as well as the canonical semispray $S$ of $L^n$. But $S$ determines the canonical nonlinear connection $N$ and the metrical $N-$linear connection $D$, given by the generalized Christoffel symbols. The structure equations of $D$ are derived. This theory is applied to the study of the electromagnetic and gravitational fields of the space $L^n$. An almost K\"{a}hlerian model is constructed. This theory suggests to define the notion of generalized Lagrange space $GL^n=(M,g(x,y))$, where $g(x,y)$ is a metric tensor on the manifold $TM$. The space $GL^n$ is not reducible to a space $L^n$.

A particular case of Lagrange space $L^n$ leads to the known concept of Finsler space $F^n=(M,F(x,y))$. It follows that the geometry of Finsler space $F^n$ can be constructed only by means of Analytical Mechanics principles.

Since a Riemann space ${\cal R}^n=(M\cdot g(x))$ is a particular Finsler space $F^n=(M,F(x,y))$ we get the following sequence of inclusions:
$$\{{\cal R}^n\}\subset\{F^n\}\subset\{L^n\}\subset\{GL^n\}.\leqno({\rm{I}})$$ The Lagrangian geometry is the geometrical study of the sequence (I).

The geometrical theory of the Hamilton spaces can be constructed step by step following the theory of Lagrange spaces. The le\-gi\-ti\-macy of this theory is due to ${\cal L}-$duality (Legendre duality) between a Lagrange space $L^n=(M,L(x,y))$ and a Hamilton space $H^n=(M,H(x,p))$.

Therefore, we begin with the geometrical theory of cotangent bundle $(T^*M,\pi^*,M)$, continue with the notion of Hamilton space $H^n(M,$ $H(x,p))$, where $H:T^*M\to R$ is a regular Hamiltonian, with the variational problem $$I(c)=\dd\int^1_0\left[p_i(t)\dfrac{dx^i}{dt}-\dfrac12 L(x(t),p(t))\right]dt,$$ from which the Hamilton--Jacobi equations are derived, Hamiltonian vector field of $H^n$, nonlinear connection $N$, $N^*-$linear connection etc. The Legendre transformation ${\cal L}:L^n\to H^n$ transforms the fundamental geometrical object fields on $L^n$ into the fundamental geometrical object fields on $H^n$. The restriction of ${\cal L}$ to the Finsler spaces $\{F^n\}$ has as image a new class of spaces ${\cal C}^n=(M,K(x,p))$ called the Cartan spaces. They have the same beauty, symmetry and importance as the Finsler spaces. A pair $GH^n=(M,g^*(x,p))$, where $g^*$ is a metric tensor on $T^*M$ is named a generalized Hamilton space. Remarking that ${\cal R}^{*n}=(M,g^*(x))$, with $g^{*ij}(x)$ the contravariant Cartan space, we obtain a dual sequence of the sequence (I): $$\{{\cal R}^{*n}\}\subset\{{\cal C}^n\}\subset\{H^n\}\subset\{GH^n\}.\leqno({\rm{II}})$$ The Hamiltonian geometry is geometrical study of the sequence II.

The Lagrangian and Hamiltonian geometries are useful for applications in: Variational calculus, Mechanics, Physics, Biology etc.

The part II of the book is devoted to the notions of Lagrange and Hamilton spaces of higher order. We study the geometrical theory of the total space of $k-$tangent bundle $T^kM$, $k\geq 1$, generalizing, step by step the theory from case $k=1$. So, we introduce the Liouville vector fields, study the variational problem for a given integral of action on $T^kM$, continue with the notions of $k-$semispray, nonlinear connection, the prolongation to $T^kM$ of the Riemannian structure defined on the base manifold $M$. The notion of $N-$metrical connections is pointed out, too.

It follows the notion of Lagrange space of order $k\geq 1$. It is defined as a pair $L^{(k)n}=(M,L(x,y^{(1)},y^{(2)},...,y^{(k)}))$, where $L$ is a Lagrangian depending on the (material point) $x=(x^i)$ and on the accelerations $y^{(k)^i}=\dfrac{1}{k!}\dfrac{d^{k}x^i}{dt}$, $k=1,2,...,k$; $i=1,2,...,n$, $n={\rm{dim}}M$. Some examples prove the existence of the spaces $L^{(k)n}$. The canonical nonlinear connection $N$ and metrical $N-$linear connection are pointed out. The Riemannian almost contact model for $L^{(k)n}$ as well as the Generalized Lagrange space of order $k$ end this theory.

The methods used in the construction of the Lagrangian geometry of higher-order are the natural extensions of those used in the theory of Lagrangian geometries of order $k=1$.

Next chapter treats the geometry of Finsler spaces of order $k\geq 1$, $F^{(k)1}=(M,F(x,y^{(1)},...,y^{(k)}))$, $F:T^kM\to R$ being positive $k-$ho\-mo\-ge\-neous on the fibres of $T^k M$ and fundamental tensor has a constant signature. Finally, one obtains the sequences $$\{{\cal R}^{(k)n}\}\subset\{F^{(k)n}\}\subset\{L^{(k)1}\}\subset\{GL^{(k)n}\}.\leqno({\rm{III}})$$
The Lagrangian geometry of order $k\geq 1$ is the geometrical theory of the sequences of inclusions (III).

The notion of Finsler spaces of order $k$, introduced by the present author, was investigated in the book {\it The Geometry of Higher-Order Finsler Spaces}, Hadronic Press, 1998. Here it is a natural extension to the manifold $\widetilde{T^kM}$ of the theory of Finsler spaces given in Section 3. It was developed by H. Shimada and S. Sab\u{a}u \cite{shimada}.

The previous theory is continued with the geometry of $k$ cotangent bundle $T^{*k}M$, defined as the fibered product: $T^{k-1}MX_M T^*M$, with the notion of Hamilton space of order $k\geq 1$ and with the particular case of Cartan spaces of order $k$. An extension to $L^{(k)n}$ and $H^{(k)n}$ of the Legendre transformation is pointed out. The Hamilton--Jacobi equations of $H^{(k)n}$, which are fundamental equations of these spaces, are presented, too.

For the Hamilton spaces of order $k$, $H^{(k)n}=(M,$ $H(x,y)^{(k)}$, ..., $y^{(k-1)},p)$ a similar sequence of inclusions with (III) is introduced. These considerations allow to clearly define what means the Hamiltonian geometry of order $k$ and what is the utility of this theory in applications.

Part III of this book is devoted to applications in Analytical Mechanics. One studies the geometrical theory of scleronomic nonconservative classical mechanical systems $\Sigma_{\cal R}=(M,T,Fe)$, it is introduced and investigated the notion of Finslerian mechanical systems $\Sigma_F=(M,F,Fe)$ and is defined the concept of Lagrangian mechanical system $\Sigma_L=(M,L,Fe)$. In all these theories $M$ is the configuration space, $T$ is the kinetic energy of a Riemannian space ${\cal R}^n=(M,g)$, $F(x,y)$ is the fundamental function of a Finsler space $F^n=(M,F(x,y))$, $L(x,y)$ is a regular Lagrangian and $Fe(x,y)$ are the external forces which depend on the material points $x\in M$ and their velocities $y^i=\dot{x}^i=\dfrac{dx^i}{dt}$.

In order to study these Mechanics we apply the methods from the Lagrangian geometries. The contents of these geometrical theory of mechanical systems $\Sigma_{{\cal R}}$, $\Sigma_F$ and $\Sigma_L$ is based on the geometrical theory of the velocity space $TM$. The base of these investigations are the Lagrange equations. They determine a canonical semispray, which is fundamental for all constructions. For every case of $\Sigma_{\cal R}$, $\Sigma_F$ and $\Sigma_L$ the law of conservation of energy is pointed out. We end with the corresponding almost Hermitian model on the velocity space $TM$.

The dual theory via Legendre transformation, leads to the geometrical study of the Hamiltonian mechanical systems $\Sigma^*_{\cal R}=(M,T^*,Fe^*)$, $\Sigma^*_{\cal C}=(M,K,Fe^*)$ and $\Sigma^*_H=(M,H,Fe^*)$ where $T^*$ is energy, $K(x,p)$ is the fundamental function of a given Cartan space and $H(x,p)$ is a regular Hamiltonian on the cotangent bundle $T^*M$. The fundamental equations in these Hamiltonian mechanical systems are the Hamilton equations $$\dfrac{dx^i}{dt}=-\dfrac{\pp {\cal H}}{\pp p_i};\quad \dfrac{dp_i}{dt}=\dfrac{\pp {\cal H}}{\pp x^i}+\dfrac12 F_i, \({\cal H}=\dfrac12 H\),$$ $F_i(x,p)$ being the covariant components of external forces. The used methods are those given by the sequence of inclusions $\{{\cal R}^*\}\subset\{{\cal C}^n\}\subset\{H^n\}\subset\{GH^n\}$.

More general, the notions of Lagrangian and Hamiltonian mechanical systems of order $k\geq 1$ are introduced and studied. Therefore, to the sequences of inclusions III correspond the analytical mechanics of order $k$ of the Riemannian, Finslerian, Lagrangian type. All these cases are the direct generalizations from case $k=1$. This is the reason why we present here shortly the principal results. Especially, we paid attention to the following question:

{\it Considering a Riemannian $($particularly Euclidian$)$ mechanical system $\Sigma_{\cal R}=(M,T,Fe)$, with $T=\dfrac{1}{2} g_{ij}(x)\dot{x}^i\dot{x}^j$ as energy, but with the external forces $Fe$ depending on material point $(x^i)$ and on higher order accelerations $\left(\dfrac{dx^i}{dt}\right.$, $\dfrac{d^2 x^i}{dt^2}$, .... $\left.\dfrac{d^k x^i}{dt^k}\right)$. What are the evolution equations of the system $\Sigma_{\cal R}$?}

\medskip

\noindent Clearly, the classical Lagrange equations are not valid. In the last part of this book we present the solution of this problem.

Finally, what is new in the present book?
\begin{itemize}
\item[$1^\circ$] A solution of the problem of geometrization of the classical nonconservative mechanical systems, whose external forces depend on velocities, based on the differential geometry of velocity space.
\item[$2^\circ$] The introduction of the notion of Finslerian mechanical system.
\item[$3^\circ$] The definition of Cartan mechanical system.
\item[$4^\circ$] The study of theory of Lagrangian and Hamiltonian mechanical systems by means of the geometry of tangent and cotangent bundles.
\item[$5^\circ$] The geometrization of the higher order Lagrangian and Hamiltonian mechanical systems.
\item[$6^\circ$] The determination of the fundamental equations of the Riemannian mechanical systems whose external forces depend on the higher order accelerations.
\end{itemize}

\newpage

\mainmatter
\begin{partbacktext}
\part{Lagrange and Hamilton Spaces}

\noindent The purpose of the part I is a short presentation of the geometrical theory of Lagrange space and of Hamilton spaces. These spaces are basic for applications to Mechanics, Physics etc. and have been introduced by the present author in 1980 and 1987, respectively.

Therefore we present here the general framework of Lagrange and Hamilton geometries based on the books by R. Miron \cite{mir11}, \cite{mir12}, \cite{mir13}, R. Miron and M. Anastasiei \cite{mirana}, I. Buc\u{a}taru and R. Miron \cite{BuMi} and R. Miron, D. Hrimiuc, H. Shimada and S. Sab\u{a}u \cite{MHS}. Also, we use the papers R. Miron, M. Anastasiei and I. Buc\u{a}taru, {\it The geometry of Lagrange spaces}, Handbook of Finsler Geometry, P. L. Antonelli ed., Kluwer Academic; R. Miron, {\it Compendium on the Geometry of Lagrange spaces}, Handbook of Differential Geometry, vol. II, pp 438--512, Edited by F.~J.~E. Dillen and L.~C.~A. Verstraelen, 2006.

The geometry of Lagrange space starts here with the study of geometrical theory of tangent bundle $(TM,\pi,M)$. It is continued with the notion of Lagrange space $L^n=(M,L(x,y))$ and with an important particular case the Finsler space $F^n=(M,F)$. The Lagrangian geometry is the study of the sequence of inclusions (I) from the Preface.

The geometry of Hamilton spaces follow the same pattern: the geometry of cotangent bundle $(T^*M,\pi^*,M)$, it continues with the notion of Hamilton space $H^n=(M,H(x,p))$, with the concept of Cartan space ${\cal C}=(M,K(x,p))$ and ends with the sequence of inclusions (II) from Introduction.

The relation between the previous sequences are given by means of the Legendre transformation.

In the following, we assume that all the geometrical object fields and mappings are $C^\infty-$differentiable and we express this by words ``differentiable'' or ``smooth''.

\end{partbacktext} 

\newpage

\chapter{The Geometry of tangent manifold}\index{Tangent manifold}

The total space $TM$ of tangent bundle $(TM,\pi,M)$ carries some natural object fields as Liouville vector field $\C$, tangent structure $J$ and vertical distribution $V$. An important object field is the semispray $S$ defined as a vector field $S$ on the manifold $TM$ with the property $J(S)=\C$. One can develop a consistent geometry of the pair $(TM,S)$.

\section{The manifold TM}\index{Bundles! Tangent $TM$}
\label{s1c1p1}

The differentiable structure on $TM$ is induced by that of the
base manifold $M$ so that the natural projection
$\pi:TM\rightarrow M$ is a differentiable submersion and the
triple $(TM,\pi,M)$ is a differentiable vector bundle. Assuming
that $M$ is a real, $n-$dimensional differentiable manifold and
$(U,\varphi=(x^i))$ is a local chart at a point $x\in M$, then any
curve $\sigma:I\rightarrow M$, (Im$\sigma\subset U$) that passes
through $x$ at $t=0$ is analytically represented by $x^i=x^i(t)$,
$t\in I$, $\varphi(x)=(x^i(0))$, $(i,j,...=1,...,n)$. The tangent
vector $[\sigma]_x$ is determined by the real numbers
$$
x^{i}=x^{i}(0),\ y^{i}=\dfrac{dx^{i}}{dt}(0).
$$
Then the pair $(\pi^{-1}(U),\Phi)$, with $\Phi ([\sigma
]_{x})=(x^{i},y^{i})\in \R^{2n}$, $\forall [\sigma]_{x}\in
\pi^{-1}(U)$ is a local chart on $TM$. It will be denoted by
$(\pi^{-1}(U),\phi=(x^i,y^i))$. The set of these ``induced'' local
charts determines a differentiable structure on $TM$ such that
$\pi:TM\rightarrow M$ is a differentiable manifold of dimension
$2n$ and $(TM,\pi, M)$ is a differentiable vector bundle.

A change of coordinate on $M$, $(U,\varphi=x^i)\rightarrow
(V,\psi=\tilde{x}^i)$ given by $\tilde{x}^i=\tilde{x}^i(x^j)$, with rank$\left(\dfrac{\partial
\tilde{x}^{i}}{\partial x^{j}}\right)=n$ has the corresponding
change of coordinates on $TM:(\pi^{-1}(U)$, $\Phi
=(x^{i},y^{i}))\rightarrow(\pi^{-1}(V),\Psi
=(\tilde{x}^{i},\tilde{y}^{i}))$, $(U\cap V\neq\phi)$, given by:
\begin{equation}
\label{1.1.1}
\left\{
\begin{array}{l}
\tilde{x}^{i}=\tilde{x}^{i}(x^{j}),\ \ {\rm rank}\left(\dfrac{\partial \tilde{x}%
^{i}}{\partial x^{j}}\right)=n,\vspace{3mm} \\
\tilde{y}^{i}=\dfrac{\partial \tilde{x}^{i}}{\partial x^{j}}y^{j}.
\end{array}
\right.
\end{equation}

The Jacobian of $\Psi_{\circ}\Phi^{-1}$ is
det$\left(\dfrac{\partial \tilde{x}^{i}}{\partial
x^{j}}\right)^{2}>0$. So the manifold $TM$ is orientable.

The tangent space $T_u TM$ at a point $u\in TM$ to $TM$ is a
$2n-$di\-men\-sional vector space, having the natural basis
$\left\{\dfrac{\partial}{\partial x^{i}},\dfrac{\partial}{\partial
y^{i}}\right\}$ at $u$. A change of coordinates (\ref{1.1.1}) on $TM$
implies the change of natural basis, at point $u$ as follows:
\begin{equation}
\label{1.1.2}
\left\{
\begin{array}{l}
\dfrac{\partial}{\partial x^{i}}=\dfrac{\partial \tilde{x}^{j}}{
\partial x^{i}}\dfrac{\partial}{\partial \tilde{x}^{j}}+\dfrac{
\partial \tilde{y}^{j}}{\partial x^{i}}\dfrac{\partial}{\partial\tilde{y}^{j}}, \vspace{3mm} \\
\dfrac{\partial}{\partial y^{i}}=\dfrac{\partial \tilde{x}^{j}}{
\partial x^{i}}\dfrac{\partial}{\partial \tilde{y}^{j}}.
\end{array}
\right.
\end{equation}
A vector $X_{u}\in T_{u}TM$ is given by
$X=X^{i}(u)\dfrac{\partial}{\partial
x^{i}}+Y^{i}(u)\dfrac{\partial}{\partial y^{i}}$. Then a vector
field $X$ on $TM$ is section $X:TM\rightarrow TTM$ of the projection $\pi_*:
TTM\rightarrow TM$. This is $\pi_*(x,y,X,Y)=(x,y)$.

From the last formula (\ref{1.1.2}) we can see that
$\left(\dfrac{\partial}{\partial y^i}\right)$ at point $u\in TM$
span an $n-$dimensional vector subspace $V(u)$ of $T_u TM$. The mapping $V:u\in TM\rightarrow
V(u)\subset T_u TM$ is an integrable distribution called the
vertical distribution. Then $VTM=\displaystyle\bigcup_{u\in
TM}V(u)$ is a subbundle of the tangent bundle $(TTM,\pi_*(u),TM)$ to
$TM$. As $\pi:TM\rightarrow M$ is a submersion it follows that
$\pi_{*,u}:T_u TM\rightarrow T_{\pi(u)}M$ is an epimorphism of
linear spaces. The kernel of $\pi_{*,u}$ is exactly the vertical
subspaces $V(u)$.

We denote by ${\cal X}^v(TM)$ the set of all vertical vector field
on $TM$. It is a real subalgebra of Lie algebra of vector fields
on $TM$, ${\cal X}(TM)$.

Consider the cotangent space $T^*_u TM$, $u\in TM$. It is the dual
of space $T_u TM$ and $(dx^i,dy^i)_u$ is the natural cobasis and with
respect to (\ref{1.1.1}) we have
\begin{equation}
\label{1.1.3}
d\tilde{x}^{i}=\dfrac{\partial\tilde{x}^{i}}{\partial
x^{j}}dx^{j}, d\tilde{y}^{i}=\dfrac{\partial
\tilde{x}^{i}}{\partial x^{j}}dy^{j}+\dfrac{
\partial\tilde{y}^{i}}{\partial x^{j}}dx^{j}.
\end{equation}

The almost tangent structure of tangent bundle is defined as
\begin{equation}
\label{1.1.4}
J=\dfrac{\partial}{\partial y^{i}}\otimes dx^{i}
\end{equation}

By means of (\ref{1.1.2}) and (\ref{1.1.3}) we can prove that $J$ is globally
defined on $TM$ and that we have
\begin{equation}
\label{1.1.4'}
J\left(\dfrac{\partial}{\partial
x^{i}}\right)=\dfrac{\partial}{\partial y^{i}}, \
J\left(\dfrac{\partial}{\partial y^{i}}\right)=0.\tag{1.1.4'}
\end{equation}

It follows that the following formulae hold: $$J^{2}=J\circ J=0,
Ker J=Im J=VTM.$$

The almost cotangent structure $J^*$ is defined by
$$J^*=dx^i\otimes\dfrac{\partial}{\partial y^i}.$$ Therefore, we
obtain $$J^*(dx^i)=0, J^*(dy^i)=dx^i.$$

The Liouville vector field on $\widetilde{TM}=TM\setminus\{0\}$ is
defined by
\begin{equation}
\label{1.1.5}
\C=y^i\dfrac{\partial}{\partial y^i}.
\end{equation}
It is globally defined on $\widetilde{TM}$ and $\C\neq 0$.

A smooth function $f:TM\rightarrow R$ is called $r\in Z$
homogeneous with respect to the variables $y^i$ if,
$f(x,ay)=a^r f(x,y)$, $\forall a\in \R^{+}$. The Euler
theorem holds: {\it A function $f\in {\cal F}(TM)$ differentiable on
$\widetilde{TM}$ is $r$ homogeneous with respect to $y^i$ if and only if}
\begin{equation}
\label{1.1.6}
{\cal L}_{\C}f=\C f=y^{i}\dfrac{\partial
f}{\partial y^{i}}=rf,
\end{equation}
${\cal L}_{\C}$ being the Lie derivation with respect to
$\C$.

A vector field $X\in {\cal X}(TM)$ is $r-$homogeneous with respect
to $y^i$ if ${\cal L}_{\C}X=(r-1)X$, where ${\cal
L}_{\C}X=[\C,X]$.

Finally an 1-form $\omega\in{\cal X}^*(TM)$ is $r-$homogeneous if
${\cal L}_{\C}\omega=r\omega$.

Evidently, the notion of homogeneity can be extended to a tensor
field $T$ of type $(r,s)$ on the manifold $TM$.

\section{Semisprays on the manifold $TM$}\index{Semisprays! on manifold $TM$}
\label{s2c1p1}
\setcounter{equation}{0}
\setcounter{definition}{0}
\setcounter{theorem}{0}
\setcounter{proposition}{0}

The notion of semispray on the total space $TM$ of the tangent
bundle is strongly related with the second order differential
equations on the base manifold $M$:
\begin{equation}
\label{1.2.1}
\dfrac{d^{2}x^{i}}{dt^{2}}+2G^{i}\left(x,\dfrac{dx}{dt}\right)=0.
\end{equation}

Writing the equation (\ref{1.2.1}), on $TM$, in the equivalent form
\begin{equation}
\label{1.2.2}
\frac{dy^{i}}{dt^{2}}+2G^{i}(x,y)=0, y^i=\dfrac{dx^i}{dt}
\end{equation}
we remark that with respect to the changing of coordinates (\ref{1.1.1})
on $TM$, the functions $G^i(x,y)$ transform according to:
\begin{equation}
\label{1.2.3}
2\widetilde{G^{i}}=\dfrac{\partial \widetilde{x}^{i}}{\partial
x^{j}}2G^{j}- \dfrac{\partial \widetilde{y}^{i}}{\partial
x^{j}}y^{j}.
\end{equation}

But (\ref{1.2.2}) are the integral curve of the vector field:
\begin{equation}
\label{1.2.4}
S=y^{i}\dfrac{\partial }{
\partial x^{i}}-2G^{i}(x,y)\dfrac{\partial }{\partial y^{i}}
\end{equation}

By means of (\ref{1.2.3}) one proves: $S$ is a vector field globally
defined on $TM$. It is called a semispray on $TM$ and $G^i$ are
the coefficients of $S$.

$S$ is homogeneous of degree 2 if and only if its coefficients
$G^i$ are homogeneous functions of degree 2. If $S$ is
2-homogeneous then we say that $S$ is a spray.

If the base manifold $M$ is paracompact, then on $TM$ there exist
semisprays.

\section{Nonlinear connections}\index{Connections! Nonlinear}
\label{s3c1p1}
\setcounter{equation}{0}
\setcounter{definition}{0}
\setcounter{theorem}{0}
\setcounter{proposition}{0}

As we have seen in the first section of this chapter, the vertical
distribution $V$ is a regular, $n$-dimensional, integrable
distribution on $TM$. Then it is naturally to look for a
complementary distribution of $VTM$. It will be called a
horizontal distribution. Such distribution is equivalent with
a nonlinear connection.

Consider the tangent bundle $(TM,\pi,M)$ of the base manifold $M$
and the tangent bundle $(TTM, \pi_{\ast}, TM)$ of the manifold
$TM$. As we know the kernel of $\pi_{\ast}$ is the vertical
subbundle $(VTM,\pi_V,TM)$. Its fibres are the vertical spaces
$V(u)$, $u\in TM$.

For a vector field $X\in{\cal X}(TM)$, in local natural basis, we
can write: $$X=X^i(x,y)\dfrac{\partial}{\partial
x^{i}}+Y^i(x,y)\dfrac{\partial}{\partial y^{i}}.$$ Shorter
$X=(x^i,y^i,X^i,Y^i)$.

The mapping $\pi_*:TTM\rightarrow TM$ has the local form
$\pi_*(x,y,X,Y)=(x,X)$. The points of the submanifold $VTM$ are of
the form $(x,y,0,Y)$.

Let us consider the pull-back bundle
$$
\pi^{\ast}(TM)=TM{\times}_{\pi}TM=\{(u,v)\in TM\times TM|\pi
(u)=\pi(v)\}.
$$

The fibres of $\pi^{\ast}(TM)$, i.e., $\pi_{u}^{\ast}(TM)$ are
isomorphic to $T_{\pi(u)}M$. Then, we can define the following
morphism $\pi!:TTM\longrightarrow \pi^{\ast}(TM)$ by $\pi
!(X_{u})=(u,\pi_{\ast,u}(X_{u}))$. Therefore we have
$$
{\rm ker}\,\pi !={\rm ker}\,\pi_{\ast }=VTM.
$$
We can prove, without difficulties that the following sequence of vector bundles over $TM$ is
exact:
\begin{equation}
\label{1.3.1}
0\longrightarrow VTM{\stackrel{i}{\longrightarrow
}}TTM{\stackrel{\pi !}{ \longrightarrow }}\pi^{\ast
}(TM)\longrightarrow 0
\end{equation}

Thus, we can give
\begin{definition}
A nonlinear connection on the tangent
manifold $TM$ is a left splitting of the exact sequence (\ref{1.3.1}).

Consequently, a nonlinear connection on $TM$ is a vector bundle
morphism $C:TTM\to VTM$, with the property that $C\circ
i=1_{VTM}$.
\end{definition}

The kernel of the morphism $C$ is a vector subbundle of the
tangent bundle $(TTM,\pi_{\ast},TM)$ denoted by $(NTM,\pi_{N},TM)$
and called the {\it horizontal} subbundle. Its fibres $N(u)$
determine a regular $n$-dimensional distribution $N:u\in
TM\rightarrow N(u)\subset T_{u}TM$, complementary to the vertical
distribution $V:u\in TM\rightarrow V(u)\subset T_{u}TM$ i.e.
\begin{equation}
\label{1.3.2}
T_u TM=N(u)\oplus V(u), \forall u\in TM.
\end{equation}
Therefore, a nonlinear connection on $TM$ induces the following
Whitney sum:
\begin{equation}
TTM=NTM\oplus VTM.  \tag{1.3.2'}
\end{equation}
The reciprocal of this property is true, too.

\noindent An adapted local basis to the direct decomposition (\ref{1.3.2}) is of the
form $\left(\dfrac{\delta}{\delta x^{i}},\dfrac{\partial}{\partial
y^{j}}\right)_u$, where
\begin{equation}
\label{1.3.3}
\dfrac{\delta}{\delta x^{i}}=\dfrac{\partial}{\partial
x^{i}}-N^j_{\ i}(x,y)\dfrac{\partial}{\partial y^{j}}
\end{equation}
and $\left.\dfrac{\delta }{\delta x^{i}}\right|_u$, $(i=1,...,n)$ are vector
fields that belong to $N(u)$.

They are $n-$linear independent vector fields and are independent
from the vector fields $\left(\dfrac{\partial }{\partial
y^{i}}\right)_u$ $(i=1,...,n)$ which belong to $V(u)$.

The functions $N^j_i(x,y)$ are called the coefficients of the
nonlinear connection, denoted in the following by $N$.

Remarking that $\pi_{*,u}:T_u TM\to T_{\pi(u)}M$ is an epimorphism
the restriction of $\pi_{*,u}$ to $N(u)$ is an isomorphism from $N(u)$ so
$T_{\pi(u)}M$. So we can take the inverse map
$l_{h,u}$ the horizontal lift determined by the nonlinear
connection $N$.

Consequently, the vector fields $\left.\dfrac{\delta}{\delta x^{i}}\right|_u$
can be given in the form $$\left(\dfrac{\delta }{\delta
x^{i}}\right)_u=l_{h,u}\left(\dfrac{\partial}{\partial
x^{i}}\right)_{\pi(u)}.$$ With respect to a change of local
coordinates on the base manifold $M$ we have
$\dfrac{\partial}{\partial
x^i}=\dfrac{\partial\widetilde{x}^j}{\partial
x^i}\dfrac{\theta}{\partial \widetilde{x}^j}$.

Consequently $\left(\dfrac{\delta}{\delta x^i}\right)_u$ are
changed, with respect to (\ref{1.1.1}), in the form
\begin{equation}
\label{1.3.4}
\dfrac{\delta }{\delta x^{i}}=\dfrac{\partial
\widetilde{x}^{j}}{\partial x^{i}}\dfrac{\delta }{\delta
\widetilde{x}^{j}}.
\end{equation}

It follows, from (\ref{1.3.3}), that the coefficients $N^i_j(x,y)$ of the
non\-li\-near connection $N$, with respect to a change of local
coordinates on the manifold $TM$, (\ref{1.1.1}) are transformed by the
rule:
\begin{equation}
\label{1.3.5}
\dfrac{\partial \widetilde{x}^{j}}{\partial
x^{k}}N_{i}^{k}=\widetilde{N}_{k}^{j} \dfrac{\partial
\widetilde{x}^{k}}{\partial x^{i}}+\dfrac{\partial \widetilde{y}^{j}}{
\partial x^{i}}.
\end{equation}
The reciprocal property is true, too.

We can prove without difficulties that there exists a
nonlinear connection on $TM$ if $M$ is a paracompact manifold. The
validity of this sentence is assured by the following theorem:

\begin{theorem}
\label{t3.1}
If $S$ is a semispray with the coefficients $G^i(x,y)$ then the system of functions
\begin{equation}
\label{1.3.6}
N^i_j(x,y)=\dfrac{\partial G^{i}}{\partial y^{j}}
\end{equation}
is the system of coefficients of a nonlinear connection $N$.
\end{theorem}

Indeed, the formula (\ref{1.2.3}) and $\dfrac{\partial}{\partial
y^{i}}=\dfrac{\partial\widetilde{x}^{j}}{\partial
x^{i}}\dfrac{\partial}{\partial \widetilde{y}^{j}}$ give the rule
of transformation (\ref{1.3.5}) for $N^i_j$ from (\ref{1.3.6}).

The adapted dual basis $\{dx^i,\delta y^i\}$ of the basis
$\left(\dfrac{\delta}{\delta x^{i}},\dfrac{\partial }{\partial
y^{i}}\right)$ has the 1-forms $\delta y^i$ as follows:
\begin{equation}
\label{1.3.7}
\delta y^i=dy^i+N^i_j dx^j.
\end{equation}

With respect to a change of coordinates, (\ref{1.1.1}), we have
\begin{equation}
d\widetilde{x}^i=\dfrac{\partial \widetilde x^i}{\partial
x^j}dx^j,\quad \delta\widetilde{y}^i=\dfrac{\partial \widetilde
x^i}{\partial x^j}\delta y^j.\tag{1.3.7'}
\end{equation}

Now, we can consider the horizontal and vertical projectors $h$
and $v$ with respect to direct decomposition (\ref{1.3.2}):
\begin{equation}
\label{1.3.8}
h=\dfrac{\delta}{\delta x^{i}}\otimes dx^{i},\ \ v=\dfrac{\partial
}{\partial y^{i}}\otimes \delta y^{i}.
\end{equation}

Some remarkable geometric structures, as the almost product
$\P$ and almost complex structure $\F$, are determined
by the nonlinear connection $N$:
\begin{equation}
\label{1.3.9}
\P=\dfrac{\delta}{\delta x^{i}}\otimes
dx^{i}-\dfrac{\partial}{\partial y^{i}}\otimes \delta y^{i}=h-v.
\end{equation}
\begin{equation}
\F=\dfrac{\delta}{\delta x^{i}}\otimes \delta
y^{i}-\dfrac{\partial}{\partial y^{i}}\otimes dx^{i}. \tag{1.3.9'}
\end{equation}

It is not difficult to see that $\P$ and $\F$ are
globally defined on $\widetilde{TM}$ and
\begin{equation}
\label{1.3.10}
\P\circ \P=Id,\ \ \F\circ \F=-Id.
\end{equation}

With respect to (\ref{1.3.2}) a vector field $X\in{\cal X}(TM)$ can be
uniquely written in the form
\begin{equation}
\label{1.3.11}
X=hX+vX=X^H+X^V
\end{equation}
with $X^H=hX$ and $X^V=vX$.

An 1-form $\omega\in{\cal X}^*(TM)$ has the similar form
$$\omega=h\omega+v\omega,$$ where $h\omega(X)=\omega(X^H)$,
$v\omega(X)=\omega(X^V)$.

A $d-$tensor field $T$ on $TM$ of type $(r,s)$ is called a
distinguished tensor field (shortly a $d-$tensor) if $$T(\omega_1,
...,\omega_r, X_1,..., X_s)= T(\varepsilon_1\omega_1,
...,\varepsilon_r\omega_r, \varepsilon_1 X_1,..., \varepsilon_s
X_s)$$ where $\varepsilon_1, ..., \varepsilon_r, ...$ are $h$ or
$v$.

Therefore $hX=X^H$, $vX=X^V$, $h\omega=\omega^H$,
$v\omega=\omega^V$ are $d-$vectors or $d-$covectors. In the
adapted basis $\left(\dfrac{\delta}{\delta x^{i}},\dfrac{\partial
}{\partial y^{i}}\right)$ we have
$$X^H=X^i(x,y)\dfrac{\delta}{\delta x^i},\quad X^V=\dot{X}^i\dfrac{\partial}{\partial
y^i}$$ and $$\omega^H=\omega_j(x,y)dx^j,\quad
\omega^V=\dot{\omega}_j\delta y^j.$$

A change of local coordinates on $TM:(x,y)\to (\widetilde{x},\widetilde{y})$ leads to the change of coordinates
of $X^H,X^V,\omega^H,\omega^V$ by the classical rules of
transformation:
$$\widetilde{X}^i=\dfrac{\partial\widetilde{x}^i}{\partial x^j}X^j,\quad \omega_j=\dfrac{\partial \widetilde{x}^i}{\partial
x^j}\widetilde{\omega}_i\ {\rm{etc}}.$$ So, a $d-$tensor $T$ of
type $(r,s)$ can be written as
\begin{equation}
\label{1.3.12}
T=T_{j_1...j_s}^{i_1...i_r}(x,y)\dfrac{\delta}{\delta
x^{i_1}}\otimes...\otimes \dfrac{\partial}{\partial
y^{i_r}}\otimes dx^{j_1}\otimes...\otimes \delta y^{j_s}.
\end{equation}

A change of coordinates (\ref{1.1.1}) implies the classical rule:
\begin{equation}
\widetilde{T}_{j_1...j_s}^{i_1...i_r}(\widetilde{x},\widetilde{y})=\dfrac{\partial
\widetilde{x}^{i_1}}{\partial x^{h_1}}...\dfrac{\partial
\widetilde{x}^{i_r}}{\partial x^{h_r}}\dfrac{\partial
x^{k_1}}{\partial \widetilde{x}^{j_1}}...\dfrac{\partial
x^{k_s}}{\partial \widetilde{y}^{j_s}}=T^{h_1...h_2}_{k_1...k_s}.
\tag{1.3.12'}
\end{equation}

Next, we shall study the integrability of the nonlinear connection
$N$ and of the structures $\P$ and $\F$.

Since $\left(\dfrac{\delta}{\delta x^{i}}\right)$ $i=1,...,n$ is
an adapted basis to $N$, according to the Frobenius theorem it
follows that $N$ is integrable if and only if the Lie brackets
$\left[\dfrac{\delta}{\delta x^{i}},\dfrac{\delta}{\delta
x^{j}}\right]$, $i,j=1,...,n$ are vector fields in the
distribution $N$.

But we have
\begin{equation}
\label{1.3.13}
\left[\dfrac{\delta}{\delta x^{i}},\dfrac{\delta}{\delta
x^{j}}\right]=R^h_{\ ij}\dfrac{\partial}{\partial y^h}
\end{equation}
where
\begin{equation}
R^h_{\ ij}=\dfrac{\delta N^h_i}{\delta x^j}-\dfrac{\delta
N^h_j}{\delta x^i}.\tag{1.3.13'}
\end{equation}

It is not difficult to prove that $R^h_{ij}$ are the components
of a $d-$tensor field of type (1,2). It is called the {\it
curvature} tensor of the nonlinear connection $N$.

We deduce:

\begin{theorem}
\label{t3.2}
The nonlinear connection $N$ is integrable
if and only if its curvature tensor $R^h_{ij}$ vanishes.
\end{theorem}

The weak torsion $t^h_{ij}$ of $N$ is defined by
\begin{equation}
\label{1.3.14}
t^h_{ij}=\dfrac{\partial N^h_{\ i}}{\partial y^i}-\dfrac{\partial
N^h_{\ j}}{\partial y^i}.
\end{equation}
It is a d-tensor field of type $(1,2)$, too. We say that $N$ is a
symmetric nonlinear connection if its weak torsion $t^h_{ij}$
vanishes.

Now is not difficult to prove:

\begin{theorem}
\label{t3.3}
\begin{enumerate} \item[$1^\circ$] The almost product
structure $\P$ is integrable if and only if the nonlinear
connection $N$ is integrable; \item[$2^\circ$] The almost complex
structure $\F$ is integrable if and only if the nonlinear connection $N$ is symmetric and integrable.
\end{enumerate}
\end{theorem}

The proof is simple using the Nijenhuis tensors $N_{\P}$ and
$N_{\F}$. For $N_{\P}$ we have the expression
$$N_{\P}(X,Y)=\P^2[X,Y]+[\P X,\P Y]-\P[\P X,Y]-\P[X,\P Y],\forall
X,Y\in\chi(TM).$$

Also we can see that each structure $\P$ or $\F$
characterizes the nonlinear connection $N$.

\bigskip

{\bf Autoparallel curves of a nonlinear connection} can be
obtained considering the horizontal curves as follows.

A curve $c:t\in I\subset R\to(x^i(t),y^i(t))\in TM$ has the
tangent vector $\dot{c}$ given by
\begin{equation}
\label{1.3.15}
\dot{c}=\dot{c}^H+\dot{c}^V=\dfrac{dx^i}{dt}\dfrac{\delta}{\delta
x^i}+\dfrac{\delta y^i}{dt}\dfrac{\partial}{\partial
y^i}
\end{equation}
where
\begin{equation}
\dfrac{\delta
y^i}{dt}=\dfrac{dy^i}{dt}+N^i_j(x,y)\dfrac{dx^j}{dt}.\tag{1.3.15'}
\end{equation}

The curve $c$ is a {\it horizontal curve} if $\dot{c}^V=0$ or
$\dfrac{\delta y^i}{dt}=0$.

Evidently, if the functions $x^i(t)$, $t\in I$ are given and
$y^i(t)$ are the solutions of this system of differential
equations, then we have an horizontal curve $x^i=x^i(t)$,
$y^i=y^i(t)$ in $TM$ with respect to $N$.

In the case $y^i=\dfrac{dx^i}{dt}$, the horizontal curves are
called the autoparallel curves of the nonlinear connection $N$.
They are characterized by the system of differential equations
\begin{equation}
\label{1.3.16}
\dfrac{d y^{i}}{dt}+N^i_j(x,y)\dfrac{dx^{j}}{dt}=0,\quad
y^i=\dfrac{dx^i}{dt}.
\end{equation}

A theorem of existence and uniqueness can be easy formulated for
the autoparallel curves of a nonlinear connection $N$ given by its
coefficients $N^i_j(x,y)$.

\section{$N$-linear connections}\index{Connections! $N-$linear}
\label{s4c1p1}
\setcounter{equation}{0}
\setcounter{definition}{0}
\setcounter{theorem}{0}
\setcounter{proposition}{0}

An $N$-linear connection on the manifold $TM$ is a special linear connection $D$ on $%
TM $ that preserves by parallelism the horizontal distribution $N$
and the vertical distribution $V$. We study such linear
connections determining the curvature, torsion and structure
equations.

Throughout this section $N$ is an a priori given nonlinear
connection with the coefficients $N^i_j$.

\begin{definition}
\label{d4.1}
A linear connection $D$ on the manifold $TM$
is called a {\it distinguished connection} (a d-connection for
short) if it preserves by parallelism the horizontal distribution
$N$.
\end{definition}

Thus, we have $Dh=0$. It follows that we have also: $Dv=0$ and
$D\P=0$.

If $Y=Y^H+Y^V$ we get $$D_X Y=(D_X Y^H)^H+(D_X Y^V)^V,\quad
\forall X,Y\in \chi(TM).$$ We can easily prove:

\begin{proposition}
\label{p4.1}
For a d-connection the following conditions are equivalent:

$1^{{\rm o}}$ $DJ=0$;

$2^{{\rm o}}$ $D\F=0$.
\end{proposition}

\begin{definition}
\label{d4.2}
A d-connection $D$ is called an {\it
$N$-linear connection} if the structure $J$ (or $\F$) is
absolute parallel with respect to $D$, i.e. $DJ=0$.
\end{definition}

In an adapted basis an $N$-linear connection has the form:
\begin{equation}
\label{1.4.1}
\left\{\begin{array}{ll}D_{\frac{\delta}{\delta
x^{j}}}\dfrac{\delta}{\delta
x^{i}}=L_{ij}^{h}\dfrac{\delta}{\delta x^{h}}; &
D_{\frac{\delta}{\delta
x^{j}}}\dfrac{\partial}{\partial y^{i}}= L_{ij}^{h}\dfrac{\partial}{\partial y^{h}}, \vspace{3mm} \\
D_{\frac{\partial}{\partial y^{j}}}\dfrac{\delta}{\delta
x^{i}}=C_{ij}^{h}\dfrac{\delta}{\delta x^{h}}; &
D_{\frac{\partial}{
\partial y^{j}}}\dfrac{\partial}{\partial y^{i}}=C_{ij}^{h}\dfrac{\partial}{\partial y^{h}}.
\end{array}\right.
\end{equation}

The set of functions $D\Gamma =(N_{j}^{i}(x,y), L_{ij}^{h}(x,y),
C_{ij}^{h}(x,y))$ are called the local coefficients of the
$N-$linear connection $D$. Since $N$ is fixed we denote sometimes
by $D\Gamma(N)=(L^h_{ij}(x,y),C^h_{ij}(x,y))$ the coefficients of
$D$.

For instance the $B\Gamma(N)=\left(\dfrac{\partial
N^h_{i}}{\partial y^j},0\right)$ are the coefficient of a special
$N-$linear connection, derived only from the nonlinear connection
$N$. It is called the {\bf Berwald connection} of the nonlinear
connection $N$. We can prove this showing that the
system of functions $\left(\dfrac{\partial N^h_{i}}{\partial
y^j}\right)$ has the same rule of transformation, with respect to
(\ref{1.1.1}), as the coefficients $L^h_{ij}$. Indeed, under a change of
coordinates (\ref{1.1.1}) on $TM$ the coefficients $(L^h_{ij},C^h_{ij})$
are transformed by the rules:
\begin{equation}
\label{1.4.2}
\left\{
\begin{array}{l}
\widetilde{L}_{ij}^{h}=\dfrac{\partial \widetilde{x}^{h}}{\partial
x^{s}} L_{pq}^{s}\dfrac{\partial x^{p}}{\partial
\widetilde{x}^{i}}\dfrac{\partial x^{q} }{\partial
\widetilde{x}^{j}}-\dfrac{\partial^{2}\widetilde{x}^{h}}{\partial x^{p}{\
\partial x^{q}}}\dfrac{\partial x^{p}}{\partial
\widetilde{x}^{i}}\dfrac{
\partial x^{q}}{\partial \widetilde{x}^{j}},\vspace{3mm} \\
\widetilde{C}_{ij}^{h}=\dfrac{\partial \widetilde{x}^{h}}{\partial
x^{s}} C_{pq}^{s}\dfrac{\partial x^{p}}{\partial
\widetilde{x}^{i}}\dfrac{\partial x^{q} }{\partial \widetilde{x}^{j}}.
\end{array}
\right.
\end{equation}
So, the horizontal coefficients $L_{ij}^{h}$ of $D$ have the same
rule of transformation of the local coefficients of a linear connection on the base manifold $M$. The vertical coefficients $C_{ij}^{h}$ are the components of a (1,2)-type d-tensor field.

But, conversely, if the set of functions
$(L^i_{jk}(x,y),C^i_{jk}(x,y))$ are given, having the property
(\ref{1.4.2}), then the equalities (\ref{1.4.1}) determine an $N-$linear
connection $D$ on $TM$.

For an $N-$linear connection $D$ on $TM$ we shall associate two
operators of $h-$ and $v-$ covariant derivation on the algebra of
$d-$tensor fields. For each $X\in\chi(TM)$ we set:
\begin{equation}
\label{1.4.3}
D_{X}^HY=D_{X^H}Y,\quad D^H_X f=X^H f,\quad \forall
Y\in\chi(TM),\forall f\in{\cal F}(TM).
\end{equation}

If $\omega \in \chi^{*}(TM)$, we obtain
\begin{equation}
(D_{X}^{H}\omega )(Y)=X^H(\omega(Y))-\omega (D_{X}^{H}Y).
\tag{1.4.3'}
\end{equation}
So we may extend the action of the operator $D_{X}^{H}$ to any
d-tensor field by asking that $D_{X}^{H}$ preserves the type of d-tensor fields, is $\R$-linear, satisfies the Leibniz rule with respect to
tensor product and commutes with the operation of contraction.
$D_{X}^{H}$ will be called the $h${\it-covariant
derivation} operator.

In a similar way, we set:
\begin{equation}
\label{1.4.4}
D_{X}^{V}Y=D_{X^V}Y,\ D_{X}^{V}f=X^V (f),\ \forall Y\in \chi
(TM),\ \forall f\in {\cal F}(TM).
\end{equation}
and $$D_{X}^{V}\omega=X^V(\omega (Y))-\omega(D_{X}^{V}Y),\ \forall
\omega\in \chi (TM).$$

Also, we extend the action of the operator $D_{X}^{V}$ to any
d-tensor field in a similar way as we did for $D_{X}^{H}$.
$D_{X}^{V}$ is called the $v${\it-covariant of
derivation operator}.

Consider now a d-tensor $T$ given by (\ref{1.3.12}). According to (\ref{1.4.1})
its $h-$covariant derivation is given by
\begin{equation}
\label{1.4.5}
D_{X}^{H}T=X^{k}T_{j_{1}\cdots j_{s}|k}^{i_{1}\cdots
i_{r}}\dfrac{\delta}{\delta x^{i_{1}}}\otimes \cdots \otimes
\dfrac{\partial}{\partial y^{i_{r}}}\otimes
dx^{j_{1}}\otimes\cdots\otimes dx^{j_{s}}
\end{equation}
where
\begin{equation}
\begin{array}{ll}
T_{j_{1}\cdots j_{s}|k}^{i_{1}\cdots i_{r}}= & \dfrac{\delta
T_{j_{1}\cdots j_{s}}^{i_{1}\cdots i_{r}}}{\delta
x^{k}}+L_{pk}^{i_{1}}T_{j_{1}\cdots j_{s}}^{pi_{2}\cdots
i_{r}}+\cdots +L_{pk}^{i_{r}}T_{j_{1}\cdots
j_{s}}^{i_{1}\cdots i_{r-1}p}-\vspace{3mm} \\
& -L_{j_{1}k}^{p}T_{pj_{2}\cdots j_{s}}^{i_{1}\cdots i_{r}}-\cdots
-L_{j_{s}k}^{p}T_{j_{1}\cdots j_{s-1}p}^{i_{1}\cdots i_{r}}.
\end{array}
\tag{1.4.5'}
\end{equation}
The $v$-covariant derivative $D_X^V T$ is as follows:
\begin{equation}
\label{1.4.6}
D_{X}^{V}T=X^{k}T_{j_{1}\cdots j_{s}}^{i_{1}\cdots
i_{r}}|_{k}\dfrac{\delta}{\delta x^{i_{1}}}\otimes \cdots \otimes
\dfrac{\partial}{\partial y^{i_{r}}}\otimes
dx^{j_{1}}\otimes\cdots\otimes dx^{j_{s}},
\end{equation}
where
\begin{equation}
\begin{array}{ll}
T_{j_{1}\cdots j_{s}}^{i_{1}\cdots i_{r}}|_{k}= & \dfrac{\partial
T_{j_{1}\cdots j_{s}}^{i_{1}\cdots i_{r}}}{\partial y^{k}}
+C_{pk}^{i_{1}}T_{j_{1}\cdots j_{s}}^{pi_{2}\cdots i_{r}}+\cdots
+C_{pk}^{i_{r}}T_{j_{1}\cdots j_{s}}^{i_{1}\cdots i_{r-1}p}-\vspace{3mm} \\
& -C_{j_{1}k}^{p}T_{pj_{2}\cdots j_{s}}^{i_{1}\cdots i_{r}}-\cdots
-C^{p}_{j_r k}T_{j_{1}\cdots j_{s-1}p}^{i_{1}\cdots i_{r}}.
\end{array}
\tag{1.4.6'}
\end{equation}

For instance, if $g$ is a d-tensor of type (0,2) having the
components $g_{ij}(x,y)$, we have
\begin{equation}
\label{1.4.7}
\begin{array}{ll}
g_{ij|k}=\dfrac{\delta g_{ij}}{\delta x^k}-L^p_{ik}g_{pj}-L^p_{jk}g_{ip},\vspace{3mm}\\
g_{ij}|k=\dfrac{\partial g_{ij}}{\partial
y^k}-C^p_{ik}g_{pj}-C^p_{jk}g_{ip}.
\end{array}
\end{equation}
For the operators ``$_{|}$'' and ``$|$'' we preserve the same
denomination of $h-$ and $v-$covariant derivations.

The torsion $T$ of an $N$-linear is given by:
\begin{equation}
\label{1.4.8}
T(X,Y)=D_{X}Y-D_{Y}X-[X,Y],\ \forall X,Y\in \chi(TM).
\end{equation}

The horizontal part $hT(X,Y)$ and the vertical one $vT(X,Y)$, for
$X\in\{X^H,X^V\}$ and $Y\in\{Y^H,Y^V\}$ determine five d-tensor
fields of torsion $T$:
\begin{equation}
\label{1.4.9}
\left\{
\begin{array}{lr}
hT(X^H,Y^H)=D_{X}^{H}Y^H-D_{Y}^{H}X^H-[X^H,Y^H]^H,\vspace{2mm}\\
vT(X^H,Y^H)=-v[X^H,Y^H]^V,\vspace{2mm}\\
hT(X^H,Y^V)=-D_{Y}^{V}X^H-[X^H,Y^V]^V,\vspace{2mm}\\
vT(X^H,Y^V)=D_{X}^{H}Y^V-[X^H,Y^V]^V,\vspace{2mm}\\
vT(X^V,Y^V)=D_{X}^{V}Y^V-D_{Y}^{V}X^V-[X^V,Y^V]^V.
\end{array}
\right.
\end{equation}

With respect to the adapted basis, the components of torsion are
as follows:
\begin{equation}
\begin{array}{llll}
hT\left(\dfrac{\delta}{\delta x^{i}},\dfrac{\delta}{\delta
x^{j}}\right)=T_{ji}^{k} \dfrac{\delta}{\delta x^{k}};
vT\left(\dfrac{\delta}{\delta x^{i}},\dfrac{\delta}{\delta
x^{j}}\right)=R_{ji}^{k}\dfrac{\partial}{\partial y^{k}};
\vspace{3mm} \\
hT\left(\dfrac{\partial}{\partial y^{i}},\dfrac{\delta}{\delta
x^{j}}\right)=C_{ji}^{k}\dfrac{\delta}{\delta x^{k}};\vspace{3mm} \\
vT\left(\dfrac{\partial}{\partial y^{i}},\dfrac{\delta}{\delta x^{j}}\right)=P_{ji}^{k}\dfrac{\partial}{\partial y^{k}};\vspace{3mm} \\
vT\left(\dfrac{\partial}{\partial y^{i}},\dfrac{\partial}{\partial
y^{j}}\right)=S_{ji}^{k}\dfrac{\partial}{\partial y^{k}},
\end{array}
\tag{1.4.9'}
\end{equation}
where $C_{jk}^i$ are the $v-$coefficients of $D$, $R^i_{jk}$ is
the curvature tensor of the nonlinear connection $N$ and
\begin{equation}
\label{1.4.10}
T^i_{\ jk}=L^i_{jk}-L^i_{kj}, S^i_{\ jk}=C^i_{jk}-C^i_{kj}, P^i_{\
jk}=\dfrac{\partial N^i_j}{\partial y^k}-L^i_{kj}.
\end{equation}

The $N-$linear connection $D$ is said to be symmetric if $T^i_{\
jk}=S^i_{\ jk}=0$.

Next, we study the curvature of an $N$-linear connection $D$:
\begin{equation}
\label{1.4.11}
R(X,Y)Z\hspace*{-0.5mm}=\hspace*{-0.5mm}D_{X}D_{Y}Z-D_{Y}D_{X}Z\hspace*{-0.5mm}-\hspace*{-0.5mm}D_{[X,Y]}Z,\ \forall X,Y,Z\in \chi
(TM).
\end{equation}
As $D$ preserves by parallelism the distribution $N$ and $V$ it
follows that the operator $R(X,Y)=D_X D_Y-D_Y D_X-D_{[X,Y]}$
carries the horizontal vector fields $Z^H$ into horizontal vector
fields and vertical vector fields into verticals. Consequently we
have the formula:
$$R(X,Y)Z=hR(X,Y)Z^H+vR(X,Y)Z^V,\ \forall X,Y,Z\in \chi (TM).$$

Since $R(X,Y)=-R(Y,X)$, we obtain:

\begin{theorem}
\label{t4.1}
The curvature $R$ of a $N$-linear connection $D$ on the tangent manifold $TM$ is completely
determined by the following six d-tensor fields:
\begin{equation}
\label{1.4.12}
\left\{
\begin{array}{l}
R(X^H,Y^H)Z^H=D_{X}^{H}D_{Y}^{H}Z^H-D_{Y}^{H}D_{X}^{H}Z^H-D_{[X^H,Y^H]}Z^H,\vspace{
3mm} \\
R(X^H,Y^H)Z^V=D_{X}^{H}D_{Y}^{H}Z^V-D_{Y}^{H}D_{X}^{H}Z^V-D_{[X^H,Y^H]}Z^V,\vspace{
3mm} \\
R(X^V,Y^H)Z^H=D_{X}^{V}D_{Y}^{H}Z^H-D_{Y}^{H}D_{X}^{V}Z^H-D_{[X^V,Y^H]}Z^H,\vspace{
3mm} \\
R(X^V,Y^H)Z^V=D_{X}^{V}D_{Y}^{H}Z^V-D_{Y}^{H}D_{X}^{V}Z^V-D_{[X^V,Y^H]}Z^V,\vspace{
3mm} \\
R(X^V,Y^V)Z^H=D_{X}^{V}D_{Y}^{V}Z^H-D_{Y}^{V}D_{X}^{V}Z^H-D_{[X^V,Y^V]}Z^H,\vspace{
3mm} \\
R(X^V,Y^V)Z^V=D_{X}^{V}D_{Y}^{V}Z^V-D_{Y}^{V}D_{X}^{V}Z^V-D_{[X^V,Y^V]}Z^V.
\end{array}
\right.
\end{equation}
As the tangent structure $J$ is absolute parallel with respect to
$D$ we have that $$JR(X,Y)Z=R(X,Y)JZ.$$
\end{theorem}

Then $R(X,Y)Z$ has only three components
$$R(X^H,Y^H)Z^H,R(X^V,Y^H)Z^H, R(X^V,Y^V)Z^H.$$ In the adapted
basis these are:
\begin{equation}
\label{1.4.13}
\begin{array}{l}
R\left(\dfrac{\delta}{\delta x^{k}},\dfrac{\delta}{\delta
x^{j}}\right)\dfrac{\delta}{\delta x^{h}}=R_{h\ kj}^{\
i}\dfrac{\delta}{\delta x^{i}};\vspace{3mm}
\\
R\left(\dfrac{\partial}{\partial y^{k}},\dfrac{\delta}{\delta
x^{j}}\right)\dfrac{\delta}{\delta x^{h}}=P_{h\ kj}^{\
i}\dfrac{\delta}{\delta x^{i}};\vspace{
3mm} \\
R\left(\dfrac{\partial}{\partial y^{k}},\dfrac{\partial}{\partial
y^{j}}\right)\dfrac{\delta}{\delta x^{h}}=S_{h\ kj}^{\
i}\dfrac{\delta}{\delta x^{i}}.
\end{array}
\end{equation}

The other three components are obtained by applying
the operator $J$ to the previous ones. So, we have $R\left(\dfrac{\delta}{\delta
x^h},\dfrac{\delta}{\delta x^j}\right)\dfrac{\partial}{\partial
y^h}=R_{h \ jk}^{\ i}\dfrac{\partial}{\partial y^h}$ etc.
Therefore, the curvature of an $N-$linear connection
$D\Gamma=(N^i_j,L^i_{jk},C^i_{jk})$ has only three local
components $R^{\ i}_{h \ jk}$, $P^{\ i}_{h \ jk}$ and $S^{\ i}_{h
\ jk}$. They are given by
\begin{equation}
\label{1.4.14}
\begin{array}{l}
R_{h\ jk}^{\ i}=\dfrac{\delta L_{hj}^{i}}{\delta
x^{k}}-\dfrac{\delta L_{hk}^{i}}{\delta x^{j}}
+L_{hj}^{s}L_{sk}^{i}-L_{hk}^{s}L_{sj}^{i}+C_{hs}^{i}R_{jk}^{s};\vspace{3mm}
\\
P_{h\ jk}^{\ i}=\dfrac{\partial L_{hj}^{i}}{\partial y^{k}}
-C_{hk|j}^{i}+C_{hs}^{i}P_{jk}^{s};\vspace{3mm} \\
S_{h\ jk}^{\ i}=\dfrac{\partial C_{hj}^{i}}{\partial
y^{k}}-\dfrac{\partial C_{hk}^{i}}{\partial
y^{j}}+C_{hj}^{s}C_{sk}^{i}-C_{hk}^{s}C_{sj}^{i}.
\end{array}
\end{equation}

If $X^{i}(x,y)$ are the components of a $d$-vector field on $TM$
then from (\ref{1.4.12}) we may derive the Ricci identities for $X^{i}$
with respect to an $N$-linear connection $D$. They are:
\begin{equation}
\label{1.4.15}
\begin{array}{l}
X_{|j|k}^{i}-X_{|k|j}^{i}=X^{s}R_{s\ jk}^{\
i}-X_{|s}^{i}T_{jk}^{s}-X^{i}|_{s}R_{\ jk}^{s},\vspace{3mm} \\
X_{|j}^{i}|_{k}-X^{i}|_{k|j}=X^{s}P_{s\ jk}^{\
i}-X_{|s}^{i}C_{jk}^{s}-X^{i}|_{s}P_{\ jk}^{s},\vspace{3mm} \\
X^{i}|_{j}|_{k}-X^{i}|_{k}|_{j}=X^{s}S_{s\ jk}^{\
i}-X^{i}|_{s}S_{\ jk}^{s}.
\end{array}
\end{equation}

We deduce some fundamental identities for the $N-$linear
connection $D$, applying the Ricci identities to the Liouville
vector field $\C=y^{i}\dfrac{\partial}{\partial y^{i}}$.
Considering the d-tensors
\begin{equation}
\label{1.4.16}
D^i_j=y^i_{\ |j}, \quad d^i_j=y^i |_j
\end{equation}
called $h-$ and $v-${\it deflection} tensors of $D$ we obtain:

\begin{theorem}
\label{t4.2}
For any $N$-linear connection $D$ the following identities hold:
\begin{equation}
\label{1.4.17}
\begin{array}{l}
D^{i}{}_{k|h}-D^{i}{}_{h|k}=y^{s}R_{s\ kh}^{\
i}-D^{i}{}_{s}T^{s}{}_{kh}-d^{i}{}_{s}R^{s}{}_{kh},\vspace{3mm} \\
D^{i}{}_{k}|_{h}-d^{i}{}_{h|k}=y^{s}P_{s\ kh}^{\
i}-D^{i}{}_{s}C^{s}{}_{kh}-d^{i}{}_{s}P^{s}{}_{kh},\vspace{3mm} \\
d^{i}{}_{k}|_{h}-d^{i}{}_{h}|_{k}=y^{s}S_{s\ kh}^{\
i}-d^{i}{}_{s}S^{s}{}_{kh}.
\end{array}
\end{equation}
\end{theorem}

Others fundamental identities are the Bianchi identities,
which are obtained by writing in the adapted basis the following Bianchi
identities:
\begin{equation}
\label{1.4.18}
\begin{array}{l}
\Sigma \lbrack D_{X}T(Y,Z)-R(X,Y)Z+T(T(X,Y),Z)]=0,\vspace{3mm} \\
\Sigma \lbrack (D_{X}R)(U,Y,Z)+R(T(X,Y),Z)U]=0,
\end{array}
\end{equation}
where $\Sigma$ means cyclic summation over $X,Y,Z$.

\section{Parallelism. Structure equations}\index{Parallelism! on manifold $TM$}\index{Structure equations}
\label{s5c1p1}
\setcounter{equation}{0}
\setcounter{definition}{0}
\setcounter{theorem}{0}
\setcounter{proposition}{0}

Let $D\Gamma(N)=(L_{jk}^{i},C_{jk}^{i})$ be an $N$-linear
connection and the adapted basis $\left(\dfrac{\delta }{\delta
x^{i}},\dfrac{\partial }{\partial y^{i}}\right)$,
$i=\overline{1,n}$.

As we know, a curve $c:t\in I\rightarrow (x^{i}(t),y^{i}(t))\in
TM,$ has the tangent vector $\dot{c}=\dot{c}^{H}+\dot{c}^{V}$
given by (3.15), i.e. $\dot{c}=\dfrac{dx^{i}}{dt}\dfrac{\delta
}{\delta x^{i}}+\dfrac{\delta y^{i}}{dt}\dfrac{\partial }{\partial
y^{i}}$. The curve $c$ is horizontal if $\dfrac{\delta
y^{i}}{dt}=0$, and it is an
autoparallel curve with respect to the nonlinear connection $N$ if $\dfrac{\delta y^{i}}{%
dt }=0,$ $y^{i}=\dfrac{dx^{i}}{dt}.$

We set
\begin{equation}
\label{1.5.1}
\dfrac{DX}{dt}=D_{\dot{c}}X,\ DX=\dfrac{DX}{dt}\,dt,\ \forall X\in
\chi (TM).
\end{equation}
Here $\dfrac{DX}{dt}$ is the covariant differential along the
curve $c$ with respect to the the $N$-linear connection $D$.

Setting $X=X^{H}+X^{V}$, $X^{H}=X^{i}\dfrac{\delta}{\delta
x^{i}}$, $X^{V}= \dot{X}^{i}\dfrac{\partial }{\partial y^{i}}$ we
have
\begin{equation}
\label{1.5.2}
\begin{array}{l}
\dfrac{DX}{dt}=\dfrac{DX^{H}}{dt}+\dfrac{DX^{V}}{dt}=\vspace{3mm} \\
=\left\{X^{i}{}_{|k}\dfrac{dx^{k}}{dt}+X^{i}|_{k}\dfrac{\delta
y^{k}}{dt}\right\}\dfrac{\delta }{\delta x^{i}}+\left\{
\dot{X}^{i}{}_{|k}\dfrac{
dx^{k}}{dt}+\dot{X}^{i}|_{k}\dfrac{\delta y^{k}}{dt}\right\}.
\end{array}
\end{equation}

Consider the {\it connection 1-forms} of $D$:
\begin{equation}
\label{1.5.3}
\omega^{i}{}_{j}=L_{jk}^{i}dx^{k}+C_{jk}^{i}\delta y^{k}.
\end{equation}

Then $\dfrac{DX}{dt}$ takes the form
\begin{equation}
\label{1.5.4}
\dfrac{DX}{dt}=\left\{ \dfrac{dX^{i}}{dt}+X^{s}\dfrac{\omega
^{i}{}_{s}}{dt} \right\} \dfrac{\delta}{\delta x^{i}}+\left\{
\dfrac{d\dot{X}^{i}}{dt}+\dot{X}^{s}\dfrac{\omega
^{i}{}_{s}}{dt}\right\}\dfrac{\partial}{\partial y^{i}}.
\end{equation}

The vector field $X$ on $TM$ is said to be parallel along the
curve $c$, with respect to the $N$-linear connection $D(N)$ if
$\dfrac{DX}{dt}=0.$ Using (\ref{1.5.2}), the equation $\dfrac{DX}{dt}=0$
is equivalent to $\dfrac{DX^{H}}{dt}=0, \dfrac{DX^{V}}{dt}=0$.
According to (\ref{1.5.4}) we obtain:

\begin{proposition}
\label{p5.1}
The vector field
$X=X^{i}\dfrac{\delta}{\delta x^{i}}+\dot{X}^{i}\dfrac{\partial}{\partial y^{i}}$ from $\chi (TM)$ is parallel
along the parametrized curve $c$ in $TM$,
with respect to the $N$-linear connection $D\Gamma(N)=(L_{jk}^{i},C_{jk}^{i})$ if and only if its coefficients $X^{i}(x(t),y(t))$ and $\dot{X }^{i}(x(t),y(t))$ are solutions of the linear system of differential equations
$$\dfrac{dZ^{i}}{dt}+Z^{s}(x(t),y(t))\dfrac{\omega
^{i}{}_{s}(x(t),y(t))}{dt}=0.$$
\end{proposition}

A theorem of existence and uniqueness for the parallel vector
fields along a curve $c$ on the manifold $TM$ can be formulated on
the classical way.

The {\it horizontal geodesic} of $D$ are the horizontal curves
$c:I\to TM$ with the property $D_{\dot c}\dot c=0$. Taking
$X^{i}=\dfrac{dx^{i}}{dt}$, $\dot{X}^{i}=\dfrac{\delta
y^{i}}{dt}=0$ we get:

\begin{theorem}
\label{t5.1}
The horizontal geodesics of an $N$-linear connection are characterized by the system of differential
equations:
\begin{equation}
\label{1.5.5}
\dfrac{d^{2}x^{i}}{dt^{2}}+L_{jk}^{i}(x,y)\dfrac{dx^{j}}{dt}\dfrac{dx^{k}}{dt}
=0,\ \dfrac{dy^{i}}{dt}+N^{i}{}_{j}(x,y)\dfrac{dx^{j}}{dt}=0.
\end{equation}
\end{theorem}

Now we can consider a curve $c_{x_{0}}^{V}$ on the fibre
$T_{x_{0}}M=\pi ^{-1}(x_{0})$. It is represented by $$
x^{i}=x_{0}^{i},\ y^{i}=y^{i}(t),\ t\in I,$$ $c_{x_{0}}^{v}$ is
called a {\it vertical} curve of $TM$ at the point $x_{0}\in M.$

The {\it vertical geodesic} of $D$ are the vertical curves
$c^v_{x_0}$ with the property $D_{\dot c^v_{x_0}}\dot
c^v_{x_0}=0.$

\begin{theorem}
\label{t5.2}
The vertical geodesics at the point $x_{0}\in M,$ of the $N$-linear connection
$D\Gamma(N)=(L_{jk}^{i},C_{jk}^{i})$ are characterized by the
following system of differential equations
\begin{equation}
\label{1.5.6}
x^{i}=x_{0}^{i},\
\dfrac{d^{2}y^{i}}{dt^{2}}+C_{jk}^{i}(x_{0},y)\dfrac{dy^{j}}{
dt}\dfrac{dy^{k}}{dt}=0.
\end{equation}
\end{theorem}

Obviously, the local existence and uniqueness of horizontal or
vertical geodesics are assured if initial conditions are given.

Now, we determine the structure equations of an $N$-linear
connection $D$, considering the connection 1-forms $\omega^i_j$, (\ref{1.5.3}).

First of all, we have:

\begin{lemma}
\label{l5.1}
The exterior differential of $1$-forms
$\delta y^{i}=dy^{i}+N^{i}{}_{j}dx^{j}$ are given by:
\begin{equation}
\label{1.5.7}
d(\delta y^{i})=\dfrac{1}{2}R^{i}{}_{js}dx^{s}\wedge
dx^{j}+B^{i}{}_{js}\delta y^{s}\wedge dx^{j}
\end{equation}
where
\begin{equation}
B^{i}{}_{jk}=\dfrac{\partial N^{i}{}_{j}}{\partial y^{k}}.
\tag{1.5.7'}
\end{equation}
\end{lemma}

\begin{remark}
$B^{i}{}_{jk}$ are the coefficients of the Berwald
connection.
\end{remark}

\begin{lemma}
\label{l5.2}
With respect to a change of local coordinate
on the manifold $TM$, the following $2$-forms
$$d(dx^{i})-dx^{s}\wedge \omega^{i}{}_{s}; d(\delta y^{i})-\delta
y^{s}\wedge \omega ^{i}{}_{s}$$ transform as the components
of a $d$-vector field.

The $2$-forms
$$d\omega ^{i}{}_{j}-\omega ^{s}{}_{j}\wedge \omega ^{i}{}_{s}$$ transform as the components of a $d$-tensor field of type
$(1,1)$.
\end{lemma}

\begin{theorem}
\label{t5.3}
The structure equations of an $N$-linear connection $D\Gamma(N)=(L_{jk}^{i},C_{jk}^{i})$ on the manifold $TM$ are given
by
\begin{equation}
\label{1.5.8}
\begin{array}{l}
d(dx^{i})-dx^{s}\wedge \omega ^{i}{}_{s}=-{\stackrel{(0)}{\Omega
}}{}^{i}
\vspace{3mm} \\
d(\delta y^{i})-\delta y^{s}\wedge \omega
^{i}{}_{s}=-{\stackrel{(1)}{\Omega
} }{}^{i}\vspace{3mm} \\
d\omega ^{i}{}_{j}-\omega ^{s}{}_{j}\wedge \omega
^{i}{}_{s}=-\Omega ^{i}{}_{j}
\end{array}
\end{equation}
where ${\stackrel{(0)}{\Omega }}{}^{i}$ and
${\stackrel{(1)}{\Omega }}{}^{i}$ are the $2$-forms of torsion:
\begin{equation}
\label{1.5.9}
\begin{array}{l}
{\stackrel{(0)}{\Omega }}{}^{i}=\dfrac{1}{2}T_{\
jk}^{i}dx^{j}\wedge
dx^{k}+C_{jk}^{i}dx^{j}\wedge \delta y^{k}\vspace{3mm} \\
{\stackrel{(1)}{\Omega }}{}^{i}=\dfrac{1}{2}R_{\
jk}^{i}dx^{j}\wedge dx^{k}+P_{\ jk}^{i}dx^{j}\wedge \delta
y^{k}+\dfrac{1}{2}S_{\ jk}^{i}\delta y^{j}\wedge \delta y^{k}
\end{array}
\end{equation}
and the $2$-forms of curvature $\Omega^{i}{}_{j}$ are given by
\begin{equation}
\label{1.5.10}
\Omega^{i}{}_{j}=\dfrac{1}{2}R_{j\ kh}^{\ i}dx^{k}\wedge
dx^{h}+P_{j\ kh}^{\ i}dx^{k}\wedge \delta y^{h}+\dfrac{1}{2}S_{j\
kh}^{\ i}\delta y^{j}\wedge \delta y^{h}.
\end{equation}
\end{theorem}

\begin{proof}
By means of Lemma (\ref{l5.2}), the general structure
equations of a linear connection on $TM$ are particularized for an
$N$-linear connection $D$ in the form (\ref{1.5.8}). Using the connection
$1$-forms $\omega^{i}{}_{j}$ (\ref{1.5.3}) and the formula (\ref{1.5.7}), we can
calculate the forms ${\stackrel{(0)}{\Omega}}{}^{i}$, ${\stackrel{(1)}{\Omega}}
{}^{i}$ and $\Omega^{i}{}_{j}$.

Then it is very easy to determine the structure equations (\ref{1.5.9}).
\end{proof}

\begin{remark}
The Bianchi identities of an $N$-linear connection
$D$ can be obtained from (\ref{1.5.8}) by calculating the exterior
differential of (\ref{1.5.8}), modulo the same system (\ref{1.5.8}) and using the
exterior differential of ${\stackrel{(0)}{\Omega }}{}^{i}$
${\stackrel{(1)}{\Omega }}{}^{i}$ and $\Omega_{j}^{i}$.
\end{remark}

\newpage

\chapter{Lagrange spaces}\index{Spaces! Lagrange}

The notion of Lagrange spaces was introduced and studied by the present
author. The term ``Lagrange geometry'' is due to J. Kern, \cite{KJ}. We study the
geometry of Lagrange spaces as a subgeometry of the geometry of
tangent bundle $(TM,\pi, M)$ of a manifold $M$, using the
principles of Analytical Mechanics given by variational problem on
the integral of action of a regular Lagrangian, the law of
conservation, N\"{o}ther theorem etc. Remarking that the Euler -
Lagrange equations determine a canonical semispray $S$ on the
manifold $TM$ we study the geometry of a Lagrange space using
this canonical semi-spray $S$ and following the methods given in
the first chapter.

Beginning with the year 1987 there were published by author, alone or in
collaborations, some books on the Lagrange spaces and the
Hamilton spaces \cite{MHS}, and on the higher-order
Lagrange and Hamilton spaces \cite{mir11}, as well.

\section{The notion of Lagrange space}
\label{s1c2p1}

First we shall define the notion of differentiable Lagrangian over
the tangent manifold $TM$ and $\widetilde{TM}=TM\setminus\{0\}$,
$M$ being a real $n-$di\-men\-si\-onal manifold.

\begin{definition}
\label{d1.1}
A differentiable Lagrangian is a mapping
$L:(x,y)\in TM\rightarrow L(x,y)\in {I\!\!R}$, of class $C^{\infty
}$ on $\widetilde{TM}$ and continuous on the null section
$0:M\rightarrow TM$ of the projection $\pi :TM\rightarrow M$.
\end{definition}

The Hessian of a differentiable Lagrangian $L$, with respect to
$y^{i}$, has the elements:
\begin{equation}
\label{1.1}
g_{ij}=\dfrac{1}{2}\dfrac{\partial^{2}L}{\partial y^{i}\partial
y^{j}}.
\end{equation}

Evidently, the set of functions $g_{ij}(x,y)$ are the components
of a $d$-tensor field, symmetric and covariant of order 2.

\begin{definition}
\label{d1.2}
A differentiable Lagrangian $L$ is called {\it regular} if:
\begin{equation}
\label{1.2}
{\rm rank}(g_{ij}(x,y))=n,\ {\rm on}\widetilde{TM}.
\end{equation}
\end{definition}

Now we can give the definition of a Lagrange space:

\begin{definition}
\label{d1.3}
A Lagrange space is a pair
$L^{n}=(M,L(x,y))$ formed by a smooth, real $n$-dimensional
manifold $M$ and a regular Lagrangian $L(x,y)$ for which the
$d$-tensor $g_{ij}$ has a constant signature over the manifold
$\widetilde{TM}$.

For the Lagrange space $L^n=(M,L(x,y))$ we say that $L(x,y)$ is
the {\it fundamental function} and $g_{ij}(x,y)$ is the {\it
fundamental} (or metric) tensor.
\end{definition}

{\bf Examples.}

\begin{itemize}
\item[$1^\circ$]  Every Riemannian manifold $(M,g_{ij}(x))$
determines a Lagrange space $L^{n}=(M,L(x,y))$, where
\end{itemize}

\begin{equation}
\label{1.3}
L(x,y)=g_{ij}(x)y^{i}y^{j}.
\end{equation}

\begin{itemize}
\item[\ ]  This example allows to say:

If the manifold $M$ is paracompact, then there exist Lagrangians
$L(x,y)$ such that $L^n=(M,L(x,y))$ is a Lagrange space.

\item[$2^\circ$]  The following Lagrangian from electrodynamics
\end{itemize}

\begin{equation}
\label{1.4}
L(x,y)=mc\gamma_{ij}(x)y^i y^j+\dfrac{2e}{m}A_i(x)y^i+{\cal U}(x),
\end{equation}
where $\gamma_{ij}(x)$ is a pseudo-Riemannian metric, $A_i(x)$ a
covector field and ${\cal U}(x)$ a smooth function, $m,c,e$ are physical constants, is a
Lagrange space $L^n$. This is called the Lagrange space of
electrodynamics.

We already have seen that $g_{ij}(x,y)$ from (2.1.1) is a d-tensor field, i.e.
with respect to (1.1.1), we have
$$\widetilde{g}_{ij}(\widetilde{x},\widetilde{y})=\dfrac{\partial x^h}{\partial \widetilde{x}^i}\dfrac{\partial
x^k}{\partial\widetilde{x}^j}g_{hk}(x,y).$$

Now we can prove without difficulties:

\begin{theorem}
\label{t1.1}
For a Lagrange space $L^n$ the following
properties hold:

\begin{itemize} \item[1$^\circ$] The system of
functions
\begin{equation}
p_i=\dfrac{1}{2}\dfrac{\partial L}{\partial y^i}\nonumber
\end{equation}
determines a d-covector field.
\item[2$^\circ$] The functions
\begin{equation}
C_{ijk}=\dfrac{1}{4}\dfrac{\partial^3 L}{\partial y^i\partial
y^j\partial y^k}=\dfrac{1}{2}\dfrac{\partial g_{ij}}{\partial
y^k}\nonumber
\end{equation}
are the components of a symmetric $d-$tensor field of type $(0,3)$.
\item[3$^\circ$] The 1-form
\begin{equation}
\label{1.5}
\oo =p_{i}dx^{i}=\dfrac{1}{2}\dfrac{\partial L}{\partial
y^{i}}dx^{i}
\end{equation}
depend on the Lagrangian $L$ only and are globally defined on the
manifold $\widetilde{TM}$
\item[4$^{\circ}$] The $2$-form
\begin{equation}
\label{1.6}
\theta =d\oo=dp_{i}\wedge dx^{i}
\end{equation}
is globally defined on $\widetilde{TM}$ and defines a symplectic
structure on $\widetilde{TM}$.
\end{itemize}
\end{theorem}

\section{Variational problem. Euler-Lagrange equations}\index{Euler--Lagrange equations}\index{Semisprays! on Lagrange space $L_n$}
\label{s2c2p1}
\setcounter{equation}{0}
\setcounter{definition}{0}
\setcounter{theorem}{0}
\setcounter{lemma}{0}

The variational problem can be formulated for differentiable
Lagrangians and can be solved in the case when the integral of
action is defined on the parametrized curves.

Let $L:TM\rightarrow R$ be a differentiable Lagrangian and $c:t\in
\lbrack 0,1]\to (x^{i}(t))\in U\subset M$ be a smooth curve, with
a fixed parametrization, having Im$c\subset U$, where $U$ is a
domain of a local chart on the manifold $M$. The curve $c$ can be
extended to $\pi^{-1}(U)\subset\widetilde{TM}$ by
$$\widetilde{c}:t\in \lbrack 0,1]\to
(x^{i}(t),\dfrac{dx^{i}}{dt}(t))\in \pi ^{-1}(U).$$ So,
Im$\widetilde{c}\subset\pi^{-1}(U)$.

The integral of action of the Lagrangian $L$ on the curve $c$ is
given by the functional:
\begin{equation}
\label{2.1}
I(c)=\int_{0}^{1}L\left(x,\dfrac{dx}{dt}\right)dt.
\end{equation}
Consider the curves
\begin{equation}
\label{2.2}
c_{\varepsilon}:t\in \lbrack 0,1]\to (x^{i}(t)+\varepsilon
V^{i}(t))\in M
\end{equation}
which have the same end points $x^{i}(0)$ and $x^{i}(1)$ as the curve $c$, $V^{i}(t)=V^{i}(x^i(t))$ being a regular vector field on the curve
$c$, with the property $V^{i}(0)=V^{i}(1)=0$ and $\varepsilon$ is
a real number, sufficiently small in absolute value, so that ${\rm
Im}c_{\varepsilon}\subset U$.

The extension of a curve $c_{\varepsilon}$ to $\widetilde{TM}$ is
given by $$\widetilde{c}_{\varepsilon}:t\in \lbrack 0,1]\mapsto
\left(x^{i}(t)+\varepsilon V^{i}(t),\dfrac{dx^{i}}{dt}+\varepsilon
\dfrac{dV^{i}}{dt}\right)\in \pi^{-1}(U).$$ The integral of
action of the Lagrangian $L$ on the curve $c_{\varepsilon}$ is
\begin{equation}
\label{2.2.1'}
I(c_{\varepsilon})=\int_{0}^{1}L\left(x+\varepsilon
V,\dfrac{dx}{dt}+\varepsilon\dfrac{dV}{dt}\right)dt.\tag{2.2.1'}
\end{equation}

A necessary condition for $I(c)$ to be an extremal value of $I(c_{\varepsilon })$ is
\begin{equation}
\label{2.3}
\dfrac{dI(c_{\varepsilon})}{d\varepsilon}|_{\varepsilon=0}=0.
\end{equation}

Under our condition of differentiability, the operator $\dfrac{d}{d\varepsilon}$ is permuting with the operator of integration.

From (\ref{2.2.1'}) we obtain
\begin{equation}
\label{2.4}
\dfrac{dI(c_\varepsilon)}{d\varepsilon}=\int_{0}^{1}\dfrac{d}{d\varepsilon}
L(x+\varepsilon V,\dfrac{dx}{dt}+\varepsilon \dfrac{dV}{dt})dt.
\end{equation}
But we have
$$\begin{array}{c}
\dfrac{d}{d\varepsilon}L\left(x+\varepsilon
V,\dfrac{dx}{dt}+\varepsilon \dfrac{dV}{dt}\right)|_{\varepsilon
=0}=\dfrac{\partial L}{\partial x^{i}}V^{i}+\dfrac{
\partial L}{\partial y^{i}}\dfrac{dV^{i}}{dt}=\vspace{3mm} \\
=\left\{ \dfrac{\partial L}{\partial
x^{i}}-\dfrac{d}{dt}\dfrac{\partial L}{
\partial y^{i}}\right\} V^{i}+\dfrac{d}{dt}\left\{\dfrac{\partial L}{
\partial y^{i}}V^{i}\right\},\ y^{i}=\dfrac{dx^{i}}{dt}.
\end{array}$$
Substituting in (\ref{2.4}) and taking into account the fact that
$V^{i}(x(t))$ is arbitrary, we obtain the following theorem.

\begin{theorem}
\label{t2.1}
In order for the functional $I(c)$ to be
an extremal value of $I(c_{\varepsilon})$ it is necessary for the curve $c(t)=(x^{i}(t))$ to satisfy the Euler-Lagrange equations:
\begin{equation}
\label{2.5}
E_{i}(L):=\dfrac{\partial L}{\partial
x^{i}}-\dfrac{d}{dt}\dfrac{\partial L}{
\partial y^{i}}=0,\ y^{i}=\dfrac{dx^{i}}{dt}.
\end{equation}
\end{theorem}

For the Euler-Lagrange operator $E_{i}=\dfrac{\partial}{\partial
x^{i}}-\dfrac{d}{dt}\dfrac{\partial}{\partial y^{i}}$ we have:

\begin{theorem}
\label{t2.2}
The following properties hold true:

1$^\circ$ $E_{i}(L)$ is a d-covector field.

2$^\circ$ $E_{i}(L+L^{\prime})=E_{i}(L)+E_{i}(L^{\prime})$.

3$^\circ$ $E_{i}(aL)=aE_{i}(L),a\in \R$.

4$^\circ$ $E_{i}\left(\dfrac{dF}{dt}\right)=0$, $\forall
F\in {\cal F} (TM)$ with $\dfrac{\partial F}{\partial
y^i}=0$.
\end{theorem}

The notion of {\it energy of a Lagrangian $L$} can be introduced
as in the Theoretical Mechanics \cite{andres}, \cite{mirana}, by
\begin{equation}
\label{2.6}
E_L=y^i\dfrac{\partial L}{\partial y^i}-L.
\end{equation}

We obtain, without difficulties:

\begin{theorem}
\label{t2.3}
For every smooth curve $c$ on the base
manifold $M$ the following formula holds:
\begin{equation}
\label{2.7}
\dfrac{dE_L}{dt}=-\dfrac{dx^i}{dt}E_i(L),\ y^i=\dfrac{dx^i}{dt}.
\end{equation}
\end{theorem}

Consequently:

\begin{theorem}
\label{t2.4}
For any differentiable Lagrangian $L(x,y)$
the energy $E_L$ is conserved along every solution curve $c$ of
the Euler-Lagrange equations
$$E_i(L)=0,\quad \dfrac{dx^i}{dt}=y^i.$$
\end{theorem}

A Noether theorem can be proved:

\begin{theorem}
\label{t2.5}
For any infinitesimal symmetry on
$M\times\R$ of the Lagrangian $L(x,y)$ and for any smooth
function $\phi(x)$ the following function: $${\cal
F}(L,\phi)=V^i\dfrac{\partial L}{\partial y^i}-\tau E_L-\phi(x)$$
is conserved on every solution curve $c$ of the Euler-Lagrange
equations $E_i(L)=0$, $y^i=\dfrac{dx^i}{dt}$.
\end{theorem}

\noindent{\bf Remark.} An infinitesimal symmetry on $M\times \R$ is
given by $x^{\prime i}=x^i+\varepsilon V^i(x,t)$,
$t'=t+\varepsilon\tau(x,t)$.

\section{Canonical semispray. Nonlinear connection}\index{Connections! Nonlinear}
\label{s3c2p1}
\setcounter{equation}{0}
\setcounter{definition}{0}
\setcounter{theorem}{0}
\setcounter{lemma}{0}

Now we can apply the previous theory in order to study the
Lagrange space $L^n=(M,L(x,y))$. As we shall see that $L^n$ determines
a canonical semispray $S$ and $S$ gives a canonical nonlinear
connection on the manifold $\widetilde{TM}$.

As we know, the fundamental tensor $g_{ij}$ of the space $L^n$ is
nondegenerate, and $E_i(L)$ is a $d$-covector field, so the
equations $g^{ij}E_j(L)=0$ have a geometrical meaning.

\begin{theorem}
\label{t3.1}
If $L^n=(M,L)$ is a Lagrange space, then
the system of differential equations
\begin{equation}
\label{3.1}
g^{ij}E_j(L)=0,\ y^i=\dfrac{dx^i}{dt}
\end{equation}
can be written in the form:
\begin{equation}
\label{3.1'}
\dfrac{d^2x^i}{dt^2} + 2G^i\left(x,\dfrac{dx}{dt}\right)=0,
y^i=\dfrac{dx^i}{dt} \tag{2.3.1'}
\end{equation}
{\it where}
\begin{equation}
\label{3.2}
2G^i(x,y)=\dfrac{1}{2} g^{ij}\left\{\dfrac{\partial^2 L}{\partial
y^j\partial x^k} y^k-\dfrac{\partial L}{\partial x^j}\right\}.
\end{equation}
\end{theorem}

\begin{proof}
We have $$E_i(L)=\dfrac{\partial L}{\partial
x^i}-\left\{\dfrac{\partial^2 L}{\partial y^i\partial x^k}+ 2g_{ij}\dfrac{dy^j}{dt}\right\}, \ \ y^i=\dfrac{dx^i}{dt}.$$
So, (\ref{3.1}) implies (\ref{3.1'}), (\ref{3.2}).

The previous theorem tells us that the Euler Lagrange equations
for a Lagrange space are given by a system of $n$ second order
ordinary differential equations. According with theory from
Section 1.2, Ch. 1, it follows that the equations (\ref{3.1}) determine
a semispray with the coefficients $G^i(x,y)$:
\begin{equation}
\label{3.3}
S=y^{i}\dfrac{\partial}{\partial x^{i}}-2G^{i}(x,y)\dfrac{\partial
}{\partial y^{i}}.
\end{equation}
$S$ is called the canonical semispray of the Lagrange space $L^n$.

By means of Theorem 1.3.1, it follows:

\begin{theorem}
\label{t3.2}
Every Lagrange space $L^{n}=(M,L)$ has a
canonical nonlinear connection $N$ which depends only on the
fundamental function $L$. The local coefficients of $N$ are given
by
\begin{equation}
\label{3.4}
N^{i}{}_{j}=\dfrac{\partial G^{i}}{\partial
y^{j}}=\dfrac{1}{4}\dfrac{\partial
}{\partial y^{j}}\left\{g^{ik}\left(\dfrac{\partial^{2}L}{\partial y^{k}\partial x^{h}}
y^{h}-\dfrac{\partial L}{\partial x^{k}}\right)\right\}.
\end{equation}
\end{theorem}

Evidently:

\begin{proposition}
\label{p3.1}
The canonical nonlinear connection $N$
is symmetric, i.e. $t^i_{jk}=\dfrac{\partial N^i_j}{\partial
y^k}-\dfrac{\partial N^i_{\ k}}{\partial y^j}=0$.
\end{proposition}

\begin{proposition}
\label{p3.2}
The canonical nonlinear connection $N$
is invariant with respect to the Carath\'{e}odory transformation
\begin{equation}
\label{2.3.4'}
L^{\prime}(x,y)=L(x,y)+\dfrac{\partial\varphi(x)}{\partial
x^{i}}y^{i}.
\end{equation}
\end{proposition}

Indeed, we have
$$E_{i}(L^{\prime})=E_{i}\left(L(x,y)+\dfrac{d\varphi}{dt}\right)=E_{i}(L).\hfill\qed$$
\end{proof}

So, $E_{i}(L^{\prime}(x,y))=0$ determines the same canonical
semispray as the one determined by $E_{i}(L(x,y))=0$. Thus,
the Carath\'{e}odory transformation (\ref{2.3.4'}) preserves the nonlinear
connection $N$.

\medskip

\noindent{\bf Example.} The Lagrange space of electrodynamics, $L^n=(M,L(x,y))$, where $L(x,y)$ is given by (\ref{1.4}) with $U(x)=0$
has the canonical semispray with the coefficients:
\begin{equation}
\label{3.5}
G^i(x,y)=\dfrac{1}{2}\gamma^i{}_{jk}(x) y^j y^k-g^{ij}(x)F_{jk}(x)
y^k,
\end{equation}
where $\gamma^i{}_{jk}(x)$ are the Christoffel symbols of the metric tensor $g_{ij}(x)=mc\gamma_{ij}(x)$ of the space $L^n$ and $F_{jk}$ is the
electromagnetic tensor
\begin{equation}
\label{3.6}
F_{jk}(x,y)=\dfrac{e}{2m}\left(\dfrac{\partial A_k}{\partial
x^j}-\dfrac{\partial A_j}{\partial x^k}\right).
\end{equation}

Therefore, the integral curves of the Euler-Lagrange equation are
given by the solution curves of the {\it Lorentz equations}:
\begin{equation}
\label{3.7}
\dfrac{d^2x^i}{dt^2}+\gamma^i{}_{jk}(x)\dfrac{dx^j}{dt}\dfrac{dx^k}{dt}=
g^{ij}(x) F_{jk}(x)\dfrac{dx^k}{dt}.
\end{equation}

According to (\ref{3.4}), the canonical nonlinear connection of the
Lagrange space of electrodynamics $L^n$ has the local coefficients
given by
\begin{equation}
\label{3.8}
N^i{}_j(x,y)=\gamma^i_{jk}(x) y^k-g^{ik}(x) F_{kj}(x).
\end{equation}

It is remarkable that the coefficients $N^i{}_j$ from (\ref{3.8}) are
linear with respect to $y^i$.

\begin{proposition}
\label{p3.3}
The Berwald connection of the
canonical nonlinear connection $N$ has the coefficients
$B\Gamma(N)=(\gamma^i_{jk}(x),0)$.
\end{proposition}

\begin{proposition}
\label{p3.4}
The solution curves of the Euler-Lagrange equations and the autoparallel curves of the
canonical nonlinear connection $N$ are given by the Lorentz
equations $($\ref{3.7}$)$.
\end{proposition}

In the last part of this section, we underline the following
theorem:

\begin{theorem}
\label{t3.3}
The autoparallel curves of the canonical nonlinear connection $N$ are given by the following system:
$$\dfrac{d^2 x^i}{dt^2}+N^i{}_j\left(x,\dfrac{dx}{dt}\right)\dfrac{dx^j}{dt}=0,
$$ where $N^i{}_j$ are given by $($\ref{3.4}$)$.
\end{theorem}

\section{Hamilton-Jacobi equations}
\label{s4c2p1}
\setcounter{equation}{0}
\setcounter{definition}{0}
\setcounter{theorem}{0}
\setcounter{lemma}{0}

Consider a Lagrange space $L^n=(M,L(x,y))$ and $N(N^i{}_j)$ its
canonical nonlinear connection. The adapted basis
$\left(\displaystyle\frac{\delta}{\delta x^i},
\displaystyle\frac{\partial}{\partial y^i}\right)$ to the
horizontal distribution $N$ and the vertical distribution $V$ has
the horizontal vector fields:
\begin{equation}
\label{4.1}
\dfrac{\delta}{\delta x^i}=\dfrac{\partial}{\partial
x^i}-N^j{}_i\dfrac{
\partial }{\partial y^j}.
\end{equation}
Its dual is $(dx^i,\delta y^i)$, with
\begin{equation}
\label{4.2}
\delta y^i=dy^i+N^i{}_j dx^j.
\end{equation}
Theorem 1.1.1 give us the momenta
\begin{equation}
\label{4.3}
p_i=\dfrac{1}{2} \dfrac{\partial L}{\partial y^i},
\end{equation}
the 1-form
\begin{equation}
\label{4.4}
\omega=p_i dx^i
\end{equation}
and the 2-form
\begin{equation}
\label{4.5}
\theta=d\omega=dp_i\wedge dx^i.
\end{equation}

These geometrical object fields are globally defined on
$\widetilde{TM}$. $\theta$ is a {\it symplectic structure} on the
manifold $\widetilde{TM}$.

\begin{proposition}
\label{p4.1}
In the adapted basis the 2-form $\theta$ is given by
\begin{equation}
\label{4.6}
\theta=g_{ij}\delta y^i\wedge dx^j.
\end{equation}
\end{proposition}

Indeed, $\theta=dp_i\wedge dx^i=\dfrac{1}{2}
\left(\dfrac{\delta}{\delta x^s}\dfrac{\partial L}{\partial
y^i}dx^s+\dfrac{\partial}{\partial y^s}\dfrac{\partial L}{\partial y^i}\delta y^s\right)\wedge dx^i$ $=\dfrac{1}{
4}\left(\dfrac{\delta}{\delta x^s}\dfrac{\partial L}{\partial
y^i}-\dfrac{\delta}{\delta x^i}\dfrac{\partial
L}{\partial y^s}\right)dx^s \wedge dx^i+g_{is}\delta y^s\wedge
dx^i$.

But is not difficult to see that the coefficient of $dx^s\wedge
dx^i$ vanishes.

The triple $(\widetilde{TM},\theta,L)$ is called a Lagrangian
system.

The energy $E_{L}$ of the space $L^{n}$ is given by (\ref{2.6}). Denoting ${\cal H}=\dfrac{1}{2}E_{L}$, ${\cal L}=\dfrac{1}{2}L$, then (\ref{2.6}) can be written as:
\begin{equation}
\label{4.7}
{\cal H}=p_{i}y^{i}-{\cal L}(x,y).
\end{equation}
But, along the integral curve of the Euler-Lagrange equations
(\ref{2.5}) we have
$$\dfrac{\partial{\cal H}}{\partial x^{i}}=-\dfrac{\partial{\cal L}}{\partial x^{i}}=-\dfrac{dp_i}{dt}.$$
And from (\ref{4.7}), we get
$$\dfrac{\partial{\cal H}}{\partial p_{i}}=y^{i}=\dfrac{dx^{i}}{dt}.$$
So, we obtain:

\begin{theorem}
\label{t4.1}
Along to integral curves of the Euler-Lagrange
equations the Hamilton-Jacobi equations:
\begin{equation}
\label{4.8}
\dfrac{dx^{i}}{dt}=\dfrac{\partial {\cal H}}{\partial p_{i}};\ \dfrac{dp_{i}}{
dt}=-\dfrac{\partial {\cal H}}{\partial x^{i}},
\end{equation}
where ${\cal H}$ is given by $($\ref{4.7}$)$ and
$p_{i}=\dfrac{1}{2}\dfrac{\partial L}{\partial y^{i}}$, are satisfied.
\end{theorem}

These equations are important in applications.

\medskip

\noindent{\bf Example.} For the Lagrange space of Electrodynamics with the
fundamental function $L(x,y)$ from (\ref{1.4}) and $U(x)=0$ we obtain
$${\cal H}=\dfrac{1}{2mc}\gamma^{ij}(x)p_i
p_j-\dfrac{e}{mc^2}A^i(x)p_i+\dfrac{e^2}{2mc^3}A^i(x)A_i(x)$$
$(A^i=\gamma^{ij}A_j)$.

Then, the Hamilton - Jacobi equations can be written without
difficulties.

Now we remark that $\theta$ being a symplectic structure on
$\widetilde{TM}$, exterior differential $d\theta$ vanishes.
But in adapted basis $$d\theta=dg_{ij}\wedge\delta y^i\wedge
dx^j+g_{ij}d\delta y^i\wedge dx^j=0$$ reduces to:
$$\dfrac{1}{2}\left(\dfrac{\delta g_{ij}}{\delta
x^k}-\dfrac{\delta g_{ik}}{\delta x^j}\right)\delta y^i\wedge
dx^j\wedge dx^k+\dfrac{1}{2}\left(\dfrac{\partial g_{ij}}{\partial
x^k}-\dfrac{\partial g_{kj}}{\partial y^i}\right)\delta y^k\wedge
\delta y^i\wedge dx^j+$$ $$+g_{ij}\left(\dfrac{1}{2}R^i_{\
rs}dx^s\wedge dx^r+B^i_{\ rs}\delta y^s\wedge dx^r\right)\wedge
dx^j=0.$$ We obtain

\begin{theorem}
\label{t4.2}
For any Lagrange space $L^n$ the following
identities hold
\begin{equation}
\label{4.9}
g_{ij\|k}-g_{ik\|j}=0,\quad g_{ij}\|_k-g_{ik}\|_j=0.
\end{equation}
\end{theorem}

Indeed, taking into account the $h-$ and $v-$covariant derivations
of the metric $g_{ij}$ with respect to Berwald connection
$B\Gamma(N)=\left(\dfrac{\partial N^i_j}{\partial y^k},0\right)$,
i.e. $$g_{ij\|k}=\dfrac{\delta g_{ij}}{\delta
x^k}-B^r_{ik}g_{kj}-B^r_{jk}g_{ir}$$ and
$g_{ij}\|_{k}=\dfrac{\partial g_{ij}}{\partial y^k}$, according
with the properties $\dfrac{\partial g_{ij}}{\partial
y^k}=2C_{ijk}$, $B^i_{jk}=B^i_{kj}$, we obtain (\ref{4.9}).

\section{Metrical $N$-linear connections}
\label{s5c2p1}
\setcounter{equation}{0}
\setcounter{definition}{0}
\setcounter{theorem}{0}
\setcounter{lemma}{0}

Let $N(N^i_j)$ be the canonical nonlinear connection of the
Lagrange space $L^n=(M,L)$ and $D$ an $N-$linear connection with
the coefficients $D\Gamma(N)=(L^i_{jk},C^i_{jk})$. Then, the $h-$
and $v-$ covariant derivations of the fundamental tensor $g_{ij}$,
$g_{ij|k}$ and $g_{ij}|_k$ are given by (1.4.7).

Applying the theory of $N-$linear connection from Chapter 1, one
proves without difficulties, the following theorem:

\begin{theorem}
\label{5.1}
\begin{itemize}\item[$1^\circ$] On the
manifold $\widetilde{TM}$ there exist only one $N-$linear
connection $D$ which verifies the following axioms:
\begin{itemize}
\item[$A_1$]\quad  $N$ is canonical nonlinear connection of the space
$L^n$.

\item[$A_2$]\quad  $g_{ij|k}=0$ ($D$ is $h$-metrical);

\item[$A_3$]\quad  $g_{ij}|_{k}=0$ ($D$ is $v$-metrical);

\item[$A_4$]\quad  $T_{\ jk}^{i}=0$ ($D$ is $h$-torsion free);

\item[$A_5$]\quad  $S_{\ jk}^{i}=0$ ($D$ is $v$-torsion free).
\end{itemize}

\item[$2^\circ$]\quad  The coefficients
$D\Gamma(N)=(L_{jk}^{i},C_{jk}^{i})$ of $D$ are expressed by the
following generalized Christoffel symbols:
\end{itemize}

\begin{equation}
\label{5.1}
\begin{array}{l}
L^i_{jk}=\dfrac{1}{2}g^{ir} \left(\dfrac{\delta g_{rk}}{\delta
x^j}+\dfrac{\delta g_{rj}}{\delta x^k}-\dfrac{\delta
g_{jk}}{\delta x^r}\right)\vspace{3mm}
\\
C^i_{jk}=\dfrac{1}{2}g^{ir} \left(\dfrac{\partial g_{rk}}{\partial
y^j} + \dfrac{
\partial g_{rj}}{\partial y^k}- \dfrac{\partial g_{jk}}{\partial y^r}\right)
\end{array}
\end{equation}

\begin{itemize}
\item[$3^\circ$] This connection depends only on the
fundamental function $L(x,y)$ of the Lagrange space $L^{n}$.
\end{itemize}
\end{theorem}

The $N$-linear connection $D$ given by the previous theorem is
called the {\it canonical metric connection} and denoted by
$C\Gamma (N)=(L^i_{jk},C^i_{jk})$.

By means of \S1.5, Ch. 1, the connection 1-forms $\omega
^{i}{}_{j}$ of the $C\Gamma(N)$ are
\begin{equation}
\label{5.2}
\omega^{i}{}_{j}=L_{jk}^{i}dx^{k}+C_{jk}^{i}\delta y^{k},
\end{equation}

\begin{theorem}
\label{5.2}
The canonical metrical connection $C\Gamma
(N)$ satisfies the following structure equations:
\begin{equation}
\label{5.3}
\begin{array}{l}
d(dx^{i})-dx^{k}\wedge\omega^{i}{}_{k}=-{\stackrel{(0)}{\Omega}}{}^{i},
\vspace{3mm} \\
d(\delta y^{i})-\delta y^{k}\wedge\omega
^{i}{}_{k}=-{\stackrel{(1)}{\Omega}}{}^{i},
\vspace{3mm} \\
d\omega^i_j-\omega^k_j\wedge \omega^i_k=-{\Omega}{}^{i}_j
\end{array}
\end{equation}
where 2-forms of torsion ${\stackrel{(0)}{\Omega }}{}^{i}$ and
${\stackrel{(1)}{\Omega}}{}^{i}$ are as follows
\begin{equation}
\label{5.4}
\begin{array}{l}
{\stackrel{(0)}{\Omega}}{}^{i}=C_{jk}^{i}dx^{j}\wedge \delta y^{k},\vspace{%
3mm} \\
{\stackrel{(1)}{\Omega}}{}^{i}=\dfrac{1}{2}R^{i}{}_{jk}dx^{j}\wedge
dx^{k}+P^{i}{}_{jk}dx^{j}\wedge \delta y^{k}
\end{array}
\end{equation}
and the 2-forms of curvature $\Omega ^{i}{}_{j}$ are
\begin{equation}
\label{5.5}
\Omega ^{i}{}_{j}=\dfrac{1}{2}R_{j\ kh}^{\ i}dx^{k}\wedge
dx^{h}+P_{j\ kh}^{\ i}dx^{k}\wedge \delta y^{h}+\dfrac{1}{2}S_{j\
kh}^{\ i}\delta y^{k}\wedge \delta y^{h}.
\end{equation}
\end{theorem}

The $d$-tensors of torsion $R_{\ jk}^{\ i}$, $P_{\ jk}^{\ i}$ are
given by (1.3.13') and (1.4.10), and the d-tensors of curvature
$R_{j\ kh}^{\ i}$, $P_{j\ kh}^{\ i}$, $S_{j\ kh}^{\ i}$ have the
expressions (1.4.14).

Starting from the canonical metrical connection
$C\Gamma(N)=(L^i_{jk},C^i_{jk})$ we can derive other $N$-linear
connections depend only on the space $L^n$: Berwald connection
$B\Gamma (N)=\left(\dfrac{\partial N^{i}{}_{j}}{\partial
y^{k}},0\right)$; Chern-Rund connection $R\Gamma
(N)=(L_{jk}^{i},0)$ and Hashiguchi connection $H\Gamma
(N)=\left(\dfrac{\partial N^{i}{}_{j}}{\partial
y^{k}},C_{jk}^{i}\right)$. For special transformations of these
connections, the following commutative diagram holds:
$$\begin{array}{ccccc} & & R\Gamma(N) & & \\ & \nearrow & & \searrow \\
C\Gamma(N) & &\longrightarrow & & B\Gamma(N)\\ & \searrow & &\nearrow \\ & &
H\Gamma(N) &
&\end{array}$$

Some properties of the canonical metrical connection
$C\Gamma(N)$ are given by:

\begin{proposition}
\label{p5.1}
We have:

\begin{itemize}
\item[$1^\circ$] $\sum_{(ijk)}R_{ijk}=0$,
($R_{ijk}=g_{ih}R^{h}{}_{jk}$).

\item[$2^\circ$] $P_{ijk}=g_{ih}P^{h}{}_{jk}$ is totally
symmetric.

\item[$3^\circ$]  The covariant curvature d-tensors
$R_{ijkh}=g_{jr}R^{\ r}_{i\ \ kh}$, $P_{ijkh}=g_{jr}P^{\ r}_{i\ \
kh}$ and $S_{ijkh}=g_{ir}S^{\ r}_{i\ \ kh}$ are skew-symmetric
with respect to the first two indices.

\item[$4^\circ$]
$S_{ijkh}=C_{iks}C^{s}{}_{jh}-C_{ihs}C^{s}{}_{jk}$.

\item[$5^\circ$] $C_{ikh}=g_{is}C^s_{jh}$.
\end{itemize}
\end{proposition}

These properties can be proved using the property $d\theta=0$,
with $\theta=g_{ij}\delta y^i\wedge dx^j$, the Ricci identities
applied to the fundamental tensor $g_{ij}$ and the equations
$g_{ij|k}=0,$ $g_{ij}|_k=0$.

By the same method we can study the metrical connections with a
priori given $h-$ and $v-$ torsions.

\begin{theorem}
\label{t5.3}
\begin{itemize}
\item[$1^\circ$]  There exists only one $N$-linear connection $\bar{D}
\Gamma (N)=(\bar{L}_{jk}^{i},\bar{C}_{jk}^{i})$ which satisfies
the following axioms:

\begin{itemize}
\item[$A^{\prime}_1$]\quad  $N$ is canonical nonlinear connection
of the space $L^{n}$,

\item[$A^{\prime}_2$]\quad  $g_{ij|k}=0$ ($\bar{D}$ is
$h$-metrical),

\item[$A^{\prime}_3$]\quad  $g_{ij}|_{k}=0$ $(\bar{D}$ is
$v$-metrical),

\item[$A^{\prime}_4$]\quad  The $h$-tensor of torsion
$\bar{T}_{jk}^{i}$ is a priori given.

\item[$A^{\prime}_5$]\quad  The $v$-tensor of torsion
$\bar{S}_{jk}^{i}$ is a priori given.
\end{itemize}

\item[$2^\circ$]\quad  The coefficients
$(\bar{L}_{jk}^{i},\bar{C}_{jk}^{i})$ of $\bar{D}$ are given by
\end{itemize}
\begin{equation}
\label{5.6}
\begin{array}{l}
\bar{L}^i_{jk}=L^i_{jk} + \dfrac{1}{2}
g^{ih}(g_{jr}\bar{T}^r{}_{kh} + g_{kr}
\bar{T}^r{}_{jh} - g_{hr}\bar{T}^r{}_{kj}), \vspace{3mm} \\
\bar{C}^i_{jk}=C^i_{jk} + \dfrac{1}{2}
g^{ih}(g_{jr}\bar{S}^r{}_{kh} + g_{kr} \bar{S}^r{}_{jh} -
g_{hr}\bar{S}^r{}_{kj})
\end{array}
\end{equation}
where $(L^i_{jk}, C^i_{jk})$ are the coefficients of
the canonical metrical connection.
\end{theorem}

From now on $\bar{T}^i_{\ jk}, \bar{S}^i_{\ jk}$ will be denoted by
$T^i_{\ jk}, S^i_{\ jk}$ and the $N-$linear connection given by
the previous theorem will be called {\it metrical $N-$connection}
of the Lagrange space $L^n$.

Some particular cases can be studied using the expressions of the
coefficients $\bar{L}^i_{jk}$ and $\bar{C}^i_{jk}$. For instance
the semi-symmetric case will be obtained taking $T^i_{\
jk}=\delta^i_j\sigma_k-\delta^i_k\sigma_j$, $S^i_{\
jk}=\delta^i_j\tau_k-\delta^i_k\tau_j$.

\begin{proposition}
\label{p5.2}
The Ricci identities of the metrical
$N$-linear connection $D\Gamma(N)$ are given by:
\begin{equation}
\label{5.6}
\begin{array}{l}
X^i{}_{|j|k}-X^i{}_{|k|j}=X^r R^{\ i}_{r\ jk} -
X^i{}_{|r}T^r{}_{jk}-X^i{}|_r
R^r{}_{jk},\vspace{3mm} \\
X^i{}_{|j}|_k-X^i|_{k|j} =X^r P^{\ i}_{r\ jk} -
X^i{}_{|r}C^r{}_{jk}-X^i{}|_r
P^r{}_{jk},\vspace{3mm} \\
X^i{}|_j|_k - X^i{}|_k|_j=X^r S^{\ i}_{r\ jk} - X^i{}|_r
S^r{}_{jk}.
\end{array}
\end{equation}
\end{proposition}

Of course these identities can be extended to a $d$-tensor field
of type $(r,s)$.

Denoting
\begin{equation}
\label{5.7}
D^i{}_j=y^i{}_{|j},\ d^i{}_j=y^i{}|_j.
\end{equation}
we have the $h-$ and $v-$ deflection tensors. They have the known
expressions:
\begin{equation}
D^i{}_j=y^s L^i{}_{sj}-N^i{}_j;\
d^i{}_j=\delta^i{}_j+y^sC^i{}_{sj}. \tag{2.5.7'}
\end{equation}

According to Ricci identities (\ref{5.6}) we obtain:

\begin{theorem}
\label{t5.4}
For any metrical $N$-linear connection the
following identities hold:
\begin{equation}
\label{5.8}
\begin{array}{l}
D^i{}_{j|k}-D^i{}_{k|j}=y^s R^{\ i}_{s\ jk}-D^i{}_s T^s_{\
jk}-d^i{}_s R^s_{\ jk},
\vspace{3mm} \\
D^i{}_{j}|_k-d^i{}_{k|j}=y^s P^{\ i}_{s\
jk}-D^i{}_s C^s{}_{jk}-d^i{}_s P^s{}_{jk}, \vspace{3mm} \\
d^i{}_{j}|_k-d^i{}_{k}|_j=y^s S^{\ i}_{s\ jk}-d^i{}_s S^s_{\ jk}.
\end{array}
\end{equation}
\end{theorem}

We will apply this theory in a next section taking into
account the canonical metrical connection $C\Gamma(N)$ and taking
$T^i_{\ jk}=0$, $S^i_{\ jk}=0.$

Of course the theory of parallelism of vector fields and the
$h-$ge\-o\-de\-sics or $v-$ge\-o\-de\-sics for the metrical connection
$N-$linear connections can be obtained as a consequence of the
corresponding theory from Ch. 1.

\section{The electromagnetic and gravitational fields}\index{Electromagnetism and gravitational fields}
\label{s6c2p1}
\setcounter{equation}{0}
\setcounter{definition}{0}
\setcounter{theorem}{0}
\setcounter{lemma}{0}

Let us consider a Lagrange spaces $L^n=(M,L)$ endowed with the
canonical nonlinear connection $N$ and with the canonical metrical
$N-$connection $C\Gamma(N)=(L^i_{jk},C^i_{jk})$.

The covariant deflection tensors $D_{ji}$ and $d_{ji}$ are given
by $D_{ij}=g_{is}D^{s}{}_{j},\ d_{ij}=g_{is}d^{s}{}_{j}$. We have:
$$D_{ij|k}=g_{is}D^{s}{}_{j|k},\ d_{ij|k}=g_{is}d^{s}{}_{j|k}$$ etc. So, we have

\begin{proposition}
\label{p6.1}
The covariant deflection tensors
$D_{ij}$ and $d_{ij}$ of the canonical metrical $N$-connection
$C\Gamma (N)$ satisfy the identities:
\begin{equation}
\label{6.1}
\begin{array}{l}
D_{ij|k}-D_{ik|j}=y^{s}R_{sijk}-d_{is}R^{s}{}_{jk},\vspace{3mm} \\
D_{ij}|_{k}-d_{ik|j}=y^{s}P_{sijk}-D_{is}C^{s}{}_{jk}-d_{is}P^{s}{}_{jk},
\vspace{3mm} \\
d_{ij}|_{k}-d_{ik}|_{j}=y^{s}S_{sijk}.
\end{array}
\end{equation}
\end{proposition}

The Lagrangian theory of electrodynamics lead us to introduce \cite{miradio}, \cite{mitava}, \cite{mitava1},
\cite{Miron}:

\begin{definition}
\label{d6.1}
The d-tensor fields:
\begin{equation}
\label{6.2}
F_{ij}=\dfrac{1}{2}(D_{ij}-D_{ji}),\
f_{ij}=\dfrac{1}{2}(d_{ij}-d_{ji})
\end{equation}
are the $h$- and $v$-{\it electromagnetic tensor} of the Lagrange
space $L^{n}=(M,L)$.
\end{definition}

The Bianchi identities for $C\Gamma(N)$ and the
identities (\ref{6.1}) lead to the following important result:

\begin{theorem}
\label{t6.1}
The following generalized Maxwell equations hold:
\begin{equation}
\label{6.3}
\begin{array}{l}
F_{ij|k} + F_{jk|i} + F_{ki|j}=-\displaystyle\sum_{(ijk)} C_{ios} R^s{}_{jk}, \vspace{3mm} \\
F_{ij}|_k + F_{jk}|_i + F_{ki}|_j=0,
\end{array}
\end{equation}
where $C_{ios}=C_{ijs}y^j$, and $\displaystyle\sum_{(ijk)}$ means cyclic sum.
\end{theorem}

\begin{corollary}
\label{c6.1}
If the canonical nonlinear connection
$N$ of the space $L^n$ is integrable then the equations (\ref{6.3})
reduce to:
\begin{equation}
\label{6.3'}
\displaystyle\sum_{(ijk)}F_{ij|k}=0,\
\displaystyle\sum_{(ijk)}F_{ij}|_k=0.\tag{2.6.3'}
\end{equation}
\end{corollary}

If we put
\begin{equation}
\label{6.4}
F^{ij}=g^{is}g^{jr}F_{sr}
\end{equation}
and
\begin{equation}
\label{6.5}
hJ^{i}=F^{ij}{}_{|j},\ vJ^{i}=F^{ij}{}|_{j},
\end{equation}
then one can prove:

\begin{theorem}
\label{t6.2}
The following laws of conservation hold:
\begin{equation}
\label{6.6}
\begin{array}{l}
hJ^i{}_{|i}=\dfrac{1}{2}\{F^{ij}(R_{ij}-R_{ji}) + F^{ij}{}|_r
R^r{}_{ij} \},
\vspace{3mm} \\
vJ^i{}|_i=0,
\end{array}
\end{equation}
where $R_{ij}$ is the Ricci tensor $R^{\ h}_{i\ jh}$.
\end{theorem}

\begin{remark}
In the Lagrange space of electrodynamics the tensor
$F_{jk}$ is given by (2.3.6). The previous theory one reduces
to the classical theory. Namely $F_{ij}(x)$ satisfy the Maxwell
equations $\displaystyle\sum_{(ijk)}F_{ij|k}=0$ and $F_{ij}|_k=0$,
$hj^i_{|i}=0$, $vj^i=0$.
\end{remark}

Now, considering the lift to $\widetilde{TM}$ of the fundamental
tensor $g_{ij}(x,y)$ of the space $L^n$, given by $${\mathbb G}=g_{ij}dx^i\otimes dx^j+g_{ij}\delta y^i\otimes\delta y^j$$
we can obtain the Einstein equations of the canonical metric
connection $C\Gamma(N)$. The curvature Ricci and scalar
curvatures:
\begin{equation}
\label{6.7}
\begin{array}{l}
R_{ij}=R^{\ h}_{i \ \ jh}, S_{ij}=S^{\ h}_{i \ \ jh}, 'P_{ij}=P^{\ h}_{i \ \ jh}, ^{''}P^{\ h}_{i \ \ hj}\vspace{3mm}\\
R=g^{ij} R_{ij}, S=g^{ij} S_{ij}.\end{array}
\end{equation}

Let us denote by $\buildrel{H}\over{T}_{ij}$,
$\buildrel{V}\over{T}_{ij}$, $\buildrel{1}\over{T}_{ij}$ and
$\buildrel{2}\over{T}_{ij}$ the components in adapted basis
$\left(\dfrac{\delta}{\delta x^i},\dfrac{\partial}{\partial
y^i}\right)$ of the energy momentum tensor on the manifold
$\widetilde{TM}$.

Thus we obtain:
\begin{theorem}
\label{t6.3}
\begin{itemize}
\item[$1^\circ$] The Einstein equations of the Lagrange space
$L^n=(M,L(x,y))$ with respect to the canonical metrical con\-nec\-tion\break
$C\Gamma(N)=$ $(L^i_{{j}{k}},C^i_{{j}{k}})$ are as follows:
\begin{equation}
\label{6.8}
\begin{array}{l}
R_{ij}-\dfrac{1}{2}Rg_{ij}=\kappa\buildrel{H}\over{T}_{ij}, 'P_{ij}=\kappa\buildrel{1}\over{T}_{ij}\vspace{3mm}\\
S_{ij}-\dfrac{1}{2}Sg_{ij}=\kappa\buildrel{V}\over{T}_{{(i)}{(j)}},''P_{ij}=\kappa\buildrel{2}\over{T}{}^i_j,
\end{array}
\end{equation}
where $\kappa$ is a real constant.
\item[$2^\circ$] The energy
momentum tensors $\buildrel{H}\over{T}_{ij}$ and
$\buildrel{V}\over{T}_{ij}$ satisfy the following laws of
conservation
\begin{equation}
\label{6.9}
\kappa\buildrel{H}\over{T}_{ij}=-\dfrac{1}{2}(P^{ih}_{js}R^s_{hi}+2R^s_{ij}P^i_s),
\kappa\buildrel{V}\over{T}^i_j |_i=0.
\end{equation}
\end{itemize}
\end{theorem}

The physical background of the previous theory is discussed by
Satoshy Ikeda in the last chapter of the book \cite{IS1}.

The previous theory is very simple in the particular Lagrange
spaces $L^n$ having $P^{\ h}_{i\ \ jk}=0$.

We have:

\begin{corollary}
\label{c6.2}
\begin{itemize} \item[$1^\circ$] If the
canonical metrical connection $C\Gamma(N)$ has the property $P^{\
i}_{j\ \ kh}=0$, then the Einstein equations are
\begin{equation}
\label{6.10}
R_{ij}-\dfrac{1}{2}Rg_{ij}=\kappa\buildrel{H}\over{T}_{ij},
S_{ij}-\dfrac{1}{2}Sg_{ij}=\kappa\buildrel{V}\over{T}_{{(i)}{(j)}}
\end{equation}
\item[$2^\circ$] The following laws of conservation hold:
$$\buildrel{H}\over{T}^i_{j|i}=0,\
\buildrel{V}\over{T}^i_{j}|_i=0.$$\end{itemize}
\end{corollary}

\begin{remark}
The Lagrange space of Electrodynamics, $L^n$, has\break
$C\Gamma(N)=(\gamma^i_{jk}(x),0)$, $P^{\ i}_{j\ kh}=0$, $S^{\
i}_{j\ kh}=0$. The Einstein equations (\ref{6.10}) reduce to the
classical Einstein equations of the space $L^n$.
\end{remark}

\section{The almost K\"ahlerian model of a Lagrange space $L^n$}\index{Almost Hermitian model}
\label{s7c2p1}
\setcounter{equation}{0}
\setcounter{definition}{0}
\setcounter{theorem}{0}
\setcounter{lemma}{0}

A Lagrange space $L^{n}=(M,L)$ can be thought as an almost
K\"{a}hler space on the manifold
$\widetilde{TM}=TM\setminus\{0\}$, called the geometrical model of
the space $L^n$.

As we know from section 3, Ch. 1 the canonical nonlinear
connection $N$ determines an almost complex structure
${\mathbb F}(\widetilde{TM})$, expressed in (1.3.9'). This is
\begin{equation}
\label{7.1}
{\mathbb F}=\dfrac{\delta}{\delta x^{i}}\otimes\delta
y^i-\dfrac{\partial}{\partial y^{i}}\otimes dx^i.
\end{equation}
${\mathbb F}$ is integrable if and only if $R^i_{jk}=0$.

${\mathbb F}$ is globally defined on $\widetilde{TM}$ and it can be
considered as a ${\cal F}(\widetilde{TM})-$ linear mapping from
$\chi(\widetilde{TM})$ to $\chi(\widetilde{TM})$:
\begin{equation}
{\mathbb F}\left(\dfrac{\delta}{\delta x^i}\right)=-\dfrac{\partial
}{\partial y^{i}},\quad {\mathbb F}\left(\dfrac{\partial}{\partial
y^i}\right)=\dfrac{\delta}{ \delta x^{i}}, (i=1,...,n). \tag{2.7.1'}
\end{equation}

The lift of the fundamental tensor $g_{ij}$ of the space $L^n$
with respect to $N$ is defined by
\begin{equation}
\label{7.2}
{\mathbb G}=g_{ij}dx^i\otimes dx^j+g_{ij}\delta y^i\otimes \delta
y^j.
\end{equation}

Evidently ${\mathbb G}$ is a (pseudo-)Riemannian metric on the
manifold $\widetilde{TM}$.

The following result can be proved without difficulties:

\begin{theorem}
\label{t7.1}
\begin{itemize}
\item[$1^\circ$] The pair $({\mathbb G},{\mathbb F})$ is an
almost Hermitian structure on $\widetilde{TM}$, determined only by
the fundamental function $L(x,y)$ of $L^n$.

\item[$2^\circ$] The almost symplectic structure associated
to the structure $({\mathbb G},{\mathbb F})$ is given by
\end{itemize}

\begin{equation}
\label{7.3}
\theta=g_{ij}\delta y^i\wedge dx^j.
\end{equation}

\begin{itemize}
\item[$3^\circ$] The space
$(\widetilde{TM},{\mathbb G},{\mathbb F})$ is almost K\"{a}hlerian.
\end{itemize}
\end{theorem}

Indeed:
\begin{itemize}
\item[$1^\circ$] $N,{\mathbb G},{\mathbb F}$ are determined only by $L(x,y)$.

We have ${\mathbb G}({\mathbb F} X,{\mathbb F} Y)={\mathbb G}(X,Y)$,
$\forall X,Y\in\chi(\widetilde{TM})$.

\item[$2^\circ$] In the adapted basis $\theta(X,Y)={\mathbb G}({\mathbb F}X,
Y)$ is (\ref{7.3}).

\item[$3^{\circ}$] Taking into account (\ref{t1.1}), it follows
that $\theta$ is a symplectic structure (i.e. $d\theta=0$).
\end{itemize}

The space $H^{2n}=(\widetilde{TM},{\mathbb G}, {\mathbb F})$ is called
{\it the almost K\"ahlerian model} of the Lagrange space $L^n$. It
has a remarkable property:

\begin{theorem}
\label{t7.2}
The canonical metrical connection $D$ with
coefficients $C\Gamma(N)=(L^{i}_{jk},C^i_{jk})$ of the Lagrange
space $L^n$ is an almost K\"{a}hlerian connection, i.e.
\begin{equation}
\label{7.4}
D{\mathbb G}=0,\quad D{\mathbb F}=0.
\end{equation}
\end{theorem}

Indeed, (\ref{7.1}), (\ref{7.2}), in the adapted basis imply (\ref{7.4}).

We can use this geometrical model to study the geometry of
Lagrange space $L^n$. For instance, the Einstein equations of the
(pseudo) Riemannian space $(\widetilde{TM},{\mathbb G})$ equipped
with the metrical canonical connection $C\Gamma(N)$ are the
Einstein equations of the Lagrange space studied in the previous
section of this chapter.

G.S. Asanov showed \cite{As} that the metric ${\mathbb G}$ given by the lift (\ref{7.2}) does not satisfies the principle of the
Post-Newtonian calculus because the two terms of ${\mathbb G}$ have not the same physical dimensions. This is the
reason to introduce a new lift which can be used in a gauge theory of physical fields.

Let us consider the scalar field:
\begin{equation}
\label{7.5}
{\cal E}=||y||^{2}=g_{ij}(x,y)y^{i}y^{j}.
\end{equation}
$\varepsilon$ is called the absolute energy of the Lagrange space
$L^n$.

We assume $||y||^2>0$ and consider the following lift of the
fundamental tensor $g_{ij}$:
\begin{equation}
\label{7.6}
\buildrel{0}\over{{\mathbb G}}=g_{ij}dx^{i}\otimes dx^{j}+\dfrac{a^{2}}{||y||^{2}}g_{ij}\delta y^{i}\otimes \delta y^{j}
\end{equation}
where $a>0$ is a constant.

Let us consider also the tensor field on $\widetilde{TM}$:
\begin{equation}
\label{7.7}
\buildrel{0}\over{{\mathbb F}}=-\dfrac{||y||}{a}
\dfrac{\partial}{\partial y^i}\otimes dx^j+\dfrac{a}{||y||}
\dfrac{\delta}{\delta x^i}\otimes \delta y^i,
\end{equation}
and 2-form
\begin{equation}
\label{7.8}
\buildrel{0}\over{\theta}=\dfrac{a}{||y||}\theta,
\end{equation}
where $\theta$ is from (\ref{7.3}).

We can prove:

\begin{theorem}
\label{t7.3}
\begin{itemize}
\item[$1^\circ$]  The pair
$(\buildrel{0}\over{{\mathbb G}},\buildrel{0}\over{{\mathbb F}})$ is
an almost Hermitian structure on the manifold $\widetilde{TM}$,
depending only on the the fundamental function $L(x,y)$ of the
space $L^n$.

\item[$2^\circ$]  The almost symplectic structure
$\buildrel{0}\over{\theta}$ associated to the structure
$(\buildrel{0}\over{{\mathbb G}},\buildrel{0}\over{{\mathbb F}})$ is
given by $($\ref{7.8}$)$.

\item[$3^\circ$]  $\buildrel{0}\over{\theta}$ being
conformal to symplectic structure $\theta$, the pair
$(\buildrel{0}\over{{\mathbb G}},\buildrel{0}\over{{\mathbb F}})$ is
conformal to the almost K\"{a}hlerian structure
$({\mathbb G},{\mathbb F})$.
\end{itemize}
\end{theorem}

\section{Generalized Lagrange spaces}\index{Generalized spaces! Lagrange}
\label{s8c2p1}
\setcounter{equation}{0}
\setcounter{definition}{0}
\setcounter{theorem}{0}
\setcounter{lemma}{0}

A first natural generalization of the notion of Lagrange space is
provided by a notion which we call a generalized Lagrange space.
This notion was introduced by author in the paper \cite{mitava1}.

\begin{definition}
\label{d8.1}
A generalized Lagrange space is a pair $GL^n=(M,$ $g_{ij}(x,y))$, where $g_{ij}(x,y)$ is a d-tensor field on
the manifold $\widetilde{TM}$, of type (0,2), symmetric, of rank
$n$ and having a constant signature on $\widetilde{TM}$.

We continue to call $g_{ij}(x,y)$ the {\it fundamental} tensor on
$GL^n$.
\end{definition}

One easily sees that any Lagrange space $L^n=(M,L(x,y))$ is a
ge\-ne\-ralized Lagrange space with the fundamental tensor
\begin{equation}
\label{8.1}
g_{ij}(x,y)=\dfrac{1}{2}\dfrac{\partial^2 L(x,y)}{\partial
y^i\partial y^j}.
\end{equation}

But not any space $GL^n$ is a Lagrange space $L^n$.

Indeed, if $g_{ij}(x,y)$ is given, it may happen that the system
of partial differential equations (\ref{8.1}) does not admits solutions
in $L(x,y)$.

\begin{proposition}
\label{p8.1}
\begin{itemize}
\item[$1^\circ$] A
necessary condition in order that the system $($\ref{8.1}$)$ admit a
solution $L(x,y)$ is that the d-tensor field $\dfrac{\partial
g_{ij}}{\partial y^k}=2C_{ijk}$ be completely symmetric.
\item[$2^\circ$] If the condition $1^\circ$ is verified and the
functions $g_{ij}(x,y)$ are 0-homogeneous with respect to $y^i$;
then the function
\begin{equation}
\label{8.2}
L(x,y)=g_{ij}(x,y)y^i
y^j+A_i(x)y^i+U(x)
\end{equation}
is a solution of the
system of partial differential equations (\ref{8.1}) for any arbitrary
d-covector field $A_i(x)$ and any arbitrary function $U(x)$, on
the base manifold $M$.
\end{itemize}
\end{proposition}

The proof of previous statement is not complicated.

In the case when the system (\ref{8.1}) does not admit solutions in the
functions $L(x,y)$ we say that the generalized Lagrange space
$GL^2=(M,g_{ij(x,y)})$ is not reducible to a Lagrange space.

\begin{remark} The Lagrange spaces $L^n$ with the fundamental
function (\ref{8.2}) give us an important class of Lagrange spaces which
includes the Lagrange space of electrodynamics.
\end{remark}

\noindent{\bf Examples.}
\begin{itemize}
\item[$1^\circ$] The pair $GL^n=(M,g_{ij})$ with the fundamental
tensor field
\begin{equation}
\label{8.3}
g_{ij}(x,y)=e^{2\sigma(x,y)}\gamma_{ij}(x)
\end{equation}
where $\sigma$ is a function on $(\widetilde{TM})$ and
$\gamma_{ij}(x)$ is a pseudo-Riemannian metric on the manifold $M$
is a generalized Lagrange space if the d-covector field
$\dfrac{\partial\sigma}{\partial y^i}$ no vanishes.

It is not reducible to a Lagrange space. R. Miron and R. Tavakol
\cite{mitava1} proved that $GL^n=(M,g_{ij}(x,y))$ defined by (\ref{8.3}) satisfies
the Ehlers - Pirani - Schild's axioms of General Relativity.

\item[$2^\circ$] The pair $GL^n=(M,g_{ij}(x,y))$, with
\begin{equation}
\label{8.4}
g_{ij}(x,y)=\gamma_{ij}(x)+\left(1-\dfrac{1}{n^2(x,y)}\right)y_i
y_j, \quad y_i=\gamma_{ij}(x)y^j
\end{equation}
where $\gamma_{ij}(x)$ is a pseudo-Riemannian metric and
$n(x,y)>1$ is a smooth function ($n$ is a refractive index) give
us a generalized Lagrange space $GL^n$ which is not reducible to a
Lagrange space. This~metric has been called by R. G. Beil, \cite{Beil2} the Miron's metric from Relativistic Optics.
\end{itemize}

The restriction of the fundamental tensor $g_{ij}(x,y)$ (\ref{8.4}) to a
section $S_V:x^i=x^i, y^i=V^i(x)$, ($V^i$ being a vector field) of
the projection $\pi:TM\to M$, is given by $g_{ij}(x,V(x))$. It
gives us the known Synge's metric tensor of the Relativistic
Optics \cite{SJ}.

For a generalized Lagrange space $GL^n=(M,g_{ij}(x,y))$ an
important problem is to determine a nonlinear connection obtained
from the fundamental tensor $g_{ij}(x,y)$. In the particular
cases, given by the previous two examples this is possible. But,
generally no.

We point out a method of determining a nonlinear connection $N$,
strongly connected to the fundamental tensor $g_{ij}$ of the space
$GL^n$, if such kind of nonlinear connection exists.

Consider the {\it absolute energy} ${\cal E}(x,y)$ of space
$GL^n$:
\begin{equation}
\label{8.5}
\varepsilon(x,y)=g_{ij}(x,y) y^i y^j
\end{equation}
$\varepsilon(x,y)$ is a Lagrangian.

The Euler - Lagrange equations of $\varepsilon(x,y)$ are
\begin{equation}
\label{8.6}
\dfrac{\partial\varepsilon}{\partial
x^i}-\dfrac{d}{dt}\dfrac{\partial\varepsilon}{\partial
y^i}=0,\quad y^i=\dfrac{dx^i}{dt}.
\end{equation}
Of course, according the general theory, the energy
$E_\varepsilon$ of the Lagrangian ${\cal E}(x,y)$, is
$E_\varepsilon=y^i\dfrac{\partial \varepsilon}{\partial
y^i}-{\cal E}$ and it is preserved along the integral curves of
the differential equations (\ref{8.6}).

If ${\cal E}(x,y)$ is a regular Lagrangian - we say that the
space $GL^n$ is weakly regular - it follows that the
Euler-Lagrange equations determine a semispray with the
coefficients
\begin{equation}
\label{8.7}
2G^i(x,y)=\buildrel{\vee}\over{g}^{is}\left(\dfrac{\partial^2\varepsilon}{\partial
y^s \partial x^j}y^j-\dfrac{\partial\varepsilon}{\partial
x^s}\right),
\left(\buildrel{\vee}\over{g}_{ij}=\dfrac{1}{2}\dfrac{\partial^2\varepsilon}{\partial
y^i\partial y^j}\right).
\end{equation}

Consequently, the nonlinear connection $N$ with the coefficients
$N^i_j=\dfrac{\partial G^i}{\partial y^j}$ is determined only by
the fundamental tensor $g_{ij}(x,y)$ of the space $GL^n$.

In the case when we can not derive a nonlinear connection from the
fundamental tensor $g_{ij}$, we give a priori a nonlinear
connection $N$ and study the geometry of pair $(GL^n,N)$ by the
methods of the geometry of Lagrange space $L^n$.

For instance, using the adapted basis $\left(\dfrac{\delta}{\delta
x^i},\dfrac{\partial}{\partial y^i}\right)$ to the distributions
$N$ and $V$, respectively and its dual $(dx^i,\delta y^i)$ we can
lift $g_{ij}(x,y)$ to $\widetilde{TM}$:
$$G(x,y)=g_{ij}(x,y)dx^i\otimes dx^j+\dfrac{a}{\|y\|^2}g_{ij}(x,y)\delta y^i\otimes \delta
y^j$$ and can consider the almost complex structure
$${\mathbb F}=-\dfrac{\|y\|}{a}\dfrac{\partial}{\partial y^i}\otimes dx^i+\dfrac{a}{\|y\|}\dfrac{\delta}{\delta x^i}\otimes \delta
y^i$$ with $\varepsilon(x,y)>0$ and
$\|y\|=\varepsilon^{1/2}(x,y)$.

The space $(\widetilde{TM},{\mathbb F})$ is an almost Hermitian space
geometrical associated to the pair $(GL^n,N)$.

J. Silagy, \cite{Szila}, make an exhaustive and interesting study of this difficult problem.

\newpage

\chapter{Finsler Spaces}\index{Spaces! Finsler}
\label{ch3p1}

An important class of Lagrange Spaces is provided by the so-called
Finsler spaces.

The notion of Finsler space was introduced by Paul Finsler in 1918
and was developed by remarkable mathematicians as L. Berwald \cite{berw}, E.
Cartan \cite{cartan}, H. Buseman \cite{buse}, H. Rund \cite{rund}, S.S. Chern  \cite{chern2}, M. Matsumoto \cite{Ma1}, and many others.

This notion is a generalization of Riemann space, which gives an
important geometrical framework in Physics, especially in the geometrical theory of
physical fields, \cite{AIM}, \cite{IS1}, \cite{kondo2}, \cite{Tak}, \cite{Vacaru}.

In the last 40 years, some remarkable books on Finsler geometry
and its applications were published by H. Rund, M. Matsumoto, R.
Miron and M. Anastasiei, A. Bejancu, Abate-Patrizio, D. Bao, S.S.
Chern and Z. Shen, P. Antonelli, R. Ingarden and M. Matsumoto, R.
Miron, D. Hrimiuc, H. Shimada and S. Sab\u{a}u, G.S. Asanov, M.
Crampin, P.L. Antonelli, S. Vacaru, S. Ikeda.

In the present chapter we will study the Finsler spaces
considered as Lagrange spaces and applying the mechanical
principles. This method simplifies the theory of Finsler spaces. So,
we will treat: Finsler metric, Cartan nonlinear connection derived
from the canonical spray, Cartan metrical connection and its
structure equations. Some examples: Randers spaces, Kropina spaces
and some new classes of spaces more general as Finsler spaces: the
almost Finsler Lagrange spaces and the Ingarden spaces will close
this chapter.

\section{Finsler metrics}

\begin{definition}
\label{d3.1}
A Finsler space is a pair $F^n=(M,F(x,y))$
where $M$ is a real $n-$dimensional differentiable manifold and
$F:TM\to R$ is a scalar function which satisfies the following
axioms:
\begin{itemize}
\item[$1^\circ$] $F$ is a differentiable function on
$\widetilde{TM}$ and $F$ is continuous on the null section of the
projection $\pi:TM\to M$.
\item[$2^\circ$] $F$ is a positive
function.
\item[$3^\circ$] $F$ is positively 1-homogeneous with
respect to the variables $y^i$
\item[$4^\circ$] The Hessian of
$F^2$ with the elements:
\begin{equation}
\label{3.1.1}
g_{ij}(x,y)=\dfrac{1}{2}\dfrac{\partial F^2}{\partial y^i\partial
y^j}
\end{equation}
is positively defined on the manifold $\widetilde{TM}$.

Of course, the axiom $4^\circ$ is equivalent with the following:
\item[4$^{\circ\prime}$] The pair $(M,F^2(x,y))=L^n_F$ is a Lagrange space with
positively defined fundamental tensor, $g_{ij}$. $L^n_F$ will be
called the Lagrange space associated to the Finsler space $F^n$. It
follows that all properties of the Finsler space $F^n$ derived
from the fundamental function $F^2$ and the fundamental tensor
$g_{ij}$ of the associated Lagrange space
$L^n_F$.
\end{itemize}
\end{definition}

\medskip

\noindent{\it Remarks}
\begin{itemize}
\item[$1^\circ$] Sometimes we will ask for $g_{ij}$ to be of
constant signature and rank$(g_{ij}(x,y))=n$ on $\widetilde{TM}$.
\item[$2^\circ$] Any Finsler space $F^n=(M,F(x,y))$ is a Lagrange space
$L^n_F=(M,F^2(x,y))$, but not conversely.
\end{itemize}

\medskip

\noindent{\it Examples}
\begin{itemize}
\item[$1^\circ$] A Riemannian manifold $(M,\gamma_{ij}(x))$
determines a Finsler space $F^n=(M,F(x,y))$, where
\begin{equation}
\label{3.1.2}
F(x,y)=\sqrt{\gamma_{ij}(x)y^i y^j}.
\end{equation}

The fundamental tensor is $g_{ij}(x,y)=\gamma_{ij}(x)$.
\item[$2^\circ$] Let us consider, in a preferential local system of
coordinates, the following function:
\begin{equation}
\label{3.1.3}
F(x,y)=\sqrt[4]{(y^1)^4+...+(y^n)^4}.
\end{equation}
Then $F$ satisfy the axioms $1^\circ$-$4^\circ$.
\end{itemize}

\medskip

\begin{remark}This example was given by B. Riemann.
\end{remark}

\begin{itemize}
\item[$3^{\circ}$] Antonelli-Shimada's ecological metric is
given, in a preferential local system of coordinates, on
$\widetilde{TM}$, by
$$F(x,y)=e^{\phi}L,\ \phi =\alpha_{i}x^{i},\ (\alpha _{i}\ {\rm
are\ positive\ constants}),$$ and where
\begin{equation}
\label{3.1.4}
L=\{(y^{1})^{m}+(y^{2})^{m}+\cdots +(y^{n})^{m}\}^{1/m}, m\geq 3,
\end{equation}
$m$ being even.

\item[$4^\circ$] Randers metric is defined by
\begin{equation}
\label{3.1.5}
F(x,y)=\alpha (x,y)+\beta (x,y),
\end{equation}
where $$\alpha^{2}(x,y):=a_{ij}(x)y^{i}y^{j},$$ $(M,a_{ij}(x))$
being a Riemannian manifold and $$\beta (x,y):=b_{i}(x)y^{i}.$$

The fundamental tensor $g_{ij}$ is expressed by:
\begin{equation}
\label{3.1.5'}
\begin{array}{c}
g_{ij}=\dfrac{\alpha +\beta }{\alpha }h_{ij}+d_{i}d_{j},\
h_{ij}:=a_{ij}-\buildrel{0}\over{l}_{i}\buildrel{0}\over{l}_{j},\\ \\
d_{i}=b_{i}+\buildrel{0}\over{l}_{i},\
\buildrel{0}\over{l}_{i}:=\dfrac{\partial \alpha }{\partial
y^{i}}.\tag{3.1.5'}
\end{array}
\end{equation}
\end{itemize}
One can prove that $g_{ij}$ is positively defined under the
condition $b^{2}=a^{ij}b_{i}b_{j}<1$. In this case the pair
$F^n=(M,\alpha+\beta)$ is a Finsler space.

The first example motivates the following theorem:

\begin{theorem}
\label{t3.1}
If the base manifold $M$ is paracompact,
then there exist functions $F:TM \to\R$ such that the pair
$(M,F)$ is a Finsler spaces.
\end{theorem}

Regarding the axioms $1^\circ-4^\circ$ we can see without
difficulties:

\begin{theorem}
\label{t3.2}
The system of axioms of a Finsler space is minimal.
\end{theorem}

Some properties of Finsler space $F^n$:

\begin{itemize}
\item[$1^\circ$] The components of the fundamental tensor $g_{ij}(x,y)$
are 0-ho\-mo\-geneous with respect to $y^i$.

\item[$2^\circ$] The components of 1-form
\begin{equation}
\label{3.1.6}
p_i=\dfrac{1}{2}\dfrac{\partial F^2}{\partial y^i}
\end{equation}
are 1-homogeneous with respect to $y^i$.

\item[$3^{\circ}$] The components of the Cartan tensor
\begin{equation}
\label{3.1.7}
C_{ijk}=\dfrac{1}{4}\dfrac{\partial^{3}F^2}{\partial y^{i}\partial y^{j}\partial y^{k}}=\dfrac{1}{2}\dfrac{\partial g_{ij}}{\partial y^{k}}
\end{equation}
are $-1-$homogeneous with respect to $y^i$.
\end{itemize}

Consequently we have
\begin{equation}
\label{3.1.8}
C_{oij}=y^s C_{sij}=0.
\end{equation}

If $X^i$ and $Y^i$ are d-vector fields, then

$\|X\|^2:=g_{ij}(x,y)X^i Y^i$ is a scalar field.

$<X,Y>:=g_{ij}(x,y)X^i Y^j$ is a scalar field.

Assuming $\|X\|_u\neq 0$, $\|Y\|_u\neq 0$ the angle
$\varphi=\angle(X,Y)$ at a point $u\in\widetilde{TM}$ is given by
$$\cos\varphi=\dfrac{<X,Y>(u)}{\|X\|_u\cdot\|Y\|_u}.$$

The vectors $X_u,Y_u$ are orthogonal if $\langle X,Y\rangle(u)=0$.

\begin{proposition}
\label{p3.1}
In a Finsler space $F^n$ the following
identities hold:
\begin{itemize}
\item[$1^\circ$]
$F^2(x,y)=g_{ij}(x,y)y^i y^j$
\item[$2^\circ$] $p_i
y^i=F^2$.
\end{itemize}
\end{proposition}

\begin{proposition}
\label{p3.2}
\begin{itemize}
\item[$1^\circ$] The 1-form
\begin{equation}
\label{3.1.9}
\omega=p_i dx^i
\end{equation}
is globally defined on $\widetilde{TM}$.
\item[$2^\circ$] The 2-form
\begin{equation}
\label{3.1.10}
\theta=d\omega=dp_i\wedge dx^i
\end{equation}
is globally defined on $\widetilde{TM}$.
\item[$3^\circ$] $\theta$ is a symplectic structure on $\widetilde{TM}$.
\end{itemize}
\end{proposition}

\begin{definition}
\label{d1.2}
A Finsler space $F^n=(M,F)$ is called reducible to a Riemann space if its fundamental tensor
$g_{ij}(x,y)$ does not depend on the variable $y^i$.
\end{definition}

\begin{proposition}
\label{p3.3}
A Finsler space $F^n$ is reducible to
a Riemann space if and only if the tensor $C_{ijk}$ is vanishing
on $\widetilde{TM}$.
\end{proposition}

\section{Geodesics}\index{Geodesics}
\setcounter{equation}{0}
\setcounter{definition}{0}
\setcounter{theorem}{0}
\setcounter{lemma}{0}

In a Finsler space $F^{n}=(M,F(x,y))$ one can define the notion
of arc length of a smooth curve.

Let $c$ be a parametrized curve in the manifold $M$:
\begin{equation}
\label{3.2.1}
c:t\in \lbrack 0,1]\to (x^{i}(t))\in U\subset M
\end{equation}
$U$ being a domain of a local chart in $M$.

The extension $\widetilde{c}$ of $c$ to $\widetilde{TM}$ has the equations
\begin{equation}
\label{3.2.1'}
x^{i}=x^{i}(t),\ y^{i}=\dfrac{dx^{i}}{dt}(t),\ t\in \lbrack
0,1].\tag{3.2.1'}
\end{equation}
Thus the restriction of the fundamental function $F(x,y)$ to
$\widetilde{c}$ is $F(x(t)$, $\dfrac{dx}{dt}(t))$, $t\in [0,1]$.

We define the ``length'' of curve $c$ with extremities $c(0),c(1)$
by the number
\begin{equation}
\label{3.2.2}
{\ell}(c)=\int_{0}^{1}F(x(t),\dfrac{dx}{dt}(t))dt.
\end{equation}

The number ${\ell}(c)$ does not depend on a changing of coordinates on $\widetilde{TM}$ and, by means of 1-homogeneity of function $F$,
${\ell}(c)$ does not depend on the parametrization of the curve
$c$. ${\ell}(c)$ depends on the curve $c$, only.

We can choose a canonical parameter on $c$, considering the
following function $s=s(t)$, $t\in[0,1]$:
$$s(t)=\int_{0}^t F(x(\tau),\dfrac{dx}{dt}(\tau))d\tau.$$

This function is derivable and its derivative is
$$\dfrac{ds}{dt}=F(x(t),\dfrac{dx}{dt}(t))>0,\ t\in (0,1).$$
So the function $s=s(t)$, $t\in [0,1]$, is invertible. Let
$t=t(s)$, $s\in[s_0,s_1]$ be its inverse. The change of parameter
$t\to s$ has the property
\begin{equation}
\label{3.2.3}
F\left(x(s), \dfrac{dx}{ds}(s)\right)=1.
\end{equation}

Variational problem on the functional $\ell$ will gives the curves
on $\widetilde{TM}$ which extremize the arc length. These
curves are geodesics of the Finsler space $F^n$.

So, the solution curves of the Euler-Lagrange equations:
\begin{equation}
\label{3.2.4}
\dfrac{d}{dt}\left(\dfrac{\partial F}{\partial y^i }\right)-\dfrac{\partial F}{\partial x^i}=0,\
y^i=\dfrac{dx^i}{dt}
\end{equation}
are the geodesics of the space $F^n$.

\begin{definition}
\label{d6.1}
The curves $c=(x^i(t))$, $t\in[0,1]$,
solutions of the Euler-Lagrange equations (\ref{3.2.4}) are called the
{\it geodesics} of the Finsler space $F^n$.
\end{definition}

The system of differential equations (\ref{3.2.4}) is equivalent to the
following system $$\dfrac{d}{dt}\dfrac{\partial F^2}{\partial
y^i}-\dfrac{\partial F^2}{\partial x^i}=2\dfrac{dF}{dt}\dfrac{\partial F}{\partial y^i},\quad
y^i=\dfrac{dx^i}{dt}.$$

In the canonical parametrization, according to (\ref{3.2.3}) we have:

\begin{theorem}
\label{2.1}
In the canonical parametrization the
geodesics of the Finsler space $F^{n}$ are given by the system of
differential equations
\begin{equation}
\label{3.2.5}
E_i(F^2):=\dfrac{d}{ds}\dfrac{\partial F^2}{\partial y^i}-\dfrac{\partial F^2}{\partial x^i}=0,
y^i=\dfrac{dx^i}{ds}.
\end{equation}
\end{theorem}

Now, remarking that $F^2=g_{ij}y^i y^j$, the previous equations
can be written in the form:
\begin{equation}
\label{3.2.6}
\dfrac{d^{2}x^{i}}{ds^{2}}+\gamma^{i}_{jk}\left(x,\dfrac{dx}{ds}\right)\dfrac{dx^j}{ds}\dfrac{dx^k}{ds}=0,\ y^{i}=\dfrac{dx^{i}}{ds},
\end{equation}
where $\gamma^i_{jk}$ are the Christoffel symbols of the
fundamental tensor $g_{ij}$:
\begin{equation}
\label{3.2.7}
\gamma^{i}{}_{jk}=\dfrac{1}{2}g^{ir}\left(\dfrac{\partial
g_{rk}}{\partial x^{j}}+\dfrac{\partial g_{jr}}{\partial
x^{k}}-\dfrac{\partial g_{jk}}{\partial x^{r}}\right).
\end{equation}

A theorem of existence and uniqueness of the solution of differential
equations (\ref{3.2.6}) can be formulated in the classical manner.

\section{Cartan nonlinear connection}\index{Cartan! nonlinear connection}
\setcounter{equation}{0}
\setcounter{definition}{0}
\setcounter{theorem}{0}
\setcounter{lemma}{0}

Considering the Lagrange space $L^n_F=(M,F^2)$ associated to the
Finsler space $F^n=(M,F)$ we can obtain some main geometrical
object field of $F^n$.

So, Theorem 2.3.1, affirms:

\begin{theorem}
\label{t3.1}
For the Finsler space $F^n$ the equations
$$g^{ij}E_j(F^2):=g^{ij}\left(\dfrac{d}{dt}\dfrac{\partial F^2}{\partial
y^j}\right)-\dfrac{\partial F^2}{\partial x^j}=0, y^i=\dfrac{dx^i}{dt}$$ can be written in the form
\begin{equation}
\label{3.3.1}
\dfrac{d^{2}x^{i}}{dt^{2}}+2G^{i}\left(x,\dfrac{dx}{dt}\right)=0,
y^i=\dfrac{dx^i}{dt}
\end{equation}
where
\begin{equation}
\label{3.3.1'}
2G^i(x,y)=\gamma^i_{jk}(x,y)y^j y^k\tag{3.3.1'}
\end{equation}
\end{theorem}

Consequently the equations (\ref{3.3.1}) give the integral curves of the
semispray:
\begin{equation}
\label{3.3.2}
S=y^{i}\dfrac{\partial}{\partial
x^{i}}-2G^{i}(x,y)\dfrac{\partial}{\partial y^{i}}.
\end{equation}

Since $G^{i}$ are 2-homogeneous functions with respect to $y^i$ it
follows that $S$ is a {\it spray}.

$S$ determine a canonical nonlinear connection $N$ with the
coefficients
\begin{equation}
\label{3.3.3}
N^i{}_j=\dfrac{\partial G^i}{\partial
y^j}=\dfrac{1}{2}\dfrac{\partial}{\partial
y_j}\{\gamma^i{}_{rs}(x,y) y^r y^s\}.
\end{equation}
$N$ is called the Cartan nonlinear connection of the space $F^n$.

The tangent bundle $T(TM)$, the horizontal distribution $N$ and
the vertical distribution $V$ give us the direct decomposition of
vectorial spaces
\begin{equation}
\label{3.3.4}
T_u(\widetilde{TM})=N(u)\oplus V(u),\quad \forall
u\in\widetilde{TM}.
\end{equation}

The adapted basis to $N$ and $V$ is $\left(\dfrac{\delta}{\delta
x^i},\dfrac{\partial}{\partial y^i}\right)$ and the dual adapted
basis is $(dx,\delta y^i)$, where
\begin{equation}
\label{3.3.4'}
\left\{\begin{array}{l} \dfrac{\delta}{\delta
x^i}=\dfrac{\partial}{\partial
x^i}-N^j_i(x,y)\dfrac{\partial}{\partial y^j} \vspace{3mm}\\
\delta y^i =dy^i+N^i_j(x,y)dx^j.
\end{array}\right.
\tag{3.3.4'}
\end{equation}

One obtains:

\begin{theorem}
\label{t3.2}
\begin{itemize}
\item[$1^\circ$] The
horizontal curves in $F^n$ are given by $$x^i=x^i(t),\quad
\dfrac{\delta y^i}{dt}=0.$$
\item[$2^\circ$] The autoparallel
curves of the Cartan nonlinear connection $N$ coincide to
the integral curves of the spray $S$, $($\ref{3.3.2}$)$.
\end{itemize}
\end{theorem}

\section{Cartan metrical connection}\index{Cartan! metrical connection}
\setcounter{equation}{0}
\setcounter{definition}{0}
\setcounter{theorem}{0}
\setcounter{lemma}{0}

Let $N(N^i_j)$ be the Cartan nonlinear connection of the Finsler
space $F^n$. According to section 5 of chapter 2 one introduces
the canonical metrical $N-$linear connection of the space $F^n$.

But, for these spaces the system of axioms, from Theorem 2.5.1 can be given in the Matsumoto's form \cite{Ma1}, \cite{Mats}.

\begin{theorem}
\label{t4.1}
\begin{itemize}
\item[$1^\circ$] On the
manifold $\widetilde{TM}$, for any Finsler space $F^n=(M,F)$ there
exists only one linear connection $D$, with the coefficients
$C\Gamma=(N^i_j,F^i_{jk},C^i_{jk})$ which verifies the following
axioms:

$A_1.$ The deflection tensor field $D^i_j=y^i_{|j}$ vanishes.

$A_2.$ $g_{ij|k}=0$, ($D$ is $h-$metrical).

$A_3.$ $g_{ij}|_k=0$, ($D$ is $v-$metrical).

$A_4.$ $T^i_{\ jk}=0$, ($D$ is $h-$torsion free).

$A_5.$ $S^i_{\ jk}=0$, ($D$ is $v-$torsion free).

\item[$2^\circ$] The coefficients $(N^i_j,F^i_{jk},C^i_{jk})$ are
as follows:

a. $N^i_{j}$ are the coefficients (3.3.3) of the Cartan nonlinear
connection.

b. $F^i_{jk},C^i_{jk}$ are expressed by the generalized
Christoffel symbols:
\begin{equation}
\label{3.4.1}
\begin{array}{l}
F^i_{jk}=\dfrac{1}{2}g^{is} \left(\dfrac{\delta g_{sk}}{\delta
x^j} + \dfrac{\delta g_{js}}{\delta x^k}-\dfrac{\delta
g_{jk}}{\delta x^s}\right)\vspace{3mm}
\\
C^i_{jk}=\dfrac{1}{2}g^{is} \left(\dfrac{\partial g_{sk}}{\partial
y^j} + \dfrac{\partial g_{js}}{\partial y^k}- \dfrac{\partial
g_{jk}}{\partial y^s}\right)
\end{array}
\end{equation}
\item[$3^\circ$] This connection depends only on the fundamental
function $F$.\end{itemize}
\end{theorem}

A proof can be find in the books \cite{Ant2}, \cite{mirana}.

The previous connection is named the {\it Cartan metrical connection} of
the Finsler space $F^n$.

Now we can develop the geometry of Finsler spaces, exactly as the
geometry of the associated Lagrange space $L^n_F=(M,F^2)$.

Also, in the case of Finsler space the geometrical model
$H^{2n}=(\widetilde{TM},\G,\F)$ is an almost
K\"{a}hlerian space.

A very good example is provided by the Randers space, introduced by
R.S. Ingarden.

A Randers space is the Finsler space $F^n=(M,\alpha+\beta)$
equipped with Cartan nonlinear connection $N$.

It is denoted by $RF^n=(M,\alpha+\beta,N)$.

The geometry of these spaces was much studied by many
geometers. A good monograph in this respect is the D. Bao, S.S.
Chern and Z. Shen's book \cite{BCS1}.

The Randers spaces $RF^n$ can be generalized considering the
Finsler spaces $F^n=(M,\alpha+\beta)$, where $\alpha(x,y)$ is the
fundamental function of a Finsler space $F^{\prime n}=(M,\alpha)$.
The Finsler space $F^n=(M,\alpha+\beta)$ equipped with the Cartan
nonlinear connection $N$ of the space $F^{\prime n}=(M,\alpha)$ is
a generalized Randers space \cite{mir2}, \cite{BCS1}. Evidently, this notion has
some advantages, since we can take some remarkable Finsler spaces
$F^{\prime n}$, (M. Anastasiei $\beta-$transformation of a Finsler space \cite{AnShi}).

As an application of the previous notions, we define the notion of
Ingarden space $IF^n$, \cite{ingarden2}, \cite{BuMi}. This is the Finsler space
$F^n=(M,\alpha+\beta)$ equipped with the nonlinear connection
$N=\gamma^i_{jk}(x)y^k-F^i_{j}(x)$, $\gamma^i_{jk}(x)$ being the
Christoffel symbols of the Riemannian metric $a_{ij}(x)$, which
defines $\alpha^2=a_{ij}(x)y^i y^j$ and the electromagnetic tensor
$F^i_j(x)$ determined by $\beta=b_i(x)y^i$. While the spaces
$RF^n$ have not the electromagnetic field ${\cal
F}=\dfrac{1}{2}(D_{ij}-D_{ji})$, the Ingarden spaces have such
kind of tensor fields and they give us the remarkable Maxwell
equations, \cite{BuMi}. Also, the autoparallel curve of the nonlinear
connection $N$ are given by the known Lorentz equations.

An example of a special Lagrange space derived from a Finsler one,
\cite{mirana} is as follows:

Let us consider the Lagrange space $L^n=(M,L(x,y))$ with the
fundamental function $$L(x,y)=F^2(x,y)+\beta,$$ where $F$ is the
fundamental function of a priori given Finsler space $F^n=(M,F)$
and $\beta=b_i(x)y^i$.

These spaces have been called {\it the almost Finsler Lagrange
Spaces} (shortly AFL-spaces), \cite{mirana}, \cite{BuMi}. They generalize the
Lagrange space from Electrodynamics.

Indeed, the Euler - Lagrange equations of AFL-spaces are exactly
the Lorentz equations
$$\dfrac{d^2
x^i}{dt^2}+\gamma^i_{{j}{k}}(x,y)\dfrac{dx^j}{dt}\dfrac{dx^k}{dt}=\dfrac{1}{2}F^i_j(x)\dfrac{d
x^j}{dt}.$$

To the end of these three chapter we can do a general remark:

The class of Riemann spaces $\{{\cal R}^n\}$ is a subclass of the
class of Finsler spaces $\{F^n\}$, the class $\{F^n\}$ is a
subclass of class of Lagrange spaces $\{L^n\}$ and this is a
subclass of class of generalized Lagrange spaces $\{GL^n\}$. So,
we have the following sequence of inclusions:
\begin{itemize}
\item[(I)] $\quad \{{\cal
R}^n\}\subset\{F^n\}\subset\{L^n\}\subset\{GL^n\}$.
\end{itemize}

Therefore, we can say: The Lagrange geometry is the geometric
study of the terms of the sequence of inclusions (I).

\newpage

\chapter{The Geometry of Cotangent Manifold}\index{Cotangent manifold}
\label{ch4p1} 

The geometrical theory of cotangent bundle $(T^* M$, $\pi^*,M)$ on a real, finite dimensional manifold $M$ is important in differential geometry. Correlated with that of tangent bundle $(TM,\pi,M)$, introduced in Ch. 1, one obtains a framework for construction of Lagrangian and Hamiltonian mechanical systems, as well as, for the duality between them, via Legendre transformation.

We study here the fundamental geometric objects on $T^*M$, as Liouville vector field $C^*$, Liouville 1-form $\omega$, symplectic structure $\theta=d\oo$, Poisson structure, $N-$linear connection etc.

\section{Cotangent bundle}\index{Bundles! Cotangent $T^*M$}
\label{s1c4}

Let $M$ be a real $n-$dimensional differentiable manifold. Its cotangent bundle $(T^*M$, $\pi^*$, $M)$ can be constructed as the dual of the tangent bundle $(TM,\pi,M)$, \cite{mir4}. If $(x^i)$ is a local coordinate system on a domain $U$ of a chart on $M$, the induced system of coordinates on $\pi^{*-1}(U)$ is $(x^i,p_i)$, $(i,j,k,..=1,2,...,n)$, $p_1,...,p_n$ are called ``momentum variables''. We denote $(x^i,p_i)=(x,p)=u$.

A change of coordinates on $T^*M$ is given by
\begin{equation}
\label{4.1.1}
\left\{\begin{array}{l}
\widetilde{x}^i=\widetilde{x}^i(x^1,...,x^n),\ {\rm{rank}}\left(\dfrac{\pp\widetilde{x}^i}{\pp x^j}\right)=n\\ \\
\widetilde{p}_i=\dfrac{\pp x^j}{\pp\widetilde{x}^i}p_j.
\end{array}\right.
\end{equation}
The natural frame $\left(\dfrac{\pp}{\pp x^i},\dfrac{\pp}{\pp p_i}\right)=(\dot{\pp}_i,\dot{\pp}^i)$ is transformed by (\ref{4.1.1}) in the form
\begin{equation}
\label{4.1.2}
\pp_i=\dfrac{\pp\widetilde{x}^j}{\pp x^i}\widetilde{\pp}_j+\dfrac{\pp\widetilde{p}_j}{\pp x^i}\dot{\pp}^j,\ \dot{\pp}^i=\dfrac{\pp\widetilde{x}^i}{\pp x^j}\dot{\pp}^j.
\end{equation}

The natural coframe $(dx^i,dp_i)$ is changed by the rule
\begin{equation}
\label{4.1.2'}
d\widetilde{x}^i=\dfrac{\pp\widetilde{x}^i}{\pp x^j}dx^j;\ d\widetilde{p}_i=\dfrac{\pp x^j}{\pp\widetilde{x}^i}dp_j+\dfrac{\pp^2 x^j}{\pp\widetilde{x}^i\pp\widetilde{x}^k}p_j d\widetilde{x}^k.\tag{4.1.2'}
\end{equation}

The Jacobian matrix of change of coordinates (\ref{4.1.1}) is $$J(u)=\left(\begin{array}{cc}\dfrac{\pp\widetilde{x}^i}{\pp x^j} & 0\\ \\ \dfrac{\pp\widetilde{p}_j}{\pp x^i} & \dfrac{\pp x^j}{\pp\widetilde{x}^i}\end{array}\right)_u.$$ It follows $$\det J(u)=1{\rm{\ for\ every\ }}u\in T^*M.$$ So we get:
\begin{theorem}
\label{t4.1.1}
The differentiable manifold $T^*M$ is orientable.
\end{theorem}

One can prove that if $M$ is a paracompact manifold, then $T^*M$ is paracompact, too.

The kernel of differential $d\pi^*:TT^*M\to T^*M$ is the vertical subbundle $VT^*M$ of the tangent bundle $TT^*M$. The fibres $V_u$ of $VT^*M$, $\forall u\in T^*M$ determine a distribution $V$ on $T^*M$, called the {\it vertical} distribution. It is locally generated by the tangent vector fields $(\dot{\pp}^1,...,\dot{\pp}^n)$. Consequently, $V$ is an integrable distribution of local dimension $n$.

Noticing the formulae (\ref{4.1.2}), (\ref{4.1.2'}) one can introduce the following geometrical object fields:
\begin{equation}
\label{4.1.3}
C^*=p_i\dot{\pp}^i,
\end{equation}
called the Liouville--Hamilton vector field on $T^*M$,
\begin{equation}
\label{4.1.4}
\omega=p_i dx^i,
\end{equation}
called the Liouville 1-form on $T^*M$,
\begin{equation}
\label{4.1.5}
\theta=d\omega=dp_i\wedge dx^i,
\end{equation}
$\theta$ is a symplectic structure on $T^*M$.
All these geometrical object fields do not depend on the change of coordinates (\ref{4.1.1}).

The Poisson brackets $\left\{,\right\}$ on the manifold $T^*M$ are defined by
\begin{equation}
\label{4.1.6}
\{f,g\}=\dfrac{\pp f}{\pp p_i}\dfrac{\pp g}{\pp x^i}-\dfrac{\pp g}{\pp p_i}\dfrac{\pp f}{\pp x^i},\ \forall f,g\in {\cal F}(T^*M).
\end{equation}
Of course, $\{f,g\}\in{\cal F}(T^*M)$ and $\{f,g\}$ does not depend on the change of coordinates (\ref{4.1.1}).

Also, the following properties hold:
\begin{itemize}
\item[$1^\circ$] $\{f,g\}=-\{g,f\}$,
\item[$2^\circ$] $\{f,g\}$ is $R-$linear in every argument,
\item[$3^\circ$] $\{\{f,g\},h\}+\{\{g,h\},f\}+\{\{h,f\},g\}=0$ (Jacobi identity),
\item[$4^\circ$] $\{\cdot,gh\}=\{\cdot,g\}h+\{\cdot,h\}g$.
\end{itemize}
The pair $\{{\cal F}(T^*M),\{,\}\}$ is a Lie algebra, called the Poisson--Lie algebra.

The relation between the structures $\theta$ and $\{,\}$ can be given by means of the notion of Hamiltonian system.

\begin{definition}
\label{d4.1.1}
A differentiable Hamiltonian is a real function $H:T^*M\to R$ which is of $C^\9$ class on $\widetilde{TM^*}=T^*M\setminus\{0\}$  and continuous on $T^*M$.
\end{definition}

An example: Let $g_{ij}(x)$ be a $C^\9-$Riemannian structure on $M$. Then $H=\sqrt{g^{ij}(x)p_i p_j}$ is a differentiable Hamiltonian on $T^*M$.

\begin{definition}
A Hamiltonian system is a triple $(T^*M,\theta,H)$.
\end{definition}

Let us consider the ${\cal F}(T^*M)-$modules $\chi(T^*M)$ (tangent vector fields on $T^* M$), $\chi^*(T^*M)$ (cotangent vector fields on $T^*M$).

The following ${\cal F}(T^*M)-$linear mapping $$S_\theta=\chi(T^*M)\to \chi^*(TM)$$ can be defined by
\begin{equation}
\label{4.1.7}
S_\theta(X)=i_X \theta,\ \ \forall X\in\chi(T^*M).
\end{equation}
One proves, without difficulties:

\begin{proposition}
\label{p4.1.1}
$S_\theta$ is an isomorphism. We have:
\begin{equation}
\label{4.1.6'}
S_\theta\left(\dfrac{\pp}{\pp x^i}\right)=-dp_i,\ S_\theta\left(\dfrac{\pp}{\pp p_i}\right)=dx^i\tag{4.1.6'}
\end{equation}
\begin{equation}
\label{4.1.7'}
S_\theta(C^*)=\omega.\tag{4.1.7'}
\end{equation}
\end{proposition}

\begin{theorem}
\label{t4.1.2}
The following properties of the Hamiltonian system $(T^*M,\theta,H)$ hold:
\begin{itemize}
\item[$1^\circ$] There exists a unique vector field $X_H\in{\cal X}(T^* M)$ for which:
\begin{equation}
\label{4.1.8}
i_H\theta=-dH.
\end{equation}
\item[$2^\circ$] The integral curves of the vector field $X_H$ are given by the Hamilton--Jacobi equations:
\begin{equation}
\label{4.1.9}
\dfrac{dx^i}{dt}=\dfrac{\pp H}{\pp p_i},\ \dfrac{dp_i}{dt}=-\dfrac{\pp H}{\pp x^i}.
\end{equation}
\end{itemize}
\end{theorem}

\begin{proof}
\begin{itemize}
\item[$1^\circ$] The existence and uniqueness of the vector field $X_H$ is assured by the isomorphism $S_\theta$:
\begin{equation}
\label{4.1.10}
X_H=S^{-1}_\theta(-dH).
\end{equation}
$X_H$ is called the {\it Hamiltonian vector} field.
\item[$2^\circ$] The local expression of $X_H$ is given (by (\ref{4.1.6})):
\begin{equation}
\label{4.1.11}
X_H=\dfrac{\pp H}{\pp p_i}\dfrac{\pp}{\pp x^i}-\dfrac{\pp H}{\pp x^i}\dfrac{\pp}{\pp p_i}.
\end{equation}
\end{itemize}
Consequently: the integral curves of the vector field $X_H$ are given by the Hamilton--Jacobi equations (\ref{4.1.8}).
\end{proof}

Along the integral curves of $X_H$, we have $$\dfrac{dH}{dt}=\{H,H\}=0.$$ Thus: {\it The differentiable Hamiltonian $H(x,p)$ is constant along the integral curves of the Hamilton vector field $X_H$.}

The structures $\theta$ and $\{,\}$ have a fundamental property:

\begin{theorem}
\label{t4.1.3}
The following formula holds:
\begin{equation}
\label{4.1.12}
\{f,g\}=\theta(X_f,X_g),\ \forall f,g\in{\cal X}^*(T^*M), \forall X\in {\cal X}(T^*M).
\end{equation}
\end{theorem}

\begin{proof}
From (\ref{4.1.10}): $$\{f,g\}=X_f g=-X_g f=-df(X_g)=(i_{X_f}\theta)(X_g)=\theta(X_f,X_g).$$
\end{proof}

As a consequence, we obtain:
\begin{equation}
\label{4.1.13}
\dfrac{dx^i}{dt}=\{H,x^i\},\ \dfrac{dp_i}{dt}=\{H,p_i\}.
\end{equation}

It is remarkable the Jacobi method for integration of Hamilton--Jacobi equations (\ref{4.1.9}). Namely, we look for a solution curves $\gamma(t)$ in $T^*M$ of the form
\begin{equation}
\label{4.1.14}
x^i=x^i(t),\ p_i=\dfrac{\pp S}{\pp x^i}(x(t))
\end{equation}
where $S\in{\cal F}(M)$. Substituting in (\ref{4.1.9}), we have
\begin{equation}
\label{4.1.15}
\dfrac{dx^i}{dt}=\dfrac{\pp H}{\pp p_i}(x(t))\dfrac{\pp S}{\pp x^i}(x(t));\ \dfrac{dp_i}{dt}=\dfrac{\pp^2 S}{\pp x^i\pp x^j}\dfrac{\pp H}{\pp p_j}=-\dfrac{\pp H}{\pp x^i}.
\end{equation}
It follows $$dH\left(x,\dfrac{\pp S}{\pp x}\right)=\left(\dfrac{\pp H}{\pp x^i}\dfrac{\pp H}{\pp p_i}-\dfrac{\pp H}{\pp p_i}\dfrac{\pp H}{\pp x^i}\right)dt=0.$$ Consequently, $H\left(x,\dfrac{\pp S}{\pp x}\right)=$const, which is called the Hamilton--Jacobi equation of Mechanics. This equation determines the function $S$ and the first equation (\ref{4.1.14}) gives us the curves $\gamma(t)$.

The Jacobi method suggests to obtain the Hamilton--Jacobi equation by the variational principle.

\section{Variational problem. Hamilton--Jacobi equations}\index{Variational problem! Hamilton--Jacobi equations}
\setcounter{equation}{0}
\setcounter{definition}{0}
\setcounter{theorem}{0}
\setcounter{lemma}{0}

The variational problem for the Hamiltonian systems $(T^*M,\theta,H)$ is defined as follows:

Let us consider a smooth curve $c$ defined on a local chart $\pi^{*-1}(U)$ of the cotangent manifold $T^*M$ by: $$c:t\in[0,1]\to (x(t);p(t))\in\pi^{*-1}(U)$$ analytically expressed by
\begin{equation}
\label{4.2.1}
x^i=x^i(t),\ p_i=p_i(t),\ \ t\in[0,1].
\end{equation}
Consider a vector field $V^i(t)$ and a covector field $\eta_i(t)$ on the domain $U$ of the local chart $(U,\varphi)$ on $M$, and assume that we have
\begin{equation}
\label{4.2.2}
\begin{array}{l}
V^i(0)=V^i(1)=0, \ \eta_i(0)=\eta_i(1)=0\\ \\
\dfrac{dV^i}{dt}(0)=\dfrac{dV^i}{dt}(1)=0.
\end{array}
\end{equation}

The variations of the curve $c$ determined by the pair $(V^i(t),\eta_i(t))$ are defined by the curves $\overline{c}(\varepsilon_1,\varepsilon_2)$:
\begin{equation}
\label{4.2.3}
\begin{array}{l}
\overline{x}^i=x^i(t)+\varepsilon_1 V^i(t),\\ \\
\overline{p}_i=p_i(t)+\varepsilon_2\eta_i(t),\ t\in[0,1],
\end{array}
\end{equation}
when $\varepsilon_1$ and $\varepsilon_2$ are constants, small in absolute value such that the image of any curve $\overline{c}$ belongs to the open set $\pi^{*-1}(U)$ in $T^*M$.

The integral of action of Hamiltonian $H(x,p)$ along the curve $c$ is defined by
\begin{equation}
\label{4.2.4}
I(c)=\dint^1_0\left[p_i(t)\dfrac{dx^i}{dt}-H(x(t),p(t))\right]dt.
\end{equation}

The integral of action $I(\overline{c}(\varepsilon_1,\varepsilon_2))$ is:
\begin{equation}
\label{4.2.5}
\begin{array}{c}
I(\overline{c}(\varepsilon_1,\varepsilon_2))=\dint^1_0\left[(p_i+\varepsilon_2\eta_i(t))\left(\dfrac{dx^i}{dt}+
\varepsilon_1\dfrac{dV^i}{dt}\right)-\right.\\ \\
\left.-H(x+\varepsilon_1V,p+\varepsilon_2\eta)\right]dt.
\end{array}
\end{equation}

The necessary conditions in order that $I(c)$ is an extremal value of $I(\overline{c}(\vare_1,\vare_2))$ are
\begin{equation}
\label{4.2.6}
\left.\dfrac{\pp I(\ov{c}(\vare_1,\vare_2))}{\pp \vare_1}\right|_{\vare_1=\vare_2=0}=0, \
\left.\dfrac{\pp I(\ov{c}(\vare_1,\vare_2))}{\pp \vare_2}\right|_{\vare_1=\vare_2=0}=0.
\end{equation}

Under our conditions of differentiability, the operators $\dfrac{\pp}{\pp\vare_1}$, $\dfrac{\pp}{\pp\vare_2}$ and the operator of integration commute. Therefore, from (\ref{4.2.5}) we deduce:
\begin{equation}
\label{4.2.7}
\begin{array}{l}
\dint^1_0\left[p_i(t)\dfrac{dV^i}{dt}(t)-\dfrac{\pp H}{\pp x^i}V^i\right]dt=0\\ \\
\dint^1_0\left(\dfrac{dx^i}{dt}-\dfrac{\pp H}{\pp p_i}\right)\eta_i dt=0.
\end{array}
\end{equation}
Denoting:
\begin{equation}
\label{4.2.8}
\overset{\circ}{E}_i(H)=\dfrac{dp_i}{dt}+\dfrac{\pp H}{\pp x^i}
\end{equation}
and noticing the conditions (\ref{4.2.2}) one obtain that the equations (\ref{4.2.7}) are equivalent to:
\begin{equation}
\label{4.2.9}
\dint^1_0\overset{\circ}{E}_i(H)V^i dt=0;\ \dint^1_0\left[\dfrac{dx^i}{dt}-\dfrac{\pp H}{\pp p_i}\right]\eta_i dt=0.
\end{equation}
But $(V^i,\eta_i)$ are arbitrary. Thus, (\ref{4.2.9}) lead to the following result:
\begin{theorem}
\label{t4.2.1}
In order to the integral of action $I(c)$ to be an extremal value for the functionals $I(\ov{c})$ is necessary that the curve $c$ to satisfy the Hamilton--Jacobi equations:
\begin{equation}
\label{4.2.10}
\dfrac{dx^i}{dt}=\dfrac{\pp H}{\pp p_i},\ \dfrac{dp_i}{dt}=-\dfrac{\pp H}{\pp x^i}.
\end{equation}
\end{theorem}
The operator $\ce_i(H)$ has a geometrical meaning:
\begin{theorem}
\label{t4.2.2}
$\ce_i(H)$ is a covector field.
\end{theorem}

\begin{proof}
With respect to a change of local coordinates $(4.1.1)$ on the manifold $T^*M$, we have $$\begin{array}{c}
\dint^1_0\widetilde{\ce_i}(\widetilde{H})\widetilde{V}^i dt-\dint^1_0\ce_i(H)V^i dt=\\ \\
=\dint^1_0\left[\widetilde{\ce_i}(\widetilde{H})\dfrac{\pp\widetilde{x}^i}{\pp x^j}-\ce_j(H)\right]V^j dt=0.\end{array}$$ Since the vector $V^i$ is arbitrary, we obtain: $$\widetilde{\ce_i}(\widetilde{H})\dfrac{\pp\widetilde{x}^i}{\pp x^j}=\ce_j(H).\qed$$
\end{proof}

A consequence of the last theorem: The Hamilton--Jacobi equation have a geometrical meaning on the cotangent manifold $T^*M$.

The notion of homogeneity for the Hamiltonian systems is defined in the classical manner \cite{MHS}.

A smooth function $f:T^*M\to R$ is called $r-$homogeneous $r\in\Z$ with respect to the momenta variables $p_i$ if we have:
\begin{equation}
\label{4.2.11}
{\cal L}_{\C^*}f=\C^* f=p:\dot{\pp}^i f=rf.
\end{equation}
A vector field $X\in\chi(T^*M)$ is $r-$homogeneous if
\begin{equation}
\label{4.2.12}
{\cal L}_{\C^*}X=(r-1)X,
\end{equation}
where ${\cal L}_{\C^*}X=[\C^*;X]$.
So, we have
\begin{itemize}
\item[$1^\circ$] $\dfrac{\pp}{\pp x^i}$, $\dfrac{\pp}{\pp p_i}=\dot{\pp}^i$ are 1- and 0-homogeneous.
\item[$2^\circ$] If $f\in{\cal F}(T^*M)$ is $s-$homogeneous and $X\in\chi(T^*M)$ is $r-$ho\-mo\-ge\-neous then $fX$ is $r+s-$homogeneous.
\item[$3^\circ$] $\C^*=p_i\dot{\pp}^i$ is 1-homogeneous.
\item[$4^\circ$] If $X$ is $r-$homogeneous and $f$ is $s-$homogeneous then $Xf$ is $r+s-1-$homogeneous.
\item[$5^\circ$] If $f$ is $s-$homogeneous, then $\dot{\pp}^i\dot{\pp}f$ is $s-1-$homogeneous.
\end{itemize}
Analogously, for $q-$forms $\omega\in \Lambda^q(T^*M)$. The $q-$form $\omega$ is $r-$ho\-mo\-ge\-neous if
\begin{equation}
\label{4.2.13}
{\cal L}_{\C^*}\omega=r\omega.
\end{equation}
It follows:
\begin{itemize}
\item[$1^\prime$] If $\omega,\omega^\prime$ are $s-$ respectively $s^\prime-$homogeneous, then $\omega\wedge\omega^{\prime}$ is $s+s'-$homogeneous.
\item[$2^\prime$] $dx^i,dp_i$ are $0-$ respectively $1-$homogeneous.
\item[$3^\prime$] The Liouville $1-$form $\omega=p_i dx^i$ is 1-homogeneous.
\item[$4^\prime$] The canonical symplectic structure $\theta$ is 1-homogeneous.
\end{itemize}

\section{Nonlinear connections}\index{Connections! Nonlinear}
\setcounter{equation}{0}
\setcounter{definition}{0}
\setcounter{theorem}{0}
\setcounter{lemma}{0}

On the manifold $T^*M$ there exists a remarkable distribution. It is the vertical distribution $V$. As we know, $V$ is integrable, having $(\dot{\pp}^1,\dot{\pp}^2,...,\dot{\pp}^n)$ as a local adapted basis and dim$V=n=$dim$M$.

\begin{definition}
\label{d4.3.1}
A nonlinear connection $N$ on the manifold $T^*M$ is a differentiable distribution $N$ on $T^*M$ supplementary to the vertical distribution $V$:
\begin{equation}
\label{4.3.1}
T_uT^*M=N_u\oplus V_u,\ \ \forall u\in T^*M.
\end{equation}
$N$ is called the {\sl horizontal} distribution on $T^*M$.
\end{definition}

It follows that dim$N=n$.

When a nonlinear connection $N$ is given we can consider an {\sl adapted} basis $\left(\dfrac{\delta}{\delta x^1},...,\dfrac{\delta}{\delta x^n}\right)$ expressed in the form
\begin{equation}
\label{4.3.2}
\dfrac{\delta}{\delta x^i}=\dfrac{\pp}{\pp x^i}+N_{ij}\dfrac{\pp}{\pp p_j}\ \ (i=1,2,...,n).
\end{equation}

The system of function $N_{ij}(x,p)$ is well determined by (\ref{4.3.1}). It is called system of coefficients of the nonlinear connection $N$. This system defines a geometrical object field on $T^*M$.

The set of vector fields $\left(\dfrac{\delta}{\delta x^i},\dot{\pp}^i\right)$ give us an adapted basis to the direct decomposition (\ref{4.3.1}). Its dual adapted basis is $(dx^i,\delta p_i)$, $(i=1,...,n)$, where
\begin{equation}
\label{4.3.3}
\delta p_i=dp_i-N_{ji}dx^j.
\end{equation}

It is not difficult to prove that if $M$ is a paracompact manifold, then on the cotangent manifold $T^*M$ there exists a nonlinear connection.

Let $N$ be a nonlinear connection with the coefficients $N_{ij}(x,p)$ and define the set of function $\tau_{ij}(x,y)$ by
\begin{equation}
\label{4.3.4}
\tau_{ij}=\dfrac12(N_{ij}-N_{ji}).
\end{equation}
It is not difficult to see that $\tau_{ij}$ is transformed, with respect to (\ref{4.1.1}) as a covariant tensor field on the base manifold $M$. So, it is a distinguished tensor field, shortly a $d-$tensor of {\sl torsion} of the nonlinear connection. $\tau_{ij}=0$ has a geometrical meaning. In this case $N$ is a symmetric nonlinear connection.

With respect to a symmetric nonlinear connection $N$ the symplectic structure $\theta$ can be written in an invariant form:
\begin{equation}
\label{4.3.5}
\theta=\delta p_i\wedge dx^i,
\end{equation}
and the Poisson structure $\{,\}$ can be expressed in the invariant form:
\begin{equation}
\label{4.3.6}
\{f,g\}=\dfrac{\pp f}{\pp p_i}\dfrac{\delta g}{\delta x^i}-\dfrac{\pp g}{\pp p_i}\dfrac{\delta f}{\delta x^i}.
\end{equation}
Of course, we can consider the curvature tensor of $N$ as the $d-$tensor field
\begin{equation}
\label{4.3.7}
R_{ijh}=\dfrac{\delta N_{ji}}{\delta x^h}-\dfrac{\delta N_{hi}}{\delta x^j}.
\end{equation}
It is given by
\begin{equation}
\label{4.3.8}
\left[\dfrac{\delta}{\delta x^j},\dfrac{\delta}{\delta x^h}\right]=R_{ijh}\dot{\pp}^i.
\end{equation}
Therefore $R_{ijh}(x,p)$ is the {\sl integrability} tensor of the horizontal distribution $N$. Thus $N$ is an integrable distribution if and only if the curvature tensor $R_{ijh}(x,p)$ vanishes.

A curve $c:I\subset R\to c(t)\in T^*M$. Im$c\subset\pi^*(U)$ expressed by ($x^i=x^i(t)$, $p_i=p_i(t)$, $t\in I$). The tangent vector $\dfrac{dc}{dt}$ can be written in the adapted basis as: $$\dfrac{dc}{dt}=\dfrac{dx^i}{dt}\dfrac{\delta}{\delta x^i}+\dfrac{\delta p_i}{dt}\dfrac{\pp}{\pp p_i},$$ where $$\dfrac{\delta p_i}{dt}=\dfrac{dp_i}{dt}-N_{ji}(x(t),p(t))\dfrac{dx^j}{dt}.$$ The curve $c$ is horizontal if $\dfrac{\delta p_i}{dt}=0$. We obtain the system of differential equations which characterize the horizontal curves:
\begin{equation}
\label{4.3.9}
x^i=x^i(t),\ \dfrac{dp_i}{dt}-N_{ji}(x,p)\dfrac{dx^j}{dt}=0.
\end{equation}

Let $g_{ij}(x,y)$ be a $d-$tensor by ($d-$means distinguished) with the properties $g_{ij}=g_{ji}$ and det$(g_{ij})\neq 0$ on $T^*M$. Its contravariant $g^{ij}(x,y)$ can be considered, $g_{ij}g^{ts}=\delta_i^s$. As usually, we put $\dfrac{\delta}{\delta x^i}=\delta_i$, $\dfrac{\pp}{\pp p_i}=\dot{\pp}$ and $\dot{\pp}_i=g_{ij}\dot{\pp}^j$. So, one can consider the following ${\cal F}(T^*M)-$linear mapping $\check{\F}:{\cal X}(T^*M)\to{\cal X}(T^*M)$:
\begin{equation}
\label{4.3.10}
\check{\F}(\delta_i)=-\dot{\pp}_i,\ \check{\F}(\dot{\pp}_i)=\delta_i,\ (i=1,...,n).
\end{equation}
It is not difficult to prove, \cite{MHS}:
\begin{theorem}
\label{t4.3.1}
\begin{itemize}
\item[$1^\circ$] The mapping $\check{\F}$ is globally defined on $T^*M$.
\item[$2^\circ$] $\check{\F}$ is a tensor field of type $(1,1)$ on $T^*M$.
\item[$3^\circ$] The local expression of $\check{\F}$ in the adapted basis $(\delta_i,\dot{\pp}^i)$ is
\begin{equation}
\label{4.3.11}
\check{\F}=-g_{ij}\dot{\pp}^i\otimes dx^j+g^{ij}\delta_i\otimes\delta p_j.
\end{equation}
\item[$4^\circ$] $\check{\F}$ is an almost complex structure on $T^*M$ determined by $N$ and by $g_{ij}(x,p)$, i.e.
\begin{equation}
\label{4.3.12}
\check{\F}\circ\check{\F}=-I.
\end{equation}
\end{itemize}
\end{theorem}

Also, nonlinear symmetric connection $N$ being considered, we can define the tensor:
\begin{equation}
\label{4.3.13}
\G(x,p)=g_{ij}(x,p)dx^i\otimes dx^j+g^{ij}\delta p_i\otimes \delta p_j.
\end{equation}

If $d-$tensor $g_{ij}(x,p)$ has a constant signature on $T^*M$ (for instance it is positively defined), then it follows:

\begin{theorem}
\label{t4.3.2}
\begin{itemize}
\item[$1^\circ$] $\G$ is a Riemannian structure on $T^*M$ determined by $N$ and $g_{ij}$.
\item[$2^\circ$] The distributions $N$ and $V$ are orthogonal with respect to $\G$.
\item[$3^\circ$] The pair $(\G,\check{\F})$ is an almost Hermitian structure on $T^*M$.
\item[$4^\circ$] The associated almost symplectic structure to $(\G,\check{\F})$ is the canonical symplectic structure $\theta=\delta p_i\wedge dx^i$.
\end{itemize}
\end{theorem}

Remarking that the vector fields $(\delta_i,\dot{\pp}^i)$ are transformed by (\ref{4.1.1}) in the form
\begin{equation}
\label{4.3.14}
\widetilde{\delta}_i=\dfrac{\pp x^j}{\pp\widetilde{x}^i}\delta_j,\ \widetilde{\dot{\pp}}^i=\dfrac{\pp\widetilde{x}^i}{\pp x^j}\dot{\pp}^j,
\end{equation}
and the 1-forms $(dx^i,\delta p_i)$ are transformed by
\begin{equation}
\label{4.3.13'}
d\widetilde{x}^i=\dfrac{\pp\widetilde{x}^i}{\pp x^j}dx^j,\ \delta\widetilde{p}_i=\dfrac{\pp x^j}{\pp\widetilde{x}^i}\delta p_j,\tag{4.3.13'}
\end{equation}
we can consider the horizontal and vertical projectors with respect to the direct decomposition (\ref{4.3.1}):
\begin{equation}
\label{4.3.14}
h=\dfrac{\de}{\de x^i}\otimes dx^i, \ v=\dot{\pp}^i\otimes\de p_i.
\end{equation}
They have the properties: $$h+v=I,\ h^2=h,\ v^2=v,\ h\circ v+v\circ h=0.$$ We set $$\begin{array}{lll}hX=X^H, & vX=X^V, & \forall X\in{\cal X}(T^*M)\\ \\
\omega^H=\omega\circ h, & \omega^v=\omega\circ v, & \forall \omega\in{\cal X}^*(T^*M).\end{array}$$ Therefore, for every vector field $X$ on $T^*M$, represented in adapted basis in the form $$X=X^i\de_i+\dot{X}_i\dot{\pp}^i$$ we have $$X^H=hX=X^i\dfrac{\de}{\de x^i},\ X^V=vX=\dot{X}_i\dot{\pp}^i$$ where the coefficients $X^i(x,p)$ and $\dot{X}_i(x,p)$ are transformed by (\ref{4.1.1}): $$\widetilde{X}^i=\dfrac{\pp\widetilde{x}^i}{\pp x^j}X^j,\ \widetilde{\dot{X}_i}=\dfrac{\pp x^j}{\pp \widetilde{x}^i}\dot{X}_j.$$ For this reason, $X^i(x,p)$ are the coefficients of a distinguished vector field and $\dot{X}_i(x,p)$ are the coefficients of a distinguished covector field, shortly denoted by $d-$vector (covector) fields.

Analogously, for the 1-form $\omega$: $$\omega=\omega_i dx^i+\dot{\oo}^i\de p_i.$$ Therefore, $$\oo^H=\oo_i dx^i,\ \oo^V=\dot{\oo}^i\de p_i.$$ Then $\oo_i(x,p)$ are component of an one $d-$form on $T^*M$ and $\dot{\oo}^i(x,p)$ are the components of a $d-$vector field on $T^*M$.

On the same way, we can define a distinguished $(d-)$ tensor field.

A $d-$tensor field can be represented in adapted basis in the form
\begin{equation}
\label{4.3.15}
T=T^{i_1...i_r}_{j_1...j_s}(x,p)\de_{i_1}\otimes...\otimes\dot{\pp}^{j_s}\otimes dx^{j_1}\otimes...\otimes\de p_r.
\end{equation}
Its coefficients are transformed by (\ref{4.1.1}) in the form:
\begin{equation}
\label{4.3.15'}
\wt{T}^{i_1...i_r}_{j_1...j_s}(\wt{x},\wt{p})=\dfrac{\pp\wt{x}^{i_1}}{\pp x^{h_1}}...\dfrac{\pp\wt{x}^{i_r}}{\pp x^{h_r}}\dfrac{\pp x^{k_1}}{\pp\wt{x}^{j_1}}...\dfrac{\pp x^{k_2}}{\pp\wt{x}^{j_s}}T^{h_1...h_r}_{k_1...k_s}(x,p).\tag{4.3.15'}
\end{equation}

So, a $d-$tensor field $T$ can be given by the coefficients $T^{i_1...i_r}_{j_1...j_s}(x,p)$ whose rule of transformation, with respect to (\ref{4.1.1}) is the same with that of a tensor of the same type on the base manifold $M$ of cotangent bundle $(T^*M,M,\pi^*)$.

One can speak of the $d-$tensor algebra on $T^*M$, which is not difficult to be defined.

\section{$N-$linear connections}\index{Connections! $N-$linear}
\setcounter{equation}{0}
\setcounter{definition}{0}
\setcounter{theorem}{0}
\setcounter{lemma}{0}

As we know, a nonlinear connection $N$ determines a direct decomposition (\ref{4.3.1}) in respect to which we have
\begin{equation}
\label{4.4.1}
X=X^H+X^V,\ \oo=\oo^H+\oo^V,\ \forall X\in{\cal X}(T^*M),\ \forall\oo\in{\cal X}^*(T^*M).
\end{equation}
Assuming that $N$ is a symmetric nonlinear connection we can give:
\begin{definition}
\label{d4.4.1.}
A linear connection $D$ on the manifold $T^*M$ is called an $N-$linear connection if:
\begin{itemize}
\item[$1^\circ$] $D$ preserves by parallelism the distributions $N$ and $V$.
\item[$2^\circ$] The canonical symplectic structure $\theta=\de p_i\wedge dx^i$ has the associate tensor $\ov{\theta}=\de p_i\otimes dx^i$ parallel with respect to $D$:
\begin{equation}
\label{4.4.2}
D\ov{\theta}=0.
\end{equation}
\end{itemize}
\end{definition}

It follows that:
\begin{equation}
\label{4.4.3}
\begin{array}{l}
D_X h=D_X v=0\\ \\
D_X Y=D_{X^{H}}Y+D_{X^V}Y.
\end{array}
\end{equation}
We obtain new operators of derivations in the algebras of $d-$tensors, defined by
\begin{equation}
\label{4.4.4}
D^H_X=D_{X^H},\ D^V_X=D_{X^V},\ \forall X\in{\cal X}(T^*M).
\end{equation}
We have
\begin{equation}
\label{4.4.5}
\begin{array}{l}
D_X=D_X^H+D_X^V;\ \ D^H_X f=X^F f,\ \ D_X^V f=X^VX\\ \\
D^H_X(fY)=X^H fY+fD^H_XY;\ D^V_X(fY)=X^V fY+fD^V_XY\\ \\
D^H_{fX}=fD^H_X,\ D^V_{fX}=fD^V_X\\ \\
D^H_X\theta=0,\ D^V_X\theta=0.
\end{array}
\end{equation}

The operators $D^H,D^V$ have the property of localization. They have the similarly properties with covariant derivative but they are not covariant derivative. $D^H$ is called $h-$covariant derivative and $D^V$ is called $v-$covariant derivative. They act on the 1-forms $\oo$ on $T^*M$ by the rules:
\begin{equation}
\label{4.4.6}
\begin{array}{l}
(D^H_X\oo)(Y)=X^H\oo(Y)-\oo(D^H_XY)\\ \\
(D^V_X\oo)(Y)=X^V\oo(Y)-\oo(D^V_X Y).
\end{array}
\end{equation}
The extension of these operators to the $d-$tensor fields is immediate.

The torsion $\T$ of an $N-$linear connection has the form
\begin{equation}
\label{4.4.7}
\Pi(X,Y)=D_XY-D_Y X-[X,Y].
\end{equation}
It is characterized by the following vector fields: $$\T(X^H,Y^H),\ \T(X^H,Y^V),\ T(X^V,X^V).$$ Taking the $h-$ and $v-$components we obtain the $d-$tensor of torsion $$T(X^H,Y^H)=h\T(X^H,Y^H)+v\T(X^H,Y^H),\ {\rm{etc.}}$$

\begin{proposition}
\label{p4.4.1}
The $d-$tensors of torsion of an $N-$linear connection $D$ are:
\begin{equation}
\label{4.4.6}
\begin{array}{l}
h\T(X^H,Y^H)=D^H_X Y^H-D^H_Y X^H-[X^H,Y^H]^H\\ \\
h\T(X^H,Y^V)=-D^V_Y X^H+[X^H,Y^V]^H\\ \\
v\T(X^H,Y^H)=-[X^H,Y^H]^V\\ \\
v\T(X^H,Y^V)=D^H_X Y^V-[X^H,Y^V]^V\\ \\
v\T(X^V,Y^V)=D^V_X Y^V-D^V_Y X^V-[X^V,Y^V]^V=0.
\end{array}
\end{equation}
\end{proposition}

The curvature $\R$ of an $N-$linear connection $D$ is given by
\begin{equation}
\label{4.4.7}
\R(X,Y)Z=(D_X D_Y-D_Y D_X)Z-D_{[X,Y]}Z.
\end{equation}

Remarking that the vector field $\R(X,Y)Z^H$ is horizontal one and $\R(X,Y)Z^V$ is a vertical one, we have
\begin{equation}
\label{4.4.8}
v(\R(X,Y)Z^H)=0,\ h(\R(X,Y)Z^V)=0.
\end{equation}

We will see that the $d-$tensors of curvature of $D$ are:
\begin{equation}
\label{4.4.9}
\R(X^H,Y^H)Z^H,\ \R(X^H,Y^V)Z^H,\ \R(X^V,Y^V)Z^H.
\end{equation}

By means of (\ref{4.4.7}) one obtains:
\begin{proposition}
\label{p4.4.2}
\begin{itemize}
\item[$1^\circ$] The {\sl Ricci} identities of $D$ are
\begin{equation}
\label{4.4.10}
[D_X,D_Y]Z=\R(X,Y)Z-D_{[X,Y]}Z.
\end{equation}
\item[$2^\circ$] The {\sl Bianchi} identities are given by
\begin{equation}
\label{4.4.11}
\begin{array}{l}
\dd\sum_{(XYZ)}\{(D_X\T)(Y,Z)-\R(X,Y)Z+\T(\T(X,Y),Z)\}=0,\\ \\
\dd\sum_{(X,Y,Z)}\{(D_X\R)(U,Y,Z)-\R(\T(X,Y),Z)U\}=0,
\end{array}
\end{equation}
where $\dd\sum_{(XYZ)}$ means the cyclic sum.
\end{itemize}
\end{proposition}

The previous formulae can be expressed in local coordinates, taking $\de_i,\dot{\pp}^i$, as the vectors $X,Y,Z,U$. But, first of all, we must introduce the local coefficients of an $N-$linear connection $D$. In the adapted basis $(\de_i,\dot{\pp}^i)$ we take into account the properties:
\begin{equation}
\label{4.4.12}
D_{\de_j}=D^H_{\de_j},\ D_{\dot{\pp}^j}=D^V_{\dot{\pp}^j}.
\end{equation}
Then, the following theorem holds:

\begin{theorem}
\label{t4.4.1}
\begin{itemize}
\item[$1^\circ$] An $N-$linear connection $D$ on $T^*M$ can be uniquely represented in adapted basis $(\de_i,\dot{\pp}^i)$ in the form:
\begin{equation}
\label{4.4.13}
\begin{array}{ll}
D_{\de_j}\de_i=H^h_{ij}\de_h, & D_{\de_j}\dot{\pp}^i=-H^i_{hj}\dot{\pp}^h,\\ \\
D_{\dot{\pp}^j}\de_i=C_i^{hj}\de_h, & D_{\dot{\pp}^j}\dot{\pp}^i=-C_h^{ij}\dot{\pp}^h.
\end{array}
\end{equation}
\item[$2^\circ$] With respect to (4.4.1) on $T^*M$ the coefficients $H^i_{jk}(x,p)$ transform by the rule
\begin{equation}
\label{4.4.14}
\wt{H}^i_{rs}\dfrac{\pp x^r}{\pp\wt{x}^j}\dfrac{\pp x^s}{\pp\wt{x}^k}=\dfrac{\pp\wt{x}^i}{\pp x^r}H^r_{jk}-\dfrac{\pp^2\wt{x}^i}{\pp x^j\pp x^k}
\end{equation}
\end{itemize}
and $C^{jk}_i(x,p)$ is a $d-$tensor of type $(2,1)$.
\end{theorem}

The proof of this theorem is not difficult.

Conversely, if $N$ is an a priori given nonlinear connection and a set of functions $H^i_{jk}(x,p),C_i^{jk}(x,p)$ on $T^*M$, verifying $2^\circ$ is given, then there exists an unique $N-$linear connection $D$ with the property (\ref{4.4.14}).

The action of $D$ on the adapted cobasis $(dx^i,\de p_i)$ is given by
\begin{equation}
\label{4.4.15}
\begin{array}{ll}
D_{\de_j}dx^i=-H^i_{kj}dx^k, & D_{\de_j}\de p_i=H^k_{ij}\de p_k,\\ \\
D_{\dot{\pp}^j}d x^i=-C_k^{ij}dx^k, & D_{\dot{\pp}^j}\de p_i=C_i^{kj}\de p_k.
\end{array}
\end{equation}
The pair $D\Gamma(N)=(H^i_{jk},C^{jk}_i)$ is called the system of coefficients of $D$.

Let us consider a $d-$tensor field $T$ with the local coefficients $T^{i,...,i_r}_{j_1,...,j_s}(x,p)$, (see (4.3.15)) and a horizontal vector field $X=X^H=X^i\de_i$.

By means of previous theorem we obtain for $h-$covariant derivation $D^H_X$ of tensor $T$:
\begin{equation}
\label{4.4.16}
D^H_XT=X^kT^{i_1...i_r}_{j_1...j_s,|k}\de_{i_1}\otimes...\otimes\dot{\pp}^{js}\otimes dx^{j_1}\otimes...\otimes\de p_{i_r},
\end{equation}
where
\begin{equation}
\label{4.4.16'}
\begin{array}{ll}
T^{i_1...i_r}_{j_1...j_s|k}=\de_k T^{i_1...i_r}_{j_1...j_s}+T^{mi_2...i_r}_{j_1...j_s}H^{i_1}_{mk}+...+T^{ir_1...m}_{j_1...j_s}H^{i_r}_{mk}-\\ \\
- T^{i_1...i_r}_{m...j_s}H^{m}_{j_1k}-...-T^{i_1...i_r}_{j_1...m}H^{m}_{j_s k}.
\end{array}\tag{4.4.16'}
\end{equation}
The operator ``$|$'' is called $h-$covariant derivative with respect to $D\Gamma(N)=(H^{i}_{jk},C_i^{jk})$.

Now, taking $X=X^V=X_i\dot{\pp}^i$, the $v-$covariant derivative $D^V_XT$ has the following form
\begin{equation}
\label{4.4.17}
D^V_X T=\dot{X}_k \left.T^{i_1...i_r}_{j_1...j_s}\right|^k\de_i\otimes...\otimes\dot{\pp}^{j_s}\otimes dx^{j_1}\otimes...\otimes \de p_{i_r},
\end{equation}
where
\begin{equation}
\label{4.4.17'}
\begin{array}{ll}
\left.T^{i_1...i_r}_{j_1...j_s}\right|^k=\dot{\pp}^k T^{i_1...i_r}_{j_1...j_s}+T^{m...i_r}_{j_1...j_s}C^{i_1 k}_{m}+...+T^{i_1...m}_{j_1...j_s}C^{i_rk}_{m}-\\ \\
- T^{i_1...i_r}_{m...j_s}C^{mk}_{j_1}-...-T^{i_1...i_r}_{j_1...m}C^{mk}_{j_s}.
\end{array}\tag{4.4.17'}
\end{equation}

The operator ``|'' is called the $v-$covariant derivative.

\begin{proposition}
\label{p4.4.3}
The following properties hold:
\begin{itemize}
\item[$1^\circ$] $T^{i_1...i_r}_{j_1...j_s|k}$ is a $d-$tensor of type $(r,s+1)$.
\item[$2^\circ$] $\left. T^{i_1...i_r}_{j_1...j_s}\right|^k$ is a $d-$tensor of type $(r+1,s)$.
\item[$3^\circ$] $f_{|m}=\dfrac{\de f}{\de x^m}$, $f|^m=\dot{\pp}^m f$.
\begin{equation}
\label{4.4.18}
\begin{array}{lll}
4^\circ & X^i_{|k}=\de_k X^i+X^m H^i_{mk};\  & X^i|^k=\dot{\pp}^k X^i+X^m C^{ik}_m,\\ \\
& \oo_{i|k}=\de_k\oo_i-\oo_m H^m_{ik};\ & \oo_1 |^k=\dot{\pp}^k\oo_i-\oo_m C^{mk}_i.
\end{array}
\end{equation}
\item[$5^\circ$] The operators ``$_|$'' and ``$|$'' are distributive with respect to addition and verify the rule of Leibnitz with respect to tensor product of $d-$tensors.
\end{itemize}
\end{proposition}

Let us consider the {\sl deflection} tensors of $D\Gamma(N)$:
\begin{equation}
\label{4.4.19}
\Delta_{ij}=p_{i|j},\ \check{\de}^j_i=p_i |^j.
\end{equation}
One gets
\begin{equation}
\label{4.4.19'}
\Delta_{ij}=N_{ij}-p_{m}H^m_{ij}; \delin^i_j=\de^i_j-p_m C^{mj}_i.\tag{4.4.19'}
\end{equation}

\begin{proposition}
\label{p4.4.4}
If $H^i_{jk}(x,p)$ is a system of differentiable functions, verifying $($\ref{4.4.14}$)$, then $N_{ij}(x,p)$ given by
\begin{equation}
\label{4.4.20}
N_{ij}=p_m H^m_{ij}
\end{equation}
determine a nonlinear connection in $T^*M$.
\end{proposition}

As in the case of tangent bundle, one can prove that if the base manifold $M$ is paracompact, then on $T^*M$ there exist the $N-$linear connections.

Indeed, on $M$ there exists a Riemannian metric $g_{ij}(x)$. Then the pair $D\Gamma(N)=(\gamma^i_{jk},0)$ is an $N-$linear connection on $T^*M$, $\gamma^i_{jk}(x)$ being the Christoffel coefficients and $N_{ij}=p_m\gamma^m_{ij}$, are the coefficients of a nonlinear connection on $T^*M$.

In the adapted basis $(\de_i,\dot{\pp}^i)$, the Ricci identitites (\ref{4.4.10}) can be written in the form
\begin{equation}
\label{4.4.21}
\begin{array}{ll}
X^i_{|j|h}-X^i_{|h|j}=X^m R^{\ i}_{m\ jh}-X^i_{|m}T^m_{jh}-X^i|^m R_{mjh}\\ \\
X^i_{|j}|^h-X^i|^h_{|j}=X^m P_{m\ j}^{\ i\ h}-X^i_{\ |m}C^{mh}_j-X^i |^m P^h_{mj}\\ \\
X^i|^j|^h-X^i|^h|^j=X^m S_m^{\ ijh}-X^i|^m S_m^{\ jh},
\end{array}
\end{equation}
where
\begin{equation}
\label{4.4.22}
T^i_{jh}=H^i_{jh}-H^i_{hj},\ S_i^{\ jh}=C^{jh}_i-C^{hj}_i,\ P_{jh}^i=H^i_{jh}-\dot{\pp}^i N_{hj}
\end{equation}
and $R_{mjh},C^{mh}_i$ are the $d-$tensors of torsion and $R^{\ i}_{k\ jh}$, $P^{\ i\ h}_{k\ j}$, $S_k^{\ ijh}$ are the $d-$tensors of curvature:
\begin{equation}
\label{4.4.22}
\begin{array}{ll}
R^{\ i}_{k\ jh}=\de_h H^i_{kj}-\de_j H^i_{kh}+H^m_{kj}H^i_{mh}-H^m_{kh}H^i_{kh}H^i_{mj}-C_k^{\ im}R_{mjh},\\ \\
P^{\ i\ h}_{k\ j}=\dot{\pp}^h H^i_{kj}-C^{ih}_{k\ |j}+C^{im}_k P^h_{\ mj},
S_{k}^{\ ijh}=\dot{\pp}^hC^{ij}_k-\dot{\pp}^j C^{ih}_k+\\ \\ \quad \qquad +C^{\ mj}_k C^{\ ih}_m-C^{\ mh}_k C_m^{\ ij}.
\end{array}
\end{equation}

Evidently, these $d-$tensors of curvature verify:
\begin{equation}
\label{4.4.23}
\begin{array}{ll}
\R(\de_j\de_h)\de_k=R^{\ i}_{k\ hj}\de_i;\ \R(\de_j,\dot{\pp}^h)\de_k=P^{\ h\ i}_{k\ j}\de_i,\\ \\
\R(\dot{\pp}^j,\dot{\pp}^h)\de_k=S^{ihj}_k\de_i,
\end{array}
\end{equation}
and
\begin{equation}
\label{4.4.23'}
\begin{array}{ll}
\R(\de_j\de_h)\dot{\pp}^k=R^{\ k}_{i\ hj}\dot{\pp}^i;\ \R(\de_j,\dot{\pp}^h)\dot{\pp}^k=-P^{\ k\ h}_{i\ j}\dot{\pp}^i,\\ \\
\R(\dot{\pp}^j,\dot{\pp}^h)\dot{\pp}^k=-S^{\ hkj}_i\dot{\pp}^j.\tag{4.4.23'}
\end{array}
\end{equation}

The Bianchi identities (\ref{4.4.11}), in adapted basis $(\de_i,\dot{\pp}^i)$ can be written without difficulties.

\bigskip

\noindent{\bf Applications:} \begin{itemize} \item[$1^\circ$] For a $d-$ tensor $g^{ij}(x,p)$ the Ricci identities are
\begin{equation}
\label{4.4.24}
\begin{array}{ll}
g^{ij}_{\ |k|h}\hspace*{-0.5mm}-\hspace*{-0.5mm}g^{ij}_{\ |h|k}\hspace*{-0.5mm}=\hspace*{-0.5mm}g^{mj}R_{m\ kh}^{\ i}\hspace*{-0.5mm}+\hspace*{-0.5mm}g^{im}R^{\ j}_{mkh}\hspace*{-0.5mm}-\hspace*{-0.5mm}g^{ij}_{\ |m}T^m_{\ kh}-g^{ij}|^m R_{mkh},\\ \\
g^{ij}_{\ |k}|^h\hspace*{-0.5mm}-\hspace*{-0.5mm}g^{ij}|^h_{\ |k}\hspace*{-0.5mm}=\hspace*{-0.5mm}g^{mj}P_{m\ k}^{\ i\ h}\hspace*{-0.5mm}+\hspace*{-0.5mm}g^{im}P_{m\ k}^{\ j\ h}\hspace*{-0.5mm}-\hspace*{-0.5mm}g^{ij}_{\ |m}C_k^{\ mh}-g^{ij}|^m P^h_{mk},\\ \\
g^{ij}|^k|^h-g^{ij}|^h|^k=g^{mj}S^{\ ikh}_m+g^{im}S_m^{\ jkh}-g^{ij}|^m S^{\ kh}_m.
\end{array}
\end{equation}
In particular, if $g^{ij}$ is covariant constant with respect to $N-$connection $D$, i.e.
\begin{equation}
\label{4.4.25}
g^{ij}_{\ |k}=0,\ \ g^{ij}|^k=0,
\end{equation}
then, from (\ref{4.4.24}) we have:
\begin{equation}
\label{4.4.25'}
\begin{array}{ll}
g^{mj}R^{\ i}_{m\ kh}+g^{im}R^{\ j}_{m\ kh}=0,\\
g^{mj}P^{\ i\ h}_{m\ k}+g^{im}P^{\ j\ h}_{m\ k}=0,\\
g^{mj}S^{\ ikh}_m+g^{im}S_m^{\ jkh}=0.
\end{array}\tag{4.4.25'}
\end{equation}
Such kind of equations will be used for the $N-$linear connections compatible with a metric structure $G$ in (\ref{4.3.13}).
\end{itemize}

The Ricci identities applied to the Liouville--Hamilton vector field $C^*=p_i\dot{\pp}^i$ lead to the important identities, which imply the deflection tensors $\Delta_i$, and ${^{-}}\!\!\!\!\!\!\de^i_j$.

\begin{theorem}
\label{t4.2.2}
Any $N-$linear connection $D$ on $T^*M$ satisfies the following identities:
\begin{equation}
\label{4.2.26}
\begin{array}{l}
\Delta_{ij|k}-\Delta_{ik|j}=-p_m R_{i\ jk}^{\ m}-\Delta_{im}T^{m}_{\ jk}-\delin^m_i R_{mjk},\\
\Delta_{ij}|^k-\delin_{i\ |j}^{\ k}=-p_m P_{i\ j}^{m\ k}-\Delta_im C^{mk}_j-\delin^m_i P^i_{\ mj},\\
\delin_i^j|^k-\delin^k_j |^j=-p_m S_i^{\ mjk}-\delin^m_i S_m^{\ jk}.
\end{array}
\end{equation}
\end{theorem}

One says that are $N-$linear connection $D$ is of Cartan type if $\Delta_{ij}=0$, $\delin^i_j=\de^i_j$.

\begin{proposition}
\label{p4.4.5}
Any $N-$linear connection $D$ of Cartan type satisfies the identities:
\begin{equation}
\label{4.4.27}
p_m R^m_{i\ jk}+R_{ijk}=0,\ p_m P^{\ m\ k}_{i\ j}+P^{k}_{\ ij}=0,\ p_m S_i^{\ mjk}+S_i^{\ jk}=0.
\end{equation}
\end{proposition}

Finally, we remark that we can explicitly write, in adapted basis, the Bianchi identities (\ref{4.4.11}).

\section{Parallelism, paths and structure equations}\index{Parallelism! on manifold $TM$}\index{Structure equations}\index{Paths}
\setcounter{equation}{0}
\setcounter{definition}{0}
\setcounter{theorem}{0}
\setcounter{lemma}{0}

Consider the Hamiltonian systems $(T^*M,H,\theta)$, an $N-$linear connection $D\Gamma(N)=(H^i_{jk},C^{\ jk}_i)$.

A curve $c:I\to T^*M$, locally represented by
\begin{equation}
\label{4.5.1}
x^i=x^i(t),\ p_i=p_i(t),
\end{equation}
has the tangent vector $\dfrac{dc}{dt}$ in adapted basis $(\de_i,\dot{\pp}^i)$:
\begin{equation}
\label{4.5.2}
\dfrac{dc}{dt}=\dfrac{dx^i}{dt}\de_i+\dfrac{\de p_i}{dt}\dot{\pp}^i,
\end{equation}
with $\dfrac{\de p_i}{dt}=\dfrac{dp_i}{dt}-N_{ji}\dfrac{dx^j}{dt}$. The operator of covariant derivative $D$ along the curve $c$ is $$D_{\dot{c}}=D^H_{\dot{c}}+D^V_{\dot{c}}=\dfrac{dx^i}{dt}D_{\de_i}+\dfrac{\de p_i}{dt}D_{\dot{\pp}^i}.$$ If $X=X^i\de_i+\dot{X}_i\dot{\pp}^i$ is a vector field then the operator $X$ of covariant differential acts on the vector fields $D$ by the rule:
\begin{equation}
\label{4.5.3}
DX=(dx^i+X^m\oo^i_m)\de_i+(d\dot{X}_i-\dot{X}_m\oo^m_i)\dot{\pp}^i,
\end{equation}
where
\begin{equation}
\label{4.5.4}
\oo^i_j=H^i_{jk}dx^k+C^{ik}_j\de p_k
\end{equation}
is called {\sl the connection $1-$form}.

Therefore, the covariant differential $\dfrac{DX}{dt}$ is:
\begin{equation}
\label{4.5.5}
\dfrac{DX}{dt}=\left(\dfrac{dX^i}{dt}+X^m\dfrac{\oo^i_m}{dt}\right)\de_i+\left(\dfrac{d\dot{X}_i}{dt}-\dot{X}_m\dfrac{\oo^m_i}{dt}\right)\dot{\pp}^i,
\end{equation}
where $\dfrac{dX^i}{dt}$ is $\dfrac{dX^i}{dt}(x(t),p(t))$, $t\in I$ and
\begin{equation}
\label{4.5.6}
\dfrac{\oo^i_j}{dt}=H^i_{\ jk}(x(t),p(t))\dfrac{dx^k}{dt}+C_j^{\ ik}(x(t),p(t))\dfrac{\de p_k}{dt}.
\end{equation}

As usually, we introduce:
\begin{definition}
\label{d4.5.1}
The vector field $X=X^i(x,p)\de_i+\dot{X}_m(x,p)\dot{\pp}^i$ is parallel, with respect to $D\Gamma(N)$, along smooth curve $c:I\to T^*M$ if $\dfrac{DX}{dt}=0$, $\forall t\in I$.
\end{definition}

From (\ref{4.5.5}) one obtains:
\begin{theorem}
\label{t4.5.1}
The vector field $X=X^i\de_i+\dot{X}_i\dot{\pp}^i$ is parallel, with respect to the $N-$linear connection $D\Gamma(H^i_{jk},C^{jk}_i)$, along of curve $c$ if, and only if, the functions $X^i,\dot{X}_i$, $(i=1,...,n)$ are solutions of the differential equations:
\begin{equation}
\label{4.5.7}
\dfrac{dX^i}{dt}+X^m\dfrac{\oo^i_m}{dt}=0,\ \dfrac{d\dot{X}_i}{dt}-\dot{X}_m\dfrac{\oo_i^m}{dt}=0.
\end{equation}
\end{theorem}

The proof is immediate, by means of (\ref{4.5.5}).

A theorem of existence and uniqueness for the parallel vector fields along a given parametrized curve $c$ in the manifold $T^*M$ can be formulated in the classical manner. Let us consider the case when a vector field $X$ is absolute parallel on $T^*M$, with respect to an $N-$linear connection $D$, i.e. $DX=0$ on $T^*M$, \cite{MHS}.

Using the formula (\ref{4.5.3}), $DX=0$ if and only if $$dX^i+X^m\oo^i_m=0,\ d\dot{X}_i-\dot{X}_m\oo^m_i=0.$$ Remarking (\ref{4.5.4}), and that we have $$\begin{array}{l}dX^i=X^i_{|k}dx^k+X^i|^k\de p_k\\ \\ d\dot{X}_i=\dot{X}_{i\ |k}dx^k+\dot{X}_i|^k\de p_k\end{array}$$ it follows that the equation $DX=0$ along any curve $c$ on $T^*M$ is equivalent to
\begin{equation}
\label{4.5.8}
\begin{array}{ll}
X^i_{|j}=0, & X^i|^j=0\\ \\
\dot{X}_{i|j}=0, & \dot{X}_i|^j=0.
\end{array}
\end{equation}
The differential consequence of the previous system are given by the Ricci identities (\ref{4.4.23}), taken modulo (\ref{4.5.8}). They are
\begin{equation}
\label{4.5.9}
\begin{array}{lll}
X^hR^{\ i}_{h\ jk}=0, & X^h P_{h\ j}^{\ i\ k}=0, & X^h S^{\ ijk}_h=0\\ \\
\dot{X}_{h}R^{\ h}_{i\ jk}=0, & \dot{X}_h P^{\ h\ k}_{i\ j}=0, & \dot{X}_h S^{\ hjk}_i=0.
\end{array}
\end{equation}
But, $\{X^i,\dot{X}^i\}$ being arbitrary, it follows
\begin{theorem}
The $N-$linear connection $D\Gamma(N)$ is with the absolute parallelism of vectors if, and only if, the curvature tensor $\R$ of $D$ vanishes.
\end{theorem}

\begin{definition}
The curve $c:I\to T^*M$ is called autoparallel for $N-$linear connection $D$ if $D_{\dot{c}}\cdot\dot{c}=0$.
\end{definition}

Taking into account the fact that $\dot{c}=\dfrac{dc}{dt}=\dfrac{dx^i}{dt}\de_i+\dfrac{\de p_i}{dt}\dot{\pp}^i$, one obtains
\begin{theorem}
A curve $c:I\to T^*M$ is autoparallel with respect to $D\Gamma(N)$ if, and only if, the functions $x^i(t), p_i(t)$, $t\in I$, are solutions of the system of differential equations
\begin{equation}
\label{4.5.10}
\dfrac{d^2 x^i}{dt^2}+\dfrac{dx^s}{dt}\dfrac{\oo^i_s}{dt}=0,\ \dfrac{d}{dt}\dfrac{\de p_i}{dt}-\dfrac{\de p_s}{dt}\dfrac{\oo^s_i}{dt}=0.
\end{equation}
\end{theorem}
Of course, a theorem of existence and uniqueness for the autoparallel curves can be formulated.

\begin{definition}
An horizontal path of $D\Gamma(N)$ is an horizontal autoparallel curve.
\end{definition}

\begin{theorem}
The horizontal paths of $D\Gamma(N)$ are characterized by the differential equations
\begin{equation}
\label{4.5.11}
\dfrac{d^2x^i}{dt^2}+H^i_{jk}(x,p)\dfrac{dx^j}{dt}\dfrac{dx^k}{dt}=0.
\end{equation}
\end{theorem}

The curve $c:I\to T^*M$ is vertical at the point $x^i_0=x^i(t_0)$, $t_0\in I$ of $M$ if $\dot{c}(t_0)\in V$ (vertical distribution).

A {\sl vertical path} at the point $(x^i_0)\in M$ is a vertical autoparallel curve at point $(x^i_0)$.

\begin{theorem}
The vertical paths at a point $x_0=(x^i_0)\in M$ with respect to $D\Gamma(N)$ are solutions of the differential equations
\begin{equation}
\label{4.5.12}
\dfrac{d^2 p_i}{dt^2}-C^{jk}_i(x_0,p)\dfrac{dp_j}{dt}\dfrac{dp_k}{dt}=0,\ x^i=x^i_0.
\end{equation}
\end{theorem}

Now, we can write the structure equations of $D\Gamma(N)=(H^i_{jk},C^{\ jk}_i)$, remarking that the following geometric objects are a $d-$vector field and a $d-$covector field, respectively: $$d(dx^i)-dx^m\wedge\oo^i_m;\ \ d\de p_i+\de p_m\wedge\oo^m_i$$ and $d\oo^i_j-\oo^m_j\wedge\oo^i_m$ is a $d-$tensor of type (1,1). Here $d$ is the operator of exterior differential.

By a straightforward calculus we can prove:
\begin{theorem}
For any $N-$linear connection $D$ with the coefficients $D\Gamma(N)=(H^i_{jk},C^{jk}_i)$ the following structure equations hold:
\begin{equation}
\label{4.5.13}
\begin{array}{ll}
d(dx^i)-dx^m\wedge\oo^i_m=-\Omega^i\\ \\
d(\de p_i)+\de p_m\wedge\oo^m_i=-\dot{\Omega}_i\\ \\
d\oo^i_j-\oo^m_j\wedge\oo^i_m=-\Omega^i_j,
\end{array}
\end{equation}
where $\Omega^i,\dot{\Omega}_i$ are $2-$forms of torsion:
\begin{equation}
\label{4.5.14}
\begin{array}{ll}
\O^i=\dfrac12 T^i_{\ jk}dx^j\wedge dx^k+C^{ik}_j dx^j\wedge \de p_k,\\ \\
\dot{\O}_i=\dfrac12 R_{ijk}dx^j\wedge dx^k+P^k_{\ ij}dx^j\wedge \de p_k+\dfrac12 S_i^{\ jk}\de p_{j}\wedge \de p_k,
\end{array}
\end{equation}
and when $\O^i_j$ is 2-forms of curvature:
\begin{equation}
\label{4.5.15}
\O^i_j\hspace*{-0.5mm}=\hspace*{-0.5mm}\dfrac12 R^{\ i}_{j\ km}dx^k\wedge dx^m\hspace*{-0.5mm}+\hspace*{-0.5mm}P^{\ i\ m}_{j\ k}dx^k\wedge\de p_m\hspace*{-0.5mm}+\hspace*{-0.5mm}\dfrac12 S_i^{\ jkm}\de p_k\wedge \de p_m.
\end{equation}
\end{theorem}

\noindent{\bf Remarks.} \begin{itemize}
\item[$1^\circ$] The torsion and curvature $d-$tensor forms (\ref{4.5.14}) and (\ref{4.5.15}) are expressed in Section 4 of this chapter.
\item[$2^\circ$] The Bianchi identities can be obtained exterior differentiating the structure equations modulo the same equations.
\end{itemize}

\newpage

\chapter{Hamilton spaces}
\label{ch5p1} 

The notion of Hamilton space was defined and investigated by R. Miron in the papers \cite{mir4}, \cite{MHS}, \cite{Miron}. It was studied by D. Hrimiuc \cite{hrimiuc}, H. Shimada \cite{shimada} et al. It was applied by P.L. Antonelli, S. Vacaru, D. Bao et al. \cite{VS}. The geometry of these spaces is the geometry of a Hamiltonian system $(T^*M,\theta,H)$, where $H(x,p)$ is a fundamental function of an Hamilton space $H^n=(M,H(x,p))$. Consequently, we can apply the geometric theory of cotangent manifold $T^*M$ presented in the previous chapter, establishing the all fundamental geometric objects of the spaces $H^n$.

As we see, the geometry of Hamilton spaces can be studied as dual geometry, via Legendre transformation of the geometry of Lagrange spaces.

In this chapter we study the geometry of Hamilton spaces, combining these two methods. So, we defined the notion of Hamilton space $H^n$, the canonical non-linear connection $N$ of $H^n$, the metrical structure and the canonical metrical linear connection of $H^n$. The fundamental equations of $H^n$  are the Hamilton--Jacobi equations. The notion of parallelism with respect to a $N-$metrical connection and its consequences are studied. As applications we study the notion of Hamilton spaces of the electrodynamics. Also the almost K\"{a}hlerian model of the spaces $H^n$ is pointed out.

\section{Notion of Hamilton space}
\setcounter{equation}{0}
\setcounter{definition}{0}
\setcounter{theorem}{0}
\setcounter{lemma}{0}

Let $M$ be a real differential manifold of dimension $n$ and $(T^*M,\pi^*,M)$ its cotangent bundle. A differentiable Hamiltonian is a mapping $H:T^*M\to R$ differentiable on $\wt{T^*M}=T^*M\setminus\{0\}$ and continuous on the null section.

The Hessian of $H$, with respect to the momenta variable $p_i$ of differential Hamiltonian $H(x,p)$ has the components
\begin{equation}
\label{5.1.1}
g^{ij}=\dfrac12\dfrac{\pp^2 H}{\pp p_i\pp p_j}.
\end{equation}
Of course, the Latin indices $i,j,...$ run on the set $\{1,...,n\}$.

The set of functions $g^{ij}(x,p)$ determine a symmetric contravariant of order 2 tensor field on $\wt{TM}^{*}$. $H(x,p)$ is called regular if
\begin{equation}
\label{5.1.2}
\det(g^{ij}(x,p))\neq 0\ {\rm{\ on\ }}\wt{T^*M}.
\end{equation}

Now, we can give
\begin{definition}
\label{d5.1}
A Hamilton space is a pair $H^n=(M,H(x,p))$ where $M$ is a smooth real, $n-$dimensional manifold, and $H$ is a function on $T^*M$ having the properties
\begin{itemize}
\item[$1^\circ$] $H:(x,p)\in T^*M\to H(x,p)\in R$ is a differentiable function on $\wt{T^*M}$ and it is continuous on the null section of the natural projection $\pi^* T^*M\to M$.
\item[$2^\circ$] The Hessian of $H$ with respect to momenta $p_i$ given by $\|g^{ij}(x,p)\|$, (\ref{5.1.1}) is nondegenerate i.e. (\ref{5.1.2}) is verified on $\wt{T^*M}$.
\item[$3^\circ$] The $d-$tensor field $g^{ij}(x,p)$ has constant signature on the manifold $\wt{T^*M}$.
\end{itemize}
\end{definition}

One can say that the Hamilton space $H^n=(M,H(x,p))$ has as fundamental function a differentiable regular Hamiltonian for which its {\it fundamental} or {\it metric} tensor $g^{ij}(x,p)$ is nonsingular and has constant signature.

\noindent{\bf Examples}
\begin{itemize}
\item[$1^\circ$] If ${\cal R}^n=(M,g_{ij}(x))$ is a Riemannian space then $H^n=(M,H(x,p))$ with
\begin{equation}
\label{5.1.3}
H(x,p)=\dfrac{1}{mc}g^{ij}(x)p_i p_j
\end{equation}
is an Hamilton space.
\item[$2^\circ$] The Hamilton space of electrodynamics is defined by the fundamental function
\begin{equation}
\label{5.1.4}
H(x,p)=\dfrac{1}{mc}g^{ij}(x)p_i p_j-\dfrac{2e}{mc^2}A^i(x)p_i-\dfrac{e^2}{mc^2}A^i(x)A_i(x),
\end{equation}
where $m,c,e$ are the known physical constants, $g_{ij}(x)$ is a pseudo-Riemannian tensor and $A_i(x)$ are a $d-$covector (electromagnetic potentials) and $A^i(x)=g^{ij}A_j$.
\end{itemize}

\begin{remark}
The kinetic energy for a Riemannian metric ${\cal R}^n=(M,g_{ij}(x))$ is given by $$T(x,\dot{x})=\dfrac12 g_{ij}(x)\dot{x}^i\dot{x}^j$$ and for the Hamilton space $H^n\hspace*{-0.5mm}=\hspace*{-0.5mm}(M,H(x,p))$ with $H(x,p)\hspace*{-0.5mm}=\hspace*{-0.5mm}g^{ij}(x)p_i p_j$ it is $$T^*(x,p)=\dfrac12 g^{ij}p_i p_j.$$
\end{remark}

Therefore, generally, for an Hamilton space $H^n=(M,H(x,p))$  is convenient to introduce the energy given by the Hamiltonian
\begin{equation}
\label{5.1.5}
{\cal H}(x,p)=\dfrac12 H(x,p).
\end{equation}
One obtain:
\begin{equation}
\label{5.1.6}
\dot{\pp}^i\dot{\pp}^j{\cal H}=g^{ij}(x,p).
\end{equation}

Consider the Hamiltonian system $(T^*M,\theta,{\cal H})$, where $\theta$ is the natural symplectic structure on $T^*M,\theta=dp_i\wedge dx^i$ (Ch. 4). The isomorphism $S_\theta$, defined by (4.1.7) can be used and Theorem 4.2 can be applied.

One obtain
\begin{theorem}
For any Hamilton space $H^n=(M,H(x,p))$ the following properties hold:
\begin{itemize}
\item[$1^\circ$] There exists a unique vector field $X_H\in{\cal X}(T^*M)$ for which
\begin{equation}
\label{5.1.7}
i_H\theta=-d{\cal H}.
\end{equation}
\item[$2^\circ$] $X_H$ is expressed by
\begin{equation}
\label{5.1.8}
X_H=\dfrac{\pp{\cal H}}{\pp p_i}\dfrac{\pp}{\pp x^i}-\dfrac{\pp{\cal H}}{\pp x^i}\dfrac{\pp}{\pp p_i}.
\end{equation}
\item[$3^\circ$] The integral curves of $X_H$ are given by the Hamilton--Jacobi equations
\begin{equation}
\label{5.1.9}
\dfrac{dx^i}{dt}=\dfrac{\pp{\cal H}}{\pp p_i};\ \dfrac{dp_i}{dt}=-\dfrac{\pp{\cal H}}{\pp x^i}.
\end{equation}
\end{itemize}
\end{theorem}

Since $$\dfrac{d{\cal H}}{dt}=\{{\cal H},{\cal H}\}=0,$$  it follows $\dfrac{dH}{dt}=0.$ Then: {\it The fundamental function} $H(x,p)$ of a Hamilton space $H^n$ is constant along the integral curves of the Hamilton - Jacobi equations (\ref{5.1.9}). The Jacobi method of integration of (\ref{5.1.9}) expound in chapter 4 is applicable and the variational problem, in \S4.2, ch. 4, can be formulated again in the case of Hamiltonian systems $(T^*M,\theta,{\cal H})$.

\section{Nonlinear connection of a Hamilton space}\index{Connections! Nonlinear}
\setcounter{equation}{0}
\setcounter{definition}{0}
\setcounter{theorem}{0}
\setcounter{lemma}{0}

Let us consider a Hamilton space $H^n\subset(M,H(x,p))$. The theory of nonlinear connection on the manifold $T^*M$, in ch. 4, can be applied in the case of spaces $H^n$. We must determine a nonlinear connection $N$ which depend only by the Hamilton space $H^n$, i.e. $N$ must be canonical related to $H^n$, as in the case of canonical nonlinear connection from the Lagrange spaces. To do this, direct method is given by using the Legendre transformation, suggested by Mechanics. We will present, in the end of this chapter, this method.

Now, following a result of R. Miron we enounce without demonstration (which can be found in ch. 2, \S2.3) the following result:

\begin{theorem}
\begin{itemize}
\item[$1^\circ$] The following set of functions
\begin{equation}
\label{5.2.1}
N_{ij}=\dfrac14\{g_{ij},H\}-\dfrac14\left(g_{ik}\dfrac{\pp^2 H}{\pp p_k\pp x^j}+g_{jk}\dfrac{\pp^2 H}{\pp p_k\pp x^i}\right)
\end{equation}
are the coefficients of a nonlinear connection of the Hamilton space $H^n=(M,H(x,p))$.
\item[$2^\circ$] The nonlinear connection $N$ with coefficients $N_{ij}$, (\ref{5.2.1}) depends only on the fundamental function $H(x,p)$.
\end{itemize}
\end{theorem}

The brackets $\{,\}$ from (\ref{5.2.1}) are the Poisson brackets $(,)$, \S5.1.

Indeed, by a straightforward computation, it follows that, under a coordinate change $(\ )$, $N_{ij}(x,p)$ from (\ref{5.2.1}) obeys the rule of transformation of the coefficients of a nonlinear connection $N$. Evidently $N_{ij}$ depend on the fundamental function $H(x,p)$ only. $N-$ will be called the canonical nonlinear connection of the Hamilton space $H^n$.

\begin{remark}
If the fundamental function $H(x,p)$ of $H^n$ is globally defined, on $T^*M$ then the horizontal distribution determined by the canonical nonlinear connection $N$ has the same property.
\end{remark}

It is not difficult to prove:

\begin{proposition}
The canonical nonlinear connection $N$ has the properties
\begin{equation}
\label{5.2.2}
\tau_{ij}=\dfrac12(N_{ij}-N_{ji})=0 \ \ ({\rm{it\ is\ symmetric}})
\end{equation}
\begin{equation}
\label{5.2.3}
R_{ijk}+R_{jki}+R_{kij}=0
\end{equation}
\end{proposition}

Taking into account that $N$ is a horizontal distribution, we have the known direct decomposition:
\begin{equation}
\label{5.2.4}
T_u\wt{T^*M}=N_u\oplus V_u,\ \forall u\in\wt{T^*M},
\end{equation}
it follows that $(\de_i,\dot{\pp}^i)$ is an adapted basis to the previous splitting, where
\begin{equation}
\label{5.2.5}
\de_i={\pp}_i-N_{ij}\dot{\pp}^j
\end{equation}
and the dual basis $(dx^i,\de p_i)$ is:
\begin{equation}
\label{5.2.5'}
\de p_i=dp_i-N_{ij}dx^j.\tag{5.2.5'}
\end{equation}
Therefore we apply the theory to investigate the notion of metric $N-$linear connection determined only by the Hamilton space $H^n$.

\section{The canonical metrical connection of Hamilton space $H^n$}
\setcounter{equation}{0}
\setcounter{definition}{0}
\setcounter{theorem}{0}
\setcounter{lemma}{0}

Let us consider the $N-$linear connection $D\Gamma(N)=(H^i_{jk},C^{jk}_i)$ for which $N$ is the canonical nonlinear connection. By means of theory from chapter 4 we can prove:
\begin{theorem}
\begin{itemize}
In a Hamilton space $H^n=(M,H(x,p))$ there exists a unique $N-$linear connection $D\Gamma(N)=(H^i_{jk},C^{jk}_i)$ verifying the axioms:
\item[$1^\circ$] $N$ is the canonical nonlinear connection.
\item[$2^\circ$] The fundamental tensor $g^{ij}$ is $h-$covariant constant
\begin{equation}
\label{5.3.1}
g^{ij}_{|k}=\de_k g^{ij}+g^{sj}H^i_{sk}+g^{is}H^j_{sk}=0.
\end{equation}
\item[$3^\circ$] The fundamental tensor $g^{ij}$ is $v-$covariant constant
\begin{equation}
\label{5.3.2}
g^{ij}|^k=\dot{\pp}^k C^{ij}+g^{sj}C^{ik}_s+g^{is}C^{jk}_s=0.
\end{equation}
\item[$4$] $D\Gamma(N)$ is $h-$torsion free:
\begin{equation}
\label{5.3.3}
T^{i}_{jk}=H^i_{jk}-H^i_{kj}=0
\end{equation}
\item[$5^\circ$] $D\Gamma(N)$ is $v-$torsion free:
$$S_i^{jk}=C_i^{jk}-C_i^{kj}=0.$$
\end{itemize}
2) The $N-$connectiom $D\Gamma(N)$ which verify the previous axioms has the coefficients $H^i_{jk}$ and $C_i^{jk}$ given by the coefficients $H^i_{jk}$ and $C_i^{jk}$ given by the following generalized Christoffel symbols:
\begin{equation}
\label{5.3.4}
\begin{array}{l}
H^{i}_{jk}=\dfrac12 g^{is}(\de_j g_{sk}+\de_k g_{js}-\de_s g_{jk}),\\ \\
C^{jk}_i=-\dfrac12 g_{is}(\dot{\pp}^j g^{sk}+\dot{\pp}^k g^{js}-\pp^s g^{jk}).
\end{array}
\end{equation}
\end{theorem}

Clearly $C\Gamma(N)$ with the coefficients (\ref{5.3.4}) is determined only by means of the Hamilton space $H^n$. It is called canonical metrical $N-$linear connection.

Now, applying theorems from ch. 4 we have:
\begin{theorem}
With respect to $C\Gamma(N)$ we have the Ricci identities:
\begin{equation}
\label{5.3.5}
\begin{array}{l}
X^i_{|j|k}-X^i_{|k|j}=X^m R_{m\ jk}^{\ i}-X^i|^m R_{mjk},\\ \\
X^i_{|j}|^k-X^i|^k_{\ |j}=X^m P_{m\ j}^{\ i\ k}-X^i_{|m}C^{\ mk}_j-X^i|^m P^{\ k}_{\ mj}\\ \\
X^i|^j|^k-X^i|^k|^j=X^m S^{\ ijk}_m,
\end{array}
\end{equation}
where the $d-$tensors of torsion $R_{ijk},P^{i}_{\ jk}$ and $d$ tensors of curvature $R^{\ i}_{j\ kh}$, $P^{\ j\ k}_{i\ h}$, $S_i^{jkh}$ are expressed in chapter 4.
\end{theorem}

The structure equations, parallelism, autoparallel curves of the Hamilton spaces are studied exactly as in chapter 4.

\bigskip

\noindent{\bf Example.} The Hamilton space of electrodynamics $H^n=(M,H(x,p))$ where the fundamental function is expressed in (\ref{5.1.4}). The fundamental tensor is $\dfrac{1}{mc}g^{ij}(x)$. The canonical nonlinear connection has the coefficients
\begin{equation}
\label{5.3.6}
N_{ij}=\gamma^h_{ij}(x)p_h+\dfrac{e}{c}(A_{i|k}+A_{k|i})
\end{equation}
where $\gamma^h_{ij}(x)$ are the Christoffel symbols of the covariant tensor of metric tensor $g^{ij}(x)$. Evidently $A_{i|k}=\pp_k A_i-A_s\gamma^s_{ik}$.

Remarking that $\dfrac{\de g_{ij}}{\de x^k}=\dfrac{\pp g_{ij}}{\pp x^k}$ we deduce that: the coefficients of canonical metrical $N-$connection $H\Gamma(N)$ are: $$H^i_{jk}(x)=\gamma^i_{jk}(x),\ C^{jk}_i=0.$$ These geometrical object fields $H, g_{ij}$, $H^i_{jk}$, $C^{jk}_i$ allow to develop the geometry of the Hamilton spaces of electrodynamics.

\section{Generalized Hamilton Spaces $GH^n$}\index{Generalized spaces! Hamilton}

\setcounter{equation}{0}
\setcounter{definition}{0}
\setcounter{theorem}{0}
\setcounter{lemma}{0}

A straightforward generalization of the notion of Hamilton space is that of generalized Hamilton space.

\begin{definition}
 Ageneralized Hamilton space is a pair $GH^n=$ $(M,$ $g^{ij}(x,p))$ where $M$ is a smooth real $n-$dimensional manifold and $g^{ij}(x,p)$ is a $d-$tensor field on $\wt{T^*M}$ of type $(0,2)$ symmetric, nondegenerate and of constant signature.
\end{definition}

The tensor $g^{ij}(x,p)$ is called fundamental. If $M$ a paracompact on $T^*M$ there exist the tensors $g^{ij}(x,p)$ which determine a generalized Hamilton space.

From the definition 5.1.1 it follows that: Any Hamilton space $H^n=(M,H)$ is a generalized Hamilton space.

The contrary affirmation is not true. Indeed, if $\overset{\circ}g^{ij}(x)$ is a Riemannian tensor metric on $M$, then $$g^{ij}(x,p)=e^{-2\sigma(x,p)}\circg^{ij}(x),\ \sigma\in{\cal F}(T^*M)$$ determine a generalized Hamilton space and $g^{ij}(x,p)$ does not the fundamental tensor of an Hamilton space.

So, it is legitime the following definition:

\begin{definition}
A generalized Hamilton space $GH^n=(M,g^{ij}(x,p))$ is called reducible to a Hamilton space if there exists a Hamilton function $H(x,p)$ on $\wt{T^*M}$ such that
\begin{equation}
\label{5.4.2}
g^{ij}(x,p)=\dfrac12\dot{\pp}^i\dot{\pp}^jH.
\end{equation}
\end{definition}

Let us consider the Cartan tensor:
\begin{equation}
\label{5.4.3}
C^{ijk}=\dfrac12\dot{\pp}^i g^{jk}.
\end{equation}
In a similar manner with the case of generalized Lagrange space (ch. 2, \S2.8) we have:
\begin{proposition}
A necessary condition that a generalized Hamilton space $GH^n=(M,g^{ij}(x,p))$ be reducible to a Hamilton one is that the Cartan tensor $C^{ijk}(x,p)$ be totally symmetric.
\end{proposition}

There exists a particular case when the previous condition is sufficient, too.

\begin{theorem}
A generalized Hamilton space $GH^n=(M,g^{ij}(x,p))$ for which the fundamental tensor $g^{ij}(x,p)$ is 0-homogeneous is reducible to a Hamilton space, if and only if the Cartan tensor $C^{ijk}(x,p)$ is totally symmetric.
\end{theorem}

Indeed, in this case $H(x,p)=g^{ij}(x,p)p_i p_j$ is a solution of the equation (\ref{5.4.2}).

Let $g_{ij}(x,p)$ be the covariant of fundamental tensor $g^{ij}(x,p)$, then the following tensor field
\begin{equation}
\label{5.4.4}
C^{jk}_i=-\dfrac12 g_{is}(\dot{\pp}^{j}g^{sk}+\dot{\pp}^k g^{js}-\dot{\pp}^s g^{sk})
\end{equation}
determine the coefficients of a $v-$covariant derivative, which is metrical, i.e.
\begin{equation}
\label{5.4.5}
g^{ij}|^k=\dot{\pp}^k g^{ij}+C^{ik}_s g^{sj}+C^{jk}_s g^{is}=0.
\end{equation}
The proof is very simple.

We use this $v-$derivation in the theory of canonical metrical connection of the spaces $GH^n$.

In general, we cannot determine a nonlinear connection from the fundamental tensor $g^{ij}$ of $GH^n$. Therefore we study the geometry of spaces $GH^n$ when a nonlinear connection $N$ is a priori given. In this case we can apply the methods used in the construction of geometry of spaces $H^n$.

Finally we remark that the class of spaces $GH^n$ include the class of spaces $H^n$:
\begin{equation}
\label{5.4.6}
\{GH^n\}\supset\{H^n\}.
\end{equation}

\section{The almost K\"{a}hlerian model of a Hamilton space}\index{Almost Hermitian model}

\setcounter{equation}{0}
\setcounter{definition}{0}
\setcounter{theorem}{0}
\setcounter{lemma}{0}

Let $H^n=(M,H(x,p))$ be a Hamilton space and $g^{ij}(x,p)$ its fundamental tensor.

The canonical nonlinear connection $N$ has the coefficients (\ref{5.2.1}). The adapted basis to the distributions $N$ and $V$ is $\left(\dfrac{\de}{\de x_i}=\right.$ $\de_i=\pp_i+$ $N_{ij}\dot{\pp}^j,$ $\left.\dot{\pp}^i\right)$ and its dual basis $(dx^i,\de p_i=dp_i-N_{ji}dx^j)$.

Thus, the following tensor on the cotangent manifold $\wt{T^*M}=T^*M\setminus\{0\}$ can be considered:
\begin{equation}
\label{5.5.1}
\G=g_{ij}(x,p)dx^i\otimes dx^j+g^{ij}\de p_i\otimes \de p_j.
\end{equation}
$\G$ determine a pseudo-Riemannian structure on $\wt{T^*M}$. If the fundamental tensor $g^{ij}(x,p)$ is positive defined, then $\G$ is a Riemannian structure on $T^*M$.
$\G$ is called the $N-$lift of the fundamental tensor $g^{ij}$. Clearly, $\G$ is determined by Hamilton space $H^N$, only.

Some properties of $\G$:
\begin{itemize}
\item[$1^\circ$] $\G$ is uniquely determined by $g^{ij}$ and $N_{ij}$.
\item[$2^\circ$] The distributions $N$ and $V$ are orthogonal.
\end{itemize}

Taking into account the mapping $\check{\F}:{\cal X}(\wt{T^* M})\to{\cal X}(T^*M)$ defined in (4.3.10) or, equivalently by:
\begin{equation}
\label{5.5.2}
\check{\F}=-g_{ij}\dot{\pp}^i\otimes dx^j+g^{ij}\de_i\otimes \de p_j,
\end{equation}
one obtain:
\begin{theorem}
\begin{itemize}
\item[$1^\circ$] $\check{\F}$ is globally defined on the manifold $\wt{T^*M}$.
\item[$2^\circ$] $\check{\F}$ is an almost complex structure on $\wt{T^*M}$:
\begin{equation}
\label{5.5.3}
\check{\F}\cdot\check{\F}=-I.
\end{equation}
\item[$3^\circ$] $\check{\F}$ depends on the Hamilton spaces $H^n$ only.
\end{itemize}
\end{theorem}

Finally, one obtain a particular form of the Theorem 4.3.2:
\begin{theorem}
The following properties hold:
\begin{itemize}
\item[$1^\circ$] The pair $(\G,\check{\F})$ is an almost Hermitian structure on the manifold $\wt{T^*M}$.
\item[$2^\circ$] The structure $(\G,\F)$ is determined only by the fundamental function $H(x,p)$ of the Hamilton space $H^n$.
\item[$3^\circ$] The associated almost symplectic structure to $(\G,\check{\F})$ is the canonical symplectic structure $\theta=dp_i\wedge dx^i=\de p_i\wedge dx^i$.
\item[$4^\circ$] The space $({\wt{T^*M}},\G,\check{\F})$ is almost K\"{a}hlerian.
\end{itemize}
\end{theorem}

The proof is similar with that for Lagrange space (cf. Ch. 3).

The space ${\cal H}^2n=({\wt{T^*M}},\G,\check{\F})$ is called the almost K\"{a}hlerian model of the Hamilton space $H^n$. By means of ${\cal H}^{2n}$ we can realize the study of gravitational and electromagnetic fields \cite{MiKa}, \cite{miradio}, \cite{mitava1}, \cite{mirzet}, on $\wt{T^*M}$.

\newpage

\chapter{Cartan spaces}

A particular class of Hamilton space is given by the class of Cartan spaces. It is formed by the spaces $H^n=(M,H(x,p))$ for which the fundamental function $H$ is $2-$homogeneous with respect to momenta $p_i$. It is remarkable that these spaces appear as dual of the Finsler spaces, via Legendre transformations. Using this duality, several important results in Cartan spaces can be obtained: the canonical nonlinear connection, the canonical metrical connection etc. Therefore, the theory of Cartan spaces has the same symmetry and beauty like Finsler geometry. Moreover, it gives a geometrical framework for the Hamiltonian Mechanics or Physics fields.

The modern formulation of the notion of Cartan space is due of R. Miron, but its geometry is based on the investigations of E. Cartan, A. Kawaguchi, H. Rund, R. Miron, D. Hrimiuc, and H. Shimada, P.L. Antonelli, S. Vacaru et. al. This concept is different from the notion of areal space defined by E. Cartan.

In the final part of this chapter we shortly present the notion of duality between Lagrange and Hamilton spaces.

\section{Notion of Cartan space}
\setcounter{equation}{0}
\setcounter{definition}{0}
\setcounter{theorem}{0}
\setcounter{lemma}{0}

\begin{definition}
A Cartan space is a pair $C^n=(M,K(x,p))$ where $M$ is a real $n-$dimensional smooth manifold and $K:T^*M\to R$ is a scalar function which satisfies the following axioms:
\begin{enumerate}
\item $K$ is differentiable on $\wt{T^*M}$ and continuous on the null section of $\pi^*:T^*M\to M$.
\item $K$ is positive on the manifold $T^*M$.
\item $K$ is positive $1-$homogeneous with respect to the momenta $p_i$.
\item The Hessian of $K^2$ having the components
\begin{equation}
\label{6.1.1}
g^{ij}(x,p)=\dfrac12\dot{\pp}^i\dot{\pp}^j K^2
\end{equation}
is positive defined on the manifold $\wt{T^*M}$.
\end{enumerate}
It follows that $g^{ij}(x,p)$ is a symmetric and nonsingular $d-$tensor field of type $(0,2)$.
\end{definition}

So, we have
\begin{equation}
\label{6.1.2}
{\rm{rank}}\|g^{ij}(y,p)\|=n\ {\rm{on\ }}\wt{T^*M}.
\end{equation}

The functions $g^{ij}(x,p)$ are $0-$homogeneous with respect to momenta $p_i$.

For a Cartan space ${\cal C}^n=(M,K(x,p))$ the function $K$ is called fundamental function and $g^{ij}$ the fundamental or metric tensor.

If the base $M$ is paracompact, then on the manifold $T^*M$ there exists real function $K$ such that the pair $(M,K)$ is a Cartan space.

Indeed, on $M$ there exists a Riemann structure $g_{ij}(x)$, $x\in M$. Then
\begin{equation}
\label{6.1.3}
K(x,p)=\{g^{ij}(x)p_i p_j\}^{1/2}
\end{equation}
determine a Cartan space.

Other examples are given by
\begin{equation}
\label{6.1.4}
K=\a^*+\b^*
\end{equation}
\begin{equation}
\label{6.1.5}
K=\dfrac{(\a^*)^2}{\b}
\end{equation}
where
\begin{equation}
\label{6.1.6}
\a^*=\{g^{ij}(x)p_i p_j\}^{1/2},\ \b^*=b^i(x)p_i.
\end{equation}
$K$ from (\ref{6.1.4}) is called a {\sl Randers} metric and $K$ from (\ref{6.1.5}) is called a Kropina metric.

A first and immediate result:
\begin{theorem}
If ${\cal C}^n=(M,K)$ is a Cartan space then the pair $H^n_{{\cal C}}=(M,K^2)$ is an Hamilton space.
\end{theorem}
$H^n_{\cal C}$ is called the associate Hamilton space with ${\cal C}^n$.

This is the reason that the geometry of Cartan space include the geometry of associate Hamilton space. So, we have the sequence of inclusions
\begin{equation}
\label{6.1.7}
\{{\cal R}^n\}\subset\{{\cal C}^n\}\subset\{H^n\}\subset\{GH^n\}
\end{equation}
where $\{{\cal R}^n\}$ is the class of Riemann spaces ${\cal R}^n=(M,g^{ij}(x))$ which give the Cartan spaces with the metric (\ref{6.1.3}).

Now we can apply the theory from previous chapter.

The canonical symplectic structure $\theta$ on $T^*M$:
\begin{equation}
\label{6.1.8}
\theta=dp_i\wedge dx^i
\end{equation}
and ${\cal C}^n$ determine the Hamiltonian system $(T^*M,\theta,K^2)$. Then, setting
\begin{equation}
\label{6.1.9}
{\cal K}=\dfrac12 K^2,
\end{equation}
we have:
\begin{theorem}
For any Cartan space ${\cal C}^n=(M,K(x,p))$, the following properties hold:
\begin{itemize}
\item[$1^\circ$] There exists a unique vector field $X_{K^2}$ on $\wt{T^*M}$ for which
\begin{equation}
\label{6.1.10}
i_{X_K}\theta=-d{\cal K}.
\end{equation}
\item[$2^\circ$] The vector field $X_K$ is expressed by
\begin{equation}
\label{6.1.11}
X_K=\dfrac{\pp{\cal K}}{\pp p_i}\dfrac{\pp}{\pp x^i}-\dfrac{\pp{\cal K}}{\pp x^i}\dfrac{\pp}{\pp p_i}.
\end{equation}
\item[$3^\circ$] The integral curves of the vector field $X_K$ are given by the Hamilton--Jacobi equations:
\begin{equation}
\label{6.1.12}
\dfrac{dx^i}{dt}=\dfrac{\pp{\cal K}}{\pp p_i},\ \dfrac{dp_i}{dt}=-\dfrac{\pp{\cal K}}{\pp x^i}.
\end{equation}
\end{itemize}
\end{theorem}

One deduce $\dfrac{d{\cal K}}{dt}=\{{\cal K},{\cal K}\}=0$. So:
\begin{proposition}
The fundamental function ${\cal K}=\dfrac{1}{2}K^2$ of a Cartan space is constant along the integral curves of the Hamilton--Jacobi equations (\ref{6.1.12}).
\end{proposition}

The Jacobi method of integration of (\ref{6.1.12}) can be applied.

The Hamilton--Jacobi equations of a Cartan space ${\cal C}^n$ are fundamental for the geometry of ${\cal C}^n$. Therefore, the integral curves of the system of differential equations (\ref{6.1.12}) are called the geodesics of Cartan space ${\cal C}^n$.

Other properties of the space ${\cal C}^n$:
\begin{itemize}
\item[$1^\circ$] $p^i=\dfrac12\dot{\pp}^i K^2$ is a 1-homogeneous $d-$vector field.
\item[$2^\circ$] $g^{ij}=\dot{\pp}^j p^i=\dfrac12\dot{\pp}^i\dot{\pp}^j K^2$ is 0-homogeneous tensor field (the fundamental tensor of ${\cal C}^n$).
\item[$3^\circ$] $C^{ijh}=-\dfrac14\dot{\pp}^i\dot{\pp}^j\dot{\pp}^h K^2$ is $(-1)-$homogeneous with respect to $p_i$.
\end{itemize}

\begin{proposition}
We have the following identities:
\begin{equation}
\label{6.1.13}
p^i=g^{ij}p_j,\ \ p_i=g_{ij}p^j,
\end{equation}
\begin{equation}
\label{6.1.14}
K^2=g^{ij}p_i p_j=p_i p^i,
\end{equation}
\begin{equation}
\label{6.1.15}
C^{ijh}p_h=C^{ihj}p_h=C^{hij}p_h=0.
\end{equation}
\end{proposition}

\begin{proposition}
The Cartan space ${\cal C}^n=(M,K(x,p))$ is Riemannian if and only if the $d-$tensor $C^{ijh}$ vanishes.
\end{proposition}

Consider $d-$tensor:
\begin{equation}
\label{6.1.16}
C^{jh}_i=-\dfrac12 g_{is}\dot{\pp}^s g^{jh}=g_{is}C^{sjh}.
\end{equation}
Thus $C^{jh}_i$ are the coefficients of a $v-$metric covariant derivation:
\begin{equation}
\label{6.1.17}
\begin{array}{l}
g^{ij}|^h=0,\ \ S^{jh}_i=0\\ \\
p_i|^k=\de^k_i.
\end{array}
\end{equation}

\section{Canonical nonlinear connection of ${\cal C}^n$}\index{Connections! Nonlinear}
\setcounter{equation}{0}
\setcounter{definition}{0}
\setcounter{theorem}{0}
\setcounter{lemma}{0}

The canonical nonlinear connection of a Cartan space ${\cal C}^n=(M,K)$ is the canonical nonlinear connection of the associate Hamilton space $H^n_{{\cal C}}=(M,K^2)$. Its coefficients $N_{ij}$ are given by (5.2.1).

Setting
\begin{equation}
\label{6.2.1}
\gamma^i_{jh}=\dfrac12 g^{is}(\pp_j g_{sh}+\pp_h g_{js}-\pp_s g_{jh})
\end{equation}
for the Christoffel symbols of the covariant of fundamental tensor of ${\cal C}^n$ and using the notations
\begin{equation}
\label{6.2.2}
\gamma^0_{jh}=\g^i_{jh}p_i,\ \ \g^0_{j0}=\g^i_{jh}p_i p^h,
\end{equation}
we obtain:

\begin{theorem}[Miron]
The canonical nonlinear connection of the Cartan space ${\cal C}^n=(M,K(x,p))$ has the coefficients
\begin{equation}
\label{6.2.3}
N_{ij}=\g^0_{ij}-\dfrac12\g^0_{h0}\dot{\pp}^h g_{ij}.
\end{equation}
\end{theorem}

The proof is based on the formula (5.2.1).

Evidently, the canonical nonlinear connection is symmetric
\begin{equation}
\label{6.2.4}
N_{ij}=N_{ji}.
\end{equation}

Let us consider the adapted basis $(\de_i,\dot{\pp}^i)$ to the distributions $N$ (determined by canonical nonlinear connection) and $V$ (vertical distribution). We have
\begin{equation}
\label{6.2.5}
\de_i=\pp_i+N_{ij}\dot{\pp}^j.
\end{equation}
The adapted dual basis $(dx^i,\de p_i)$ has the forms $\de p_i$:
\begin{equation}
\label{6.2.5'}
\de p_i=dp_i-N_{ij}dx^j.\tag{6.2.5'}
\end{equation}

The $d-$tensor of integrability of horizontal distribution $N$ is
\begin{equation}
\label{6.2.6}
R_{ijh}=\de_h N_{ji}-\de_j N_{hi}.
\end{equation}
By a direct calculus we obtain
\begin{equation}
\label{6.2.7}
R_{ijh}+R_{jki}+R_{kij}=0.
\end{equation}

\begin{proposition}
The horizontal distribution $N$ determined by the canonical nonlinear connection of a Cartan space ${\cal C}^n$ is integrable if and only if the $d-$tensor field $R_{ijk}$ vanishes.
\end{proposition}

\begin{proposition}
The canonical nonlinear connection of a Cartan space ${\cal C}^n(M,K(x,p))$ depends only on the fundamental function $K(x,$ $p)$.
\end{proposition}

\section{Canonical metrical connection of ${\cal C}^n$}\index{Connections! $N-$linear}
\setcounter{equation}{0}
\setcounter{definition}{0}
\setcounter{theorem}{0}
\setcounter{lemma}{0}

Consider the canonical metrical $N-$linear connection of the Hamilton space $H^N_e=(M,K^2)$. It is the canonical metrical $N-$linear connection of the Cartan space ${\cal C}^n=(M,K)$. It is denoted by $C\Gamma(N)=(H^i_{jk},C_i^{jk})$ and shortly is named canonical metrical connection of ${\cal C}^n$.

Then, theorem 5.3.1. implies:
\begin{theorem}
1) In a Cartan space ${\cal C}^n=(M,K(x,p))$ there exists a unique $N-$linear connection $C\Gamma(N)=(H^i_{jk},C^{jk}_i)$ verifying the axioms:
\begin{itemize}
\item[$1^\circ$] $N$ is the canonical nonlinear connection of ${\cal C}^n$.
\item[$2^\circ$] $C\Gamma(N)$ is $h-$metrical
\begin{equation}
\label{6.3.1}
g^{ij}_{|h}=0.
\end{equation}
\item[$3^\circ$] $C\Gamma(N)$ is $v-$metrical:
\begin{equation}
\label{6.3.2}
g^{ij}|^h=0.
\end{equation}
\item[$4^\circ$] $C\Gamma(N)$ is $h-$torsion free: $T^i_{\ jh}=H^i_{jh}-H^i_{jh}=0$.
\item[$5^\circ$] $C\Gamma(N)$ is $v-$torsion free: $S_i^{\ jh}=C_i^{jh}-C_i^{hj}=0$.
\end{itemize}
2) The connection $C\Gamma(N)$ has the coefficients given by the generalized Christoffel symbols:
\begin{equation}
\label{6.3.3}
\begin{array}{l}
H^i_{jh}=\dfrac12 g^{is}(\de_j g_{sh}+\de_h g_{js}-\de_s g_{jh})\\ \\
C^{jh}_i=-\dfrac12 g_{is}(\dot{\pp}^j g^{sh}+\dot{\pp}^h g^{js}-\dot{\pp}^s g^{jh}).
\end{array}
\end{equation}
3) $C\Gamma(N)$ depends only on the fundamental function $K$ of the Cartan space ${\cal C}^n$.

\noindent The connection $C\Gamma(N)$ is called canonical metrical connection of Cartan space ${\cal C}^n$.
\end{theorem}

We denote $C\Gamma=(N_{ij},H^i_{jk},C^{jk}_i)$.

Evidently, the $d-$tensor $C_i^{jh}$ has the properties (\ref{6.1.15}) and the coefficients $C\Gamma$ have $1,0,-1$ homogeneity degrees.

\begin{proposition}
The canonical metrical connections has the tensor of deflection:
\begin{equation}
\label{6.3.4}
\Delta_{ij}=p_{i|j}=0,\ \ \delin^i_j=p_j |^i=\de^i_j.
\end{equation}
\end{proposition}

But, the previous result allows to give a characterization of $C\Gamma$ by a system of axioms of Matsumoto type:
\begin{theorem}
$1^\circ$ For any space ${\cal C}^n=(M,K(x,p))$ there exists a unique linear connection $C\Gamma=(N_{ij},H^i_{jk},C_i^{jk})$ on the manifold $\wt{TM}$ verifying the following axioms:
$$1'.\ \Delta_{ij}=0;\ \ 2'.\ g^{ij}_{|h}=0;\ \ 3'.\ g^{ij}|^h=0;\ \ 4'.\ T^i_{jh}=0;\ \ 5'.\ S^{jh}_i=0.$$

$2^\circ$ The previous metrical connection is exactly the canonical metrical connection $C\Gamma$.
\end{theorem}

The following properties of $C\Gamma(N)$ are immediately
\begin{itemize}
\item[$1^\circ$] $K_{|h}=0$, $K|^h=\dfrac{p^h}{K}$;
\item[$2^\circ$] $K^2_{|h}=2p^h$;
\item[$3^\circ$] $p_{i|j}=0$, $p_i |^j=\de^j_i$;
\item[$4^\circ$] $p^i_{|j}=0$, $p^i |^j=g^{ij}$.
\end{itemize}
And the $d-$tensors of torsion of $C\Gamma(N)$ are
\begin{equation}
\label{6.3.5}
R_{ijh}=\delta_h N_{ij}-\de_j N_{ih}, C^{jh}_i, T^i_{jh}=0, S^{jh}_i=0, P^i_{jh}=H^i_{jh}-\dot{\pp}^iN_{jh}.
\end{equation}
Of course, we have
\begin{equation}
\label{6.3.6}
R_{ijh}=-R_{ihj}, P^i_{jh}=P^i_{hj},\ C^{jh}_i=C^{hj}_i.
\end{equation}

\begin{proposition}
The $d-$tensor of curvature $S^{ijh}_k$ is given by:
\begin{equation}
\label{6.3.7}
S^{ijh}_k=C^{mh}_k C^{ij}_m-C^{mj}_k C^{ih}_m.
\end{equation}
\end{proposition}

Denoting by $$R^{ij}_{\ kh}=g^{is} R^{\ j}_{s\ kh},\ etc.$$ and applying the Ricci identities (5.3.5), one obtains:
\begin{theorem}
The canonical metrical connection $C\Gamma(N)$ of the Cartan space ${\cal C}^n$ satisfies the identities:
\begin{equation}
\label{6.3.8}
R^{ij}_{\ kh}+R^{ji}_{\ kh}=0,\ P^{ij\ h}_{\ k}+P^{ji\ h}_{\ k}=0,\ S^{ijkh}+S^{jikh}=0.
\end{equation}
\begin{equation}
\label{6.3.9}
R^{\ o}_{i\ jk}+R_{ijk}=0,\ P^{\ o\ k}_{i\ j}+P^{k}_{\ ij}=0,\ S^{ojk}_i=0.
\end{equation}
\begin{equation}
\label{6.3.10}
R_{ojk}=0,\ P^k_{\ oj}=0.
\end{equation}
\end{theorem}
Of course, the index ``$o$'' means the contraction by $p_i$ or $p^i$.

The 1-form connections of $C\Gamma(N)$ are:
\begin{equation}
\label{6.3.11}
\oo^i_{\ j}=H^i_{\ jh}dx^h + C^{ih}_j\de p_h.
\end{equation}
Taking into account (\ref{6.2.5'}), one obtains
\begin{theorem}
The structure equations of the canonical metrical connection $C\Gamma(N)$ of the Cartan space ${\cal C}^n=(M,K(x,p))$ are
\begin{equation}
\label{6.3.12}
\begin{array}{l}
d(dx^i)-dx^m\wedge\oo^i_m=-\O^i;\\ \\
d(\de p_i)+\de p_m\wedge\oo^m_i=-\overset{\circ}\O_i;\\ \\
d\oo^i_j-\oo^m_j\wedge\oo^i_m=-\O^i_j,
\end{array}
\end{equation}
$\O^i$, $\overset{\circ}\O_i$ being the 2-forms of torsion:
\begin{equation}
\label{6.3.13}
\begin{array}{l}
\O^i=C^{ik}_j dx^j\wedge\de p_k\\ \\
\overset{\circ}\O_i=\dfrac12 R_{ijk}dx^j\wedge dx^k+P^k_{\ ij}dx^j\wedge\de p_k
\end{array}
\end{equation}
and $\O^i_j$ is 2-form of curvature:
\begin{equation}
\label{6.2.14}
\O^i_j\hspace*{-0.5mm}=\hspace*{-0.5mm}\dfrac12 R^i_{j\ km}dx^k\wedge dx^m+P^{\ i\ m}_{j\ k}dx^k\wedge\de p_m\hspace*{-0.5mm}+\hspace*{-0.5mm}\dfrac12 S^{\ ikm}_j\de p_k\wedge\de p_m.
\end{equation}
\end{theorem}

Applying Proposition 4.4.2, we determine the Bianchi identities of $C\Gamma(N)$.

Now, we can develop the geometry of the associated Hamilton space $H^n_{\cal C}=(M,K^2(x,p))$. Also, in the case of Cartan space, the geometrical model ${\cal K}^{2n}=(\wt{T^*M},\G,\F)$ is an almost K\"{a}hlerian one.

Before finish this chapter is opportune to say some words on the Legendre transformation.

\section{The duality between Lagrange and Hamilton spaces}\index{$L-$duality}
\setcounter{equation}{0}
\setcounter{definition}{0}
\setcounter{theorem}{0}
\setcounter{lemma}{0}

The duality, via Legendre transformation, between Lagrange and Hamilton space was formulated by R. Miron in the papers and it was developed by D. Hrimiuc, P.L. Antonelli, D. Bao, et al. Of course, it was suggested by Theoretical Mechanics.

The theory of Legendre duality is presented here follows the Chapter 7 of the book \cite{MHS}.

Let $L$ be a regular Lagrangian, ${\cal L}=\dfrac12 L$ on a domain ${\cal D}\subset TM$ and let $H$ be a regular Hamiltonian ${\cal H}=\dfrac12 H$, on a domain ${\cal D}^*\subset T^*M$.

Hence for $\dot{\pp}_i=\dfrac{\pp}{\pp y^i}$, $\dot{\pp}^i=\dfrac{\pp}{\pp p_i}$ the matrices with entries:
\begin{equation}
\label{6.4.1}
\begin{array}{l}
g_{ij}(x,y)=\dot{\pp}_i\dot{\pp}_j{\cal L}(x,y),\\ \\
g^{*ij}(x,p)=\dot{\pp}^i\dot{\pp}^j{\cal H}(x,p)
\end{array}
\end{equation}
are nondegenerate on ${\cal D}$ and on ${\cal D}^*$, $(i,j,...=1,2,...,n)$.

Since ${\cal L}\in{\cal F}({\cal D})$ is a differentiable map consider the {\sl fiber derivative} of ${\cal L}$ locally given by
\begin{equation}
\label{6.4.2}
\varphi(x,y)=(x^i,\dot{\pp}_j{\cal L}(x,y))
\end{equation}
which is called the {\it Legendre transformation}.
It is easy to see that: ${\cal L}$ is a regular Lagrangian if and only if $\varphi$ is a local diffeomorphism, \cite{MHS}.
In the same manner, for ${\cal H}\in{\cal F}{\cal D}^*$, the fiber derivative is locally given by
\begin{equation}
\label{6.4.3}
\psi(x,y)=(x^i,\dot{\pp}^i{\cal H}(x,y)),
\end{equation}
which is a local diffeomorphism if and only if ${\cal H}$ is regular.

Now, let us consider $\dfrac12 L(x,y)={\cal L}(x,y)$ the fundamental function of a Lagrange space. Then $\varphi$ defined by (\ref{6.4.2}) is a diffeomorphism between two open set $U\subset{\cal D}$ and $U^*\subset{\cal D}^*$. In this case, we can define
\begin{equation}
\label{6.4.4}
{\cal H}^*(x,p)=p_i y^i-{\cal L}(x,y),
\end{equation}
where $y=(y^i)$ is solution of the equation
\begin{equation}
\label{6.4.4'}
p_i=\dot{\pp}_i{\cal L}(x,y).\tag{6.4.4'}
\end{equation}

\begin{remark}
In the theory of Lagrange space the function ${\cal E}(x,p)=y^i\dot{\pp}_i{\cal L}-{\cal L}(x,y)$ is the energy of Lagrange space $L^n$.
\end{remark}

It follows without difficulties that $H^{*n}=(M,H^*(x,p))$, $H^*=2{\cal H}^*$ is a Hamilton space. Its fundamental tensor $g^{*ij}(x,p)$ is given by ${g^{ij}(x,\varphi^{-1}(x,p))}$.

We set $H^{*n}=Leg\ L^n$ and say that $H^{*n}$ is the dual of $L^n$ via Legendre transformation determined by $L^n$.

\begin{remark}
The mapping $Leg$ was used in chapter 2 to transform the Euler--Lagrange equations (ch. 2, part I) into the Hamilton--Jacobi equations (ch. 2, part I).
\end{remark}

Analogously, for a Hamilton space $H^n=(M,H(x,p))$, ${\cal H}=\dfrac12 H$, consider the function $${\cal L}^*(x,y)=p_i y^j-{\cal H}(x,p),$$ where $p=(p_i)$ is the solution of the equation (\ref{6.4.4}). Thus $L^{*n}=(M,L^*(x,y))$, $L^*=2{\cal L}^*$, is a Lagrange space, dual, via Legendre transformation of the Hamilton space $H^n$. So, $L^{*n}=Leg\ H^n$.

One proves: $Leg(Leg\ L^n)=L^n$, $Leg(Leg\ H^n)=H^n$.

The diffeomorphisms $\varphi$ and $\psi$ have the property $\varphi=\psi^{-1}$. And they transform the fundamental object fields from $L^n$ in the fundamental object fields of $H^{*n}=Leg\ L^n$, and conversely. For instance, the Euler--Lagrange equation of $L^n$ are transformed in the Hamilton--Jacobi equations of $H^n$. The canonical nonlinear connection of $L^n$ is transformed in the canonical nonlinear connection of $H^{*n}$ etc.

\bigskip

\noindent{\bf Examples.}

$1^\circ$ The Lagrange space of Electrodynamics $L^n_0=(M,L_0(x,y))$: $$L_0(x,y)=mc\g_{ij}(x)y^iy^j+\dfrac{2e}{m}A_i(x)y^i$$ where $\g_{ij}(x)$ is a pseudo-Riemannian metric, has the Legendre transformation: $$\varphi:\ x^i=x^i,\ p_i=\dfrac12\dfrac{\pp L_0}{\pp y^i}=mc\g_{ij}(x)y+\dfrac{e}{m}A_i(x)$$ and $$\varphi^{-1}=\psi:\ x^i=x^i,\ y^i=\dfrac1{mc}\g^{ij}(x)\left(p_j-\dfrac{e}{m}A_j(x)\right).$$ The Hamilton space $H^*_0=Leg\ L_0$ has the fundamental function:
\begin{equation}
\label{6.4.6}
H^*_0=\dfrac1{mc}\g^{ij}(x)p_ip_j-\dfrac{2e}{mc^2}A^i(x)p_i+\dfrac{e^2}{mc^3}A^i(x)A_j(x).
\end{equation}

$2^\circ$ The Lagrange space of Electrodynamics $L^n=(M,L(x,y))$ in which $L=L_0+U(x)$, ((2.1.4, Ch. 2)), i.e. $$L(x,y)=mc\g_{ij}(x)y^iy^j+\dfrac{2e}{m}A_i(x)y^i+U(x),$$ has the Legendre transformation: $$\varphi_i:\ x^i=x^i,\ p_i=mc\g_{ij}(x)y^j+\dfrac{e}{m}A_i(x)$$ $$\varphi^{-1}=\psi:\ x^i=x^i,\ y^i=\dfrac1{mc}\g^{ij}(x)\left(p_j-\dfrac{e}{m}A_j\right).$$  And $H^{*n}=Leg\ L^n$ has the fundamental function $H^*(x,p)$ given by $$H^*(x,p)=\dfrac{1}{c(m+\frac{U}{c^2})}\g^{ij}(x)p_ip_j-\dfrac{2e}{c^2(m+\frac{U}{c^2})}p_iA^i(x)+$$ $$+\dfrac{e^2}{c^3(m+\frac{U}{c^2})}A_i(x)A^i(x)-\dfrac{1}{c}\dfrac{U(x)}{c}.$$

$3^\circ$ The class of Finsler space $\{F^n\}$ is inclosed in the calss of Lagrange space $\{L^n\}$, that is $\{F^n\}\subset\{L^N\}$, one can consider the restriction of the Legendre transformation $Leg\ L^n$ to the class of Finsler spaces.

In this case, the mapping $\varphi:(x,y)\in D\to (x,p)\in D^*$ is given by $$\varphi(x,y)=(x,p)\ {\rm{with}}$$
\be
p_i=\dfrac12\dot{\pp}^i F^2
\ee
and one obtains: $Leg(F^n)={\cal C}^{*n}$, ${\cal C}^{*n}=(M,K^*(x,p))$ with $$K^{*2}=2p_i y^i-F^2(x,y)=y^i\dot{\pp}_i F^2(x,y)=F^2(x,y)$$ and $y^i$ is solution of the equation (6.4.6).

\begin{theorem}
The dual, via Legendre transformation, of a Finsler space $F^n=(M,F(x,y))$ is a Cartan space ${\cal C}^{*n}=(M,F(x,\varphi^{-1}(x,p)))$.
\end{theorem}

\newpage

\begin{partbacktext}
\part{Lagrangian and Hamiltonian Spaces of higher order}

\noindent In this part of the book we study, the notions of Lagrange and Hamilton spaces of order $k$. They was introduced and investigated by the author \cite{mir11}. Without explicitly formulated a clear definition of these spaces, a major contributions to edifice of these geometrics have been done by M. Crampin and colab. \cite{Cr}, M. de Leone and colab. \cite{leonrodri}, A. Kawaguchi \cite{kawa2}, I. Vaisman \cite{VI2} etc.

For details, we refer to the books: {\it The Geometry of Higher Order Lagrange Spaces. Applications to Mechanics and Physics}, Kluwer Acad. Publ. FTPH, 82, 1997, \cite{mir11}; {\it The Geometry of Higher--Order Finsler Spaces}, Hadronic Press, Inc. USA, 1998, \cite{mir12}; {\it The Geometry of Higher Order Hamilton Spaces. Applications to Hamiltonian Mechanics}, Kluwer Acad. Publ., FTPH 132, 2003 \cite{mir13}; as well as the papers \cite{mirata}--\cite{mirata6}.

This part contents: The geometry of the manifold of accelerations $T^k M$, Lagrange spaces of order $k$, $L^{(k)n}$, Finsler spaces of order $k$, $F^{(k)n}$ and dual, via Legendre transformation Hamilton spaces $H^{(k)n}$ and Cartan spaces ${\cal C}^{(k)n}$.

\end{partbacktext}

\setcounter{chapter}{0}
\chapter{The Geometry of the manifold $T^kM$}
\label{ch1p1} 


The importance of Lagrange geometries consists of the fact that
the variational problems for Lagrangians have numerous
applications: Mathematics, Physics, Theory
of Dynamical Systems, Optimal Control, Biology, Economy etc.

But, all of the above mentioned applications have imposed also the
introduction of the notions of higher order Lagrange spaces. The
base manifold of this space is the bundle of accelerations of
superior order. The methods used in the construction of the
geometry of higher order Lagrange spaces are the natural
extensions of those used in the edification of the Lagrangian
geometries exposed in chapters 1, 2 and 3.

The concept of higher order Lagrange space was given by author in
the books \cite{mir11}, \cite{mir5}. The problems raised by the geometrization of
Lagrangians systems of order $k>1$ were been investigated by many
scholars: Ch. Ehresmann \cite{ehre2}, P. Libermann \cite{Libermann}, J. Pommaret \cite{pomm}, J. T. Synge \cite{SJ},
M. Crampin \cite{Cr}, P. Saunders \cite{saunders}, G.S. Asanov \cite{As}, D. Krupka \cite{krupka}, M. de L\'{e}on \cite{deLeon},
H. Rund \cite{rund}, A. Kawaguchi \cite{kawa}, K. Yano \cite{YI}, K. Kondo \cite{kondo}, D.
Grigore \cite{grigore}, R. Miron \cite{mir5}, \cite{mir6} et al.

In this chapter we shall present, briefly, the following problems:
\begin{itemize}
\item[$1^\circ$] The geometry of total space of the bundle of
higher order accelerations. \item[$2^\circ$] The definition of
higher order Lagrange space, based on the nondegenerate
Lagrangians of order $k\geq 1$. \item[$3^\circ$] The solving of
the old problem of prolongation of the Riemannian structures,
given on the base manifold $M$, to the Riemannian structures on
the total space of the bundle of accelerations of order $k\geq 1$,
we prove for the first time the existence of Lagrange spaces of
order $k\geq 1$. \item[$4^\circ$] The elaboration of the
geometrical ground for variational calculus involving Lagrangians
which depend on higher order accelerations. \item[$5^\circ$] The
introduction of the notion of higher order energies and proof of
the law of conservation. \item[$6^\circ$] The notion of
$k-$semispray. Nonlinear connection the canonical metrical
connection and the structure equations. \item[$7^\circ$] The
Riemannian $(k-1)n-$almost contact model of a Lagrange space of
order $k$.
\end{itemize}

Evidently, we can not sufficiently develop these subjects. For
much more informations one can see the books \cite{mir11}, \cite{mir12}.

Throughout in this chapter the differentiability of manifolds and
of mappings means the class $C^\infty$.

\section{The bundle of acceleration of order $k\geq 1$}\index{Acceleration of order $T^kM$}
\setcounter{equation}{0}\setcounter{theorem}{0}\setcounter{definition}{0}\setcounter{proposition}{0}\setcounter{remark}{0}

In Analytical Mechanics a real $n-$dimensional differentiable
manifold $M$ is considered as space of configurations of a
physical system. A point $(x^i)\in M$ is called a material point. A
mapping $c:t\in I\rightarrow (x^i(t))\in U\subset M$ is a law of
moving (a law of evolution), $t$ is time, a pair $(t,x)$ is an event and
the $k$-uple $\left(\displaystyle\frac{dx^i}{dt},\cdots,\displaystyle\frac 1{k!}\displaystyle\frac{d^k x^i}{dt^k}\right)$
gives the velocity and generalized accelerations of order $2$,
..., $k-1$. The factors $\displaystyle\frac 1{h!}$ ($h=1,...,k$)
are introduced here for the simplicity of calculus. In this
chapter we omit the word ``generalized'' and say shortly, the
acceleration of order $h$, for $\displaystyle\frac
1{h!}\displaystyle\frac{d^h x^i}{dt^h}$. A law of moving $c:t\in
I\rightarrow c(t)\in U$ will be called a curve parametrized by
time $t$.

In order to obtain the differentiable bundle of accelerations of
order $k$, we use the accelerations of order $k$, by means of
geometrical concept of contact of order $k$ between two curves in
the manifold $M$.

Two curves $\rho,\sigma:I\to M$ in $M$ have at the point $x_0\in
M$, $\rho(0)=\sigma(0)=x_0\in U$ (and $U$ a domain of local chart
in $M$) have a {\it contact} of order $k$ if we have
\begin{equation}
\dfrac{d^\alpha(f\circ\rho)(t)}{dt^\alpha}|_{t=0}=\dfrac{d^\alpha(f\circ\sigma)(t)}{dt^\alpha}|_{t=0},
(\alpha=1,...,k).
\end{equation}

It follows that: the curves $\rho$ and $\sigma$ have at the point
$x_0=\rho(0)=\sigma(0)$ a contact of order $k$ if and only if the
accelerations of order $1,2,...,k$ on the curve $\rho$ at $x_0$
have the same values with the corresponding accelerations on the
curve $\sigma$ at point $x_0$.

The relation ``{\it to have a contact of order $k$}'' is an
equivalence. Let $[\rho]_{x_0}$ be a class of equivalence and
$T^k_{\ x_0}M$ the set of equivalence classes. Consider the set
\begin{equation}
T^k M=\displaystyle\bigcup_{x_0\in M}T^k_{\ x_0}M
\end{equation}
and the mapping:
\begin{equation}
\pi^k:[\rho]_{x_0}\in T^k M\to x_0\in M,\quad
\forall[\rho]_{x_0}.\tag{1.1.2'}
\end{equation}

Thus the triple $(T^k M,\pi^k, M)$ can be endowed with a natural
differentiable structure exactly as in the cases $k=1$, when
$(T^1M,\pi^1,M)$ is the tangent bundle.

If $U\subset M$ is a coordinate neighborhood on the manifold $M$,
$x_0\in U$ and the curve $\rho: I\to U$, $\rho_0=x_0$ is
analytical represented on $U$ by the equations $x^i=x^i(t)$, $t\in
I$, then $T^k_{x_0}M$ can be represented by:
\begin{equation}
x_0^i=x^i(0),\ y^{(1)i}_0=\displaystyle\frac{dx^i}{dt}(0),...,\
y^{(k)i}_0=\displaystyle\frac 1{k!}\displaystyle\frac{d^kx^i}{%
dt^k}(0).
\end{equation}

Setting
\begin{equation}
\phi:([\rho]_{x_0})\in
T^kM\to\phi([\rho]_{x_0})=(x^i_0,y^{(1)i}_0,...,y^{(k)i}_0)\in
R^{(k+1)n},
\end{equation}
it follows that the pair $(\pi^k)^{-1}(U),\phi)$ is a local chart
on $T^kM$ induced by the local chart $(U,\varphi)$ on the manifold
$M$.

So a differentiable atlas of the manifold $M$ determine a
differentiable atlas on $T^kM$ and the triple $(T^kM,\pi^k,M)$ is
a differentiable bundle. Of course the mapping $\pi^k:T^kM\to M$
is a submersion.

$(T^kM,\pi^k,M)$ is called the $k$ accelerations bundle or tangent
bundle of order $k$ or $k-$osculator bundle. A change of
local coordinates $(x^i,y^{(1)i},...,y^{(k)i})\to
(\widetilde{x}^i,\widetilde{y}^{(1)i},...,\widetilde{y}^{(k)i})$
on the manifold $T^k M$, according with (1.1.3), is given by:
\begin{equation}
\left\{
\begin{array}{l}
\widetilde{x}^i=\widetilde{x}^i(x^1,...,x^n),\ rank\left(
\displaystyle
\frac{\partial \widetilde{x}^i}{\partial x^j}\right) =n \\
\\
\widetilde{y}^{(1)i}=\displaystyle\frac{\partial
\widetilde{x}^i}{\partial
x^j}y^{(1)j} \\
\\
2\widetilde{y}^{(2)i}=\displaystyle\frac{\partial
\widetilde{y}^{(1)i}}{
\partial x^j}y^{(1)j}+2\displaystyle\frac{\partial \widetilde{y}^{\left(
1\right)i}}{\partial y^{\left(1\right)j}}y^{\left( 2\right) j} \\
................................................. \\
k\widetilde{y}^{(k)i}\hspace*{-1mm}=\hspace*{-1mm}\displaystyle\frac{\partial
\widetilde{y}^{(k-1)i}}{
\partial x^j}y^{(1)j}\hspace*{-0.8mm}+\hspace*{-0.8mm}2\displaystyle\frac{\partial \widetilde{y}^{\left(
k-1\right)i}}{\partial y^{\left( 1\right) j}}y^{\left(2\right)j}\hspace*{-0.7mm}+\hspace*{-1mm}...\hspace*{-1mm}+\hspace*{-0.5mm}k %
\displaystyle\frac{\partial\widetilde{y}^{\left(k-1\right)
i}}{\partial y^{\left(k-1\right)j}}y^{\left(k\right)j}.
\end{array}
\right.
\end{equation}

And remark that we have the following identities:
\begin{equation}
\begin{array}{l}
\displaystyle\frac{\partial\widetilde{y}^{(\alpha)i}}{\partial x^j}=
\displaystyle\frac{\partial\widetilde{y}^{(\alpha+1)i}}{\partial
y^{(1)j}}=...=\displaystyle\frac{\partial\widetilde{y}^{(k)i}}{\partial y^{(k-\alpha)j}},\\ \\
\hfill (\alpha=0,...,k-1;y^{(0)}=x).
\tag{1.1.5'}
\end{array}
\end{equation}

We denote a point $u\in T^kM$ by $u=(x,y^{(1)},...,y^{(k)})$ and
its coordinates by $(x^i,y^{(1)i},...,y^{(k)i})$.

A section of the projection $\pi^k$ is a mapping $S:M\to T^kM$
with the property $\pi^k\circ S=1_M$. And a local section $S$ has
the property $\pi^k\circ S_{|U}=1_U$.

If $c:I\rightarrow M$ is a smooth curve, locally represented by
$x^i=x^i(t)$, $t\in I$, then the mapping
$\widetilde{c}:I\rightarrow T^kM$ given by:
\begin{equation}
x^i=x^i(t),\ \ y^{\left( 1\right) i}=\displaystyle\frac
1{1!}\displaystyle
\frac{dx^i}{dt}(t),...,\ \ y^{\left( k\right) i}=\displaystyle\frac 1{k!} %
\displaystyle\frac{d^{\left( k\right)}x^{i}}{dt^k}(t),\ t\in I
\end{equation}
is {\it the extension of order} $k$ to $T^kM$ of $c$. We have $\pi
^k\circ \widetilde{c} =c. $

The following property holds:

If the differentiable manifold $M$ is paracompact, then $T^kM$ is
a paracompact manifold.

We shall use the manifold $\widetilde{TM}=T^kM\setminus\{0\}$,
where $0$ is the null section of $\pi^k$.

\section{The Liouville vector fields}
\setcounter{equation}{0}\setcounter{theorem}{0}\setcounter{definition}{0}\setcounter{proposition}{0}\setcounter{remark}{0}

The natural basis at point $u\in T^kM$ of $T_u(T^kM)$ is given by
$$\left(\displaystyle\frac{\partial}{\partial
x^i},\displaystyle\frac{\partial}{\partial
y^{(1)i}},...,\displaystyle\frac{\partial}{\partial
y^{(k)i}}\right)_u.$$ A local coordinate changing (1.1.5) transform
the natural basis by the following rule:
\begin{equation}
\begin{array}{l}
\displaystyle\frac \partial {\partial
x^i}=\displaystyle\frac{\partial \widetilde{x}^j}{\partial
x^i}\displaystyle\frac \partial {\partial
\widetilde{x}^j}+\displaystyle\frac{\partial
\widetilde{y}^{(1)j}}{\partial
x^i}\displaystyle\frac \partial {\partial \widetilde{y}^{(1)j}}+...+ %
\displaystyle\frac{\partial \widetilde{y}^{(k)j}}{\partial x^i}\displaystyle %
\frac \partial {\partial \widetilde{y}^{(k)j}} \\
\\
\displaystyle\frac \partial {\partial y^{(1)i}}=\qquad \qquad
\displaystyle \frac{\partial \widetilde{y}^{(1)j}}{\partial
y^{(1)i}}\displaystyle\frac
\partial {\partial \widetilde{y}^{(1)j}}+...+\displaystyle\frac{\partial
\widetilde{y}^{(k)j}}{\partial y^{(1)i}}\displaystyle\frac
\partial
{\partial \widetilde{y}^{(k)j}} \\
............................................................... \\
\displaystyle\frac \partial {\partial y^{(k)i}}=\qquad \qquad
\qquad \qquad \qquad \qquad \displaystyle\frac{\partial
\widetilde{y}^{(k)j}}{\partial y^{(k)i}}\displaystyle\frac
\partial {\partial \widetilde{y}^{(k)j}},
\end{array}
\end{equation}
calculated at the point $u\in T^kM$.

The natural cobasis $(dx^i,dy^{(1)i},...,dy^{(k)i})_u$ is
transformed by (1.1.5) as follows:
\begin{equation}
\begin{array}{l}
d\widetilde{x}^i=\displaystyle\frac{\partial
\widetilde{x}^i}{\partial x^j}
dx^j, \\
\\
d\widetilde{y}^{\left( 1\right) i}=\displaystyle\frac{\partial
\widetilde{y} ^{(1)i}}{\partial
x^j}dx^j+\displaystyle\frac{\partial \widetilde{y}^{(1)i}}{
\partial y^{(1)j}}dy^{\left( 1\right) j}, \\
............................................................... \\
d\widetilde{y}^{\left( k\right) i}=\displaystyle\frac{\partial
\widetilde{y} ^{(k)i}}{\partial
x^j}dx^j+\displaystyle\frac{\partial \widetilde{y}^{(k)i}}{
\partial y^{(1)j}}dy^{\left( 1\right) j}+...+\displaystyle\frac{\partial
\widetilde{y}^{(k)i}}{\partial y^{(k)j}}dy^{\left( k\right) j}.
\end{array}
\tag{1.2.1'}
\end{equation}

The formulae (1.2.1) and (1.2.1') allow to determine some important
geometric object fields on the total space of accelerations bundle
$T^kM$.

The vertical distribution $V_1$ is local generated by the vector
fields $\left\{\displaystyle\frac{\partial}{\partial
y^{(1)i}},\right.$ $\left.
...,\displaystyle\frac{\partial}{\partial y^{(k)i}}\right\}$,
$i=1,...,n$. $V_1$ is integrable and of dimension $kn$. The
distribution $V_2$ local generated by
$\left\{\displaystyle\frac{\partial}{\partial
y^{(2)i}},...,\displaystyle\frac{\partial}{\partial
y^{(k)i}}\right\}$ is also integrable, of dimension $(k-1)n$ and
it is a subdistribution of $V_1$. And so on.

The distribution $V_k$ local generated by $\left\{\displaystyle\frac
{\partial}{\partial y^{(k)i}}\right\}$ is integrable and of dimension
$n$. It is a subdistribution of the distribution $V_{k-1}$. So we
have the following sequence:
$$V_1\supset V_2\supset ...\supset V_k$$
Using again (1.2.1) we deduce:

\begin{theorem}
The following operators in the algebra of
functions ${\mathcal{F}}(T^kM):$
\begin{equation}
\begin{array}{l}
\stackrel{1}{\Gamma }=y^{\left( 1\right) i}\displaystyle\frac
\partial
{\partial y^{\left( k\right) i}}, \\
\stackrel{2}{\Gamma }=y^{\left( 1\right) i}\displaystyle\frac
\partial {\partial y^{\left( k-1\right) i}}+2y^{\left( 2\right)
i}\displaystyle\frac
\partial {\partial y^{\left( k\right) i}}, \\
.............................................. \\
\stackrel{k}{\Gamma }=y^{\left( 1\right) i}\displaystyle\frac
\partial {\partial y^{\left( 1\right) i}}+2y^{\left( 2\right)
i}\displaystyle\frac
\partial {\partial y^{\left( 2\right) i}}+...+ky^{\left( k\right) i}%
\displaystyle\frac \partial {\partial y^{\left( k\right) i}}
\end{array}
\end{equation}
are the vector fields on $T^kM$. They are independent on the
manifold $\widetilde{T^kM}$ and $\stackrel{1}{\Gamma}\subset V_k$,
$\stackrel{2}{\Gamma}\subset
V_{k-1}$,...,$\stackrel{k}{\Gamma}\subset V_1$.
\end{theorem}

The vector fields $\stackrel{1}{\Gamma}$, $\stackrel{2}{\Gamma}$,
..., $\stackrel{k}{\Gamma}$ are called the {\it Liouville vector
fields}.

Also we have:

\begin{theorem}
For any function $L\in{\cal
F}(\widetilde{T^kM})$, the following entries are $1 $-form fields
on the manifold $\widetilde{T^kM}$:
\begin{equation}
\begin{array}{l}
d_0L=\displaystyle\frac{\partial L}{\partial y^{\left( k\right) i}}dx^i, \\
\\
d_1L=\displaystyle\frac{\partial L}{\partial y^{\left( k-1\right) i}}dx^i+%
\displaystyle\frac{\partial L}{\partial y^{\left( k\right)
i}}dy^{\left(
1\right) i}, \\
.............................................. \\
d_kL=\displaystyle\frac{\partial L}{\partial x^i}dx^i+\displaystyle\frac{%
\partial L}{\partial y^{\left( 1\right) i}}dy^{\left( 1\right) i}+...+%
\displaystyle\frac{\partial L}{\partial y^{\left( k\right)
i}}dy^{\left( k\right) i}.
\end{array}
\end{equation}
\end{theorem}

Evidently, $d_kL=dL$.

In applications we shall use also the following nonlinear operator
\begin{equation}
\Gamma =y^{\left( 1\right) i}\displaystyle\frac \partial {\partial
x^i}+2y^{\left( 2\right) i}\displaystyle\frac \partial {\partial
y^{\left( 1\right) i}}+...+ky^{\left( k\right)
i}\displaystyle\frac \partial {\partial y^{(k-1)i}}.
\end{equation}
$\Gamma$ is not a vector field on $\widetilde{T^kM}$.

\begin{definition}
A $k$-tangent structure $J$ on $T^kM$ is
the $\mathcal{F}(T^kM)$-linear mapping $J:\mathcal{X}(T^kM)\rightarrow \mathcal{X}(T^kM)$, defined by:
\begin{equation}
\begin{array}{l}
J\left(\displaystyle\frac \partial {\partial
x^i}\right)=\displaystyle\frac
\partial {\partial y^{\left( 1\right) i}},\quad
J\left(\displaystyle\frac \partial {\partial y^{\left( 1\right)
i}}\right)=\displaystyle\frac \partial {\partial
y^{\left( 2\right) i}},..., \\
\\
J\left(\displaystyle\frac \partial {\partial y^{\left( k-1\right) i}}\right)= %
\displaystyle\frac \partial {\partial y^{\left( k\right) i}},\quad J\left( %
\displaystyle\frac \partial {\partial y^{\left( k\right)
i}}\right)=0.
\end{array}
\quad
\end{equation}
\end{definition}

It is not difficult to see that $J$ has the properties:

\begin{proposition}
We have:

1$^\circ$. $J$ is globally defined on $T^kM$,

2$^\circ$. $J$ is an integrable structure,

3$^\circ$. $J$ is locally expressed by
\begin{equation}
J=\displaystyle\frac \partial {\partial y^{\left( 1\right) i}}\otimes dx^i+ %
\displaystyle\frac \partial {\partial y^{\left( 2\right)
i}}\otimes dy^{\left( 1\right) i}+...+\displaystyle\frac \partial
{\partial y^{\left( k\right) i}}\otimes dy^{\left( k-1\right) i}
\end{equation}

4$^{\circ}$. Im$J=Ker J,\quad KerJ=V_k$,

5$^{\circ }$. $rankJ=kn$,

6$^{\circ }$. $J\stackrel{k}{\Gamma }=\stackrel{k-1}{\Gamma
},...,$ $J \stackrel{2}{\Gamma }=\stackrel{1}{\Gamma },$
$J\stackrel{1}{\Gamma }=0$,

7$^{\circ }$. $J\circ J\circ...\circ J=0$, ($k+1$ factors).
\end{proposition}

In the next section we shall use the functions
\begin{equation}
I^1(L)=\mathcal{L}_{\stackrel{1}{\Gamma }}L,...,I^k(L)=\mathcal{L}_{%
\stackrel{k}{\Gamma }}L,  \quad \forall L\in{\cal
F}(T^kM),
\end{equation}
where ${\cal L}_{\stackrel{\alpha}{\Gamma}}$ is the operator of
Lie derivation with respect to the Liouville vector field
$\stackrel{\alpha}{\Gamma}$.

The functions $I^1(L)$, ..., $I^k(L)$ are called the \textit{main
invariants} of the function $L$. They play an important role in
the variational calculus.

\section{Variational Problem}
\setcounter{equation}{0}\setcounter{theorem}{0}\setcounter{definition}{0}\setcounter{proposition}{0}\setcounter{remark}{0}

\begin{definition}
A differentiable Lagrangian of order $k$ is
a mapping $L:(x,y^{(1)},...,y^{(k)})\in T^kM\rightarrow
L(x,y^{(1)},y^{(k)})\in R$, differentiable on $\widetilde{T^kM}$
and continuous on the null section
$0:M\rightarrow\widetilde{T^kM}$ of the projection
$\pi^k:T^kM\rightarrow M$.
\end{definition}

If $c:t\in [0,1]\rightarrow (x^i(t))\in U\subset M$ is a curve,
with extremities $c(0)=(x^i(0))$ and $c(1)=(x^i(1))$ and
$\widetilde{c}:[0,1]\rightarrow\widetilde{T^kM}$ from (1.1.6) is its extension.
Then the integral of action of $L\circ\widetilde{c}$ is defined by
\begin{equation}
I(c)=\displaystyle\int_0^1L\left(x(t),\displaystyle\frac{dx}{dt}(t),...,\dfrac{1}{k!}\dfrac{d^kx}{dt^k}(t)\right)
dt.
\end{equation}

\begin{remark} One proves that if $I(c)$ does not depend on
the parametrization of curve $c$ then the following Zermelo
conditions holds:
\begin{equation}
I^1(L)=...=I^{k-1}(L)=0,\quad I^k(L)=L.
\end{equation}
\end{remark}

Generally, these conditions are not verified.

The variational problem involving the functional $I(c)$ from (1.3.1)
will be studied as a natural extension of the theory expounded in
\S2.2, Ch. 2.

On the open set $U$ we consider the curves
\begin{equation}
c_\varepsilon:t\in [0,1]\rightarrow (x^i(t)+\varepsilon V^i(t))\in
M,
\end{equation}
where $\varepsilon$ is a real number, sufficiently small in
absolute value such that Im$c_\varepsilon\subset U$,
$V^i(t)=V^i(x(t))$ being a regular vector field on $U$, restricted
to $c$. We assume all curves $c_\varepsilon$ have the same end
points $c(0)$ and $c(1)$ and their osculator spaces of order
$1,2,...,k-1$ coincident at the points $c(0)$, $c(1)$. This means:
\begin{equation}
\begin{array}{l}
V^i(0)=V^i(1)=0;\quad\dfrac{d^\alpha
V^i}{dt^\alpha}(0)=\dfrac{d^\alpha V^i}{dt^\alpha}(1)=0,\\ \\
\hfill (\alpha=1,...,k-1).
\end{array}
\tag{1.3.3'}
\end{equation}

The integral of action $I(c_\varepsilon)$ of the Lagrangian $L$
is:
\begin{equation}
I(c_\varepsilon)\hspace*{-0.7mm}=\hspace*{-0.7mm}\displaystyle\int_0^1\hspace*{-0.7mm} L\left(x+\varepsilon
V,\displaystyle\frac{dx}{dt}+\varepsilon\dfrac{dV}{dt},...,\dfrac{1}{k!}\left(\dfrac{d^kx}{dt^k}+
\varepsilon\dfrac{d^kV}{dt^k}\right)\right) dt.
\end{equation}

A necessary condition for $I(c)$ to be an extremal value for
$I(c_\varepsilon)$ is
\begin{equation}
\dfrac{dI(c_\varepsilon)}{d\varepsilon}|_{\varepsilon=0}=0.
\end{equation}

Thus, we have
$$
\dfrac{dI(c_\varepsilon)}{d\varepsilon}=\displaystyle\int_0^1\dfrac{d}{d\varepsilon}L\left(x+\varepsilon
V,\displaystyle\frac{dx}{dt}+\varepsilon\dfrac{dV}{dt},...,\dfrac{1}{k!}\left(\dfrac{d^kx}{dt^k}+
\varepsilon\dfrac{d^kV}{dt^k}\right)\right) dt.
$$

The Taylor expansion of $L$ for $\varepsilon=0$, gives:
\begin{equation}
\dfrac{dI(c_\varepsilon)}{d\varepsilon}|_{\varepsilon=0}\hspace*{-1mm}=\hspace*{-1mm}\displaystyle\int_0^1\hspace*{-1mm}
\left(\hspace*{-0.7mm}\dfrac{\partial L}{\partial x^i}V^i\hspace*{-0.7mm}+\hspace*{-0.7mm}\displaystyle\frac{\partial L}{\partial
y^{(1)i}}\dfrac{dV^i}{dt}\hspace*{-0.7mm}+\hspace*{-0.7mm}...\hspace*{-0.7mm}+\hspace*{-0.7mm}\dfrac{1}{k!}\dfrac{\partial
L}{\partial y^{(k)i}}\dfrac{d^kV^i}{dt^k}\hspace*{-0.7mm}\right)dt.
\end{equation}

Now, with notations
\begin{equation}
\stackrel{\circ }{E}_i(L):=\dfrac{\partial L}{\partial
x^i}-\dfrac{d}{dt}\dfrac{\partial L}{\partial
y^{(1)i}}+...+(-1)^k\dfrac{1}{k!}\dfrac{d^k}{dt^k}\dfrac{\partial
L}{\partial y^{(k)i}}
\end{equation}
and
\begin{equation}
\begin{array}{l}
I^1_V L=V^i\dfrac{\partial L}{\partial y^{(k)i}},\quad
I^2_V(L)=V^i\dfrac{\partial L}{\partial
y^{(k-1)i}}+\dfrac{dV^i}{dt}\dfrac{\partial L}{\partial
y^{(k)i}},...,\vspace{3mm}\\ I^k_V =V^i\dfrac{\partial L}{\partial
y^{(1)i}}+\dfrac{dV^i}{dt}\dfrac{\partial L}{\partial
y^{(2)i}}+...+\dfrac{1}{(k-1)!}\dfrac{d^{k-1}V^v}{dt^{k-1}}\dfrac{\partial
L}{\partial y^{(k)i}}
\end{array}
\end{equation}
we obtain an important identity:
\begin{equation}
\begin{array}{l}
\dfrac{\partial L}{\partial x^i}V^i+\displaystyle\frac{\partial
L}{\partial
y^{(1)i}}\dfrac{dV^i}{dt}+...+\dfrac{1}{k!}\dfrac{\partial
L}{\partial y^{(k)i}}\dfrac{d^kV^i}{dt^k}=\stackrel{\circ
}{E}_i(L)+\vspace{3mm}\\+
\dfrac{d}{dt}\left\{I^k_V(L)\hspace*{-0.7mm}-\hspace*{-0.7mm}\dfrac{1}{2!}\dfrac{d}{dt}I^{k-1}_V(L)+\hspace*{-0.7mm}...\hspace*{-0.7mm}+
(-1)^{k-1}\dfrac{1}{k!}\dfrac{d^{k-1}}{dt^{k-1}}I^1_V(L)\right\}
\end{array}
\end{equation}

Now, applying (1.3.7) and taking into account (1.3.3')
with $$I^\alpha_V(L)(c(0))=I^\alpha_V(L)c(1)=0,\quad
(\alpha=1,2,...,k)$$ we obtain
\begin{equation}
\dfrac{dI(c_\varepsilon)}{d\varepsilon}|_{\varepsilon=0}=\dd\int^1_0\stackrel{0}{E}_i(L)V^i
dt.
\end{equation}
But $V^i(t)$ is an arbitrary vector field. Therefore the
equalities (1.3.5) and (1.3.10) lead to the following result:

\begin{theorem}
In order that the integral of action
$I(c)$ be an extremal value for the functionals
$I(c_\varepsilon)$, $(1.3.4)$ is necessary that the following Euler -
Lagrange equations hold: \begin{equation}
\left\{\begin{array}{l}\stackrel{0}{E}_i(L):=\hspace*{-0.7mm}\dfrac{\partial
L}{\partial x^i}-\dfrac{d}{dt}\dfrac{\partial L}{\partial
y^{(1)i}}+\hspace*{-0.7mm}...\hspace*{-0.7mm}+(-1)^k\dfrac{1}{k!}\dfrac{d^k}{dt^k}\dfrac{\partial
L}{\partial y^{(k)}}=0,\vspace{3mm}\\
y^{(1)i}=\dfrac{dx^i}{dt},\cdots,y^{(k)i}=\dfrac{1}{k!}\dfrac{d^k
x^i}{dt^k}.\end{array}\right.\end{equation}
\end{theorem}

One proves \cite{mir11} that $\stackrel{0}{E}_i(L)$ is a covector field.
Consequence the equation $\stackrel{0}{E}_i(L)=0$ has a
geometrical meaning.

Consider the scalar field
\begin{equation}
{\cal
E}^k(L)\hspace*{-0.7mm}=\hspace*{-0.7mm}I^k(L)-\hspace*{-0.7mm}\dfrac{1}{2!}\dfrac{d}{dt}I^{k-1}(L)+\hspace*{-0.7mm}...\hspace*{-0.7mm}
+\hspace*{-0.7mm}(-1)^{k-1}\dfrac{1}{k!}\dfrac{d^{k-1}}{dt^{k-1}}I^1(L)\hspace*{-0.7mm}-\hspace*{-0.7mm}L.
\end{equation}
It is called the {\it energy of order $k$} of the Lagrangian $L$.

The next result is known, \cite{mir11}:

\begin{theorem}
For any Lagrangian
$L(x,y^{(1)},\cdots,y^{(k)})$ the energy of order $k$, ${\cal
E}^k(L)$ is conserved along every solution curve of the Euler -
Lagrange equations $\stackrel{0}{E}_i(L)=0$,
$y^{(1)i}=\dfrac{dx^i}{dt},\cdots,y^{(k)i}=\dfrac{1}{k!}\dfrac{d^k
x^i }{dt^k}$.
\end{theorem}

\begin{remark} Introducing the notion of energy of order
$1,2,\cdots,k-1$ we can prove a N\"{o}ther theorem for the
Lagrangians of order $k$.
\end{remark}

Now we remark that for any $C^\infty-$function $\phi(t)$ and any
differentiable Lagrangian $L(x,y^{(1)},\cdots,y^{(k)})$ the
following equality holds:
\begin{equation}
\stackrel{0}{E}_i(\phi
L)=\phi\stackrel{0}{E_i}(L)+\displaystyle\frac{d\phi}{dt}\stackrel{1}{E_i}(L)+\cdots
+\displaystyle\frac{d^k\phi }{
dt^k}\stackrel{k}{E_i}(L),
\end{equation}
where $\stackrel{1}{E_i}(L),\cdots,\stackrel{k}{E_i}(L)$ are
$d$-covector fields - called Graig - Synge covectors, \cite{craig}. We
consider the covector $\stackrel{k-1}{E_i}(L)$:
\begin{equation}
\stackrel{k-1}{E}_i(L)=(-1)^{k-1}\dfrac{1}{(k-1)!}\left(\displaystyle\frac{\partial
L}{\partial y^{(k-1)i}}-\displaystyle\frac{d}{ dt}\dfrac{\partial
L}{\partial y^{(k)i}}\right),
\end{equation}

It is important in the theory of $k-$semisprays from the Lagrange
spaces of order $k$.

The Hamilton - Jacobi equations, of a space $L^n=(M,L(x,y))$
introduced in the section 4 of Chapter 2 can be extended in the
higher order Lagrange spaces by using the Jacobi - Ostrogradski
momenta. Indeed, the energy of order $k$, ${\cal E}^k(L)$ from
(1.3.13) is a polynomial function in
$\dfrac{dx^i}{dt},\cdots,\dfrac{d^k x^i}{dt^k}$ given by
\begin{equation}
{\mathcal{E}}^k(L)=p_{(1)i}\displaystyle\frac{dx^i}{dt}+p_{(2)i}\displaystyle
\frac{d^2x^i}{dt^2}+\cdots
+p_{(k)i}\displaystyle\frac{d^kx^i}{dt^k}-L,
\end{equation}
where
\begin{equation}
\begin{array}{l}
p_{(1)i}=\displaystyle\frac{\partial L}{\partial y^{(1)i}}-\displaystyle %
\frac 1{2!}\displaystyle\frac d{dt}\displaystyle\frac{\partial
L}{\partial y^{(2)i}}\hspace*{-0.7mm}+\hspace*{-0.7mm}...\hspace*{-0.7mm}+\hspace*{-0.7mm}(-1)^{k-1}\displaystyle\frac
1{k!}\displaystyle\frac{
d^{k-1}}{dt^{k-1}}\displaystyle\frac{\partial L}{\partial y^{(k)i}}, \\
\\
p_{(2)i}\hspace*{-0.7mm}=\hspace*{-0.7mm}\displaystyle\frac 1{2!}\displaystyle\frac{\partial
L}{\partial y^{(2)i}}\hspace*{-0.7mm}-\hspace*{-0.7mm}\displaystyle\frac 1{3!}\displaystyle\frac
d{dt}\displaystyle\frac{
\partial L}{\partial y^{(3)i}}\hspace*{-0.7mm}+\hspace*{-0.7mm}...\hspace*{-0.7mm}+\hspace*{-0.7mm}(-1)^{k-2}\displaystyle\frac 1{k!}
\displaystyle\frac{d^{k-2}}{dt^{k-2}}\displaystyle\frac{\partial
L}{\partial
y^{(k)i}}, \\
.......................................................................
\\
p_{(k)i}=\displaystyle\frac 1{k!}\displaystyle\frac{\partial
L}{\partial y^{(k)i}}.
\end{array}
\end{equation}
$p_{(1)i}$, ..., $p_{(k)i}$ are called \textit{the
Jacobi-Ostrogradski momenta.}

\medskip

The following important result has been established by M. de
L\'eon and others, \cite{mir11}:

\begin{theorem}
Along the integral curves of Euler
Lagrange equations $\stackrel{0}{E_i}(L)=0$ the following Hamilton
- Jacobi - Ostrogradski equations holds:
$$
\begin{array}{l}
\displaystyle\frac{\partial {\mathcal{E}}^k(L)}{\partial p_{(\alpha )i}}=%
\displaystyle\frac{d^\alpha x^i}{dt^\alpha },\quad (\alpha =1,...,k), \\
\\
\displaystyle\frac{\partial {\mathcal{E}}^k(L)}{\partial
x^i}=-\displaystyle
\frac{dp_{(1)i}}{dt}, \\
\\
\dfrac{1}{\alpha!}\displaystyle\frac{\partial {\mathcal{E}}^k(L)}{\partial y^{(\alpha )i}}%
=-\displaystyle\frac{dp_{(\alpha +1)i}}{dt},\quad (\alpha
=1,...,k-1).
\end{array}
$$
\end{theorem}

\begin{remark} The Jacobi - Ostrogradski momenta allow the
introduction of the 1-forms:
$$\begin{array}{l}
p_{(1)}=p_{(1)i}dx^i+p_{(2)i}dy^{(1)i}+\cdots +p_{(k)i}dy^{(k-1)i}, \\
\\
p_{(2)}=p_{(2)i}dx^i+p_{(3)i}dy^{(1)i}+\cdots +p_{(k)i}dy^{(k-2)i}, \\
..............................................................
\\
p_{(k)}=p_{(k)i}dx^i.
\end{array}$$
\end{remark}

\section{Semisprays. Nonlinear connections}\index{Semisprays! on the manifold $T^kM$}
\setcounter{equation}{0}\setcounter{theorem}{0}\setcounter{definition}{0}\setcounter{proposition}{0}\setcounter{remark}{0}

A vector field $S\in\chi(T^kM)$ with the property
\begin{equation}
JS=\stackrel{k}{\Gamma}
\end{equation}
is called a $k-$semispray on $T^kM$. $S$ can be uniquely written
in the form:
\begin{equation}
S=y^{\left( 1\right) i}\displaystyle\frac \partial {\partial
x^i}+...+ky^{\left( k\right) i}\displaystyle\frac \partial
{\partial
y^{(k-1)i}}-(k+1)G^i(x,y^{\left( 1\right) },...,y^{\left( k\right) }) %
\displaystyle\frac \partial {\partial y^{\left( k\right) i}},
\end{equation}
or shortly
\begin{equation}
S=\Gamma -(k+1)G^i\displaystyle\frac\partial {\partial y^{\left(
k\right) i}}.  \tag{1.4.2'}
\end{equation}

The set of functions $G^i$ is the set of \textit{coefficient}s of
$S.$ With respect to (1.1.5) Ch. 4 $G^i$ are transformed as
following:
\begin{equation}
(k+1)\widetilde{G}^i=(k+1)G^j\displaystyle\frac{\partial
\widetilde{x}^i}{
\partial x^j}-\Gamma \widetilde{y}^{\left( k\right) i}.
\end{equation}

A curve $c:I\rightarrow M$ is called a $k$\textit{-path} on $M$
with respect to $S$ if its extension $\widetilde{c}$ is an
integral curve of $S.$ A $k$-path is characterized by the $(k+1)-$
differential equations:
\begin{equation}
\displaystyle\frac{d^{k+1}x^i}{dt^{k+1}}+(k+1)G^i\left(x,\displaystyle\frac{dx}{dt},...,
\dfrac{1}{k!}\displaystyle\frac{d^kx}{dt^k}\right)=0.
\end{equation}

We shall show that a $k-$semispray determine the main geometrical
object fields on $T^kM$ as: the nonlinear connections $N$, the
$N-$linear connections $D$ and them structure equations.
Evidently, $N$ and $D$ are basic for the geometry of manifold
$T^kM$.

\begin{definition} A subbundle $HT^kM$ of the tangent bundle
$(TT^kM,$ $d\pi^k$, $T^kM)$ which is supplementary to the vertical
subbundle $V_1T^kM$:
\begin{equation}
TT^kM=HT^kM\oplus V_1T^kM
\end{equation}
is called a {\it nonlinear connection}.
\end{definition}

The fibres of $HT^kM$ determine a horizontal distribution
$$N:u\in T^kM\to N_u=H_u T^kM\subset T_u T^kM,\quad \forall u\in T^kM$$ supplementary to the vertical distribution $V_1$, i.e.
\begin{equation}
T_uT^kM=N_u\oplus V_{1,u},\quad \forall u\in T^kM.\tag{1.4.5'}
\end{equation}

If the base manifold $M$ is paracompact on $T^kM$ there exist the
nonlinear connections.

The local dimension of $N$ is $n=\dim M$.

Consider a nonlinear connection $N$ and denote by $h$ and $v$ the
horizontal and vertical projectors with respect to $N$ and $V_1$:
$$h+v=I,\quad hv=vh=0,\quad h^2=h,\quad v^2=v.$$ As usual we denote $$X^H=hX,\quad X^V=vX,\quad \forall X\in\chi(T^kM).$$

An horizontal lift, with respect to $N$ is a
${\mathcal{F}}(M)$-linear mapping $l_h:{\mathcal{X}}(M)\rightarrow
{\mathcal{X}}(T^{k}M)$ which has the properties
$$v\circ l_h=0,\ d\pi^{k}\circ l_h=I_d.$$

There exists an unique local basis adapted to the horizontal
distribution $N$. It is given by
\begin{equation}
\displaystyle\frac \delta {\delta x^i}=l_h\left(\displaystyle\frac
\partial {\partial x^i}\right),(i=1,...,n).
\end{equation}

The linearly independent vector fields of this basis can be
uniquely written in the form:
\begin{equation}
\displaystyle\frac \delta {\delta x^i}=\displaystyle\frac \partial
{\partial x^i}-\underset{(1)}{N_i^j}\displaystyle\frac \partial
{\partial y^{\left( 1\right)
j}}-...-\underset{(k)}{N_i^j}\displaystyle\frac \partial {\partial
y^{\left( k\right) j}}.
\end{equation}

The systems of differential functions on
$T^kM:(\underset{(1)}{N_i^j},...,\underset{(k)}{ N_i^j})$, gives
the \textit{coefficients} of the nonlinear connection $N$.

By means of (1.4.6) it follows:

\begin{proposition}
With respect of a change of local
coordinates on the manifold $T^kM$ we have
\begin{equation}
\displaystyle\frac \delta {\delta x^i}=\displaystyle\frac{\partial
{\widetilde{x}}^j}{\partial x^i}\displaystyle\frac \delta {\delta
{\widetilde{x}} ^j} , \tag{1.4.7'}
\end{equation}
and
\begin{equation}
\begin{array}{l}
\underset{(1)}{\widetilde{N}_m^i}\displaystyle\frac{\partial \widetilde{x}%
^m}{\partial x^j}=\underset{(1)}{N_j^m}\displaystyle\frac{\partial
\widetilde{x}^i}{\partial x^m}-\displaystyle\frac{\partial \widetilde{y}%
^{(1)i}}{\partial x^j}, \\
\\
.............................................................. \\
\underset{(k)}{\widetilde{N}_m^i}\displaystyle\frac{\partial \widetilde{x%
}^m}{\partial
x^j}=\underset{(k)}{N_j^m}\displaystyle\frac{\partial
\widetilde{x}^i}{\partial x^m}+...+\underset{(1)}{N_j^m}%
\displaystyle\frac{\partial \widetilde{y}^{(k-1)i}}{\partial x^m}-%
\displaystyle\frac{\partial \widetilde{y}^{(k)i}}{\partial x^j}, \\
\end{array}
\end{equation}
\end{proposition}

\begin{remark} The equalities (1.4.8) characterize a nonlinear
connection $N$ with the coefficients
$\underset{(1)}{N_i^j},\cdots,\underset{(k)}{N_i^j}$.
\end{remark}

These considerations lead to an important result, given by:

\begin{theorem}
[I. Buc\u{a}taru \cite{Buc}] If $S$ if a
$k-$semispray on $T^kM$, with the coefficients $G^i$, then the
following system of functions:
\begin{equation}
\underset{(1)}{N_j}^i=\displaystyle\frac{\partial G^i}{\partial
y^{(k)i}},\underset{(2)}{N_j}^i=\displaystyle\frac{\partial
G^i}{\partial y^{(k-1)i}},\cdots,
\underset{(k)}{N_j}^i=\displaystyle\frac{\partial G^i}{\partial
y^{(1)j}}\tag{4.9}
\end{equation} gives the coefficients of a
nonlinear connection $N$.
\end{theorem}

The $k-$tangent structure $J$, defined in (1.2.5), Ch. 4. applies
the horizontal distribution $N_0=N$ into a vertical distribution
$N_1\subset V_1$ of dimension $n$, supplementary to the
distribution $V_2$. Then it applies the distribution $N_1$ in
distribution $N_2\subset V_2$, supplementary to the distribution
$V_3$ and so on. Of course we have $\dim N_0=\dim N_1=\cdots=\dim
N_{k-1}=n$.

Therefore we can write:
\begin{equation}
N_1=J(N_0),\quad N_2=J(N_1),\cdots,N_{k-1}=J(N_{k-2})
\end{equation}
and we obtain the direct decomposition:
\begin{equation}
T_uT^kM=N_{0,u}\oplus N_{1,u}\oplus\cdots\oplus N_{k-1,u}\oplus
V_{k,u},\quad \forall u\in T^kM.
\end{equation}

An adapted basis to $N_0$, $N_1$, ..., $N_{k-1}$, $V_k$ is given
by:
\begin{equation}
\left\{\displaystyle\frac \delta {\delta x^i},\displaystyle\frac
\delta {\delta y^{(1)i}},\cdots,\displaystyle\frac \delta {\delta
y^{(k-1)i}},\dfrac{\partial}{\partial y^{(k)i}}\right\},
(i=1,...,n),
\end{equation}
where $\displaystyle\frac \delta {\delta x^i}$ is in (1.4.7) and
\begin{equation}
\left\{
\begin{array}{l}
\displaystyle\frac \delta {\delta y^{(1)i}}=\displaystyle\frac
\partial {\partial y^{(1)i}}-\underset{(1)}{N_i^j}\displaystyle\frac
\partial {\partial y^{(2)i}}-\cdots
-\underset{(k-1)}{N_i^j}\displaystyle\frac \partial
{\partial y^{(k)i}}, \\
\\
.................................................................. \\
\displaystyle\frac \delta {\delta y^{(k-1)i}}=\displaystyle\frac
\partial {\partial
y^{(k-1)i}}-\underset{(1)}{N_i^j}\displaystyle\frac \partial
{\partial y^{(k)j}}.
\end{array}
\right.
\end{equation}

With respect to (1.1.5), Ch. 4 we have:
\begin{equation}
\displaystyle\frac\delta{\delta y^{(\alpha
)i}}=\displaystyle\frac{
\partial x^j}{\partial \widetilde{x}^i}\displaystyle\frac \delta {\delta
\widetilde{y}^{(\alpha )j}},\quad \left(\alpha
=0,1,...,k;y^{(0)i}=x^i, \dfrac{\delta}{\delta
y^{(k)i}}=\dfrac{\partial}{\partial y^{(k)i}}\right).
\end{equation}

Let $h,v_1,...,v_k$ be the projectors determined by (1.4.10):
$$
h+\dd\sum_1^k v_\alpha =I,h^2=h,v_\alpha v_\alpha =v_\alpha,
hv_\alpha =0;$$ $$v_\alpha h=0, v_\alpha v_\beta=v_\beta v_\alpha=0,
(\alpha \neq \beta).
$$

If we denote
\begin{equation}
X^H=hX,\ X^{V_\alpha }=v_\alpha X,\ \forall X\in
{\mathcal{X}}(T^kM)
\end{equation}
we have, uniquely,
\begin{equation}
X=X^H+X^{V_1}+...+X^{V_k}.
\end{equation}

In the adapted basis (1.4.12) we have:
$$
X^H=X^{\left( 0\right) i}\displaystyle\frac \delta {\delta
x^i},X^{V\;_\alpha }=X^{\left( \alpha \right) i}\displaystyle\frac
\delta {\delta y^{\left( \alpha \right) i}},(\alpha =1,...,k).
$$

A first result on the nonlinear connection $N$ is as follows:

\begin{theorem}
The nonlinear connection $N$ is integrable
if, and only if: $$[X^H,Y^H]^{V_\alpha}=0,\quad \forall
X,Y\in\chi(T^kM), \quad(\alpha = 1,\cdots k).$$
\end{theorem}

\section{The dual coefficients of a nonlinear connection}
\setcounter{equation}{0}\setcounter{theorem}{0}\setcounter{definition}{0}\setcounter{proposition}{0}\setcounter{remark}{0}

Consider a nonlinear connection $N$, with the coefficients
$(\underset{\left(1\right)}{N_j^i},...,\underset{\left( k\right)
}{N_j^i}).$ The dual basis, of the adapted basis (1.4.12) is of the
form
\begin{equation}
(\delta x^i,\delta y^{(1)i},\cdots,\delta y^{(k)i})
\end{equation}
where
\begin{equation}
\left\{
\begin{array}{l}
\delta x^i=dx^i, \\
\\
\delta y^{\left(1\right)i}=dy^{\left(1\right)i}+\underset{\left(
1\right)}{M_j^i}dx^j, \\
............................................. \\
\delta y^{\left(k\right)i}=dy^{\left(k\right)i}+\underset{\left(
1\right)}{M_j^i}dy^{\left(k-1\right)j}+...+\underset{\left(
k-1\right)}{M_j^i}dy^{\left(1\right)j}+\underset{\left(k\right)
}{M_j^i}dx^j,
\end{array}
\right.
\end{equation}
and where
\begin{equation}
\left\{
\begin{array}{l}
\underset{\left(1\right)}{M_j^i}=\underset{\left( 1\right)
}{N_j^i},\ \underset{\left( 2\right) }{M_j^i}=\underset{\left(
2\right) }{N_j^i}+ \underset{\left( 1\right)
}{N_j^m}\underset{\left( 1\right) }{M_m^i}
,...., \\
\\
\underset{\left( k\right) }{M_j^i}=\underset{\left( k\right)
}{N_j^i}+ \underset{\left( k-1\right) }{N_j^m}\underset{\left(
1\right) }{M_m^i} +....+\underset{\left( 1\right)
}{N_j^m}\underset{\left( k-1\right) }{ M_m^i}.
\end{array}
\right.
\end{equation}

The system of functions ($\underset{\left( 1\right) }{M_j^i},...,
\underset{\left( k\right) }{M_j^i})$ is called the system of dual
coefficients of the nonlinear connection $N.$ If the dual
coefficients of $N$ are given, then we uniquely obtain from (1.5.3)
the {\it primal} coefficients $(\underset{\left(1\right)}{N_j^i}
,...,\underset{\left(k\right)}{N_j^i},)$ of $N$.

With respect to (4.1.4), Ch. 4 the dual coefficients of $N$ are
transformed by the rule
\begin{equation}
\begin{array}{l}
\underset{(1)}{M_j^m}\displaystyle\frac{\partial
\widetilde{x}^i}{\partial
x^m}=\underset{(1)}{\widetilde{M}_m^i}\displaystyle\frac{\partial
\widetilde{x}^m}{\partial x^j}+\displaystyle\frac{\partial
\widetilde{y}
^{(1)i}}{\partial x^j}, \\
............................................. \\
\underset{(k)}{M_j^m}\displaystyle\frac{\partial
\widetilde{x}^i}{\partial
x^m}=\underset{(k)}{\widetilde{M}_m^i}\displaystyle\frac{\partial
\widetilde{x}^m}{\partial x^j}+\underset{(k-1)}{\widetilde{M}_m^i} %
\displaystyle\frac{\partial \widetilde{y}^{(1)m}}{\partial
x^j}+\cdots +
\underset{(1)}{\widetilde{M}_m^i}\displaystyle\frac{\partial
\widetilde{y} ^{(k-1)m}}{\partial x^j}+\displaystyle\frac{\partial
\widetilde{y}^{(k)i}}{
\partial x^j}.
\end{array}
\end{equation}

These transformations of the dual coefficients characterize the
nonlinear connection $N$. This property allows to prove an
important result:

\begin{theorem}
[R. Miron]  For any
$k-$semispray $S$ with the coefficients $G^i$ the following system
of functions \begin{equation}
\begin{array}{l}\underset{(1)}{M_j^i}=\displaystyle\frac{\partial
G^i}{\partial
{y}^{(k)j}},\underset{(2)}{M_j^i}=\dfrac{1}{2}(S\underset{(1)}{M_j^i}+\underset{(1)}{M_m^i}\underset{(1)}{M_j^m}),\cdots,
\vspace{3mm}\\
\underset{(k)}{M_j^i}=\dfrac{1}{k}(S\underset{(k-1)}{M_j^i}+\underset{(1)}{M_m^i}\underset{(k-1)}{M_j^m})\end{array}
\end{equation} gives the system of dual coefficients of a nonlinear connection, which depend on the $k-$semispray $S$,
only.
\end{theorem}

\begin{remark} Gh. Atanasiu \cite{mirata2} has a real contribution in demonstration of this theorem for $k=2$.

\end{remark}

As an application we can prove:

\begin{theorem}
\begin{enumerate}
\item[1)] In the adapted basis (4.4.12) Ch. 4, the Liouville vector
fields $\stackrel{1}{\Gamma}$, ..., $\stackrel{k}{\Gamma}$ can be
expressed in the form
\begin{equation}
\begin{array}{l}
\stackrel{1}{\Gamma }=z^{(1)i}\displaystyle\frac \delta {\delta
y^{(k)i}}, \quad \stackrel{2}{\Gamma}=z^{(1)i}\displaystyle\frac
\delta {\delta
y^{(k-1)i}}+2z^{(2)i}\displaystyle\frac \delta {\delta y^{(k)i}}, \\
................................................................. \\
\stackrel{k}{\Gamma }=z^{(1)i}\displaystyle\frac \delta {\delta
y^{(1)i}}+2z^{(2)i}\displaystyle\frac \delta {\delta
y^{(2)i}}+\cdots +kz^{(k)i}\displaystyle\frac \delta {\delta
y^{(k)i}},
\end{array}
\end{equation}
where
\begin{equation}
\left\{
\begin{array}{l}
z^{(1)i}=y^{(1)i}, \quad
2z^{(2)i}=2y^{(2)i}+\underset{(1)}{M_m^i}y^{(1)m},...,
\vspace{3mm}\\
kz^{(k)i}=ky^{(k)i}+(k-1)\underset{(1)}{M_m^i}y^{(k-1)m}+\cdots +%
\underset{(k-1)}{M_m^i}y^{(1)m}
\end{array}\right.
\end{equation}

\item[2)] With respect to (1.1.4) we have:
\begin{equation}
\widetilde{z}^{(\alpha)i}=\displaystyle\frac{\partial \widetilde{x}^i}{%
\partial x^j}z^{(\alpha)j},\quad (\alpha =1,...,k).  \tag{1.5.7'}
\end{equation}
\end{enumerate}
\end{theorem}

This is reason in which we call $z^{(1)i},...,z^{(k)i}$
\textit{the distinguished Liouville vector fields} (shortly,
$d$\textit{-vector fields}). These vectors are important in the
geometry of the manifold $T^kM$.

A field of $1$-forms $\omega \in \mathcal{X}^{*}(T^kM)$ can be
uniquely written as
$$
\omega =\omega^H+\omega^{V_1}+\cdots +\omega ^{V_k}
$$
where
$$
\omega^H=\omega \circ h, \omega^{V_\alpha }=\omega \circ
v_\alpha,\quad (\alpha =1,...,k).
$$

For any function $f\in {\cal F}(T^kM)$, the 1-form $df$ is
$$df=(df)^H+(df)^{V_1}+\cdots+(df)^{V_k}.$$

In the adapted cobasis we have:
\begin{equation}
(df)^H=\displaystyle\frac{\delta f}{\delta x^i}dx^i,(df)^{V_\alpha }= %
\displaystyle\frac{\delta f}{\delta y^{(\alpha )i}}\delta
y^{(\alpha )i},\quad (\alpha =1,...,k).
\end{equation}

Let $\gamma :I\rightarrow T^kM$ be a parametrized curve, locally
expressed by
$$
x^i=x^i(t),\quad y^{(\alpha )i}=y^{(\alpha )i}(t),\quad (t\in
I),\quad (\alpha =1,...,k).
$$

The tangent vector field $\displaystyle\frac{d\gamma}{dt}$ is
given by:
$$
\displaystyle\frac{d\gamma }{dt}=\left(
\displaystyle\frac{d\gamma}{dt}\right)^H+\left(\displaystyle\frac{d\gamma
}{dt}\right) ^{V_1}+\cdots +\left( \displaystyle\frac{d\gamma
}{dt}\right) ^{V_k}=
$$
$$
=
\displaystyle\frac{dx^i}{dt}\displaystyle\frac \delta {\delta x^i}+ %
\displaystyle\frac{\delta y^{(1)i}}{dt}\displaystyle\frac \delta
{\delta y^{(1)i}}+\cdots +\displaystyle\frac{\delta
y^{(k)i}}{dt}\displaystyle\frac \delta {\delta y^{(k)i}}.
$$

The curve $\gamma$ is called horizontal if
$\displaystyle\frac{d\gamma }{dt}=\left(\displaystyle\frac{d\gamma}{dt}\right)^H$. It is characterized by
the system of differential equations
\begin{equation}
x^i=x^i(t), \displaystyle\frac{\delta y^{(1)i}}{dt}=0,...,
\displaystyle\frac{\delta y^{(k)i}}{dt}=0.
\end{equation}

A horizontal curve $\gamma$ is called {\bf autoparallel} curve of
the nonlinear connection if $\gamma=\widetilde{c}$, where
$\widetilde{\widetilde{c}}$ is the extension of a curve $c:I\to
M$.

The {\it autoparallel} curves of the nonlinear connection $N$ are
characterized by the system of differential equations
\begin{equation}
\begin{array}{l}
\displaystyle\frac{\delta y^{(1)i}}{dt}=0,\cdots
,\displaystyle\frac{\delta y^{(k)i}}{dt}=0,\vspace{3mm}\\
y^{(1)i}=\displaystyle\frac{dx^i}{dt},\cdots
,y^{(k)i}=\displaystyle\frac
1{k!}\displaystyle\frac{d^kx^i}{dt^k}.
\end{array}
\tag{1.5.9'}
\end{equation}

\section{Prolongation to the manifold $T^kM$ of the Rie\-man\-nian structures given on the base manifold $M$}\index{Prolongations of ${\cal R}^n$}
\sectionmark{Prolongation to the manifold $T^kM$ of the Riemannian structures}
\setcounter{equation}{0}\setcounter{theorem}{0}\setcounter{definition}{0}\setcounter{proposition}{0}\setcounter{remark}{0}

Applying the previous theory of the notion of nonlinear connection
on the total space of acceleration bundle $T^kM$ we can solve the
old problem of the prolongation of the Riemann (or pseudo Riemann)
structure $g$ given on the base manifold $M$. This problem was
formulated by L. Bianchi and was studied by several remarkable
mathematicians as: E. Bompiani, Ch. Ehresmann, A. Morimoto, S.
Kobayashi. But the solution of this problem, as well as the
solution of the prolongation to $T^kM$ of the Finsler or Lagrange
structures were been recently given by R. Miron \cite{mirata6}. We will expound it here with very few demonstrations.

Let ${\cal R}^n=(M,g)$ be a Riemann space\index{Spaces! Riemann}, $g$ being a Riemannian
metric defined on the base manifold $M$, having the local
coordinate $g_{ij}(x)$, $x\in U\subset M$. We extend $g_{ij}$ to
$\pi^{-1}(U)\subset T^kM$, setting
$$(g_{ij}\circ\pi^k)(u)=g_{ij}(x),\quad \forall u\in
\pi^{-1}(U),\pi^k(u)=x.$$$ g_{ij}\circ \pi^k$ will be denoted by
$g_{ij}$.

The problems of prolongation of the Riemannian structure $g$ to
$T^kM$ can be formulated as follows:

The Riemannian structure $g$ on the manifold $M$ being a priori
given, determine a Riemannian structure $G$ on $T^kM$ so that $G$
be provided only by structure $g$.

As usually, we denoted by $\gamma^i_{\ jk}(x)$ the Christoffel
symbols of $g$ and prove, according theorem 1.5.1:

\begin{theorem}
There exists nonlinear connections $N$ on
the manifold $\widetilde{T^kM}$ determined only by the given
Riemannian structure $g(x)$. One of them has the following dual
coefficients
\begin{equation}
\begin{array}{l}
\underset{(1)}{M_j^i}=\gamma _{jm}^i(x)y^{(1)m}, \\
\\
\underset{(2)}{M_j^i}=\displaystyle\frac 12\left\{
\Gamma \underset{(1)}{M_j^i}+\underset{(1)}{M_m^i}\underset{(1)}{M_j^m}%
\right\} , \\
......................................... \\
\underset{(k)}{M_j^i}=\displaystyle\frac 1k\left\{
\Gamma \underset{(k-1)}{M_j^i}+\underset{(1)}{M_m^i}\underset{(k-1)}{%
M_j^m}\right\},
\end{array}
\end{equation}
where $\Gamma$ is the operator $(1.2.4)$, Ch. 4.
\end{theorem}

\begin{remark} $\Gamma$ can be substituted with any $k-$semispray
$S$, since $\underset{(1)}{M_j^i},...,\underset{(k-1)}{M_j^i}$ do
not depend on the variables $y^k$.\end{remark}

One proves, also: $N$ is integrable if and only if the Riemann
space ${\cal R}^n=(M,g)$ is locally flat.

Let us consider the adapted cobasis $(\delta x^i,$ $\delta
y^{(1)i},$ ..., $\delta y^{(k)i})$ (1.5.2) to the nonlinear connection $N$ and
to the vertical distributions
$\underset{(1)}{N}$, ..., $\underset{(k-1)}{N}$, $V_k$. It depend on
the dual coefficients (1.6.1). So it depend on the structure $g(x)$
only.

Now, on $T^kM$ consider the following {\it lift} of $g(x)$:
\begin{equation}
\begin{array}{l}
G=g_{ij}(x)dx^i\otimes dx^j+g_{ij}(x)\delta y^{(1)i}\otimes \delta
y^{(1)j}+\\ \\ \quad \quad +g_{ij}(x)\delta y^{(k)i}\otimes \delta y^{(k)j}
\end{array}
\end{equation}

Finally, we obtain:

\begin{theorem}
The pair Prol$^k{\cal
R}^n=(\widetilde{T^kM},G)$ is a Riemann space of dimension
$(k+1)n$, whose metric $G$, $(1.6.2)$ depends on the a priori given
Riemann structure $g(x)$, only.
\end{theorem}

The announced problem is solved.

Some remarks:
\begin{enumerate}
    \item[$1^\circ$] The $d-$Liouville vector fields
    $z^{(1)i},\cdots,z^{(k)i}$, from (1.5.7) are constructed only by
    means of the Riemannian structure $g$;
    \item[$2^\circ$] The following function
    \begin{equation}
L(x,y^{(1)},...,y^{(k)}) = g_{ij}(x)z^{(k)i}z^{(k)j}
\end{equation}
is a regular Lagrangian which depend only on the Riemann structure
$g_{ij}(x)$.
    \item[$3^\circ$] The previous theory, for pseudo-Riemann
    structure $g_{ij}(x)$, holds.
\end{enumerate}

\section{$N-$linear connections on $T^kM$}
\setcounter{equation}{0}\setcounter{theorem}{0}\setcounter{definition}{0}\setcounter{proposition}{0}\setcounter{remark}{0}

The notion of $N-$linear connection on the manifold $T^kM$ can be
studied as a natural extension of that of $N-$linear connection on
$TM$, given in the section 1.4.

Let $N$ be a nonlinear connection on $T^kM$ having the primal
coefficients $\underset{(1)}{N^i_j},\cdots,\underset{(k)}{N^i_j}$
and the dual coefficients
$\underset{(1)}{M^i_j},\cdots,\underset{(k)}{M^i_j}$.

\begin{definition} A linear connection $D$ on the manifold
$T^kM$ is called distinguished if $D$ preserves by parallelism the
horizontal distribution $N$. It is an $N-$connection if has the
following property, too:
\begin{equation}
DJ=0.
\end{equation}
\end{definition}

We have:

\begin{theorem}
A linear connection $D$ on $T^kM$ is an
$N$-linear connection if and only if
\begin{equation}
\left\{
\begin{array}{l}
\left( D_XY^H\right)^{V_\alpha }=0,\ (\alpha=1,\cdots,k),\
(D_XY^{V_\alpha})^H=0;\\ \\ \hfill (D_X Y^{V_\alpha})^{V_\beta}=0\
(\alpha\neq\beta)\vspace{3mm}
\\
D_X(JY^H)=JD_XY^H;\ D_X(JV^\alpha)=JD_XV^\alpha.
\end{array}
\right.
\end{equation}
\end{theorem}

Of course, for any $N-$linear connection $D$ we have $$Dh=0,\
Dv^\alpha=0, \ (\alpha=1,\cdots,k).$$

Since $$D_XY=D_{X^H}Y+D_{X^{V_1}}Y+\cdots +D_{X^{V_k}}Y,$$ setting
$$D_X^H=D_{X^H},D^{V_\alpha}_X=D_{X^{V_\alpha}},\
(\alpha=1,\cdots,k),$$ we can write:
\begin{equation}
D_XY=D_X^HY+D_X^{V_1}Y+\cdots+D^{V_k}_XY.
\end{equation}

The operators $D^H,D^{V_\alpha}$ are not the covariant derivations
but they have similar properties with the covariant derivations.
The notion of $d-$tensor fields ($d-$means ``distinguished'') can be
introduced and studied exactly as in the section 1.3.

In the adapted basis (1.4.12) and in adapted cobasis (1.5.1) we represent a $d-$tensor field of type $(r,s)$ in the form
\begin{equation}
T=T_{j_1...j_s}^{i_1...i_r}\displaystyle\frac \delta {\delta
x^{i_1}}\otimes \cdots \otimes \displaystyle\frac \delta {\delta
y^{(k)i_r}}\otimes dx^{j_1}\otimes \cdots \otimes \delta
y^{(k)j_s}.
\end{equation}

A transformation of coordinates (1.1.5), Ch. 4, has as effect the
following rule of transformation:
\begin{equation}
\widetilde{T}_{j_1...j_s}^{i_{1...}i_r}=\displaystyle\frac{\partial
\widetilde{x}^{i_1}}{\partial x^{h_1}}\cdots
\displaystyle\frac{\partial \widetilde{x}^{i_r}}{\partial
x^{hr}}\displaystyle\frac{\partial x^{k_1}}{
\partial x^{j_1}}\cdots \displaystyle\frac{\partial x^{k_s}}{
\partial {x}^{j_s}}T_{k_1...k_s}^{h_{1...}h_r}.  \tag{1.7.3'}
\end{equation}

So, $\left\{1,\dfrac{\delta}{\delta
x^i},\cdots,\dfrac{\delta}{\delta y^{(k)i}}\right\}$ generate the
tensor algebra of $d-$tensor fields.

The theory of $N-$linear connection described in the chapter 1 for
case $k=1$ can be extended step by step for the $N-$linear
connection on the manifold $T^kM$.

In the adapted basis (1.4.11) an $N-$linear connection $D$ has
the following form:
\begin{equation}
\begin{array}{l}
D_{\frac \delta {\delta x^j}}\displaystyle\frac \delta {\delta
x^i}\hspace*{-0.7mm}=\hspace*{-0.7mm}L_{ij}^m\displaystyle\frac \delta {\delta x^m},\ D_{\frac
\delta {\delta x^j}}\displaystyle\frac \delta {\delta y^{(\alpha
)i}}\hspace*{-0.7mm}=\hspace*{-0.7mm}L_{ij}^m\displaystyle\frac \delta {\delta y^{(\alpha
)m}},(\alpha=1,\cdots,k), \\
\\
D_{\frac \delta {\delta y^{(\beta )j}}}\displaystyle\frac\delta
{\delta x^i}=\underset{(\beta)}{C_{\ ij}^m}\displaystyle\frac
\delta{\delta x^m},\ D_{\frac \delta {\delta y^{(\beta )j}}}\displaystyle\frac\delta {\delta y^{(\alpha)i}}=
\underset{(\beta)}{C_{\ ij}^m}\displaystyle\frac\delta {\delta y^{(\alpha
)m}},\\ \\ \hfill (\alpha ,\beta = 1,\cdots,k).
\end{array}
\end{equation}

The system of functions
\begin{equation}
D\Gamma(N)=(L_{\ ij}^m,\underset{(1)}{C_{\ ij}^m},...,
\underset{(k)}{C_{\ ij}^m})
\end{equation}
represents \textit{the coefficients} of $D$.

With respect to (1.1.5), $L_{\ ij}^m$ are transformed by the
same rule as the coefficients of a linear connection defined on
the base manifold $M$. Others coefficients $\underset{(\alpha
)}{C_{\ ij}^h},$ $(\alpha =1,..,k)$ are transformed like
$d$-tensors of type $(1,2)$.

If $T$ is a $d$-tensor field of type $(r,s)$, given by (1.7.4) and
$X=X^H=X^i\dfrac{\delta}{\delta x^i}$, then, by means of (1.7.5),
$D^H_XT$ is:
\begin{equation}
D_X^HT=X^mT_{j_1...j_s|m}^{i_1...i_r}\displaystyle\frac \delta
{\delta x^{i_1}}\otimes \cdots \otimes \displaystyle\frac \delta
{\delta y^{(k)i_r}}\otimes dx^{j_1}\otimes \cdots \otimes \delta
y^{(k)j_s},
\end{equation}
where
\begin{equation}
T_{j_1...j_s|m}^{i_1...i_r}=\displaystyle\frac{\delta
T_{j_1...j_s}^{i_1...i_r}}{\delta x^m}
+L_{hm}^{i_1}T_{j_1...j_s}^{hi_2...i_r}+\cdots
-L_{j_sm}^hT_{j_1...h}^{i_1...i_r}.
\end{equation}

The operator ``$_{\mid}$'' will be called \textit{the
}$h$\textit{-covariant derivative}.

Consider the $v_\alpha $-covariant derivatives $D_X^{V_\alpha }$,
for $X=\stackrel{(\alpha)}{X^i}\displaystyle\frac \delta {\delta
y^{(\alpha )i}}$, $(\alpha=1,$ ..., $k)$. Then, (1.7.3) and (1.7.5) lead
to:
\begin{equation}
D_X^{V_\alpha }T=\stackrel{(\alpha
)}{X^m}T_{j_1...j_s}^{i_1...i_r}\stackrel{ (\alpha )}{\mid
}_m\displaystyle\frac \delta {\delta x^{i_1}}\otimes \cdots
\otimes \displaystyle\frac \delta {\delta y^{(k)i_r}}\otimes
dx^{j_1}\otimes \cdots \otimes dx^j_s,
\end{equation}
where
$$T_{j_1...j_s}^{i_1...i_r}\stackrel{(\alpha )}{\mid }_m=\displaystyle\frac{ \delta T_{j_1...j_s}^{i_1...i_r}}{\delta y^{(\alpha
)m}}+\underset{(\alpha
)}{C_{hm}^{i_1}}T_{j_1...j_s}^{hi_2...i_r}+\cdots
-\underset{(\alpha )}{ C_{j_sm}^h}T_{j_1...j_{s-1}h}^{i_1...i_r},
\ (\alpha=1,..,k).$$

The operators ``$\stackrel{(\alpha )}{\mid }$'', in number of $k$
are called $v_\alpha $\textit{-covariant derivatives}.

Each of operators ``$_\mid$'' and ``$\stackrel{(\alpha )}{\mid
}$'' has the usual properties with respect to sum of $d$-tensor or
them tensor product.

Now, in the adapted basis (1.4.11) we can determine the
torsion $T$ and curvature $R$ of an $N-$linear connection $D$,
follows the same method as in the case $k=1$.

We remark the following of $d-$tensors of torsion:
\begin{equation}
T_{\ jk}^i=L_{jk}^i-L_{kj}^i,\ \underset{(\alpha)}{S_{\
jk}^i}=\underset{(\alpha)}{C_{\ jk}^i}-\underset{(\alpha)}{C_{\
kj}^i}=0, \ (\alpha =1,..,k)
\end{equation}
and $d-$tensors of curvature
\begin{equation}
R_{h\ jm}^{\ i},\quad \underset{(\alpha)}{P}_{h\ jm}^{\ i},\
\underset{(\beta\alpha)}{S}_{h\ jm}^{\ i},\ \
(\alpha,\beta=1,\cdots,k)
\end{equation}

The 1-forms connection of the $N-$linear connection $D$ are:
\begin{equation}
\omega_{\ j}^i=L_{jh}^i dx^h+\underset{(1)}{C_{\ jh}^{i}}\delta
y^{(1)h}+\cdots +\underset{(k)}{C_{\ jh}^i}\delta
y^{(k)h}.
\end{equation}
The following important theorem holds:

\begin{theorem}
The structure equations of an $N-$linear
connection $D$ on the manifold $T^kM$ are given by:
\begin{equation}
\begin{array}{l}
d(dx^i)-dx^m\wedge\omega_{\ m}^i=-\stackrel{(0)}{\Omega^i}\, \\
\\
d(\delta y^{(\alpha)i})-\delta y^{(\alpha)m}\wedge\omega_{\ m}^i=-\stackrel{(\alpha)}{\Omega^i}, \\
\\
d\omega_{\ j}^i-\omega_{\ j}^m\wedge\omega_{\ m}^i=-\Omega_{\
j}^i,
\end{array}
\end{equation}
where $\stackrel{(0)}{\Omega^i}$, $\stackrel{(\alpha)}{\Omega^i}$ are the $2$-forms of torsion and
where $\Omega^{i}_{\ j}$ are the 2-forms of curvature:
$$\Omega _{\ j}^i\ =\displaystyle\frac 12 R_{j\ pq}^{\ i}dx^p\wedge dx^q+
$$
$$+\sum\limits_{\alpha =1}^k\underset{(\alpha )}{P_{j\ pq}^{\ i}}
dx^p\wedge \delta y^{(\alpha )q}+\sum\limits_{\alpha ,\beta =1}^k
\underset{(\alpha \beta )}{S_{j\ pq}^{\ i}}\delta y^{(\alpha
)p}\wedge \delta y^{(\beta )q}.$$
\end{theorem}

Now, the Bianchi identities of $D$ can be derived from (1.7.13).

The nonlinear connection $N$ and the $N-$linear connection $D$
allow to study the geometrical properties of the manifold $T^kM$
equipped with these two geometrical object fields.

\newpage

\chapter{Lagrange Spaces of Higher--order}\index{Higher--order spaces! Lagrange}
\label{ch1p1} 


The concept of higher - order Lagrange space was introduced and
studied by the author of the present monograph, \cite{mir5}, \cite{mir11}.

A Lagrange space of order $k$ is defined as a pair
$L^{(k)n}=(M,L)$ where $L:T^kM\to R$ is a differentiable regular
Lagrangian having the fundamental tensor of constant signature.
Applying the variational problem to the integral of action of $L$
we determine: a canonical $k-$semispray, a canonical nonlinear
connection and a canonical metrical connection. All these are
basic for the geometry of space $L^{(k)n}$.

\setcounter{section}{0}
\section{The spaces $L^{(k)n}=(M,L)$}
\setcounter{equation}{0}\setcounter{theorem}{0}\setcounter{definition}{0}\setcounter{proposition}{0}\setcounter{remark}{0}

\begin{definition} A Lagrange space of order $k\geq 1$ is a
pair $L^{(k)m}=(M,L)$ formed by a real $n-$dimensional manifold
$M$ and a differentiable Lagrangian
$L:(x,y^{(1)},\cdots,y^{(k)})\in T^kM\to
L(x,y^{(1)},\cdots,y^{(k)})\in R$ for which the Hessian with the
elements:
\begin{equation}
g_{ij}=\displaystyle\frac
12\displaystyle\frac{\partial^2L}{\partial y^{(k)i}\partial
y^{(k)j}}
\end{equation}
has the property
\begin{equation}
rank(g_{ij})=n\ \ on\ \ \widetilde{T^kM}
\end{equation}
and the $d-$tensor $g_{ij}$ has a constant signature.
\end{definition}

Of course, we can prove that $g_{ij}$ from (2.1.1) is a $d-$tensor
field, of type (0,2), symmetric. It is called the {\it
fundamental} (or {\it metric}) tensor of the space $L^{(k)n}$,
while $L$ is called its {\it fundamental function}.

The geometry of the manifold $T^kM$ equipped with
$L(x,y^{(1)},\cdots,y^{(k)})$ is called the geometry of the space
$L^{(k)n}$. We shall study this geometry using the theory from the
last chapter. Consequently, starting from the integral of action
$I(c)=\dd\int^1_0
L\left(x,\dfrac{dx}{dt},\cdots,\dfrac{1}{k!}\dfrac{d^kx}{dt^k}\right)dt$
we determine the Euler - Lagrange equations
$\stackrel{0}{E}_i(L)=0$ and the Craig - Synge covectors
$\stackrel{1}{E}_i(L),\cdots,\stackrel{k}{E}_i(L)$. According to
(3.1.14) one remarks that we have:

\begin{theorem}
The equations
$g^{ij}\stackrel{k-1}{E}_i(L)=0$ determine a $k-$semi\-spray
\begin{equation}
S=y^{(1)i}\displaystyle\frac \partial {\partial x^i}+2y^{(2)i}\displaystyle %
\frac \partial {\partial y^{(1)i}}+\cdots
+ky^{(k)i}\displaystyle\frac
\partial {\partial y^{(k-1)i}}-(k+1)G^i\displaystyle\frac \partial {\partial
y^{(k)i}}
\end{equation}
where the coefficients $G^i$ are given by
\begin{equation}
(k+1)G^i=\displaystyle\frac 1{2}g^{ij}\left\{ \Gamma \left( %
\displaystyle\frac \partial {\partial y^{(k)i}}\right)
-\displaystyle\frac
\partial {\partial y^{(k-1)i}}\right\}
\end{equation}
$\Gamma$ being the operator (1.2.4).
\end{theorem}

The semispray $S$ depend on the fundamental function $L$, only.
$S$ is called canonical. It is globally defined on the manifold
$\widetilde{T^kM}$.

Taking into account Theorem 1.5.1, we have:

\begin{theorem}
The systems of functions
\begin{equation}
\begin{array}{l}
\underset{(1)}{M_j^i}=\displaystyle\frac{\partial G^i}{\partial y^{(k)j}},~%
\underset{(2)}{M_j^i}=\displaystyle\frac 12\left(S\underset{(1)}{M_j^i}+%
\underset{(1)}{M_m^i}\underset{(1)}{M_j^m}\right) ,..., \\
\\
\underset{(k)}{M_j^i}=\displaystyle\frac 1k\left( S\underset{(k-1)}{M_j^i%
}+\underset{(1)}{M_m^i}\underset{(k-1)}{M_j^m}\right)
\end{array}
\end{equation} are the dual coefficients of a nonlinear connection $N$ determined only on the fundamental function $L$ of the space $L^{(k)n}$.
\end{theorem}

$N$ is the canonical nonlinear connection of $L^{(k)n}$.

The adapted basis $\left\{\displaystyle\frac \delta {\delta
x^i},\displaystyle\frac \delta {\delta y^{(1)i}},\cdots,\displaystyle\frac \delta {\delta y^{(k)i}}\right\}$ has its dual
$\{\delta x^i,\delta y^{(1)i}$, ..., $\delta y^{(k)i}\}$. They are
constructed by the canonical nonlinear connection. So, the
horizontal curves are characterized by the system of differential
equations Part 2, Ch. 2, and the autoparallel curves of $N$ are
given by Part II, Ch. 1.

The condition that $N$ be integrable is expressed by
$\left[\dfrac{\delta}{\delta x^i},\dfrac{\delta}{\delta
x^j}\right]^{V_\alpha}=0$, $(\alpha=1,\cdots,k)$.

\section{Examples of spaces $L^{(k)n}$}
\setcounter{equation}{0}\setcounter{theorem}{0}\setcounter{definition}{0}\setcounter{proposition}{0}\setcounter{remark}{0}

\begin{itemize}
  \item[$1^\circ$] Let us consider the Lagrangian:
  \begin{equation}
L(x,y^{(1)},...,y^{(k)})=g_{ij}(x)z^{(k)i}z^{(k)j}
\end{equation}
where $g_{ij}(x)$ is a Riemannian (or pseudo Riemannian) metric on
the base manifold $M$ and the $z^{(k)i}$ is the Liuoville
$d-$vector field:
\begin{equation}
kz^{(k)i}=ky^{(k)i}+(k-1)\underset{(1)}{M_m^i}y^{(k-1)m}+\cdots +%
k\underset{(k-1)}{M_m^i}y^{(1)m}
\end{equation}
constructed by means of the dual coefficients (Part II, Ch. 1) of the
canonical nonlinear connection $N$ from the problem of
prolongation to $T^kM$ of $g_{ij}(x)$. So that the Lagrangian
(2.2.1) depend on $g_{ij}(x)$ only.

The pair $L^{(k)n}=(M,L)$, (2.2.1) is a Lagrange space of order
$k$. Its fundamental tensor is $g_{ij}(x)$, since the $d-$vector
$z^{(k)i}$ is linearly in the variables $y^{(k)}$.
  \item[$2^\circ$] Let $\stackrel{\circ}{L}(x,y^{(1)})$ be the Lagrangian from electrodynamics
\begin{equation}
\stackrel{\circ }{L}(x,y^{(1)})=mc\gamma
_{ij}(x)y^{(1)i}y^{(1)j}+\displaystyle\frac{2e}mb_i(x)y^{(1)i}
\end{equation}
Let $N$ be the nonlinear connection given by the theorem Part II, Ch. 1, from the problem of prolongations to $T^kM$ of the Riemannian (or pseudo Riemannian) structure $\gamma_{ij}(x)$ and the
Liouville tensor $z^{(k)i}$ constructed by means of $N$. Then the
pair $L^{(k)n}=(M,L)$, with
\begin{equation}
L(x,y^{\left( 1\right) },...,y^{\left( k\right) })=mc\gamma
_{ij}(x)z^{\left( k\right) i}z^{\left( k\right)
j}+\displaystyle\frac{2e} mb_i(x)z^{\left( k\right) i}
\end{equation}
is a Lagrange space of order $k$. It is the prolongation to the
manifold $T^kM$ of the Lagrangian $\stackrel{\circ}{L}$ (2.2.3) of
electrodynamics.

These examples prove the existence of the Lagrange spaces of order
$k$.
\end{itemize}

\section{Canonical metrical $N-$connection}
\setcounter{equation}{0}\setcounter{theorem}{0}\setcounter{definition}{0}\setcounter{proposition}{0}\setcounter{remark}{0}

Consider the canonical nonlinear connection $N$ of a Lagrange
space of order $k$, $L^{(k)n}=(M,L)$.

An $N-$linear connection $\D$ with the coefficients
$D\Gamma(N)=(L^{i}_{jk},\underset{(1)}{C_{\ jh}^i},$
$\cdots,\underset{(k)}{C_{\ jh}^i})$ is called metrical with
respect to metric tensor $g_{ij}$ if
\begin{equation}
g_{ij|h}=0,\quad g_{ij}\stackrel{(\alpha)}{|}_h=0,\quad
(\alpha=1,\cdots,k).
\end{equation}

Now we can prove the following theorem:

\begin{theorem}
The following properties hold:

1) There exists a unique $N$-linear connection $D$ on $\widetilde{T^kM}$
verifying the axioms:
\begin{enumerate}
\item[$1^\circ$] $N-$ is the canonical nonlinear connection of
space $L^{(k)n}$. \item[$2^\circ$] $g_{ij|h}=0$, ($D$ is
$h-$metrical) \item[$3^\circ$] $g_{ij}\stackrel{(\alpha)}{|}_h=0$,
$(\alpha=1,\cdots,k)$, ($D$ is $v_\alpha-$metrical)
\item[$4^\circ$] $T_{\ jh}^i=0$, ($D$ is $h-$torsion free)
\item[$5^\circ$] $\underset{(\alpha)}{S_{\ jh}^i}=0$,
$(\alpha=1,\cdots,k)$, ($D$ is $v_\alpha-$torsion free).
\end{enumerate}

2) The coefficients $C\Gamma (N)=(L_{\ ij}^h,\underset{(1)}{C_{\ ij}^h},...,%
\underset{(k)}{C_{\ ij}^h})$ of $D$ are given by the generalized
Christoffel symbols:
\begin{equation}
\begin{array}{l}
L_{\ ij}^h=\displaystyle\frac 12 g^{hs}\left(\displaystyle\frac{\delta g_{is}}{%
\delta x^j}+\displaystyle\frac{\delta g_{sj}}{\delta x^i}-\displaystyle\frac{%
\delta g_{ij}}{\delta x^s}\right), \\
\\
\underset{\left(\alpha \right)}{C_{\ ij}^h}=\displaystyle\frac 12 g^{hs}\left(%
\displaystyle\frac{\delta g_{is}}{\delta y^{\left( \alpha \right) j}}+%
\displaystyle\frac{\delta g_{sj}}{\delta y^{\left( \alpha \right) i}}-%
\displaystyle\frac{\delta g_{ij}}{\delta y^{\left( \alpha \right)
s}}\right),\ (\alpha =1,...,k).
\end{array}
\end{equation}

3) $D$ depends only on the fundamental function $L$ of the space
$L^{\left( k\right) n}$.
\end{theorem}

The connection $D$ from the previous theorem is called
\textit{canonical metrical} $N$-{\it connection} and its
coefficients (2.3.2) are denoted by $C\Gamma(N)$.

Now, the geometry of the Lagrange spaces $L^{(k)n}$ can be
developed by means of these two canonical connection $N$ and $D$.

\section{The Riemannian $(k-1)n-$contact model of the space $L^{(k)n}$}
\setcounter{equation}{0}\setcounter{theorem}{0}\setcounter{definition}{0}\setcounter{proposition}{0}\setcounter{remark}{0}

The almost K\"{a}hlerian model of the Lagrange spaces $L^n$
expound in the section 7, Ch. 2, can be extended in a
corresponding model of the higher order Lagrange spaces. But now,
it is a Riemannian almost $(k-1)n-$contact structure on the
manifold $\widetilde{T^kM}$.

The canonical nonlinear connection $N$ of the space
$L^{(k)n}=(M,L)$ determines the following ${\cal
F}(\widetilde{T^kM})-$linear mapping $\F:{\cal
X}(\widetilde{T^kM})\to{\cal X}(\widetilde{T^kM})$ defined on the
adapted basis to $N$ and to $N_\alpha$, by
\begin{equation}
\begin{array}{l}
\F\left(\displaystyle\frac{\delta}{\delta
x^i}\right)=-\displaystyle\frac{\partial}{\partial y^{(k)i}}\\
\\
F\left(\displaystyle\frac{\delta}{\delta
y^{(\alpha)i}}\right)=0,\quad (\alpha=1,\cdots,k-1)\\
\\
F\left(\displaystyle\frac{\partial}{\partial
y^{(k)i}}\right)=\displaystyle\frac{\delta}{\delta x^i},\quad
(i=1,\cdots,n)
\end{array}
\end{equation}

We can prove:

\begin{theorem}
We have: \begin{enumerate}
\item[1$^{\circ}$] $\F$ is globally defined on
$\widetilde{T^kM}$.

\item[2$^{\circ}$] $\F$ is a tensor field of type $(1,1)$ on
$\widetilde{T^kM}$

\item[3$^\circ$] Ker$\F=N_1\oplus N_2\oplus\cdots\oplus
N_{k-1}$, Im$\F=N_0\oplus V_k$

\item[4$^\circ$] rank$\|\F\|=2n$

\item[5$^\circ$] $\F^3+\F=0$.\end{enumerate}

Thus $\F$ is an almost $(k-1)n-$contact structure on
$\widetilde{T^kM}$ determined by $N$.
\end{theorem}

Let
$\left(\underset{1a}{\xi},\underset{2a}{\xi},...,\underset{(k-1)a}{\xi_a}\right)
$, $(a=1,...,n)$ be a local basis adapted to the direct
decomposition $N_1\oplus \cdots \oplus N_{k-1}$ and
$\left(\stackrel{1a}{\eta},\stackrel{2a}{\eta},...,\stackrel{(k-1)a}{\eta}\right)$,
its dual.

Thus the set
\begin{equation}
\left(\F,\underset{1a}{\xi},...,\underset{(k-1)a}{\xi},\stackrel{1a}{\eta},...,\stackrel{(k-1)a}{\eta}\right),
\quad (a=1,\cdots,n-1)
\end{equation}
is a $(k-1)n$ almost contact structure.

Indeed, (2.4.1) imply:
$$\F(\underset{\alpha a}{\xi})=0,\stackrel{\alpha a}{\eta}(\underset{\beta b}{\xi})=\left\{\begin{array}{ccc}
\delta^a_{\ b}& {\rm{\ for\ }} & \alpha=\beta \\ 0 & {\rm{\ for\ }}& \alpha\neq\beta,
(\alpha,\beta=1,\cdots,(k-1)).\end{array}\right.$$

$$\F^2(X)=-X+\dd\sum^n_{a=1}\dd\sum^{k-1}_{\alpha=1}\stackrel{\alpha a}{\eta}(X)\underset{\alpha a}{\xi},\quad
\forall X\in{\cal X}(T^kM),\stackrel{\alpha
a}{\eta}\circ\F=0.$$

Let $N_\F$ be the Nijenhuis tensor of the structure
$\F$.
$$N_\F(X,Y)=[\F X,\F Y]+\F^2[X,Y]-\F[\F X,Y]-
\F[X,\F Y].$$ The structure (2.4.2) is said to be normal
if:
$$N_{\F}(X,Y)+\dd\sum^n_{a=1}\dd\sum^{k-1}_{\alpha=1}d\stackrel{\alpha
a}{\eta}(X,Y)=0,\quad \forall X,Y\in{\cal X}(T^kM).$$ So we obtain
a characterization of the normal structure $\F$ given by the
following theorem:

\begin{theorem}
The almost $(k-1)n-$contact structure
(2.4.2) is normal if and only if for any $X,Y\in{\cal
X}(\widetilde{T^kM})$ we have:
$$N_{\F}(X,Y)+\dd\sum^n_{a=1}\dd\sum^{k-1}_{\alpha=1}d(\delta y^{(\alpha)a})\ (X,Y)=0.$$
\end{theorem}

The lift of fundamental tensor $g_{ij}$ of the space $L^{(k)n}$
with respect to $N$ is defined by:
\begin{equation}
\G=g_{ij}dx^i\otimes dx^j+g_{ij}\delta y^{(1)i}\otimes
\delta y^{(1)j}+\cdots +g_{ij}\delta y^{(k)i}\otimes \delta
y^{(k)j}.
\end{equation}

Evidently, $\G$ is a pseudo-Riemannian structure on the
manifold $\widetilde{T^kM}$, determined only by space $L^{(k)n}$.

Now, it is not difficult to prove:

\begin{theorem}
The pair $(\G,\F)$ is a
Riemannian $(k-1)n$-almost contact structure on
$\widetilde{T^kM}$.
\end{theorem}

In this case, the next condition holds:
$$\G(\F X,Y)=-\G(\F Y,X),\quad \forall X,Y\in{\cal
X}(\widetilde{T^kM}).$$ Therefore the triple
$(\widetilde{T^kM},\G,\F)$ is an metrical
$(k-1)n$-almost contact space named the {\it geometrical model of
the Lagrange space of order $k$, $L^{(k)n}$}.

Using this space we can study the electromagnetic and
gravitational fields in the spaces $L^{(k)n}$, \cite{mir11}.

\section{The generalized Lagrange spaces of order $k$}\index{Generalized spaces! Lagrange}
\setcounter{equation}{0}\setcounter{theorem}{0}\setcounter{definition}{0}\setcounter{proposition}{0}\setcounter{remark}{0}

The notion of generalized Lagrange space of higher order is a
natural extension of that studied in chapter 2.

\begin{definition} A generalized Lagrange space of order $k$ is
a pair $GL^{(k)n}=(M,g_{ij})$ formed by a real differentiable
$n-$dimensional manifold $M$ and $a$ $C^\infty-$covariant of type
(0,2), symmetric $d-$tensor field $g_{ij}$ on $\widetilde{T^kM}$,
having the properties:
\begin{enumerate}
\item[a.] $g_{ij}$ has a constant signature on $\widetilde{T^kM}$;
\item[b.] rank$(g_{ij})=n$ on $\widetilde{T^kM}$.
\end{enumerate}
$g_{ij}$ is called the \textit{fundamental tensor} of $GL^{(k)n}$.
\end{definition}

Evidently, any Lagrange space of order $k$, $L^{(k)n}=(M,L)$
determines a space $GL^{(k)n}$ with fundamental tensor
\begin{equation}
g_{ij}=\displaystyle\frac 12\displaystyle\frac{\partial
^2L}{\partial y^{(k)i}\partial y^{(k)j}}.
\end{equation}

But not and conversely. If $g_{ij}(x,y^{(1)},...,y^{(k)})$ is a
priori given, it is possible that the system of differential
partial equations (2.5.1) does not admit any solution in
$L(x,y^{(1)},...,y^{(k)})$. A necessary condition that the system
(2.5.1) admits solutions in the function $L$ is that $d$-tensor
field
\begin{equation}
\underset{(k)}{C}_{\ ijh}=\displaystyle\frac 12\displaystyle\frac{
\partial g_{ij}}{\partial y^{(k)h}}
\end{equation}
be completely symmetric.

If the system (2.5.1) has solutions, with respect to $L$ we say that
the space $GL^{(k)n}$ is reducible to a Lagrange space of order
$k$. If this property is not true, then $GL^n$ is said to be
nonreducible to a Lagrange space $L^{(k)n}$.

\bigskip

{\bf Examples}
\begin{enumerate}
\item[$1^\circ$] Let ${\cal R}=(M,\gamma_{ij}(x))$ be a Riemannian
space and $\sigma \in{\cal F}(T^kM)$. Consider the $d$-tensor
field:
\begin{equation}
g_{ij}=e^{2\sigma}(\gamma_{ij}\circ \pi ^k).
\end{equation}
If $\dfrac{\partial\sigma}{\partial y^{(k)h}}$ is a nonvanishes
$d-$covector on the manifold $\widetilde{T^kM}$, then the pair
$GL^{(k)n}=(M,g_{ij})$ is a generalized Lagrange space of order
$k$ and it is not reducible to a Lagrange space $L^{(k)n}$.

\item[$2^\circ$] Let ${\cal R}^n=(M,\gamma_{ij}(x))$ be a Riemann
space and Prol$^k{\cal R}^n$ be its prolongation of order $k$ to
$\widetilde{T^kM}$.

Consider the Liouville $d-$vector field $z^{(k)i}$ of Prol$^k{\cal
R}^n$. It is expressed in the formula (2.2.2). We can
introduce the $d-$covector field $z^{(k)}_{\
i}=\gamma_{ij}z^{(k)j}$.

We assume that exists a function $n(x,y^{(1)},\cdots,y^{(k)})\geq
1$ on $\widetilde{T^kM}$.

Thus
\begin{equation}
g_{ij}=\gamma_{ij}+\left( 1-\displaystyle\frac 1{n^2}\right)
z_i^{(k)}z_j^{(k)}
\end{equation}
is the fundamental tensor of a space $GL^{(k)n}$. Evidently this
space is not reducible to a space $L^{(k)n}$, if the function
$n\neq 1$.
\end{enumerate}

These two examples prove the existence of the generalized Lagrange
space of order $k$.

In the last example, $k=1$ leads to the metric Part I, Ch. 2 of the
Relativistic Optics, ($n$ being the refractive index).

In a generalized Lagrange space $GL^{(k)n}$ is difficult to find a
nonlinear connection $N$ derived only by the fundamental tensor
$g_{ij}$. Therefore, assuming that $N$ is a priori given , we
shall study the pair $(N, GL^{(k)n})$. Thus of theorem of the
existence and uniqueness metrical $N-$linear connection holds:

\begin{theorem}
We have: \begin{enumerate}
\item[$1^\circ$] There exists an unique $N$-linear connection $D$
for which
$$
\begin{array}{l}
g_{ij|h}=0,\ g_{ij}\stackrel{(\alpha)}{|}_h=0,\ (\alpha =1,...,k), \\
\\
T_{\ jk}^i=0,\ \underset{(\alpha)}{S_{\ jk}^i}=0,\ (\alpha
=1,...,k).
\end{array}
$$
\item[$2^\circ$] The coefficients of $D$ are given by the
generalized Christoffel symbols Part II, Ch. 2.

\item[$3^\circ$] $D$ depends on $g_{ij}$ and $N$ only.
\end{enumerate}
\end{theorem}

Using this theorem it is not difficult to study the geometry of
Generalized Lagrange spaces.

\newpage

\chapter{Higher-Order Finsler spaces}\index{Higher--order spaces! Finsler}
\label{ch1p1} 


The notion of Finsler spaces of order $k$, introduced by the
author of this monograph and presented in the book {\it The
Geometry of Higher-Order Finsler Spaces}, Hadronic Press, 1998, is
a natural extension to the manifold $\widetilde{T^kM}$ of the
theory of Finsler spaces given in the Part I, Ch. 3. A substantial
contribution in the studying of these spaces have H. Shimada and
S. Sab\u{a}u \cite{shimada}.

The impact of this geometry in Differential Geometry, Variational
Calculus, Analytical Mechanics and Theoretical Physics is
decisive. Finsler spaces play a role in applications to Biology,
Engineering, Physics or Optimal Control. Also, the introduction of
the notion of Finsler space of order $k$ is demanded by the
solution of problem of prolongation to $T^kM$ of the Riemannian or
Finslerian structures defined on the base manifold $M$.

\setcounter{section}{0}
\section{Notion of Finsler space of order $k$}
\setcounter{equation}{0}

In order to introduce the Finsler space of order $k$ are necessary
some considerations on the concept of homogeneity of functions on
the manifold $T^kM$, \cite{mir11}.

A function $f:T^kM\to R$ of $C^\infty-$class on $\widetilde{T^kM}$
and continuous on the null section of $\pi^k:T^kM\to M$ is called
homogeneous of degree $r\in Z$ on the fibres of $T^kM$ (briefly
$r-$homogeneous) if for any $a\in R^+$ we have
$$f(x,ay^{(1)},a^2 y^{(2)},\cdots,a^k y^{(k)})=a^r
f(x,y^{(1)},\cdots,y^{(k)}).$$

An Euler Theorem hold:

{\it A function $f\in{\cal F}(T^kM)$, differentiable on
$\widetilde{T^kM}$ and continuous on the null section of $\pi^k$
is $r-$homogeneous if and only if
\begin{equation}
{\cal L}_{\stackrel{k}{\Gamma }}f=rf,
\end{equation}
${\cal L}_{\stackrel{k}{\Gamma }}$ being the Lie derivative with
respect to the Lioville vector field $\stackrel{k}{\Gamma}$.}

A vector field $X\in{\cal X}(T^kM)$ is $r-$homogeneous if
\begin{equation}
{\cal L}_{\stackrel{k}{\Gamma }}X=(r-1)X\tag{3.1.1'}
\end{equation}

\begin{definition}
A Finsler space of order $k$, $k$ $\geq 1$,
is a pair $F^{\left( k\right) n}=(M,F)$ determined by a real
differentiable manifold $M$ of dimension $n$ and a function
$F:T^kM\rightarrow R$ having the following properties:

\begin{enumerate}
\item[1$^{\circ }$] $F$ is differentiable on $\widetilde{T^kM}$
and continuous on the null section on $\pi^k$.

\item[2$^{\circ}$] $F$ is positive.

\item[3$^{\circ}$] $F$ is $k$-homogeneous on the fibres of the
bundle $T^kM$.

\item[4$^{\circ}$] The Hessian of $F^2$ with the elements:
\begin{equation}
g_{ij}=\displaystyle\frac 12\displaystyle\frac{\partial
^2F^2}{\partial y^{\left( k\right) i}\partial y^{\left( k\right)
j}}
\end{equation}
is positively defined on $\widetilde{T^kM}.$
\end{enumerate}
\end{definition}

From this definition it follows that {\it the fundamental tensor
$g_{ij}$} is nonsingular and 0-homogeneous on the fibres of
$T^kM$.

Also, we remark: Any Finlser space $F^{(k)n}$ can be considered as
a Lagrange space $L^{(k)n}=(M,L)$, whose fundamental function $L$
is $F^2$.

By means of the solution of the problem of prolongation of a
Finsler structure $F(x,y^{(1)})$ to $T^kM$ we can construct some
important examples of spaces $F^{(k)n}$.

A Finsler space with the property $g_{ij}$ depend only on the
points $x\in M$ is called a Riemann space of order $k$ and denoted
by ${\cal R}^{(k)n}$.

Consequently, we have the following sequence of inclusions,
similar with that from Ch. 3:
\begin{equation}
\left\{{\cal R}^{(k)n}\right\} \subset \left\{ F^{(k)n}\right\}
\subset \left\{ L^{(k)n}\right\} \subset \left\{
GL^{(k)n}\right\}.
\end{equation}

So, the Lagrange geometry of order $k$ is the geometrical theory
of the sequence (1.1.3).

Of course the geometry of $F^{(k)n}$ can be studied as the
geometry of Lagrange space of order $k$, $L^{(k)n}=(M,F^2)$. Thus
the canonical nonlinear connection $N$ is {\it the Cartan
nonlinear connection of $F^{(k)n}$} and the metrical $N-$linear
connection $D$ is the Cartan $N-$metrical connection of the space
$F^{(k)n}$, \cite{mir11}.

\newpage

\chapter{The Geometry of $k-$cotangent bundle}
\label{ch1p1} 


\setcounter{section}{0}
\section{Notion of $k-$cotangent bundle, $T^{*k}M$}\index{$k-$cotangent}
\setcounter{equation}{0}\setcounter{theorem}{0}\setcounter{definition}{0}\setcounter{proposition}{0}

The $k-$cotangent bundle $(T^{*k}M,\pi^{*k},M)$ is a natural extension of that of cotangent bundle $(T^*M,\pi^*,M)$. It is basic for the Hamilton spaces of order $k$. The manifold $T^{*k}M$ must have some important properties:
\begin{enumerate}
  \item[$1^\circ$] $T^{*1}M=T^*M$;
  \item[$2^\circ$] dim$T^{*k}M=$dim$T^kM=(k+1)n$;
  \item[$3^\circ$] $T^{*k}M$ carries a natural Poisson structure;
  \item[$4^\circ$] $T^{*k}M$ is local diffeomorphic to $T^kM$.
\end{enumerate}
These properties are satisfied by considering the differentiable bundle $(T^{*k}M,\pi^{*k},M)$ as the fibered bundle $(T^{k-1}M x_M T^*M,\pi^{k-1}x_M\pi^*,M)$. So we have
\begin{equation}
T^{*k}M=T^{k-1}Mx_M T^*M,\ \ \pi^{*k}=\pi^{k-1}x_M\pi^*.
\end{equation}
A point $u\in T^{*k}M$ is of the form $u=(x,y^{(1)},...,y^{(k-1)},p)$. It is determined by the point $x=(x^i)\in M$, the acceleration $y^{(1)i}=\dfrac{dx^i}{dt}$, ..., $y^{(k-1)i}=\dfrac{1}{(k-1)!}\dfrac{d^{k-1}x^i}{dt^{k-1}}$ and the momenta $p=(p_i)$. The geometries of the manifolds $T^kM$ and $T^{*k}M$ are dual via Legendre transformation. Then, $(x^i,y^{(1)i},...,y^{(k-1)i},p_i)$ are the local coordinates of a point $u\in T^{*k}M$.

The change of local coordinates on $T^{*k}M$ is:
\begin{equation}
\begin{array}{l}
\wt{x}^i=\wt{x}^i(x^1,...,x^n),\ \det\(\dfrac{\pp\wt{x}^i}{\pp x^j}\)\neq 0,\\ \\
\wt{y}^{(1)i}=\dfrac{\pp\wt{x}^i}{\pp x^j}y^{(1)j},\\ \\
...........\\
(k-1)\wt{y}^{(k-1)i}=\dfrac{\pp\wt{y}^{(k-1)i}}{\pp x^j}y^{(1)j}+...+(k-1)\dfrac{\pp\wt{y}^{(k-2)i}}{\pp y^{(k-2)j}}y^{(k-1)j}\\ \\
\wt{p}_i=\dfrac{\pp{x}^j}{\pp\wt{x}^i}p_j
\end{array}
\end{equation}
where the following equalities hold
\begin{equation}
\dfrac{\pp\wt{y}^{(\a)i}}{\pp x^j}=\dfrac{\pp\wt{y}^{(\a+1)i}}{\pp y^{(1)j}}=...=\dfrac{\pp\wt{y}^{(k-1)i}}{\pp y^{(k-1-\a)j}};\ (\a=0,...,k-2; y^{(0)}=x).
\end{equation}

The Jacobian matrix $J_k$ of (4.1.2) have the property
\begin{equation}
\det J_k(u)=\[\det\(\dfrac{\pp\wt{x}^j}{\pp x^i}(u)\)\]^{k-1}.
\end{equation}
The vector fields $\dot{\pp}^i=(\dot{\pp}^1,..,\dot{\pp}^n)=\dfrac{\pp}{\pp p_i}$ generate a vertical distribution $W_k$, while $\left\{\dfrac{\pp}{\pp y^{(k-1)i}}\right\}$ determine a vertical distribution $V_{k-1}$, ..., $\left\{\dfrac{\pp}{\pp y^{(1)i}},...,\dfrac{\pp}{\pp y^{(k-1)i}}\right\}$ determine a vertical distribution $V_1$. We have the sequence of inclusions: $$V_{k-1}\subset V_{k-2}\subset ...\subset V_1\subset V,$$ $$V_u=V_{1,u}\oplus W_{k,u},\ \forall u\in T^{*k}M.$$ We obtain without difficulties:

\begin{theorem}
\begin{enumerate}
\item[$1^\circ$] The following operators in the algebra of functions on $T^{*k}M$ are the independent vector field on $T^{*k}M$:
\begin{equation}
\begin{array}{l}
\overset{1}\Gamma=y^{(1)i}\dfrac{\pp}{\pp y^{(k-1)i}}\\ \\
\overset{2}\Gamma=y^{(1)i}\dfrac{\pp}{\pp y^{(k-2)i}}+2y^{(2)i}\dfrac{\pp}{\pp y^{(k-1)i}}\\ \\
...........\\
\overset{k-1}\Gamma=y^{(1)i}\dfrac{\pp}{\pp y^{(1)i}}+...+(k-1)y^{(k-1)i}\dfrac{\pp}{\pp y^{(k-1)i}}\\ \\
\C^*=p_i\dot{\pp}^i.
\end{array}
\end{equation}
\item[$2^\circ$] The function
\begin{equation}
\varphi=p_i y^{(1)i}
\end{equation}
is a scalar function on $T^{*k}M$.
\end{enumerate}
$\overset{1}\Gamma,...,\overset{k-1}\Gamma$ are called the {\bf Liouville vector fields}.
\end{theorem}

\begin{theorem}
\begin{enumerate}
\item[$1^\circ$] For any differentiable function $H:\wt{T^{*k}M}\to R$, $\wt{T^*kM}=T^{*k}M\setminus\{0\}$, $d_0H,...,d_{k-2}H$ defined by
\begin{equation}
\begin{array}{l}
d_0H=\dfrac{\pp H}{\pp y^{(k-1)i}}dx^i\\ \\
d_1H=\dfrac{\pp H}{\pp y^{(k-2)i}}dx^i+\dfrac{\pp H}{\pp y^{(k-1)i}}dy^{(1)i}\\ \\
...........\\
d_{k-2}H=\dfrac{\pp H}{\pp y^{(1)i}}dx^i+\dfrac{\pp H}{\pp y^{(2)i}}dy^{(1)i}+...+\dfrac{\pp H}{\pp y^{(k-1)i}}dy^{(k-2)i}
\end{array}
\end{equation}
are fields of $1-$forms on $\wt{T^{*k}M}$.
\item[$2^\circ$] While
\begin{equation}
  d_{k-1}H=\dfrac{\pp H}{\pp x^i}dx^i+...+\dfrac{\pp H}{\pp y^{(k-1)i}}dy^{(k-1)i}
\end{equation}
is not a field of $1-$form.
\item[$3^\circ$] We have
\begin{equation}
dH=d_{k-1}H+\dot{\pp}^i H p_i.
\end{equation}
\end{enumerate}
\end{theorem}

If $H=\varphi=p_i y^{(1)i}$, then $d_0H=...=d_{k-1}H=0$ and
\begin{equation}
\oo=d_{k-2}\varphi=p_i dx^i
\end{equation}
$\oo$ is called the Liouville $1-$form. Its exterior differential is expressed by
\begin{equation}
\theta=d\oo=dp_i\wedge dx^i.
\end{equation}
$\theta$ is a $2-$form of rank $2n<(k+1)n=\dim T^{*k}M$, for $k>1$. Consequence $\theta$ is a presymplectic structure on $T^{*k}M$.

Let us consider the tensor field of type $(1,1)$ on $T^{*k}M$:
\begin{equation}
J=\dfrac{\pp}{\pp y^{(1)i}}\otimes dx^i+\dfrac{\pp}{\pp y^{(2)i}}\otimes dy^{(1)i}+...+\dfrac{\pp}{\pp y^{(k-1)i}}\otimes dy^{(k-2)i}.
\end{equation}

\begin{theorem}
We have:
\begin{enumerate}
  \item[$1^\circ$] $J$ is globally defined.
  \item[$2^\circ$] $J$ is integrable.
  \item[$3^\circ$] $J\circ J\circ...\circ J=J^k=0$.
  \item[$4^\circ$] $\ker J=V_{k-1}\oplus W_k$.
  \item[$5^\circ$] ${\rm{rank}}J=(k-1)n$.
  \item[$6^\circ$] $J(\overset{1}\Gamma)=0$, $J(\overset{2}\Gamma)=\overset{1}\Gamma$, ..., $J(\overset{k-1}\Gamma)=\overset{k-2}\Gamma$, $J(\C^*)=0$.
\end{enumerate}
\end{theorem}

\begin{theorem}
\begin{enumerate}
\item[$1^\circ$] For any vector field $X\in{\cal X}(T^{*k}M)$, $\overset{1}X,...,\overset{k-1}X$ given by
\begin{equation}
\overset{1}X=J(X),\ \overset{2}X=J^2 X,...,\overset{k-1}X=J^{k-1}X
\end{equation}
are vector fields.
\item[$2^\circ$] If $X=\overset{(0)i}X\dfrac{\pp}{\pp x^i}+\overset{(1)i}X\dfrac{\pp}{\pp y^{(1)i}}+...+\overset{(k-1)i}X\dfrac{\pp}{\pp y^{(k-1)i}}+X_i\dot{\pp}^i$. Then
\begin{equation}
\overset{1}X=J(X)=\overset{(0)i}X\dfrac{\pp}{\pp y^{(1)i}}+...+\overset{(k-2)i}X\dfrac{\pp}{\pp y^{(k-1)i}},...,\overset{k-1}X=\overset{0}X^i\dfrac{\pp}{\pp y^{(k-1)i}}.
\end{equation}
\end{enumerate}
\end{theorem}

On the manifold $T^{*k}M$ there exists a {\it Poisson structure} given by
\begin{equation}
\begin{array}{l}
\{f,g\}_0=\dfrac{\pp f}{\pp x^i}\dfrac{\pp g}{\pp p_i}-\dfrac{\pp f}{\pp p_i}\dfrac{\pp g}{\pp x^i}\\ \\
\{f,g\}_\a=\dfrac{\pp f}{\pp y^{(\a)i}}\dfrac{\pp g}{\pp p_i}-\dfrac{\pp f}{\pp p_i}\dfrac{\pp g}{\pp y^{(\a)i}}\ (\a=1,...,k-1).
\end{array}
\end{equation}
It is not difficult to study the properties of this structure.

For details see the book \cite{mir13}.

\newpage

\begin{partbacktext}
\part{Analytical Mechanics of Lagrangian and Hamiltonian Mechanical Systems}
This part is devoted to applications of the Lagrangian and Hamiltonian geometries of order $k=1$ and $k>1$ to Analytical Mechanics.

Firstly, to classical Mechanics of Riemannian mechanical systems $\Sigma_{\cal R}=(M,T,Fe)$ for which the external forces $Fe$ depend on the material point $x\in M$ and on the velocities $\dfrac{dx^i}{dt}$. So, in general, $\Sigma_{\cal R}$ are the nonconservative systems. Then we are obliged to take $Fe$ as a vertical field on the phase space $TM$ and apply the Lagrange geometry for study the geometrical theory of $\Sigma_{\cal R}$. More general, we introduce the notion on Finslerian mechanical system $\Sigma_F=(M,F,Fe)$, where $M$ is the configuration space, $TM$ is the velocity space, $F$ is the fundamental function of a given Finsler space $F^n=(M,F(x,y))$ and $Fe\(x,\dfrac{dx}{dt}\)$ are the external forces defined as vertical vector field on the velocity space $TM$. The fundamental equations of $\Sigma_F$ are the {\it Lagrange} equations $$\dfrac{d}{dt}\dfrac{\pp F^2}{\pp y^i}-\dfrac{\pp F^2}{\pp x^i}=F_i(x,y),\ \ y^i=\dfrac{dx^i}{dt},$$ $F^2$ being the energy of Finsler space $F^n$.

More general, consider a triple $\Sigma_L=(M,L,Fe)$, where $M$ is space of configurations, $L^n=(M,L)$ is a Lagrange space, and $Fe$ are the external forces. The fundamental equations ate the Lagrange equations, too.

The dual theory leads to the Cartan and Hamiltonian mechanical systems and is based on the Hamilton equations.

Finally, we remark the extension of such kind of analytical mechanics, to the higher order. Some applications will be done.

These considerations are based on the paper \cite{Miron} and on the papers of J. Klein \cite{Klein1}, M. Crampin \cite{Cr}, Manuel de Leon \cite{leonrodri}. Also, we use the papers of R. Miron, M. Anastasiei, I. Bucataru \cite{miranabuc}, and of R. Miron, H. Shimada, S. Sabau and M. Roman \cite{MHS}, \cite{MSSR}.
\end{partbacktext}

\setcounter{chapter}{0}
\chapter{Riemannian mechanical systems}
\setcounter{section}{0}
\section{Riemannian mechanical systems}\index{Analytical Mechanics of! Riemannian systems}
\setcounter{equation}{0}\setcounter{theorem}{0}\setcounter{proposition}{0}\setcounter{definition}{0}



Let $g_{ij}(x)$ be Riemannian tensor field on the configuration space $M$. So its kinetic energy is
\begin{equation}
T=\dfrac12 g_{ij}(x)y^i y^j,\ \ y^i=\dfrac{dx^i}{dt}=\dot{x}^i.
\end{equation}
Following J. Klein \cite{Klein1}, we can give:
\begin{definition}
A Riemannian Mechanical system (shortly RMS) is a triple $\Sigma_{\cal R}=(M,T,Fe)$, where
\begin{itemize}
  \item[$1^\circ$] $M$ is an $n-$dimensional, real, differentiable manifold (called configuration space).
  \item[$2^\circ$] $T=\dfrac12 g_{ij}(x)\dot{x}^i\dot{x}^j$ is the kinetic energy of an a priori given Riemannian space ${\cal R}^n=(M,g_{ij}(x))$.
  \item[$3^\circ$] $Fe(x,y)=F^i(x,y)\dfrac{\pp}{\pp y^i}$ is a vertical vector field on the velocity space $TM$ ($Fe$ are called external forces).
\end{itemize}
\end{definition}
Of course, $\Sigma_{\cal R}$ is a scleronomic mechanical system. The covariant components of $Fe$ are:
\begin{equation}
F_i(x,y)=g_{ij}(x)F^i(x,y).
\end{equation}

{\bf Examples:}

\begin{enumerate}
  \item RMS - for which $Fe(x,y)=a(x,y)\C$, $a\neq 0$. Thus $F^i=a(x,y)y^i$ and $\Sigma_{\cal R}$ is called a Liouville RMS.
  \item The RMS $\Sigma_{\cal R}$, where $Fe(x,y)=F^i(x)\frac{\partial}{\partial y^i}$, and $F_i(x)=grad_i f(x)$, called conservative systems.
  \item The RMS $\Sigma_{\cal R}$, where $Fe(x,y)=F^i(x)\frac{\partial}{\partial y^i}$, but $F_i(x)\neq grad_i f(x)$, called non-conservative systems.
\end{enumerate}

\begin{remark}
\begin{enumerate}
\item A conservative system $\Sigma_{\cal R}$ is called by J. Klein \cite{Klein1} {\it a Lagrangian system}.
\item One should pay attention to not make confusion of this kind of mechanical systems with the ``Lagrangian mechanical systems'' $\Sigma_L=(M,L(x,y),Fe(x,y))$ introduced by R. Miron \cite{Miron}, where $L:TM\to\R$ is a regular Lagrangian.
\end{enumerate}
\end{remark}

Starting from Definition 1.1.1, in a very similar manner as in the geometrical theory of mechanical systems, one introduces

\vspace{.3cm}

\noindent{\bf Postulate.} The evolution equations of a RSM $\Sigma_{\cal R}$ are the {\it Lagrange equations}:
\begin{equation}
\dfrac{d}{dt}\dfrac{\partial L}{\partial y^i}-\dfrac{\partial L}{\partial x^i}=F_i(x,y),\ \ \ \ y^i=\dfrac{d x^i}{dt},\ L=2T.
\end{equation}

This postulate will be geometrically justified by the existence of a semispray $S$ on $TM$ whose integral curves are given by the equations (1.1.3). Therefore, the integral curves of Lagrange equations will be called the {\it evolution curves} of the RSM $\Sigma_{\cal R}$.

The Lagrangian $L=2T$ has the fundamental tensor $g_{ij}(x)$.

\begin{remark}
In classical Analytical Mechanics, the coordinates $(x^i)$ of a material point $x\in M$ are denoted by $(q^i)$, and the velocities $y^i=\dfrac{d x^i}{dt}$ by $\dot q^i=\dfrac{d q^i}{dt}$. However, we prefer to use the notations $(x^i)$ and $(y^i)$ which are often used in the geometry of the tangent manifold $TM$.
\end{remark}

The external forces $Fe(x,y)$ give rise to the one-form
\begin{equation}
\sigma=F_i(x,y)dx^i.
\end{equation}

Since $Fe$ is a vertical vector field it follows that $\sigma$ is semibasic one form. Conversely, if $\sigma$ from (1.1.4) is semibasic one form, then $Fe=F^i(x,y)\frac{\partial}{\partial y^i}$, with $F^i=g^{ij}F_j$, is a vertical vector field on the manifold $TM$. J Klein introduced the the external forces by means of a one-form $\sigma$, while R. Miron \cite{Miron} defined $Fe$ as a vertical vector field on $TM$.

The RMS $\Sigma_{\cal R}$ is a  regular mechanical system because the Hessian matrix with elements $\dfrac{\partial^2 T}{\partial y^i \partial y^j}=g_{ij}(x)$ is nonsingular.

We have the following important result.
\begin{proposition}
The system of evolution equations $(1.1.3)$ are equivalent to the following second order differential equations:
\begin{equation}
\dfrac{d^2 x^i}{dt^2}+\gamma^i{}_{jk}(x)\dfrac{dx^j}{dt}\dfrac{dx^k}{dt}=\dfrac12 F^i(x,\dfrac{dx}{dt}),
\end{equation}
where $\gamma_{ij}^k(x)$ are the Christoffel symbols of the metric tensor $g_{ij}(x)$.
\end{proposition}

In general, for a RMS $\Sigma_{\cal R}$, the system of differential equations (1.1.5) is not autoadjoint, consequently, it can not be written as the Euler-Lagrange equations for a certain Lagrangian.

In the case of conservative RMS, with $\dfrac12 F_i(x)=-\dfrac{\partial U(x)}{\partial x^i}$, here $U(x)$ a potential function, the equations (1.1.2) can be written as Euler-Lagrange equations for the Lagrangian $T+U$. They have $T+U=$constant as a prime integral.

This is the reason that the nonconservative RMS $\Sigma_{\cal R}$, with $Fe$ depending on $y^i=\dfrac{dx^i}{dt}$ cannot be studied by the methods of classical mechanics. A good geometrical theory of the RMS $\Sigma_{\cal R}$ should be based on the geometry of the velocity space $TM$.

From (1.1.4) we can see that in the canonical parametrization $t=s$ ($s$ being the arc length in the Riemannian space ${\cal R}^n$), we obtain the following result:
\begin{proposition}
If the external forces are identically zero, then the evolution curves of the system $\Sigma_{\cal R}$ are the geodesics of the Riemannian space ${\cal R}^n$.
\end{proposition}

In the following we will study how the evolution equations change when the space ${\cal R}^n=(M,g)$ is replaced by another Riemannian space $\bar{\cal R}^n=(M,\bar g)$ such that:

$1^{\circ}$ ${\cal R}^n$ and $\bar{\cal R}^n$ have the same parallelism of directions;

$2^{\circ}$ ${\cal R}^n$ and $\bar{\cal R}^n$have same geodesics;

$3^{\circ}$ $\bar{\cal R}^n$ is conformal to ${\cal R}^n$.

\noindent In each of these cases, the Levi-Civita connections of these two Riemmanian spaces are transformed by the rule:

$1^{\circ}$ $\bar\gamma^i{}_{jk}(x)=\gamma^i{}_{jk}(x)+\delta^i_j\alpha_k(x)$;

\bigskip

$2^{\circ}$ $\bar\gamma^i{}_{jk}(x)=\gamma^i{}_{jk}(x)+\delta^i_j\alpha_k(x)+\delta^i_k\alpha_j(x)$;

\bigskip

$3^{\circ}$ $\bar\gamma^i{}_{jk}(x)=\gamma^i{}_{jk}(x)+\delta^i_j\alpha_k(x)+\delta^i_k\alpha_j(x)-g_{jk}(x)\alpha^i(x)$,

\noindent where $\alpha_k(x)$ is an arbitrary covector field on $M$, and $\alpha^i(x)=g^{ij}(x)\alpha_j(x)$.

It follows that the evolution equations (1.1.3) change to the evolution equations of the system $\Sigma_{\bar{\cal R}}$ as follows:

$1^{\circ}$ In the first case we obtain
\begin{equation}
\dfrac{d^2 x^i}{dt^2}+\gamma^i{}_{jk}(x)\dfrac{dx^j}{dt}\dfrac{dx^k}{dt}= -\alpha\dfrac{dx^i}{dt}+\dfrac12 F^i(x,\dfrac{dx}{dt}); \alpha = \alpha_k(x)\dfrac{dx^k}{dt}
\end{equation}
Therefore, even though $\Sigma_{{\cal R}}$ is a conservative system, the mechanical system $\Sigma_{\bar{\cal R}}$ is nonconservative system having the external forces $$\bar Fe=\left(-\alpha(x,y)y^i+\dfrac12 F^i(x,y)\right)\dfrac{\partial}{\partial y^i},\quad \alpha = \alpha_k(x)\dfrac{dx^k}{dt}.$$

$2^{\circ}$ In the second case we have
\begin{equation}
\dfrac{d^2 x^i}{dt^2}+\gamma^i{}_{jk}(x)\dfrac{dx^j}{dt}\dfrac{dx^k}{dt}=-2\alpha\dfrac{dx^i}{dt}+\dfrac12 F^i\(x,\dfrac{dx}{dt}\); \quad \alpha=\alpha_k(x)\dfrac{dx^k}{dt}
\end{equation}
and
$$\bar Fe=\left(-2\alpha(x,y)y^i+\dfrac12 F^i(x,y)\right)\dfrac{\partial}{\partial y^i}, \quad \alpha = \alpha_k(x)\dfrac{dx^k}{dt}.$$

$3^{\circ}$ In the third case $\bar Fe$ is $$\bar Fe=\left\{2(-\alpha y^i+T\alpha^i)+\dfrac12 F^i\right\} \dfrac{\partial}{\partial y^i}, \quad \alpha = \alpha_k(x)\dfrac{dx^k}{dt},$$ $$\hfill \alpha^i(x)=g^{ij}(x)\alpha_j(x).$$

The previous properties lead to examples with very interesting properties.


\section{Examples of Riemannian mechanical systems}
\setcounter{equation}{0}\setcounter{theorem}{0}\setcounter{proposition}{0}\setcounter{definition}{0}

Recall that in the case of classical {\it conservative mechanical systems} we have
$2Fe=grad\ \cal U,$ where ${\cal U}(x)$ is a potential function. Therefore, the Lagrange equations are given by
$$\dfrac{d}{dt}\dfrac{\partial}{\partial y^i}(T+{\cal U})-\dfrac{\partial}{\partial x^i}(T+{\cal U})=0.$$

We obtain from here a prime integral $T+{\cal U}=h$ (constant) which give us {\it the energy conservation law}.

In the nonconservative case we have numerous examples suggested by $1^\circ$, $2^\circ$, $3^\circ$ from the previous section, where we take $F^i(x,y)=0$.

Other examples of RMS can be obtained as follows

\begin{enumerate}
\item
\begin{equation}
2Fe=-\beta(x,y)y^i\dfrac{\partial}{\partial y^i},
\end{equation}
where $\beta =\beta_i(x)y^i$ is determined by the electromagnetic potentials $\beta_i(x)$, $(i=1,..,n)$.

\item
\begin{equation}
Fe=(T-\beta)y^i\dfrac{\partial}{\partial y^i},
\end{equation}
where $\beta =\beta_i(x)y^i$ and $T$ is the kinetic energy.

\item In the three-body problem, M. B\u arbosu \cite{barb} applied the following conformal transformation:
$$d\bar s^2 = (T+{\cal U})ds^2$$ to the classic Lagrange equations and had obtained a nonconservative mechanical system with external force field
$$Fe=(T+{\cal }U)y^i\dfrac{\partial}{\partial y^i}.$$

\item The external forces $Fe=F^i(x)\dfrac{\partial}{\partial y^i}$ lead to
 classical nonconservative Riemannian mechanical systems. For instance, for $F_i=-grad_i\ {\cal U}+R_i(x)$ where $R_i(x)$ are the resistance forces, and the configuration space $M$ is $\R^3$.

\item If $M=\R^3,$ $T=\dfrac{1}{2}m\delta_{ij}y^iy^j$ and $Fe=2F^i(x)\dfrac{\partial}{\partial y^i}$, then the evolution equations are
$$m\dfrac{d^2x^i}{dt^2}=F^i(x),$$ which is {\it the Newton's law}.

\item The harmonic oscillator.

$M=\R^n,\ \ g_{ij}=\delta_{ij},\ \ 2F_i=-\omega_i^2x^i$ (the summation convention is not applied) and $\omega_i$ are positive numbers, $ (i=1,..,n)$.

The functions
$$h_i=(x^i)^2+\omega_i^2x^i,\ \ {\rm and}\ \ H=\sum_{i=1}^n h_i$$ are prime integrals.

\item Suggested by the example $6^{\circ}$, we consider a system $\Sigma_{\cal R}$ with $Fe=-2\omega(x)\C$, where $\omega(x)$ is a positive function and $\C$ is the Liouville vector field.

The evolution equations, in the case $M=\R^n,$ are given by
$$\dfrac{d^2 x^i}{dt^2}+\omega(x)\dfrac{dx^i}{dt}=0.$$
Putting $y^i=\dfrac{dx^i}{dt},$ we can write $$\dfrac{dy^i}{dt}+\omega(x)y^i=0,\ \ \ \ \ \ (i=1,..,n).$$
So, we obtain $y^i=C^ie^{-\int\omega(x(t))dt}$ and therefore $x^i=C_0^i+$\break $C^i\displaystyle\int e^{-\int\omega(x(t))dt}dt$.

\item We can consider the systems $\Sigma_{\cal R}$ having $$Fe=2a^i_{jk}(x)y^jy^k\dfrac{\partial}{\partial y^i},$$ where $a^i_{jk}(x)$ is a symmetric tensor field on $M$. The external force field $Fe$ has homogeneous components of degree 2 with respect to $y^i$.

\item Relativistic nonconservative mechanical systems can be obtained for a Minkowski metric in the space-time $\R^4.$

\item A particular case of example $1^{\circ}$ above \cite{shimada} is the case when the external force field coefficients $F^i(x,y)$ are linear in $y^i$, i.e.
$$Fe=2F^i(x,y)\frac{\partial}{\partial y^i}=2Y^i_k(x)y^k\frac{\partial}{\partial y^i},$$ where $Y:TM\to TM$ is a fiber diffeomorphism called {\it Lorentz force}, namely for any $x\in M$, we have
$$Y_x:T_xM\to T_xM,\qquad Y_x\(\frac{\partial}{\partial x^i}\)=Y_i^j(x)\frac{\partial}{\partial x^j}.$$

Let us remark that in this case, formally, we can write the Lagrange equations of this RMS in the form
$$\nabla_{\dot\gamma}{\dot\gamma}=Y(\dot\gamma),$$ where $\nabla$ is the Levi-Civita connection of the Riemannian space $(M,g)$ and $\dot\gamma$ is the tangent vector along the evolution curves $\gamma:[a,b]\to M$.

This type of RMS is important because of the global behavior of its evolution curves.

Let us denote by $S$ the evolutionary semispray, i.e. $S$ is a vector field on $TM$ which is tangent to the canonical lift $\hat \gamma=(\gamma,\dot \gamma)$ of the evolution curves (see the following section for a detailed discussion on the evolution semispray).

We will denote by $T^a$ the energy levels of the Riemannian metric $g$, i.e.
$$T^a=\left\{(x,y)\in TM:T(x,y)=\frac{a^2}{2}\right\},$$
where $T$ is the kinetic energy of $g$, and $a$ is a positive constant. One can easily see that $T^a$ is the hypersurface in $TM$ of constant Riemannian length vectors, namely for any $X=(x,y)\in T^a$, we must have $|X|_g=a$, where $|X|_g$ is the Riemannian length of the vector field $X$ on $M$.

If we restrict ourselves for a moment to the two dimensional case, then it is known that for sufficiently small values of $c$ the restriction of the flow of the semispray $S$ to $T^c$ contains no less than two closed curves when $M$ is the 2-dimensional sphere, and at least three otherwise. These curves projected to the base manifold $M$ will give closed evolution curves for the given Riemannian mechanical system.
\end{enumerate}


\section{The evolution semispray of the mechanical system $\Sigma_{\cal R}$}\index{Semisprays! canonical mechanical system $\Sigma_{\cal R}$}
\setcounter{equation}{0}\setcounter{theorem}{0}\setcounter{proposition}{0}\setcounter{definition}{0}

Let us assume that $Fe$ is global defined on $M$, and consider the mechanical system $\Sigma_{\cal R}=(M,T,Fe)$. We have

\begin{theorem} [\cite{Miron}]
The following properties hold good:

$1^{\circ}$  The quantity $S$ defined by
\begin{equation}
\left\{
\begin{array}{l}
S=y^i\dfrac{\partial}{\partial x^i}-(2\stackrel{\circ}{G^i}-\dfrac12 F^i)\dfrac{\partial}{\partial y^i},\\
\\
2\stackrel{\circ}{G^i}=\gamma^i{}_{jk}y^jy^k,
\end{array} \right.
\end{equation}
is a vector field on the velocity space $TM$.

$2^{\circ}$  $S$ is a semispray, which depends on $\Sigma_{\cal R}$ only.

$3^{\circ}$  The integral curves of the semispray $S$ are the evolution curves of the system $\Sigma_{\cal R}$.
\end{theorem}

\begin{proof}
$1^{\circ}$ Writing $S$ in the form
\begin{equation}
S=\stackrel{\circ}{S}+\dfrac12 Fe,
\end{equation}
where $\stackrel{\circ}{S}$ is the canonical semispray with the coefficients $\stackrel{\circ}{G^i}$,
we can see immediately that $S$ is a vector field on $TM$.

$2^{\circ}$ Since $\stackrel{\circ}{S}$ is a semispray and $Fe$ a vertical vector field, it follows $S$ is a semispray. From (1.3.1) we can see that $S$ depends on $\Sigma_{\cal R}$, only.

$3^{\circ}$ The integral curves of $S$ are given by
\begin{equation}
\dfrac{dx^i}{dt}=y^i;\ \ \ \  \dfrac{dy^i}{dt}+2\stackrel{\circ}{G^i}(x,y)=\dfrac12 F^i(x,y).
\end{equation}
Replacing $y^i$ in the second equation we obtain (1.1.2)
\end{proof}

$S$ will be called {\it the evolution} or {\it canonical semispray} of the nonconservative Riemannian mechanical system $\Sigma_{\cal R}$. In the terminology of J. Klein \cite{Klein1}, $S$ is the dynamical system of $\Sigma_{\cal R}$.

Based on $S$ we can develop the geometry of the mechanical system $\Sigma_{\cal R}$ on $TM$.

Let us remark that $S$ can also be written as follows:
\begin{equation}
S=y^i\dfrac{\partial}{\partial x^i}-2{G^i}(x,y)\dfrac{\partial}{\partial y^i},\tag{1.3.1'}
\end{equation}
with the coefficients
\begin{equation}
2G^i=2\stackrel{\circ}{G^i}-\dfrac12 F^i.
\end{equation}

We point out that $S$ is homogeneous of degree 2 if and only if $F^i(x,y)$ is 2-homogeneous with respect to $y^i.$ This property is not satisfied in the case $\dfrac{\partial F^i}{\partial y^j}\equiv 0$, and it is satisfied for examples $1$ and $7$ from section 1.2.

We have
\begin{theorem}
The variation of the kinetic energy $T$ of a mechanical system $\Sigma_{\cal R},$ along the evolution curves {\rm (1.1.2)}, is given by:
\begin{equation}
\dfrac{dT}{dt}=\dfrac12 F_i \dfrac{dx^i}{dt}.
\end{equation}
\end{theorem}

\begin{proof}
A straightforward computation gives
$$\begin{array}{l}
\dfrac{dT}{dt}=\dfrac{\partial T}{\partial x^i}\dfrac{dx^i}{dt}+
\dfrac{\partial T}{\partial y^i}\dfrac{dy^i}{dt}=
\left(\dfrac{d}{dt}\dfrac{\partial T}{\partial y^i}-\dfrac12 F_i\right)\dfrac{dx^i}{dt}+
\dfrac{\partial T}{\partial y^i}\dfrac{dy^i}{dt}=\\
\\
=\dfrac{d}{dt}\left(y^i\dfrac{\partial T}{\partial y^i}\right)-\dfrac12 F_i\dfrac{dx^i}{dt}=2\dfrac{dT}{dt}-\dfrac12 F_i\dfrac{dx^i}{dt},
\end{array}
$$
and the relation (1.3.5) holds good.
\end{proof}

\begin{corollary}
$T=$constant along the evolution curves if and only if the Liouville vector $\C$ and the external force $Fe$ are orthogonal vectors along the evolution curves of $\Sigma$.
\end{corollary}

\begin{corollary}
If $F_i=grad_i\ {\cal U}$ then $\Sigma_{\cal R}$ is conservative and $T+{\cal U}=h$ (constant) on the evolution curves of $\Sigma_{\cal R}$.
\end{corollary}

If the external forces $Fe$ are dissipative, i.e. $\langle\C,Fe\rangle<0$, then from the previous theorem, it follows a result of Bucataru-Miron (see \cite{BuMi}):

\begin{corollary}
The kinetic energy $T$ decreases along the evolution curves if and only if the external forces $Fe$ are dissipative.
\end{corollary}

Since the energy of $\Sigma_{\cal R}$ is $T$ (the kinetic energy), the Theorem 1.3.2 holds good in this case. The variation of $T$ is given by (1.3.5) and hence we obtain: $T$ is conserved along the evolution curves of $\Sigma_{\cal R}$ if and only if the vector field $Fe$ and the Liouville vector field $\C$ are orthogonal.


\section{The nonlinear connection of $\Sigma_{\cal R}$}\index{Connections! Nonlinear}
\setcounter{equation}{0}\setcounter{theorem}{0}\setcounter{proposition}{0}\setcounter{definition}{0}

Let us consider the evolution semispray $S$ of $\Sigma_{\cal R}$ given by
\begin{equation}
S=y^i\dfrac{\partial}{\partial x^i}-(2\stackrel{\circ}G^i-\dfrac12 F^i)\dfrac{\partial}{\partial y^i}
\end{equation}
with the coefficients
\begin{equation}
2G^i=2\stackrel{\circ}{G^i}-\dfrac12 F^i.
\end{equation}
Consequently, the evolution nonlinear connection $N$ of the mechanical system $\Sigma_{\cal R}$ has the coefficients:
\begin{equation}
N^i{}_j=\stackrel{\circ}{N^i}\!{}_j-\dfrac{1}{4}\dfrac{\partial F^i}{\partial y^j}=\gamma_{jk}^iy^k-\dfrac{1}{4}\dfrac{\partial F^i}{\partial y^j}.
\end{equation}

If the external forces $Fe$ does not depend by velocities $y^i=\dfrac{dx^i}{dt},$ then $N=\stackrel{\circ}{N}.$

Let us consider the {\it helicoidal vector field} (see Bucataru-Miron, \cite{BuMi}, \cite{BuMi1})
\begin{equation}
P_{ij}=\dfrac{1}{2} \(\dfrac{\partial F_i}{\partial y^j}-\dfrac{\partial F_j}{\partial y^i}\)
\end{equation}
and the symmetric part of tensor $\dfrac{\partial F_i}{\partial y^j}:$
\begin{equation}
Q_{ij}=\dfrac{1}{2} \(\dfrac{\partial F_i}{\partial y^j}+\dfrac{\partial F_j}{\partial y^i} \).
\end{equation}

On $TM,$ $P$ gives rise to the $2$-form:
$$P=P_{ij}\ dx^i\wedge dx^j\leqno(1.4.4')$$
and $Q$ is the symmetric vertical tensor:
$$Q=Q_{ij}\ dx^i\otimes dx^j.\leqno(1.4.5')$$
Denoting by $\nabla$ the dynamical derivative with respect to the pair $(S,N),$ one proves the theorem of Bucataru-Miron:

\begin{theorem}
For a Riemannian mechanical system $\Sigma_{\cal R}=(M,T,Fe)$ the evolution nonlinear connection is the unique nonlinear connection that satisfies the following conditions:
$$(1)\nabla g=-\dfrac{1}{2} Q$$
$$(2)\theta_L(hX,hY)=\dfrac{1}{2}P(X,Y),\ \ \forall X,Y\in \chi(TM),$$
where
\begin{equation}
\theta_L=4g_{ij}\stackrel{\circ}{\delta}y^j\wedge dx^i
\end{equation}
is the symplectic structure determined by the metric tensor $g_{ij}$ and the nonlinear connection
$\stackrel{\circ}{N}$.
\end{theorem}

The adapted basis of the distributions $N$ and $V$ is given by $\(\dfrac{\delta}{\delta x^i},\dfrac{\partial}{\partial y^i}\)$,
where
\begin{equation}
\dfrac{\delta }{\delta x^i}=\dfrac{\partial }{\partial x^i}-N^j{}_i\dfrac{\partial }{\partial y^j}=\dfrac{\stackrel{\circ}{\delta}}{\delta x^i}+\frac{1}{4}\dfrac{\partial F^j}{\partial y^i}\dfrac{\partial }{\partial y^j}
\end{equation}
and its dual basis $(dx^i,\delta y^i)$ has the 1-forms $\delta y^i$
expressed by
\begin{equation}
\delta y^i=dy^i+N^i{}_j dx^j=\stackrel{\circ}{\delta}\!\!y^i-\frac{1}{4}\dfrac{\partial F^i}{\partial y^j}dx^j.
\end{equation}
It follows that the curvature tensor ${\cal R}^i{}_{jk}$ of $N$ (from (1.4.3) is
\begin{equation}
{\cal R}^k{}_{ij}=\dfrac{\delta N^k\!{}_i}{\delta x^j}-\dfrac{\delta N^k\!{}_j}{\delta x^i}=\(\dfrac{\delta}{\delta x^j}\dfrac{\partial }{\partial y^i}-\dfrac{\delta}{\delta x^i}\dfrac{\partial }{\partial y^j}\)(\stackrel{\circ}{G^k}-\frac{1}{4}F^k)
\end{equation}
and the torsion tensor of $N$ is:
\begin{equation}
t^k{}_{ij}=\dfrac{\partial N^k\!{}_i}{\partial y^j}-\dfrac{\partial N^k\!{}_j}{\partial y^i}=\(\dfrac{\partial}{\partial y^j}\dfrac{\partial }{\partial y^i}-\dfrac{\partial}{\partial y^i}\dfrac{\partial }{\partial y^j}\)(\stackrel{\circ}{G^k}-\frac{1}{4}F^k)=0.
\end{equation}

These formulas have the following consequences.
\begin{enumerate}
\item The evolution nonlinear connection $N$ of $\Sigma_{\cal R}$ is integrable if and only if  the curvature tensor ${\cal R}^i{}_{jk}$ vanishes.
\item The nonlinear connection is torsion free, i.e. $t_{ij}^k=0$.
\end{enumerate}

The autoparallel curves of the evolution nonlinear connection $N$ are given by the system of differential equations
$$\dfrac{d^2x^i}{dt^2}+N^i\!{}_j\(x,\frac{dx}{dt}\)\dfrac{dx^j}{dt}=0,$$
which is equivalent to
\begin{equation}
\dfrac{d^2x^i}{dt^2}+\stackrel{\circ}{N^i}\!\!{}_j\(x,\frac{dx}{dt}\)\dfrac{dx^j}{dt}= \frac{1}{4}\dfrac{\partial F^i}{\partial y^j}\dfrac{dx^j}{dt}.
\end{equation}

In the initial conditions $\(x_0,\(\dfrac{dx}{dt}\)_0\)$, locally one uniquely determines the autoparallel curves of $N$.

If $F^i$ is $2$-homogeneous with respect to $y^i,$ then the previous system coincides with the Lagrange equations (1.1.4).

Therefore, we have:

\begin{theorem}
If the external forces $Fe$ are $2$-homogeneous with respect velocities $y^i=\dfrac{dx^i}{dt},$ then the evolution curves of $\Sigma_{\cal R}$ coincide to the autoparallel curves of the evolution nonlinear connection $N$ of $\Sigma_{\cal R}$.
\end{theorem}

In order to proceed further, we need the exterior differential of $1$-forms $\delta y^i$.

One obtains
\begin{equation}
d(\delta y^i)=dN^i\!{}_j\wedge dx^j=\dfrac{1}{2}R^i\!{}_{kj} dx^j\wedge dx^k+B^i\!{}_{kj}
\delta y^j\wedge d x^k,
\end{equation}
where
\begin{equation}
B^i\!{}_{kj}=B^i\!{}_{jk}=\dfrac{\partial^2 G^i}{\partial y^k \partial y^j}
\end{equation}
are the coefficients of the Berwald connection determined by the nonlinear connection $N$.


\section{The canonical metrical connection $C\Gamma(N)$}\index{Connections! $N-$linear}
\setcounter{equation}{0}\setcounter{theorem}{0}\setcounter{proposition}{0}\setcounter{definition}{0}

The coefficients of the canonical metrical connection $C\Gamma(N)=\left(F^i\!{}_{jk}\right.$, $\left.C^i\!{}_{jk}\right)$ are given by the generalized Christoffel symbols \cite{miranabuc}:
\begin{equation}
\left\{
\begin{array}{l}
F^i\!{}_{jk}=\dfrac{1}{2}g^{is}\(\dfrac{\delta g_{sk}}{\delta x^j}+
\dfrac{\delta g_{js}}{\delta x^k}-\dfrac{\delta g_{jk}}{\delta x^s}
\),\\ \\
C^i\!{}_{jk}=\dfrac{1}{2}g^{is}\(\dfrac{\partial g_{sk}}{\partial y^j}+
\dfrac{\partial g_{js}}{\partial y^k}-\dfrac{\partial g_{jk}}{\partial y^s}\)
\end{array}
\right.
\end{equation}
where $g_{ij}(x)$ is the metric tensor of $\Sigma_{\cal R}$.

On the other hand, we have $\dfrac{\delta g_{jk}}{\delta x^i}=\dfrac{\partial g_{jk}}{\partial x^i}$ and $\dfrac{\partial g_{jk}}{\partial y^i}=0$, and therefore we obtain:

\begin{theorem}
The canonical metrical connection $C\Gamma(N)$ of the mechanical system $\Sigma_{\cal R}$ has the coefficients
\begin{equation}
F^i\!{}_{jk}(x,y)=\gamma^i\!{}_{jk}(x),\ \ \ \ \ \ C^i\!{}_{jk}(x,y)=0.
\end{equation}
\end{theorem}

Let $\omega^i\!{}_{j}$ be the connection forms of $C\Gamma(N)$:
\begin{equation}
\omega^i\!{}_{j}=F^i\!{}_{jk}dx^k+C^i\!{}_{jk}\delta y^k=\gamma^i\!{}_{jk}(x)dx^k.
\end{equation}

Then, we have (\cite{miranabuc}):
\begin{theorem}
The structure equation of $C\Gamma(N)$ can be expressed by
\begin{equation}
\left\{
\begin{array}{ll}
d(dx^i)-dx^k\wedge \omega^i\!{}_k & =-\stackrel{1}{\Omega}\!{}^i,\\
d(\delta y^i)-\delta y^k\wedge \omega^i\!{}_k & =-\stackrel{2}{\Omega}\!{}^i,\\
d \omega^i\!{}_j\ \  -\ \omega^k\!{}_j\wedge\omega^i\!{}_k & =-\stackrel{}{\Omega}\!{}_j^i,
\end{array}
\right.
\end{equation}
where the $2$-forms of torsion $\stackrel{1}{\Omega}\!{}^i,\ \stackrel{2}{\Omega}\!{}^i$ are as follows
\begin{equation}
\begin{array}{ll}
\stackrel{1}{\Omega}\!{}^i=C^i\!{}_{jk}dx^j\wedge \delta y^k=0,\\
\\
\stackrel{2}{\Omega}\!{}^i= R^i\!{}_{jk}dx^j\wedge dx^k+P^i\!{}_{jk}dx^j\wedge \delta y^k.
\end{array}
\end{equation}
Here $R^i\!{}_{jk}$ is the curvature tensor of $N$ and $P^i\!{}_{jk}=\gamma^i\!{}_{jk}-\gamma^i\!{}_{kj}=0.$
\end{theorem}

The curvature $2$-form $\Omega^i_j$ is given by
\begin{equation}
\Omega^i_j=\dfrac{1}{2}R_j^{\ i}\!{}_{kh}\ dx^k\wedge dx^h+P_j^{\ i}\!{}_{kh}\ dx^k\wedge \delta y^h+\dfrac{1}{2}S_j^{\ i}\!{}_{kh}\ \delta y^k\wedge \delta y^h,
\end{equation}
where
\begin{equation}
\begin{array}{ll}
R_h^{\ i}{}_{jk}&=\dfrac{\delta F^i\!{}_{hj}}{\delta x^k}-\dfrac{\delta F^i\!{}_{hk}}{\delta x^j}+F^s\!{}_{hj}F^i\!{}_{sk}-F^s\!{}_{hk}F^i\!{}_{sj}+
C^i\!{}_{hs}R^s\!{}_{jk}=\\
\\
&=\dfrac{\partial \gamma^i\!{}_{hj}}{\partial x^k}-\dfrac{\partial \gamma^i\!{}_{hk}}{\partial x^j}+\gamma^s\!{}_{hj}\gamma^i\!{}_{sk}-\gamma^s\!{}_{hk}\gamma^i\!{}_{sj}=
{\mathbf r}_h^{\ i}{}_{jk}
\end{array}
\end{equation}
is the Riemannian tensor of curvature of the Levi-Civita connection $\gamma_{jk}^i(x)$ and the curvature tensors $P_j^{\ i}{}_{kh},S_j^{\ i}{}_{kh}$ vanish.

Therefore, the tensors of torsion of $C\Gamma(N)$ are
\begin{equation}
R^{ i}{}_{jk},\ T^{ i}{}_{jk}=0, \ S^{ i}{}_{jk}=0,\ P^{ i}{}_{jk}=0,\ C^{ i}{}_{jk}=0
\end{equation}
and the curvature tensors of $C\Gamma(N)$ are
\begin{equation}
R_j^{\ i}{}_{kh}(x,y)={\bf r}_j^{\ i}{}_{kh}(x),\ P_j^{\ i}{}_{kh}(x,y)=0,\ S_j^{\ i}{}_{kh}(x,y)=0.
\end{equation}

The Bianchi identities can be obtained directly from (1.5.4), taking into account the conditions (1.5.8) and (1.5.9).

The $h$- and $v$-covariant derivatives of $d$-tensor fields with respect to $C\Gamma(N)=(\gamma^i\!{}_{jk},0)$ are expressed, for instance, by
$$\begin{array}{l}
\nabla_k t_{ij}=\dfrac{\delta t_{ij}}{\delta x^k}-\gamma^s\!{}_{ik}t_{sj}-\gamma^s\!{}_{jk}t_{is},\\
\\
\dot\nabla_k t_{ij}=\dfrac{\partial t_{ij}}{\partial y^k}-C^s\!{}_{ik}t_{sj}-C^s\!{}_{jk}t_{is}=\dfrac{\partial t_{ij}}{\partial y^k}.
\end{array}$$
Therefore, $C\Gamma(N)$ being a metric connection with respect to $g_{ij}(x),$ we have
\begin{equation}
\nabla_k\ g_{ij}=\overset{\circ}{\nabla}_k g_{ij}=0,
\end{equation}
$(\stackrel{\circ}{\nabla}$ is the covariant derivative with respect to Levi-Civita connection of $g_{ij})$ and
\begin{equation}
\dot{\nabla}{}_k\ g_{ij}=0.\tag{1.5.10'}
\end{equation}

The deflection tensors of $C\Gamma(N)$ are
$$D^i\!{}_j=\nabla_jy^i=\dfrac{\delta y^i}{\delta x^j}+y^s\gamma^i\!{}_{sj}=-N^i\!{}_j+y^s\gamma^i\!{}_{sj}$$
and $$d^i\! {}_j=\dot \nabla_j\ y^i=\delta^i\!{}_j.$$

The evolution nonlinear connection of a Riemannian mechanical system $\Sigma_{\cal R}$ given by (1.4.3) implies
$$D^i\!{}_j=-\stackrel{\circ}{N}_j^i+\frac{1}{4}\frac{\partial F^i}{\partial y^j}+y^s\gamma_{sj}^i=\frac{1}{4}\frac{\partial F^i}{\partial y^j},$$
where we have used $\stackrel{\circ}{D}_jy^i=\stackrel{\circ}{D}_j^i=0$. It follows

\begin{proposition}
For a Riemannian mechanical system the deflection tensors $D^i\!{}_j$ and $d^i\! {}_j$ of the connection $C\Gamma(N)$ are expressed by
\begin{equation}
D^i\! {}_j=\dfrac{1}{4}\dfrac{\partial F^i}{\partial y^j},\quad d^i\! {}_j=\delta^i_j.
\end{equation}
\end{proposition}


\section{The electromagnetism in the theory of the Riemannian mechanical systems $\Sigma_{\cal R}$}\index{Electromagnetism and gravitational fields}
\setcounter{equation}{0}\setcounter{theorem}{0}\setcounter{proposition}{0}\setcounter{definition}{0}

In a Riemannian mechanical system $\Sigma_{\cal R}=\(M,T,Fe\)$ whose external forces $Fe$ depend on the point $x$ and on the velocity $y^i=\dfrac{dx^i}{dt},$ the electromagnetic phenomena appears because the deflection tensors $D^i\! {}_j$ and $d_j^i$ from (1.5.11) nonvanish. Hence the $d$-tensors $D_{ij}=g_{ih}D^h\! {}_j,\ d_{ij}=g_{ih}\delta^h_j=g_{ij}$ determine the $h$-electromagnetic tensor ${\cal F}_{ij}$ and $v$-electromagnetic tensor $f_{ij}$ by the formulas, \cite{miranabuc}:
\begin{equation}
{\cal F}_{ij}=\dfrac{1}{2}(D_{ij}-D_{ji}),\ \ \ \ f_{ij}=\dfrac{1}{2}(d_{ij}-d_{ji}).
\end{equation}

By means of (1.5.11), we have
\begin{proposition}
The $h$- and $v$-tensor fields ${\cal F}_{ij}$ and $f_{ij}$ are given by
\begin{equation}
{\cal F}_{ij}=\dfrac{1}{4} P_{ij},\ \ \ \ f_{ij}=0.
\end{equation}
where $P_{ij}$ is the helicoidal tensor $(1.4.4)$ of $\Sigma_{\cal R}$.
\end{proposition}

Indeed, we have
$${\cal F}_{ij}=\dfrac{1}{4}(D_{ij}-D_{ji})=\dfrac{1}{8}\(\dfrac{\partial F^i}{\partial y^j}-\dfrac{\partial F^j}{\partial y^i}\)=\dfrac{1}{4}P_{ij}.$$

If we denote $R_{ijk}:=g_{ih}R^h\!{}_{jk}$, then we can prove:

\begin{theorem}
The electromagnetic tensor ${\cal F}_{ij}$ of the mechanical system $\Sigma_{\cal R}=\(M,T,Fe\)$ satisfies the following generalized Maxwell equations:
\begin{equation}
\nabla_k {\cal F}_{ji}+\nabla_i {\cal F}_{kj}+\nabla_j {\cal F}_{ik}=-(R_{kji}+R_{ikj}+R_{jik}).
\end{equation}
\begin{equation}
\dot\nabla_k {\cal F}_{ji}=0.\tag{1.6.3'}
\end{equation}
\end{theorem}

\begin{proof}
Applying the Ricci identities to the Liouville vector field, we obtain
$$\nabla_i D^k\! {}_j-\nabla_j D^k\! {}_i=y^h{\bf r}^{\ k}_{h\ ij}-R^k_{ij},\ \nabla_i d^k\! {}_j-\dot \nabla_j D^k\! {}_i=0$$
and this leads to
\begin{equation}
\nabla_i D_{kj}-\nabla_j D_{ki}=y^h{\bf r}_{hkij}-R_{kij},
\end{equation}
\begin{equation}
\dot \nabla_jD_{ki}=0.\tag{1.6.4'}
\end{equation}
By taking cyclic permutations of the indices i,k,j and adding in (1.6.4), by taking
into account the identity ${\bf r}_{hijk}+{\bf r}_{hjki}+{\bf r}_{hkij}=0$ we deduce (1.6.3), and analogously (1.6.3').
\end{proof}

From equations (1.6.2) and (1.6.4') we obtain as consequences:

\begin{corollary}
The electromagnetic tensor ${\cal F}_{ij}$ of the mechanical system $\Sigma_{\cal R}=\(M,T,Fe\)$ does not depend on the velocities $y^i=\dfrac{dx^i}{dt}.$
\end{corollary}

Indeed, by means of (1.6.3') we have $\dot\nabla_j {\cal F}_{ik}=\dfrac{\partial{\cal F}_{ik}}{\partial y^j}=0$.

In other words, the helicoidal tensor $P_{ij}$ of $\Sigma_{\cal R}$ does not depend on the velocities $y^i=\dfrac{dx^i}{dt}.$

We end the present section with a remark: this theory has applications to the mechanical systems given by example $1^{\circ}$ in Section 1.2.


\section{The almost Hermitian model of the RMS $\Sigma_{\cal R}$}\index{Almost Hermitian model}
\setcounter{equation}{0}\setcounter{theorem}{0}\setcounter{proposition}{0}\setcounter{definition}{0}

Let us consider a RMS $\Sigma_{\cal R}=(M,T(x,y),Fe(x,y))$ endowed with the evolution nonlinear connection with coefficients $N_j^i$ from (1.4.3) and with the canonical $N$-metrical connection $C\Gamma(N)=(\gamma_{jk}^i(x),0)$. Thus, on the velocity space $\widetilde{TM}=TM\setminus\{0\}$ we can determine an almost Hermitian structure $H^{2n}=(\widetilde{TM}, \G, \F)$ which depends on the RMS only.

Let $\(\dfrac{\delta}{\delta x^i},\dfrac{\partial}{\partial y^i}\)$ be the adapted basis to the distributions $N$ and $V$ and its adapted cobasis $(dx^i,\delta y^i)$, where
\begin{equation}
\begin{array}{l}
\dfrac{\delta}{\delta x^i}=\stackrel{\circ}{\dfrac{\delta }{\delta x^i}} +\frac{1}{4}\frac{\partial F^s}{\partial y^i}\frac{\partial}{\partial y^s}\\
\\
\delta y^i= \stackrel{\circ}{\delta} y^i-\dfrac{1}{4}\frac{\partial F^i}{\partial y^s}dx^s
\end{array}
\end{equation}

The lift of the fundamental tensor $g_{ij}(x)$ of the Riemannian space ${\cal R}^n=(M,g_{ij}(x))$ is defined by
\begin{equation}
\G=g_{ij}dx^i\otimes dx^j+g_{ij}\delta y^i\otimes \delta y^j,
\end{equation}
and the almost complex structure $\F$, determined by the nonlinear connection $N$, is expressed by
\begin{equation}
\F=-\dfrac{\partial}{\partial y^i}\otimes dx^i+\dfrac{\delta}{\delta x^i}\otimes {\delta y^i}.
\end{equation}

Thus, the following theorems hold good.

\begin{theorem}
We have:
\begin{enumerate}
\item The pair $(\widetilde{TM},\G)$ is a pseudo-Riemannian space.
\item The tensor $\G$ depends on $\Sigma_{\cal R}$ only.
\item The distributions $N$ and $V$ are orthogonal with respect to $\G$.
\end{enumerate}
\end{theorem}

\begin{theorem}
\begin{enumerate}
\item The pair $(\widetilde{TM},\F)$ is an almost complex space.
\item The almost complex structure $\F$ depends on $\Sigma_{\cal R}$ only.
\item $\F$ is integrable on the manifold $\widetilde{TM}$ if and only if the d-tensor field
$R^i_{jk}(x,y)$ vanishes.
\end{enumerate}
\end{theorem}

Also, it is not difficult to prove

\begin{theorem}
We have
\begin{enumerate}
\item The triple $H^{2n}=(\widetilde{TM}, \G, \F)$ is an almost Hermitian space.
\item The space  $H^{2n}$ depends on $\Sigma_{\cal R}$ only.
\item The almost symplectic structure of $H^{2n}$ is
\begin{equation}
\theta=g_{ij}\delta y^i\wedge dx^j.
\end{equation}
\end{enumerate}
\end{theorem}

If the almost symplectic structure $\theta$ is a symplectic one (i.e. $d\theta=0$), then the space
$H^{2n}$ is almost K\"ahlerian.

On the other hand, using the formulas (1.7.4), (1.4.12) one obtains
$$\begin{array}{l}
d\theta=\dfrac{1}{3!}(R_{ijk}+R_{jki}+R_{kij})dx^i\wedge dx^j\wedge dx^k\\ \\
\qquad +\dfrac{1}{2}(g_{is}B^s_{jk}-g_{js}B_{ik}^s)\delta y^k\wedge dx^j\wedge dx^i.
\end{array}$$

Therefore, we deduce

\begin{theorem}
The almost Hermitian space $H^{2n}$ is almost K\"ahlerian if and only if the following relations hold good
\begin{equation}
R_{ijk}+R_{jki}+R_{kij}=0,\quad g_{is}B^s_{jk}-g_{js}B^s_{ik}=0.
\end{equation}
\end{theorem}

The space $H^{2n}=(\widetilde{TM},\G,\F)$ is called the {\it almost Hermitian model} of the
RMS $\Sigma_{\cal R}$.

One can use the almost Hermitian model $H^{2n}$ to study the geometrical theory of the mechanical system $\Sigma_{\cal R}$. For instance the Einstein equations of the RMS $\Sigma_{\cal R}$ are the Einstein equations of the pseudo-Riemannian space $(\widetilde{TM},\G)$ (cf. Chapter 6, part 1).

\begin{remark} The previous theory can be applied without difficulties to the examples 1-8 in Section 1.2.
\end{remark}

\newpage

\chapter{Finslerian Mechanical systems}\index{Analytical Mechanics of! Finslerian systems}

The present chapter is devoted to the Analytical Mechanics of the Finslerian Mechanical systems. These systems are defined by a triple $\Sigma_F=(M,F^2,Fe)$ where $M$ is the configuration space, $F(x,y)$ is the fundamental function of a semidefinite Finsler space $F^n=(M,F(x,y))$ and $Fe(x,y)$ are the external forces. Of course, $F^2$ is the kinetic energy of the space. The fundamental equations are the Lagrange equations: $$E_i(F^2)\equiv \dfrac{d}{dt}\dfrac{\pp F^2}{\pp \dot{x}^i}-\dfrac{\pp F^2}{\pp x^i}=F_i(x,\dot{x}).$$ We study here the canonical semispray $S$ of $\Sigma_F$ and the geometry of the pair $(TM,S)$, where $TM$ is velocity space.

One obtain a generalization of the theory of Riemannian Mechanical systems, which has numerous applications and justifies the introduction of such new kind of analytical mechanics.

\setcounter{section}{0}
\section{Semidefinite Finsler spaces}
\setcounter{equation}{0}\setcounter{theorem}{0}\setcounter{definition}{0}\setcounter{proposition}{0}\setcounter{remark}{0}

\begin{definition}
A Finsler space with semidefinite Finsler metric is a pair $F^n=(M,F(x,y))$ where the function $F:TM\to\R$ satisfies the following axioms:
\begin{itemize}
  \item[$1^\circ$] $F$ is differentiable on $\wt{TM}$ and continuous on the null section of $\pi:TM\to M$;
  \item[$2^\circ$] $F\geq 0$ on $TM$;
  \item[$3^\circ$] $F$ is positive 1-homogeneous with respect to velocities $\dot{x}^i=y^i$.
  \item[$4^\circ$] The fundamental tensor $g_{ij}(x,y)$
  \begin{equation}
  g_{ij}=\dfrac12\dfrac{\pp^2 F^2}{\pp y^i\pp y^j}
  \end{equation}
  has a constant signature on $\wt{TM}$;
  \item[$5^\circ$] The Hessian of fundamental function $F^2$ with elements $g_{ij}(x,y)$ is nonsingular:
  \begin{equation}
  \det(g_{ij}(x,y))\neq 0\ {\rm{on\ }}\wt{TM}.
  \end{equation}
\end{itemize}
\end{definition}

\noindent{\bf Example.} If $g_{ij}(x)$ is a semidefinite Riemannian metric on $M$, then
\begin{equation}
F=\sqrt{|g_{ij}(x)y^i y^j|}
\end{equation}
is a function with the property $F^n=(M,F)$ is a semidefinite Finsler space.

Any Finsler space $F^n=(M,F(x,y))$, in the sense of definition 3.1.1, part I, is a definite Finsler space. In this case the property $5^\circ$ is automatical verified.

But, these two kind of Finsler spaces have a lot of common properties. Therefore, we will speak in general on Finsler spaces. The following properties hold:
\begin{itemize}
  \item[$1^\circ$] The fundamental tensor $g_{ij}(x,y)$ is 0-homogeneous;
  \item[$2^\circ$] $F^2=g_{ij}(x,y)y^i y^j$;
  \item[$3^\circ$] $p_i=\dfrac12\dfrac{\pp F^2}{\pp y^i}$ is $d-$covariant vector field;
  \item[$4^\circ$] The Cartan tensor
  \begin{equation}
  C_{ijk}=\dfrac14\dfrac{\pp^3 F^2}{\pp y^i\pp y^j\pp y^k}=\dfrac12\dfrac{\pp g_{ij}}{\pp y^k}
  \end{equation}
  is totally symmetric and
  \begin{equation}
  y^i C_{ijk}=C_{0jk}=0.
  \end{equation}
  \item[$5^\circ$] $\oo=p_i dx^i$ is 1-form on $\wt{TM}$ (the Cartan 1-form);
  \item[$6^\circ$] $\theta=d\oo=dp_i\wedge dx^i$ is 2-form (the Cartan 2-form);
  \item[$7^\circ$] The Euler-Lagrange equations of $F^n$ are
  \begin{equation}
  E_i(F^2)=\dfrac{d}{dt}\dfrac{\pp F^2}{\pp y^i}-\dfrac{\pp F^2}{\pp x^i}=0,\ y^i=\dfrac{dx^i}{dt}
  \end{equation}
  \item[$8^\circ$] The energy ${\cal E}_F$ of $F^n$ is
  \begin{equation}
  {\cal E}_F=y^i\dfrac{\pp F^2}{\pp y^i}-F^2=F^2
  \end{equation}
  \item[$9^\circ$] \quad The energy ${\cal E}_F$ is conserved along to every integral curve of Euler-Lagrange equations (2.1.6);
  \item[$10^\circ$] \quad In the canonical parametrization, the equations (2.1.6) give the geodesics of $F^n$;
  \item[$11^\circ$] \quad The Euler-Lagrange equations (2.1.6) can be written in the equivalent form
\begin{equation}
\dfrac{d^2x^i}{dt^2}+\gamma^i_{jk}(x,\frac{dx}{dt})\dfrac{dx^j}{dt}\dfrac{dx^k}{dt}=0,
\end{equation}
where $\gamma^i_{jk}\(x,\dfrac{dx}{dt}\)$ are the Christoffel symbols of the fundamental tensor $g_{ij}(x,y)$.
\item[$12^\circ$] \quad The canonical semispray $S$ is
\begin{equation}
S=y^i\dfrac{\partial }{\partial x^i}-2G^i(x,y)\dfrac{\partial }{\partial y^i}
\end{equation}
with the coefficients:
\begin{equation}
2G^i(x,y)=\gamma^i_{jk}(x,y)y^jy^k=\gamma^i_{00}(x,y),\tag{2.1.9'}
\end{equation}
(the index ``0'' means the contraction with $y^i$).
\item[$13^\circ$] \quad The canonical semispray $S$ is 2-homogeneous with respect to $y^i$. So, $S$ is a spray.
\item[$14^\circ$] \quad The nonlinear connection $N$ determined by $S$ is also canonical and it is exactly the famous Cartan nonlinear connection of the space $F^n$. Its coefficients are
\begin{equation}
N^i{}_j(x,y)= \dfrac{\partial G^i(x,y)}{\partial y^j} =\dfrac{1}{2}\dfrac{\partial}{\partial y^j}\gamma^i_{00}(x,y).
\end{equation}
An equivalent form for the coefficients $N^i_j$ is as follows
\begin{equation}
N^i{}_j=\gamma^i_{j0}(x,y)-C^i_{jk}(x,y)\gamma^k_{00}(x,y).
\tag{2.1.10'}
\end{equation}
Consequently, we have
\begin{equation}
N^i{}_0=\gamma^i_{00}=2G^i.
\end{equation}
Therefore, we can say: The semispray $S'$ determined by the Cartan nonlinear connection $N$ is the canonical spray $S$ of space $F^n$.
\item[$15^\circ$] \quad The Cartan nonlinear connection $N$ determines a splitting of vector space $T_u TM,$ $\forall u\in TM$ of the form:
\begin{equation}
T_uTM=N_u\oplus V_u,\ \ \ \forall u\in TM
\end{equation}
Thus, the adapted basis $\left(\dfrac{\delta}{\delta x^i},\dfrac{\partial}{\partial y^i} \right)$, $(i=1,..,n),$ to the previous splitting has the local vector fields $\dfrac{\delta}{\delta x^i}$ given by:
\begin{equation}
\dfrac{\delta}{\delta x^i}=\dfrac{\partial}{\partial x^i}-N^j{}_i(x,y)\dfrac{\partial}{\partial y^j},\ \ \ i=1,..,n,
\end{equation}
with the coefficients $N^i{}_j(x,y)$ from (2.1.6).

Its dual basis is $\left(dx^i,\delta y^i\right)$, where
\begin{equation}
\delta y^i=dy^i+N^i{}_j(x,y)dx^j.
\end{equation}

The autoparallel curves of the nonlinear connection $N$ are given by, \cite{miranabuc},
\begin{equation}
\dfrac{d^2 x^i}{dt^2}+N^i{}_j( x,\dfrac{dx}{dt})\dfrac{dx^j}{dt}=0.
\end{equation}

Using the dynamic derivative $\nabla$ defined by $N$, the equations (2.1.11) can be written as follows
\begin{equation}
\nabla (\dfrac{dx^i}{dt})=0.
\tag{2.1.11'}
\end{equation}
\item[$16^\circ$] \quad The variational equations of the autoparallel curves (2.1.11) give the Jacobi equations:
\begin{equation}
\nabla^2 \xi^i+\left( \dfrac{\partial N^i{}_j}{\partial y^k}\dfrac{dx^j}{dt}-N^i{}_k\right) \nabla \xi^k+R^i_{jk}\dfrac{dx^j}{dt}\dfrac{dx^k}{dt}=0.
\end{equation}

The vector field $\xi^i(t)$ along a solution $c(t)$ of the equations (2.1.11) and which verifies the previous equations is called a Jacobi field. In the Riemannian case, $\dfrac{\partial g_{ij}}{\partial y^k}=0$, the Jacobi equations (2.1.12) are exactly the classical Jacobi equations:
\begin{equation}
\nabla^2 \xi^i+R^i_{jlk}(x)\dfrac{dx^l}{dt}\dfrac{dx^j}{dt}\xi^k=0
\end{equation}
\item[$17^\circ$] \quad A distinguished metric connections $D$ with the coefficients\break $C\Gamma(N)=\left(F^i_{jk},C^i_{jk}\right)$ is defined as a $N$-linear connection on $TM$, metric with respect to the fundamental tensor $g_{ij}(x,y)$ of Finsler space $F^n$, i.e. we have
\begin{equation}
\begin{array}{l}
g_{ij|k}=\dfrac{\delta g_{ij}}{\delta x^k}-F^s_{ik}g_{sj}-F^s_{jk}g_{is}=0,\\
\\
g_{ij}|_{k}=\dfrac{\partial g_{ij}}{\partial y^k} -C^s_{ik}g_{sj}-C^s_{jk}g_{is}=0.\\
\end{array}
\end{equation}

\item[$18^\circ$] \quad The following theorem holds:

\begin{theorem}
\begin{itemize}
\item[$1^{\circ}$] There is an unique $N$-linear connection $D$, with coefficients $C\Gamma(N)$ which satisfies the following system of axioms:
$A_1.$ $N$ is the Cartan nonlinear connection of Finsler space $F^n$.

$A_2.$ $D$ is metrical, $($i.e. $D$ satisfies $(2.1.14))$.

$A_3.$ $T^i_{jk}=F^i_{jk}-F^i_{kj}=0$, $S^i_{jk}=C^i_{jk}-C^i_{kj}=0$.

\item[$2^{\circ}$] The metric $N$-linear connection $D$ has the coefficients $C\Gamma(N)=\left(F^i_{jk},C^i_{jk}\right)$ given by the generalized Christoffel symbols
\end{itemize}
\begin{equation}
\begin{array}{l}
F^i_{jk}=\dfrac{1}{2}g^{is}\left(\dfrac{\delta g_{sj}}{\delta x^k}+\dfrac{\delta g_{sk}}{\delta x^j}-\dfrac{\delta g_{jk}}{\delta x^s}\right),\\
\\
C^i_{jk}=\dfrac{1}{2}g^{is}\left(\dfrac{\partial g_{sj}}{\partial y^k}+\dfrac{\partial g_{sk}}{\partial y^j}-\dfrac{\partial g_{jk}}{\partial y^s}
\right).
\end{array}
\end{equation}
\end{theorem}
\item[$20^\circ$] \quad By means of this theorem, it is not difficult to see that we have
\begin{equation}
C^i_{jk}=g^{is}C_{sjk}
\end{equation}
and
\begin{equation}
y^i{}_{|k}=0.
\end{equation}
\item[$21^\circ$] \quad The Cartan nonlinear connection $N$ determines on $\widetilde{TM}$ an almost complex structure $\F$, as follows:
\begin{equation}
\F(\dfrac{\delta}{\delta x^i})=-\dfrac{\partial}{\partial y^i},\ \ \
\F(\dfrac{\partial}{\partial y^i})=\dfrac{\delta}{\delta x^i}, \ \ \ \ i=1,..,n.
\end{equation}

But one can see that $\F$ is the tensor field on $\widetilde{TM}$:
\begin{equation}
\F=-\dfrac{\partial}{\partial y^i}\otimes dx^i+\dfrac{\delta}{\delta x^i}\otimes \delta y^i,
\tag{2.1.22'}
\end{equation}
with the $1$-forms $\delta y^i$ and the vector field $\dfrac{\delta}{\delta x^i}$ given by (2.1.10), (2.1.9), (2.1.6).

It is not difficult to prove that: The almost complex structure $\F$ is integrable if and only if the distribution $N$ is integrable on $TM$.

\item[$22^\circ$] \quad The Sasaki-Matsumoto lift of the fundamental tensor $g_{ij}$ of Fin\-sler space $F^n$ is
\begin{equation}
\G (x,y)=g_{ij}(x,y)dx^i\otimes dx^j+g_{ij}(x,y)\delta y^i\otimes \delta y^j.
\end{equation}
The tensor field $\G$ determines a pseudo-Riemannian structure on $TM$.

\item[$23^\circ$] \quad The following theorem is known:
\begin{theorem}
$1^{\circ}$ The pair $(\G,\F)$ is an almost Hermitian structure on $\widetilde{TM}$ determined only by the Finsler space $F^n$.

$2^{\circ}$ The symplectic structure associate to the structure $(\G,\F)$ is the Cartan 2-form:
\begin{equation}
\theta=2g_{ij}\delta y^i \wedge dx^j.
\end{equation}

$3^{\circ}$ The space $(\widetilde{TM},\G,\F)$ is almost K\"ahlerian.
\end{theorem}
The space $H^{2n}=(\widetilde{TM},\G,\F)$ is called {\it the almost K\"ahlerian model} of the Finsler space $F^n$.
\end{itemize}

G.S. Asanov in the paper \cite{As} proved that the metric $\G$ from (2.1.23) does not satisfies the principle of the Post-Newtonian Calculus. This is due to the fact that the horizontal and vertical terms of $\G$ do not have the same physical dimensions.

This is the reason for R. Miron to introduce a new lift of the fundamental tensor $g_{ij},$ \cite{miranabuc}, in the form:
$$
\widetilde{\G}(x,y)=g_{ij}(x,y)dx^i\otimes dx^j+\dfrac{a^2}{||y||^2}g_{ij}(x,y) \delta y^i\otimes \delta y^j
$$
where $a>0$ is a constant imposed by applications in Theoretical Physics and where $\|y\|^2=g_{ij}(x,y)y^iy^j=F^2$ has the property $F^2>0$. The lift $\G$ is 2-homogeneous with respect to $y^i$. The Sasaki-Matsumoto lift $\G$ has not the property of homogeneity.

\bigskip

\noindent{\bf Two examples:}~1. {\it Randers spaces.} They have been defined by R. S. Ingarden as a triple $RF^n=(M,\alpha+\beta,N)$, where $\alpha+\beta$ is a Randers metric and $N$ is the Cartan nonlinear connection of the Finsler space $F^n=(M,\alpha+\beta)$, \cite{Miron}.

2. {\it Ingarden spaces.}  These spaces have been defined by R. Miron, \cite{miranabuc}, as a triple $IF^n=(M,\alpha+\beta,N_L)$, where $\alpha+\beta$ is a Randers metric and $N_L$ is the Lorentz nonlinear connection of $F^n=(M,\alpha+\beta)$ having the coefficients
\begin{equation}
N^i_j(x,y)=\stackrel{\circ}{\gamma}\!{}^i_{jk}(x)y^k-\stackrel{\circ}{F}\!{}^i_j(x),\ \ \ \ \ \stackrel{\circ}{F}\!{}^i_j=\dfrac{1}{2}a^{is}(x) \(\dfrac{\partial b_s}{\partial x^j}-\dfrac{\partial b_j}{\partial x^s}\).
\end{equation}

The Christoffel symbols are constructed with the Riemannian metric tensor $a_{ij}(x)$ of the Riemann space $(M,\alpha^2)$ and $\stackrel{\circ}{F}\!{}^i_j(x)$ is the electromagnetic tensor determined by the electromagnetic form $(\alpha+\beta)$.

\section{The notion of Finslerian mechanical system}
\setcounter{equation}{0}\setcounter{theorem}{0}\setcounter{definition}{0}\setcounter{proposition}{0}\setcounter{remark}{0}

As we know from the previous chapter, the Riemannian mechanical systems $\Sigma_{\cal R}=(M,T,Fe)$ is defined as a triple
in which $M$ is the configuration space, $T$ is the kinetic energy and $Fe$ are the external forces, which depend on the material point $x\in M$ and depend on velocities $y^i=\dfrac{dx^i}{dt}$.

Extending the previous ideas, we introduce the notion of Finslerian Mechanical System, studied by author in the paper \cite{Miron}. The shortly theory of this analytical mechanics can be find in the joint book {\it Finsler-Lagrange Geometry. Applications to Dynamical Systems}, by Ioan Bucataru and Radu Miron, Romanian Academy Press, Bucharest, 2007.

In a different manner, M. de Leon and colab. \cite{leonrodri}, M. Crampin et colab. \cite{cr2}, have studied such kind of new Mechanics.

A Finslerian mechanical system $\Sigma_F$ is defined as a triple
\begin{equation}
\Sigma_F=(M,{\cal E}_{F^2}, Fe)
\end{equation}
where $M$ is a real differentiable manifold of dimension $n$, called {\it the configuration space}, ${\cal E}_{F^2}$ is {\it the energy} of an a priori given Finsler space $F^n=(M,F(x,y))$, which can be positive defined or semidefined, and $Fe(x,y)$ are the external forces given as a vertical vector field on the tangent manifold $TM$. We continue to say that $TM$ is the velocity space of $M$.

Evidently, the Finslerian mechanical system $\Sigma_F$ is a straightforward generalization of the known notion of Riemannian mechanical system $\Sigma_{\cal R}$ obtained for ${\cal E}_{F^2}$ as kinetic energy of a Riemann space ${\cal R}^n=(M,g)$.

Therefore, we can introduce the evolution (or fundamental) equations of $\Sigma_F$ by means of the following Postulate:

\bigskip

\noindent{\bf Postulate.} {\it The evolution equations of the Finslerian mechanical system $\Sigma_F$ are the Lagrange equations:
\begin{equation}
\dfrac{d}{dt}\dfrac{\partial {\cal E}_{F^2}}{\partial y^i}-
\dfrac{\partial {\cal E}_{F^2}}{\partial x^i}=F_i(x,y),\ \ \ \ y^i=\dfrac{dx^i}{dt}
\end{equation}
where the energy is
\begin{equation}
{\cal E}_{F^2}=y^i\dfrac{\partial {F^2}}{\partial y^i}-F^2=F^2,
\end{equation}
and $F_i(x,y),\ (i=1,..,n)$, are the covariant components of the external forces $Fe$:
\begin{equation}
\left\{
\begin{array}{l}
Fe(x,y)=F^i(x,y)\dfrac{\partial}{\partial y^i}\\
\\
F_i(x,y)=g_{ij}(x,y)F^i(x,y),
\end{array}\right.
\end{equation}
and
\begin{equation}
g_{ij}(x,y)=\dfrac{1}{2}\dfrac{\partial^2 F^2}{\partial y^i \partial y^j},\ \ \ \det(g_{ij}(x,y))\neq 0,
\end{equation}
is the fundamental (or metric) tensor of Finsler space $F^n=(M,F(x,y))$.}

Finally, the Lagrange equations of the Finslerian mechanical system are:
\begin{equation}
\dfrac{d}{dt}\dfrac{\partial {F^2}}{\partial y^i}-
\dfrac{\partial {F^2}}{\partial x^i}=F_i(x,y),\ \ \ \ y^i=\dfrac{dx^i}{dt}.
\end{equation}

A more convenient form of the previous equations is given by:
\begin{theorem}
The Lagrange equations $(2.2.6)$ are equivalent to the second order differential equations:
\begin{equation}
\dfrac{d^2x^i}{dt^2}+\gamma^i_{jk}(x,\dfrac{dx}{dt})
\dfrac{dx^j}{dt}\dfrac{dx^k}{dt}=\dfrac{1}{2}F^i(x,\dfrac{dx}{dt}),
\end{equation}
where $\gamma^i_{jk}(x,y)$ are the Christoffel symbols of the metric tensor $g_{ij}(x,y)$ of the Finsler space $F^n$.
\end{theorem}

\begin{proof}
Writing the kinetic energy $F^2(x,y)$ in the form:
\begin{equation}
F^2(x,y)=g_{ij}(x,y)y^iy^j,
\end{equation}
the equivalence of the systems of equations (2.2.6) and (2.2.7) is not difficult to establish.
\end{proof}

But, the form (2.2.7) is very convenient in applications. So, we obtain a first result expressed in the following theorems:
\begin{theorem}
 The trajectories of the Finslerian mechanical system $\Sigma_F$, without external forces $(Fe\equiv 0)$, are the geodesics of the Finsler space $F^n$.
\end{theorem}

Indeed, $F^i(x,y)\equiv 0$ and the SODE (2.2.7) imply the equations (2.2.4) of geodesics of space $F^n$.

A second important result is a consequence of the Lagrange equations, too.
\begin{theorem}
The variation of kinetic energy ${\cal E}_{F^2}=F^2$ of the mechanical system $\Sigma_F$ along the evolution curves (2.2.6) is given by
\begin{equation}
\dfrac{d {\cal E}_{F^2}}{dt}=\dfrac{dx^i}{dt}F_i.
\end{equation}
\end{theorem}

Consequently:
\begin{theorem}
The kinetic energy ${\cal E}_{F^2}$ of the system $\Sigma_F$ is conserved along the evolution curves (2.2.6) if the external forces $Fe$ are orthogonal to the evolution curves.
\end{theorem}

The external forces $Fe$ are called {\it dissipative} if the scalar product $\langle\C,Fe\rangle$ is negative, \cite{Miron}.

The formula (2.2.9) leads to the following property expressed by:
\begin{theorem}
The kinetic energy $ {\cal E}_{F^2} $ decreases along the evolution curves of the Finslerian mechanical system $\Sigma_F$ if and only if the external forces $Fe$ are dissipative.
\end{theorem}

\section*{Some examples of Finslerian mechanical systems}
\setcounter{equation}{0}\setcounter{theorem}{0}\setcounter{definition}{0}\setcounter{proposition}{0}\setcounter{remark}{0}

$1^{\circ}$ The systems $\Sigma_F=(M,{\cal E}_{F^2},Fe)$ given by $F^n=(M,\alpha+\beta)$ as a Randers space and $Fe=\beta\C=\beta y^i\dfrac{\partial}{\partial y^i}$. Evidently $Fe$ is $2$-homogeneous with respect to  $y^i$.

$2^{\circ}$ $\Sigma_F$ determined by $F^n=(M,\alpha+\beta)$ and $Fe=\alpha\C$.

$3^{\circ}$ $\Sigma_F$ with $F^n=(M,\alpha+\beta)$ and $Fe=(\alpha+\beta)\C$.

$4^{\circ}$ $\Sigma_F$ defined by a Finsler space $F^n=(M,F)$ and $Fe=a^i_{jk}(x)y^jy^k\dfrac{\partial}{\partial y^i}$, $a^i_{jk}(x)$ being a symmetric tensor on the configuration space $M$ of type $(1,2)$.

\section{The evolution semispray of the system $\Sigma_F$}\index{Semisprays! of mechanical system $\Sigma_F$}
\setcounter{equation}{0}\setcounter{theorem}{0}\setcounter{definition}{0}\setcounter{proposition}{0}\setcounter{remark}{0}

The Lagrange equations (2.2.6) give us the integral curves of a remarkable semispray on the velocity space $TM$, which governed the geometry of Finslerian mechanical system $\Sigma_F$. So, if the external forces $Fe$ are global defined on the manifold $TM$, we obtain:
\begin{theorem}
[Miron, \cite{Miron}] For the Finslerian mechanical systems $\Sigma_F$, the following properties hold good:

$1^{\circ}$ The operator $S$ defined by
\begin{equation}
S=y^i\dfrac{\partial}{\partial x^i}-\left( 2 \stackrel{\circ}{G}\!{}^i-\dfrac{1}{2}F^i\right) \dfrac{\partial}{\partial y^i};\ \ \ \  2\stackrel{\circ}{G}\!{}^i=\gamma^i_{jk}(x,y)y^jy^k
\end{equation}
is a vector field, global defined on the phase space $TM$.

$2^{\circ}$ $S$ is a semispray which depends only on $\Sigma_F$ and it is a spray if $Fe$ is $2$-homogeneous with respect to $y^i$.

$3^{\circ}$ The integral curves of the vector field $S$ are the evolution curves given by the Lagrange equations (2.2.7) of $\Sigma_F$.
\end{theorem}

\begin{proof}
$1^{\circ}$ Let us consider the canonical semispray $\stackrel{\circ}{S}$ of the Finsler space $F^n$. Thus from (2.3.1) we have
\begin{equation}
S=\stackrel{\circ}{S} +\dfrac{1}{2}Fe.
\end{equation}
It follows that $S$ is a vector field on $TM$.
\end{proof}

$2^{\circ}$ Since $Fe$ is a vertical vector field, then $S$ is a semispray. Evidently, $S$ depends on $\Sigma_F$, only.

$3^{\circ}$ The integral curves of $S$ are given by:
\begin{equation}
\dfrac{dx^i}{dt}=y^i;\ \ \ \  \dfrac{dy^i}{dt}+2\stackrel{\circ}{G^i}(x,y)=\dfrac{1}{2}F^i(x,y).
\end{equation}

The previous system of differential equations is equivalent to system (2.2.7).

In the book of I. Bucataru and R. Miron \cite{BuMi}, one proves the following important result, which extend a known J. Klein theorem:
\begin{theorem}
The semispray $S$, given by the formula $(2.3.1)$, is the unique vector field on $\widetilde{TM}$, solution of the equation:
\begin{equation}
i_S\stackrel{\circ}{\omega}=-dT+\sigma,
\end{equation}
where $\stackrel{\circ}{\omega}$ is the symplectic structure of the Finsler space $F^n=(M,F)$, $T=\dfrac{1}{2}F^2=\dfrac{1}{2}g_{ij}\dfrac{dx^i}{dt}\dfrac{dx^j}{dt}$ and $\sigma$ is the 1-form of external forces:
\begin{equation}
\sigma=F_i(x,y)dx^i.
\end{equation}
\end{theorem}

In the terminology of J. Klein, \cite{Klein1}, $S$ is the dynamical system of $\Sigma_F$, defined on the tangent manifold $TM$. We will say that $S$ is the evolution semispray of $\Sigma_F$.

By means of semispray $S$ (or spray $S$) we can develop the geometry of the Finslerian mechanical system $\Sigma_F$. So, all geometrical notion derived from $S$, as nonlinear connections, N-linear connections etc. will be considered as belong to the system $\Sigma_F$.

\section{The canonical nonlinear connection of the Finslerian mechanical systems $\Sigma_F$}\index{Connections! Nonlinear}
\setcounter{equation}{0}\setcounter{theorem}{0}\setcounter{definition}{0}\setcounter{proposition}{0}\setcounter{remark}{0}

The evolution semispray  $S$ (2.3.1) has the coefficients $G^i$ expressed by
\begin{equation}
2G^i(x,y)=2\stackrel{\circ}{G^i}(x,y)-\dfrac{1}{2}F^i(x,y)
=\gamma^i_{jk}(x,y)y^jy^k-\dfrac{1}{2}F^i(x,y).
\end{equation}

Thus, the evolution nonlinear connection $N$ (or canonical nonlinear connection) of the Finslerian mechanical system $\Sigma_F$ has the coefficients:
\begin{equation}
N^i{}_j=\dfrac{\partial G^i}{\partial y^j}=\stackrel{\circ}{N^i}\!{}_j-\dfrac{1}{4}\dfrac{\partial F^i}{\partial y^j}
\end{equation}
where $\stackrel{\circ}{N}$ with coefficients $\stackrel{\circ}{N^i}\!{}_j$ is the Cartan nonlinear connection of Finsler space $F^n =(M,F(x,y))$.

$N$ depends on the mechanical system $\Sigma_F$, only. It is called canonical for $F^n$.

The nonlinear connection $N$ determines the horizontal distribution, denoted by $N$ too, with the property
\begin{equation}
T_uTM=N_u\oplus V_u,\ \ \ \forall u \in TM,
\end{equation}
$V_u$ being the natural vertical distribution on the tangent manifold $TM$.

A local adapted basis to the horizontal and vertical vector spaces $N_u$ and $V_u$ is given by $\left(\dfrac{\delta}{\delta x^i},\dfrac{\partial}{\partial y^i}\right)$,  $(i=1,..,n)$, where
\begin{equation}
\dfrac{\delta}{\delta x^i}=\dfrac{\partial }{\partial x^i}-N^j{}_i\dfrac{\partial }{\partial y^j}=\dfrac{\partial}{\partial x^i} -
\left( \stackrel{\circ}{N}\!\!{}^j{}_i-
\frac{1}{4}\dfrac{\partial F^j}{\partial y^i}\right) \dfrac{\partial }{\partial y^j},\ \ i=1,..,n,
\end{equation}
and the adapted cobasis $(dx^i, \delta y^i)$ with
\begin{equation}
\delta y^i=dy^i+N^i{}_j dx^j=dy^i+\left(\stackrel{\circ}{N}\!\!{}^i{}_j-\frac{1}{4}\dfrac{\partial F^i}{\partial y^j}\right)dx^j,\ \ i=1,..,n.
\tag{2.4.4'}
\end{equation}

From (2.4.4) and (2.4.4') it follows
\begin{equation}
\left\{\begin{array}{l}
\dfrac{\delta }{\delta x^i}=\dfrac{\stackrel{\circ}{\delta}}{\delta x^i}+\frac{1}{4}\dfrac{\partial F^j}{\partial y^i}\dfrac{\partial }{\partial y^j},\\
\\
\delta y^i=\stackrel{\circ}{\delta}\!\!y^i-\dfrac{1}{4}\dfrac{\partial F^i}{\partial y^j}dx^j.
\end{array}\right.
\end{equation}

Now we easy determine the curvature ${\cal R}^i{}_{jk}$ and torsion $t^i{}_{jk}$ of the canonical nonlinear connection $N$. One obtains
\begin{equation}
\begin{array}{l}
{\cal R}^i{}_{jk}=\dfrac{\delta N^i\!{}_j}{\delta x^k}-\dfrac{\delta N^i\!{}_k}{\delta x^j}=\( \dfrac{\delta}{\delta x^k}\dfrac{\partial}{\partial y^j}-\dfrac{\delta}{\delta x^j}\dfrac{\partial}{\partial y^k}\)(\stackrel{\circ}{G^i}-\frac{1}{2}F^i),\\
\\
t^i{}_{jk}=\dfrac{\partial N^i\!{}_j}{\partial y^k}-\dfrac{\partial N^i\!{}_k}{\partial y^j}=0
\end{array}
\end{equation}
such that:

$1^{\circ}$ The torsion of the canonical nonlinear connection $N$ vanishes.

$2^{\circ}$ The condition ${\cal R}^i{}_{jk}=0$ is necessary and sufficient for $N$ to be integrable.

Another important geometric object field determined by the canonical nonlinear connection $N$ is the Berwald connection
$B\Gamma(N)=(B^i{}_{jk},0)=\(\dfrac{\partial N^i{}_j}{\partial y^k},0\)$. Its coefficients are:
\begin{equation}
B^i{}_{jk}=\stackrel{\circ}{B}\!\!{}^i{}_{jk}-\dfrac{1}{4}\dfrac{\partial^2 F^i}{\partial y^j \partial y^k},
\end{equation}
where $\stackrel{\circ}{B}\!\!{}^i{}_{jk}$ are coefficients of Berwald connection of Finsler space $F^n$. The coefficients (2.4.7) are symmetric.

We can prove:
\begin{theorem}
The autoparallel curves of the evolution nonlinear connection $N$ are given by the following SODE:
\begin{equation}
y^i=\dfrac{dx^i}{dt},\ \ \
\dfrac{\delta y^i}{dt}=\dfrac{\stackrel{\circ}{\delta}y^i}{dt} -\dfrac{1}{4}\dfrac{\partial F^i}{\partial y^j}\dfrac{dx^j}{dt}=0.
\end{equation}
\end{theorem}

\begin{corollary}
If the external forces $Fe$ vanish, then the evolution nonlinear connection $N$ is the Cartan nonlinear connection of Finsler space $F^n.$
\end{corollary}

\begin{corollary}
If the external forces $Fe$ are $2$-homogeneous with respect to velocities $y^i,$ then the equations $(2.4.8)$ coincide with the evolution equations of the Finslerian mechanical system $\Sigma_F$.
\end{corollary}

It is not difficult to determine the evolution nonlinear connection $N$ from examples in section 2.3.

\begin{lemma}
The exterior differential of the $1$-forms $\delta y^i$ are given by formula:
\begin{equation}
d(\delta y^i)=d N^i{}_j\wedge dx^j=\dfrac{1}{2}{\cal R}^i{}_{kj}dx^j \wedge dx^k+B^i{}_{kj}\delta y^j \wedge dx^k.
\end{equation}
\end{lemma}

\section{The dynamical derivative determined by the evolution nonlinear connection $N$}\index{Dynamical derivative}
\setcounter{equation}{0}\setcounter{theorem}{0}\setcounter{definition}{0}\setcounter{proposition}{0}\setcounter{remark}{0}

The dynamical covariant derivation induced by the evolution nonlinear connection $N$ is expressed by (?) using the coefficients $N^i{}_j$ from (2.4.2). It is given by
\begin{equation}
\nabla\left(X^i\dfrac{\partial}{\partial y^i}\right)=\left( SX^i+X^jN^i{}_j\right)\dfrac{\partial}{\partial y^i}.
\end{equation}

Applied to a $d$-vector field $X^i(x,y)$, we have the formula:
\begin{equation}
\nabla X^i=SX^i{}_{|}=SX^i+X^jN^i{}_j
\end{equation}
and for a $d$-tensor $g_{ij}$, we have
\begin{equation}
\nabla g_{ij}=g_{ij\ |}=Sg_{ij}-g_{sj}N^s{}_i-g_{is}N^s{}_j,
\tag{2.5.2'}
\end{equation}
where
\be
\begin{array}{l}
S=\stackrel{\circ}{S} +\dfrac{1}{2}Fe\\
\\
N^i{}_j=\stackrel{\circ}{N}\!{}^i{}_j-\dfrac{1}{4}\dfrac{\partial F^i}{\partial y^j}.
\end{array}
\ee

Remarking that $\dfrac{\partial F^i}{\partial y^j}$ is a $d$-tensor field of type (1,1), one can introduce two $d$-tensors, important in the geometrical theory of the Finslerian mechanical systems $\Sigma_F:$
\be
P_{ij}=\dfrac{1}{2}\(\dfrac{\partial F_i}{\partial y^j}-\dfrac{\partial F_j}{\partial y^i}\),\ \ \
Q_{ij}=\dfrac{1}{2}\(\dfrac{\partial F_i}{\partial y^j}+\dfrac{\partial F_j}{\partial y^i}\).
\ee

The first one $P_{ij}$ is called the helicoidal tensor of the Finslerian mechanical system $\Sigma_F$, \cite{miranabuc}.

Also, on the phase space $TM$, the elicoidal $d$-tensor $P_{ij}$ give rise to the $2$-form
\be
P=P_{ij}\ dx^i\wedge dx^j
\ee
and $Q_{ij}$ allows to consider the symmetric tensor
\be
Q=Q_{ij}\ dx^i\otimes dx^j.
\ee
The following Bucataru-Miron theorem holds:
\begin{theorem}
For a Finslerian mechanical system $\Sigma_F=(M,{\cal E}_{F^2},Fe)$ the evolution nonlinear connection $N$ is the unique nonlinear connection that satisfies the conditions:
\be
\nabla g=-\dfrac{1}{2}\ Q,\ \ \
\stackrel{\circ}{\omega}(hX,hY)=\dfrac{1}{2}P(X,Y),\ \ \forall X,Y\in \chi(TM),
\ee
where $\stackrel{\circ}{\omega}$ is the symplectic structure of the Finsler space $F^n$:
\begin{equation}
\stackrel{\circ}{\omega}=2g_{ij}\stackrel{\circ}{\delta}y^j\wedge dx^i.
\end{equation}
\end{theorem}

If $Fe$ does not depend on velocities $y^i=\dfrac{dx^i}{dt}$ we have $P\equiv 0,\  Q\equiv 0$. One obtains the the case of Riemannian mechanical systems $\Sigma_R$, studied in the previous chapter.

\section{Metric N-linear connection of $\Sigma_F$}\index{Connections! $N-$linear}
\setcounter{equation}{0}\setcounter{theorem}{0}\setcounter{definition}{0}\setcounter{proposition}{0}\setcounter{remark}{0}

The metric, or canonic, $N$-linear connection $D,$ with coefficients\break $C\Gamma(N)=(F^i{}_{jk},C^i{}_{jk})$ of the Finslerian mechanical system $\Sigma_F$ is uniquely determined by the following axioms, (Miron \cite{mir11}):

$A_1.$ $N$ is the canonical nonlinear connection of $\Sigma_F.$

$A_2.$ $D$ is $h$-metric, i.e $g_{ij\ |k}=0.$

$A_3.$ $D$ is $h$-symmetric, i.e. $T^i{}_{jk}=F^i{}_{jk}-F^i{}_{kj}=0.$

$A_4.$ $D$ is $v$-metric, i.e. $g_{ij}|_k=0.$

$A_5.$ $D$ is $v$-symmetric, i.e. $S^i{}_{jk}=C^i{}_{jk}-C^i{}_{kj}=0.$

The following important result holds:
\begin{theorem}
The local coefficients $D\Gamma(N)=(F^i{}_{jk},C^i{}_{jk})$ of the canonical $N$-connection $D$ of the Finslerian mechanical system $\Sigma_F$ are given by the generalized Christoffel symbols:
\begin{equation}
\left\{
\begin{array}{l}
F^i\!{}_{jk}=\dfrac{1}{2}g^{is}\(\dfrac{\delta g_{sk}}{\delta x^j}+
\dfrac{\delta g_{js}}{\delta x^k}-\dfrac{\delta g_{jk}}{\delta x^s}\),\\
\\
C^i\!{}_{jk}=\dfrac{1}{2}g^{is}\(\dfrac{\partial g_{sk}}{\partial y^j}+
\dfrac{\partial g_{js}}{\partial y^k}-\dfrac{\partial g_{jk}}{\partial y^s}
\).
\end{array}
\right.
\end{equation}
\end{theorem}

Using the expression (?) of the operator $\dfrac{\delta}{\delta x^i}$, one obtains:
\begin{theorem}
The coefficients $F^i{}_{jk},\ C^i{}_{jk}$ of canonical $N$-connection $D$ have the expressions:
\be
\begin{array}{l}
F^i{}_{jk}=\stackrel{\circ}{F}\!{}^i{}_{jk}+\check{C}^i{}_{jk}, \ \ \ C^i{}_{jk}=\stackrel{\circ}{C}\!{}^i{}_{jk}\\
\\
\check{C}^i{}_{jk}=\dfrac{1}{4}g^{is}\(\stackrel{\circ}{C}_{skh}\dfrac{\partial F^h}{\partial y^j}+\stackrel{\circ}{C}_{jsh}\dfrac{\partial F^h}{\partial y^k}-
\stackrel{\circ}{C}_{jkh}\dfrac{\partial F^h}{\partial y^s}\),
\end{array}
\ee
where $C\Gamma(\stackrel{\circ}{N})=(\stackrel{\circ}{F}\!{}^i{}_{jk}, \stackrel{\circ}{C}\!{}^i{}_{jk})$ are the coefficients of canonic Cartan metric connection of Finsler space $F^n$.
\end{theorem}

A consequence of the previous formulas is the following relation:
\be
y^j\check{C}^i{}_{jh}=\dfrac{1}{4}\stackrel{\circ}{C}\!{}^i_{hk}y^s\dfrac{\partial F^k}{\partial y^s}.
\ee

Let $\omega^i{}_j$ be the connection forms of $C\Gamma(N)$:
\be
\omega^i{}_j=F^i{}_{jk}dx^k+C^i{}_{jk}\delta y^k.
\ee

Then, we have
\begin{theorem}
The structure equations of the canonical connection $C\Gamma(N)$ are given by:
\begin{equation}
\left\{
\begin{array}{ll}
d(dx^i)-dx^k\wedge \omega^i\!{}_k & =-\stackrel{1}{\Omega}\!{}^i,\\
d(\delta y^i)-\delta y^k\wedge \omega^i\!{}_k & =-\stackrel{2}{\Omega}\!{}^i,\\
d \omega^i\!{}_j\ \  -\ \omega^k\!{}_j\wedge\omega^i\!{}_k & =-\stackrel{}{\Omega}\!{}^i{}_j,
\end{array}
\right.
\end{equation}
where the $2$-forms of torsion $\stackrel{1}{\Omega}\!{}^i,\ \stackrel{2}{\Omega}\!{}^i$ are as follows
\begin{equation}
\begin{array}{ll}
\stackrel{1}{\Omega}\!{}^i=C^i\!{}_{jk}dx^j\wedge \delta y^k,\\
\\
\stackrel{2}{\Omega}\!{}^i= \dfrac{1}{2}R^i\!{}_{jk}dx^j\wedge dx^k+P^i\!{}_{jk}dx^j\wedge \delta y^k,
\end{array}
\end{equation}
with
\be
R^i\!{}_{jk}=\dfrac{\delta N^i{}_{j}}{\delta x^k}-\dfrac{\delta N^i{}_{k}}{\delta x^j},\ \ \ \ P^i\!{}_{jk}=\dfrac{\partial N^i{}_j}{\partial y^k} -F^i\!{}_{kj}
\ee
and the $2$-form of curvature $\Omega^i{}_j$ is given by
\begin{equation}
\Omega^i{}_j=\dfrac{1}{2}R_j^{\ i}\!{}_{kh}\ dx^k\wedge dx^h+P_j^{\ i}\!{}_{kh}\ dx^k\wedge \delta y^h+\dfrac{1}{2}S_j^{\ i}\!{}_{kh}\ \delta y^k\wedge \delta y^h,
\end{equation}
where
\begin{equation}
\left\{
\begin{array}{l}
R_j^{\ i}{}_{kh}=\dfrac{\delta F^i\!{}_{jk}}{\delta x^h}-\dfrac{\delta F^i\!{}_{jh}}{\delta x^k}+F^s\!{}_{jk}F^i\!{}_{sh}-F^s\!{}_{jh}F^i\!{}_{sk}+
C^i\!{}_{hs}R^s\!{}_{kh},\\
\\
P_j^{\ i}{}_{kh}=\dfrac{\partial F^i\!{}_{jk}}{\partial y^h}-C^i_{jh\ |k}+C^i{}_{js}P^s{}_{kh},\\
\\
S_j^{\ i}{}_{kh}=\dfrac{\partial C^i\!{}_{jk}}{\partial y^h}-\dfrac{\partial C^i\!{}_{jh}}{\partial y^k}+C^s_{jk}C^i_{sh}-C^s_{jh}C^i_{jk}.
\end{array}
\right.
\end{equation}
\end{theorem}

Taking into account that the coefficients $F^i\!{}_{jk}$ and $C^i\!{}_{jk}$ are expressed in the formulas (2.6.2), the calculus of curvature tensors is not difficult.

Exterior differentiating (2.6.5) and using them again one obtains the Bianchi identities of $C\Gamma(N)$.

The $h$- and $v$-covariant derivatives, denoted by ``${}_{|}$'' and ``$|$'', with respect to canonical connection $D$ have the properties given by the axioms $A_1-A_4$.

So, we obtain
\be
\left\{
\begin{array}{l}
g_{ij\ |k}=\dfrac{\delta g_{ij}}{\delta x^k}-g_{sj}F^s\!{}_{ik}-g_{is}F^s\!{}_{jk}=0,\\
\\
g_{ij}|_k=\dfrac{\partial g_{ij}}{\partial y^k}
-g_{sj}C^s\!{}_{ik}-g_{is}C^s\!{}_{jk}=0.
\end{array}\right.
\ee

The Ricci identities applied to the fundamental tensor $g_{ij}$ give us:
$$
R_{ijkh}+R_{jikh}=0,\ \ \ \ P_{ijkh}+P_{jikh}=0,\ \ \ \ S_{ijkh}+S_{jikh}=0,
$$
where $R_{ijkh}=g_{js}R_i{}^s{}_{kh},$ etc.

Also, $P_{ijk}=g_{is}P^s{}_{jk}$ is totally symmetric.

The deflection tensors of $D$ are
\be
D^i\!{}_j=y^i_{\ |j}=\dfrac{\delta y^i}{\delta x^j}+y^sF^i\!{}_{sj};\ \ \
d^i\! {}_j=y^i|_j.
\ee
Taking into account (2.6.11), we obtain:
\begin{equation}
D^i\! {}_j=y^s\check{C}{}^i_{sj}+\dfrac{1}{4}\dfrac{\partial F^i}{\partial y^j},\ \ \ \ \ \ d^i\! {}_j=\delta^i_j.
\end{equation}

\section{The electromagnetism in the theory of the Finslerian mechanical systems $\Sigma_{F}$}\index{Electromagnetism and gravitational fields}
\setcounter{equation}{0}\setcounter{theorem}{0}\setcounter{definition}{0}\setcounter{proposition}{0}\setcounter{remark}{0}

For a Finslerian mechanical system $\Sigma_{F}=( M,{\cal E}_{F^2},Fe)$ whose external forces $Fe$ depend on the material point $x\in M$ and on velocity $y^i=\dfrac{dx^i}{dt},$ the electromagnetic phenomena appears because the deflection tensors $D^i\! {}_j$ and $d^i\! {}_j$ nonvanish.

Setting $D_{ij}=g_{ih}D^h\! {}_j,\ d_{ij}=g_{ih}\delta^h_j,$
the $h$-electromagnetic tensor ${\cal F}_{ij}$ and the $v$-electromagnetic tensor $f_{ij}$ are defined by
\begin{equation}
{\cal F}_{ij}=\dfrac{1}{2}(D_{ij}-D_{ji}),\ \ \ \ f_{ij}=\dfrac{1}{2}(d_{ij}-d_{ji}).
\end{equation}

By using the equalities (2.6.11), we have
\begin{equation}
{\cal F}_{ij}=\dfrac{1}{4} P_{ij},\ \ \ \ f_{ij}=0.
\end{equation}

If we denote $R_{ijk}:=g_{ih}R^h\!{}_{jk}$, then
one proves:
\begin{theorem}
The electromagnetic tensor ${\cal F}_{ij}$ of the Finslerian mechanical system $\Sigma_{F}=( M,T,Fe )$ satisfies the following generalized Maxwell equations:
\begin{equation}
\begin{array}{l}
{\cal F}_{ij\ |k}+{\cal F}_{jk\ |i}+{\cal F}_{ki\ |j}=
\dfrac{1}{2}\{ y^s(R_{sijk}+R_{sjki}+R_{skij})-\\ \\ \hfill
-(R_{ijk}+R_{jki}+R_{kij})\},\\
\\
{\cal F}_{ij} |_k+{\cal F}_{jk} |_{i}+{\cal F}_{ki} |_{j}=
\dfrac{1}{2}\{ y^s [(P_{sijk}-P_{sikj})+\\ \\ \hfill +(P_{sjki}-P_{sjik})+(P_{skij}-P_{skji})]\}
\end{array}
\end{equation}
\end{theorem}
where ${\cal F}_{ij}$ is expressed in (2.7.2).

\begin{corollary}
If the external forces $Fe$ does not depend on the velocity $y^i$ then the electromagnetic fields ${\cal F}_{ij}$ and ${f}_{ij}$ vanish.
\end{corollary}

\begin{remark}
$1^{\circ}$ The application of the previous theory to the examples from section 2.2 is immediate.

$2^{\circ}$ The theory of the gravitational field $g_{ij}$ and the Einstein equations can be realized by same method as in the papers, \cite{miranabuc}.
\end{remark}

\section{The almost Hermitian model on the tangent manifold $TM$ of the Finslerian mechanical systems $\Sigma_{F}$}\index{Almost Hermitian model}
\setcounter{equation}{0}\setcounter{theorem}{0}\setcounter{definition}{0}\setcounter{proposition}{0}\setcounter{remark}{0}

Consider a Finslerian mechanical system $\Sigma_F=(M,F(x,y),Fe(x,y))$ endowed with the evolution nonlinear connection $N$ and also endowed with the canonical $N-$metrical connection having the coefficients $C\Gamma(N)= (F^i\!{}_{jk},C^i\!{}_{jk})$. On the velocity manifold $\widetilde{TM}=TM\setminus\{ 0\}$ we can see that the previous geometrical object fields determine an almost Hermitian structure $H^{2n}=(\widetilde{TM},\G,\F)$. Moreover, the theory of gravitational and electromagnetic fields can be geometrically studied much better on such model, since the symplectic structure $\theta$, the almost complex structure $\F$ and the Riemannian structure $\G$ on $\widetilde{TM}$ are well determined by the Finslerian mechanical system $\Sigma_F$.

One obtains:
\begin{equation}
\begin{array}{l}
\G= g_{ij} dx^i\otimes dx^j +g_{ij}\delta y^i \otimes \delta y^j\\
\\
\F= -\dfrac{\partial}{\partial y^i}\otimes dx^i+\dfrac{\delta }{\delta x^i}\otimes \delta y^i\\
\\
\P= -\dfrac{\partial}{\partial y^i}\otimes \delta y^i+\dfrac{\delta }{\delta x^i}\otimes dx^i,
\end{array}
\end{equation}
where
\begin{equation}
\left\{
\begin{array}{l}
\dfrac{\delta }{\delta x^i}=\dfrac{\stackrel{\circ}{\delta} }{\delta x^i}+\dfrac{1}{4}\dfrac{\partial F^s}{\partial y^i}\dfrac{\partial }{\partial y^s}\\
\\
\delta y^i=\stackrel{\circ}{\delta}\! y^i-\dfrac{1}{4}\dfrac{\partial F^i}{\partial y^s}dx^s.
\end{array} \right.
\end{equation}

Thus, the following theorem holds:
\begin{theorem}
We have:

$1^{\circ}$  The pair $(\widetilde{TM},\G)$ is a pseudo Riemannian space.

$2^{\circ}$  The tensor $\G$ depends on $\Sigma_F,$ only.

$3^{\circ}$  The distributions $N$ and $V$ are orthogonal with respect to $\G$.
\end{theorem}

\begin{theorem}
$1^{\circ}$  The pair $(\widetilde{TM},\F)$ is an almost complex space.

$2^{\circ}$  $\F$ depends on $\Sigma_R,$ only.

$3^{\circ}$  $\F$ is integrable on $\widetilde{TM}$ iff the tensors $R^i{}_{jk}$ vanishes.
\end{theorem}

\begin{theorem}
$1^{\circ}$  The pair $(\widetilde{TM},\P)$ is an almost product space.

$2^{\circ}$  $\P$ depends on $\Sigma_F,$ only.

$3^{\circ}$  $\P$ is integrable iff the tensors $R^i{}_{jk}$ vanishes.
\end{theorem}

The integrability of $\F$ and $\P$ are studied by means of Nijenhuis tensors ${\cal N}_{\P}$ and ${\cal N}_{\F}:$
$$
{\cal N}_{\F}(X,Y)=\F^2(X,Y)+\left[\F X, \F Y \right]-\F \left[\F X,Y\right]-$$ $$-\F \left[X,\F Y\right],\ \ \forall X,Y \in\chi(TM).
$$
In the adapted basis $\left(\dfrac{\delta}{\delta x^i},\dfrac{\partial}{\partial y^i}\right)$ ${\cal N}_{\F}=0$ if, and only if $R^i{}_{jk}=0,\ t^i{}_{jk}=0.$ But the torsion tensor $t^i{}_{jk}$ vanishes.

It is not difficult to prove the following results:
\begin{theorem}
$1^{\circ}$  The triple $\left( \widetilde{TM},\G,\F\right)$ is an almost Hermitian space.

$2^{\circ}$  $\left(\widetilde{TM},\G,\F\right)$ depends on Finslerian mechanical system $\Sigma_F,$ only.

$3^{\circ}$  The almost symplectic structure $\theta$ of the space $\left(\widetilde{TM},\G,\F\right)$ is

\begin{equation}
\theta=g_{ij} \delta y^i\wedge dx^j,
\end{equation}
with $\delta y^i$ from (2.8.2).
\end{theorem}

If the almost symplectic structure $\theta$ is a symplectic one (i.e. $d\theta=0$), then the space $H^{2n}=\left(\widetilde{TM},\G, \F\right)$ is almost K\"ahlerian.

But, using the formulas (2.8.3) and (2.4.9) one obtains:
\begin{equation}
\begin{array}{l}
d\theta=\dfrac{1}{3!}(R_{ijk}+R_{jki}+R_{kij})dx^i\wedge dx^j \wedge dx^k+  \\
 \ \ \ \ \ \ \ \ \ \ \ \ \ \ \ \ \ \ \ \ \ \ \ \  +\dfrac{1}{2} (g_{is}B^s{}_{jk}-g_{js}B^s{}_{ik})\delta y^k\wedge dx^j\wedge dx^i.
\end{array}
\end{equation}

Therefore we deduce
\begin{theorem}
The space $H^{2n}=\left(\widetilde{TM},\G, \F \right)$ is almost K\"ahlerian if, and only if the following equations hold:
\begin{equation}
R_{ijk}+R_{jki}+R_{kij}=0;\ \ \ \ g_{is}B^s{}_{jk}-g_{js}B^s{}_{ik}=0.
\end{equation}
\end{theorem}

The space $H^{2n}=\left(\widetilde{TM},\G, \F \right)$ is called the almost Hermitian model of the Finslerian Mechanical System $\Sigma_F$.

\medskip

\noindent {\bf Remark.} We can study the space $H^{2n}=\left( \widetilde{TM},\G, \P \right)$ by similar way.

One can use the model $H^{2n}=\left(\widetilde{TM},\G, \F \right)$ to study the geometrical theory of Finslerian Mechanical system $\Sigma_F=\left( M, F(x,y), Fe(x,y)\right)$. For instance, the Einstein equations of the pseudo Riemannian space $(\widetilde{TM}, \G)$ can be considered as the Einstein equations of the Finslerian Mechanical system $\Sigma_F$.

\medskip

\noindent {\it Remark.} G.S. Asanov showed, \cite{As}, that the metric $\G$ given by the lift (2.8.1) does not satisfy the principle of the Post-Newtonian calculus. This is due to the fact that the horizontal and vertical terms of the metric $\G$ do not have the same physical dimensions. This is the reason for the author to introduce a new lift, \cite{MHS}, of the fundamental tensor $g_{ij}(x,y)$ of $\Sigma_F$ that can be used in a gauge theory. This lift is similar with that introduced in section 2.1, adapted to $\Sigma_R$.

It is expressed by
\begin{equation}
\widetilde{G}(x,y)=g_{ij}(x,y)dx^i\otimes dx^j+\dfrac{a^2}{F^2}g_{ij}(x,y)\delta y^i\otimes \delta y^j,
\end{equation}
where $\delta y^i=\stackrel{\circ}{\delta}\! y^i+\dfrac{1}{4}\dfrac{\partial F^i}{\partial y^j}dx^j$ and $a>0$ is a constant imposed by applications. This is to preserve the physical dimensions to the both terms of $\widetilde{G}$.

Let us consider also the tensor field on $\widetilde{TM}$
\begin{equation}
\widetilde F=-\dfrac{F}{a}\dfrac{\partial}{\partial y^i}\otimes dx^i+\dfrac{a}{F} \dfrac{\delta}{\delta x^i}\otimes \delta y^i
\end{equation}
and the $2$-form
\begin{equation}
\widetilde{\theta}=\dfrac{a}{F}\theta.
\end{equation}

Then, we have:
\begin{theorem}${}$

$1^{\circ}$ The triple $\left(\widetilde{TM},\widetilde{\G},\widetilde{\F}\right)$ is an almost Hermitian space.

$2^{\circ}$ The $2$-form $\widetilde{\theta}$ given by (2.8.8) is the almost symplectic structure determined by $\left(\widetilde{\G},\widetilde{\F}\right)$.

$3^{\circ}$ $\widetilde{\theta}$ is conformal to $\theta$.
\end{theorem}

The space $H^{2n}=\left(\widetilde{TM},\widetilde{\G},\widetilde{\F}\right)$ can be used to study the geometrical theory of the Finslerian mechanical system $\Sigma_F$, too.

\newpage

\chapter{Lagrangian Mechanical systems}\index{Analytical Mechanics of! Lagrangian systems}

A natural extension of the notion of the Finslerian mechanical system is that of the Lagrangian mechanical system. It is defined as a triple $\Sigma_L=(M,L(x,y),Fe(x,y))$ where: $M$ is a real $n-$dimensional $C^\9$ manifold called the configuration space; $L(x,y)$ is a regular Lagrangian with the property that the pair $L^n=(M,L(x,y))$ is a Lagrange space; $Fe(x,y)$ is an a priori given vertical vector field on the velocity space $TM$ called the external forces. The number $n$ is the number of freedom degree of $\Sigma_L$. The equations of evolution, or fundamental equations of Lagrangian mechanical system $\Sigma_L$ are the Lagrange equations:
$$\dfrac{d}{dt}\dfrac{\pp L}{\pp y^i}-\dfrac{\pp L}{\pp x^i}=F_i(x,y),\ \ y^i=\dfrac{dx^i}{dt}\ \ (i=1,2,...,n)\leqno(I)$$ where $F_i$ are the covariant components of the external forces $Fe$. These equations determine the integral curves of a canonical semispray $S$. Thus, the geometrical theory of the semispray $S$ is the geometrical theory of the Lagrangian mechanical system $\Sigma_L$. The Lagrangian $L(x,y)$ and the external forces $Fe(x,y)$ do not explicitly depend on the time $t$. Therefore $\Sigma_L$ is a scleronomic Lagrangian mechanical system.

\setcounter{section}{0}
\section{Lagrange Spaces. Preliminaries}
\setcounter{equation}{0}\setcounter{theorem}{0}\setcounter{definition}{0}\setcounter{proposition}{0}\setcounter{remark}{0}

Let $L^n=(M,L(x,y))$ be a Lagrange space (see ch. 2, part I). $L:TM\to R$ being a regular Lagrangian for which
\begin{equation}
g_{ij}(x,y)=\dfrac12\dot{\pp}_i\dot{\pp}_j L(x,y)
\end{equation}
is the fundamental tensor. Thus $g_{ij}(x,y)$ is a covariant of order 2 symmetric $d-$tensor field of constant signature and non singular:
\be
\det(g_{ij}(x,y))\neq 0\ {\rm{on\ }}\wt{TM}=TM\setminus\{0\}.
\ee

The Euler-Lagrange equations of the space $L^n$ are:
\begin{equation}
\dfrac{d}{dt}\dfrac{\pp L}{\pp y^i}-\dfrac{\pp L}{\pp x^i}=0,\ y^i=\dfrac{dx^i}{dt}.
\end{equation}

As we know, the system of differential equations (3.1.3) can be written in the equivalent form:
\be
\dfrac{d^2 x^i}{dt^2}+2\overset{\circ}{G^i}\(x,\dfrac{dx}{dt}\)=0
\ee
with
\be
2\overset{\circ}G^i=\dfrac12 g^{is}\{(\dot{\pp}_s\pp_h L)y^h-\pp_s L\}\tag{3.1.4'}
\ee
$\(\dot{\pp}_i=\dfrac{\pp}{\pp y^i},\pp_i=\dfrac{\pp}{\pp x^i}\)$.

$\overset{\circ}G^i$ are the coefficients of a semispray $\overset{\circ}S$:
\be
\overset{\circ}{S}=y^i\pp_i-2\overset{\circ}G^i(x,y)\dot{\pp}_i
\ee
and $\overset{\circ}S$ depend on the space $L^n$, only.

$\oci{S}$ is called the canonical semispray of Lagrange space $L^n$.

The energy of the Lagrangian $L(x,y)$ is:
\be
{\cal E}_L=y^i\dot{\pp}_i L-L.
\ee

In the chapter 2, part I, we prove the Law of conservation energy:
\begin{theorem}
Along the solution curves of Euler-Lagrange equations $(3.1.3)$ the energy ${\cal E}_L$ is conserved.
\end{theorem}

The following properties hold:
\begin{itemize}
  \item[$1^\circ$] The canonical nonlinear connection $\oci{N}$ of Lagrange space $L^n$ has the coefficients
  \be
  \oci{N}^i_j=\pp_j\oci{G}^i
  \ee
  \item[$2^\circ$] $\oci{N}$ determine a differentiable distribution on the velocity space $TM$ supplementary to the vertical distribution $V$:
  \be
  T_u TM=\oci{N}_u\oplus V_u,\ \ \forall u\in TM.
  \ee
  \item[$3^\circ$] An adapted basis to (3.1.8) is $(\oci{\de}_i,\dot{\pp}_i)_u$, $(i=1,...,n)$, where
  \be
  \oci{\de}_i=\pp_i-\oci{N}^j_i(x,y)\dot{\pp}_j
  \ee
  and an adapted cobasis $(dx^i,\oci{\de}y^i)_u$, $(i=1,...,n)$ with
  \be
  \oci{\de}y^i=dy^i+N^i_j(x,y)dx^j\tag{3.1.9'}
  \ee
  \item[$4^\circ$] The canonical metrical $\oci{N}-$connection $C\Gamma(\oci{N})=(\oci{L}^i_{jk},\oci{C}^i_{jk})$ has the coefficients expressed by the generalized Christoffel symbols:
      \be
      \begin{array}{l}
      \oci{L}^i_{jh}=\dfrac12 g^{is}(\oci{\de}_j g_{sh}+\oci{\de}_h g_{js}-\oci{\de}_s g_{jh})\\ \\
      \oci{C}^i_{jh}=\dfrac12 g^{is}(\dot{\pp}_j g_{sh}+\dot{\pp}_h g_{js}-\dot{\pp}_s g_{jh})
      \end{array}
      \ee
  \item[$5^\circ$] The Cartan 1-form of $L^n$ is
  \be
  \oci{\oo}=\dfrac12(\dot{\pp}_i L)dx^i
  \ee
  \item[$6^\circ$] The Cartan-Poincar\'{e} 2-form of $L^n$ is
  \be
  \oci{\theta}=d\oci{\oo}=g_{ij}(x,y)\oci{\de}y^i\wedge dx^j
  \ee
  \item[$7^\circ$] The 2-form $\oci{\theta}$ determine a symplectic structure on the velocity manifold $TM$.
\end{itemize}

\medskip

\noindent{\bf Example 3.1.1.} The function
\be
L(x,y)=mc\g_{ij}(x)y^i y^j+\dfrac{2e}{m}A_i(x)y^i+{\cal U}(x)
\ee
with $m,c,e$ the known physical constants, $\g_{ij}(x)$ are the gravitational potentials, $\g_{ij}(x)$ being a Riemannian metric, $A_i(x)$ are the electromagnetic potentials and ${\cal U}(x)$ is a potential function. Thus, $L^n=(M,L(x,y))$ is a Lagrange space. It is the Lagrange space of electrodynamics (cd. ch. 2, part I).

\section{Lagrangian Mechanical systems, $\Sigma_L$}
\setcounter{equation}{0}\setcounter{theorem}{0}\setcounter{definition}{0}\setcounter{proposition}{0}\setcounter{remark}{0}

\begin{definition}
A Lagrangian mechanical system is a triple:
\be
\Sigma_L=(M,L(x,y),Fe(x,y)),
\ee
where $L^n=(M,L(x,y))$ is a Lagrange space, $Fe(x,y)$ is an a priori given vertical vector field:
\be
Fe(x,y)=F^i(x,y)\dot{\pp}_i,
\ee
the number $n=\dim M$ is the number of freedom degree of $\Sigma_L$; the manifold $M$ is real $C^\9$ differentiable manifold called the configuration space; $Fe(x,y)$ is the external forces and $F^i(x,y)$ are the contravariant components of $Fe$. Of course, $F^i(x,y)$ is a vertical vector field and
\be
F_i(x,y)=g_{ij}(x,y)F^j(x,y)
\ee
are the covariant components of $Fe$. $F_i(x,y)$ is a $d-$covector field.

The 1-form:
\be
\sigma=F_i(x,y)dx^i
\ee
is a vertical 1-form on the velocity space $TM$.
\end{definition}

The Riemannian mechanical systems $\Sigma_{\cal R}$ and the Finslerian mechanical systems $\Sigma_F$ are the particular Lagrangian mechanical systems $\Sigma_L$.

\begin{remark}
The notion of Lagrangian mechanical system $\Sigma_L$ is different of that of ``Lagrangian mechanical system'' defined by Joseph Klein \cite{Klein1}, which is a Riemannian conservative mechanical system.
\end{remark}

The fundamental equations of $\Sigma_L$ are an extension of the Lagrange equations of a Finslerian mechanical system $\Sigma_F$, from the previous chapter.

So, we introduce the following Postulate:

\medskip

\noindent{\bf Postulate.} {\it The evolution equations of the Lagrangian mechanical system $\Sigma_L=(M,L,Fe)$ are the following Lagrange equations:
\be
\dfrac{d}{dt}\(\dfrac{\pp L}{\pp y^i}\)-\dfrac{\pp L}{\pp x^i}=F_i(x,y),\ y^i=\dfrac{dx^i}{dt}.
\ee}

But the both members of the Lagrange equations (3.2.5) are $d-$co\-vec\-tors. Consequently, we have:

\begin{theorem}
The Lagrange equations $(3.2.5)$ of a Lagrangian mechanical system $\Sigma_L=(M,L,Fe)$ have a geometrical meaning.
\end{theorem}

\begin{theorem}
The trajectories without external forces of the Lagrangian mechanical system $\Sigma_L=(M,L,Fe)$ are the geodesics of the Lagrange space $L^n=(M,L)$.
\end{theorem}

Indeed, $F_i(x,y)\equiv 0$ implies the previous affirmations.

For a regular Lagrangian $L(x,y)$ (in this case (3.1.2) holds), the Lagrange equations (3.2.5) are equivalent to the second order differential equations (SODE):
\be
\dfrac{d^2 x^i}{dt^2}+2\oci{G}^i\(x,\dfrac{dx}{dt}\)=\dfrac12 F^i\(x,\dfrac{dx}{dt}\),
\ee
where the functions $\oci{G}^i(x,y)$ are the local coefficients (3.1.4') of canonical semispray $\oci{S}$ of Lagrange space $L^n=(M,L(x,y))$.

The equations (3.2.6) are called fundamental equations of $\Sigma_L$, too.

\section{The evolution semispray of $\Sigma_L$}\index{Semisprays! of mechanical system $\Sigma_L$}
\setcounter{equation}{0}\setcounter{theorem}{0}\setcounter{definition}{0}\setcounter{proposition}{0}\setcounter{remark}{0}

The Lagrange equations (3.2.6) determine a semispray $S$ which depend on the Lagrangian mechanical system $\Sigma_L$, only.

Indeed, the vector field on $TM$:
\be
S=y^i\pp_i-2G^i(x,y)\dot{\pp}_i
\ee
with
\be
2G^i(x,y)=2\oci{G}^i(x,y)-\dfrac12 F^i(x,y)
\ee
has the property $JS=\C$. So it is a semispray depending only on $\Sigma_L$.

\begin{theorem}[Miron]
For a Lagrangian mechanical system $\Sigma_L$ the following properties hold:
\begin{itemize}
  \item[$1^\circ$] The semispray $S$ is given by
  \be
  S=\oci{S}+\dfrac12 Fe\tag{3.3.1'}
  \ee
  \item[$2^\circ$] $S$ is a dynamical system on the velocity space $TM$.
  \item[$3^\circ$] The integral curves of $S$ are the evolution curves $(3.2.6)$ of $\sil$.
\end{itemize}
\end{theorem}

In fact, $1^\circ$ derives from the formulas (3.3.1) and (3.3.2).\newline $2^\circ$ $S$ being a vector field on $TM$, compatible with the geometric structure of $TM$, (i.e. $JS=\C$), it is a dynamical system on the manifold $TM$. \newline $3^\circ$ The integral curves of $S$ are determined by the system of differential equations
\be
\dfrac{dx^i}{dt}=y^i,\ \ \dfrac{dy^i}{dt}+2G^i(x,y)=0.
\ee
By means of the expression (3.3.2) of the coefficients $G^i(x,y)$ the system (3.3.4) is coincident to (3.2.6).

The vector field $S$ is called the {\it evolution} (or {\it canonical}) semispray of the Lagrangian mechanical system $\sil$. Exactly as in the case of Riemannian or Finslerian mechanical systems, one can prove:

\begin{theorem} [\cite{BuMi}]
The evolution semispray $S$ is the unique vector field, on the velocity space $TM$, solution of the equation
\be
i_S\oci{\theta}=-d{\cal E}_L+\sigma
\ee
with ${\cal L}$ from (3.1.6) and $\sigma$ from $(3.2.4)$.
\end{theorem}

But $S$ being a solution of the previous equations, we get: $$S(\sil)=d{\cal E}_L(S)=\sigma(S)=F_i y^i=g_{ij}F^i y^j.$$ So, we have:

\begin{theorem}
The variation of energy ${\cal E}_L$ along the evolution curves of mechanical system $\sil$ is given by
\be
\dfrac{d{\cal E}_L}{dt}=F_i\(x,\dfrac{dx}{dt}\)\dfrac{dx^i}{dt}.
\ee
\end{theorem}

The external forces field $Fe$ is called dissipative if $g(\C,Fe)=g_{ij}F^i y^j\leq 0$. Thus, the previous theorem implies:

\begin{theorem}
The energy of Lagrange space $L^n=(M,L)$ is decreasing along the evolution curves of the mechanical system $\sil$ if and only if the external forces field $Fe$ is dissipative.
\end{theorem}

Evidently, the semispray $S$ being a dynamical system on the velocity space $TM$ it can be used for study the important problems, as the stability of evolution curves of $\sil$, the equilibrium points etc.

\section{The evolution nonlinear connection of $\sil$}\index{Connections! Nonlinear}
\setcounter{equation}{0}\setcounter{theorem}{0}\setcounter{definition}{0}\setcounter{proposition}{0}\setcounter{remark}{0}

The geometrical theory of the Lagrangian mechanical system $\sil$ is based on the evolution semispray $S$, with the coefficients $G^i(x,y)$ expressed in formula (3.3.2). Therefore, all notions or properties derived from $S$ will be considered belonging to the mechanical system $\sil$. Such that, the {\it evolution nonlinear connection} $N$ of $\sil$ is characterized by the coefficients:
\be
N^i_j(x,y)=\dot{\pp}_j G^i(x,y)=\dot{\pp}\oci{G}(x,y)-\dfrac14\dot{\pp}F^i(x,y)=\oci{N}^i_j(x,y)-\dfrac14\dot{\pp}_j F^i(x,y).
\ee

Since $\dot{\pp}_j F_i$ is a $d-$tensor field, consider its symmetric and skewsymmetric parts:
\be
P_{ij}=\dfrac12 (\dot{\pp}_j F_i+\dot{\pp}_i F_j),\ \ F_{ij}=\dfrac12(\dot{\pp}_j F_i-\dot{\pp}_i F_j)
\ee
The $d-$tensor $F_{ij}$ is the {\it elicoidal tensor} field of $\sil$, \cite{miranabuc}.

The evolution nonlinear connection $N$ allows to determine the dynamic derivative of the fundamental tensor $g_{ij}(x,y)$ of $\sil$:
\be
g_{ij|}=S(g_{ij})-g_{i|s}N^s_j-g_{sj}N^s_i.
\ee

It is not difficult to prove the following formula:
\be
g_{ij|}=\dfrac12 P_{ij}.
\ee
$N$ is metric nonlinear connection if $g_{ij|}=0$. So, we have

\begin{proposition}
The evolution nonlinear connection $N$ is metric if, and only if the $d-$tensor field $P_{ij}$ vanishes.
\end{proposition}

The adapted basis $(\de_i,\dot{\pp}_i)$ to the distributions $N$ and $V$ has the operators $\de_i$ of the form:
\be
\de_i=\pp_i-N^j_i\dot{\pp}_j=\oci{\de}_i+\dfrac14\dot{\pp}_i F^s\dot{\pp}_s.
\ee
The dual adapted cobasis $(dx^i,\de y^i)$ has 1-forms $\de y^i$:
\be
\de y^i=dy^i+N^i_j dx^j=\oci{\de}y^i-\dfrac14\dot{\pp}_j F^i dx^j\tag{3.4.5'}
\ee

In the adapted basis the evolution semispray $S$ has the expression:
\be
S=y^i\dfrac{\de}{\de x^i}-\[2\oci{G}^i-\oci{N}^i_j y^j-\dfrac12\(F^i-\dfrac12\dot{\pp}_j F^i y^j\)\]\dfrac{\pp}{\pp y^i}.\tag{3.4.5''}
\ee

Consequences:
\begin{itemize}
  \item[$1^\circ$] $S$ cannot be a vertical vector field.
  \item[$2^\circ$] If the coefficients $\oci{G}^i$ are 2-homogeneous with respect to the vertical variables $y^i$ and the contravariant components $F^i(x,y)$ are 2-homogeneous in $y^i$, then the canonical semispray $S$ belongs to the horizontal distribution $N$.
\end{itemize}

The tensor of integrability of the distribution $N$ is
\be
R^i_{jh}=\de_h N^i_j-\de_j N^i_h=(\de_h\dot{\pp}_j-\de_j\dot{\pp}_h)\(\oci{G}^i-\dfrac14 F^i\)
\ee
Thus: $N$ is integrable iff $R^i_{jh}=0$.

The $d-$tensor of torsion of the nonlinear connection $N$ vanishes.

Indeed:
\be
t^i_{jh}=\dot{\pp}_h N^i_j-\dot{\pp}_j N^i_h=0.
\ee

The Berwald connection $B\Gamma(N)=(B^i_{jk},0)$ of the canonical nonlinear connection $N$ has the coefficients
\be
B^i_{jh}=\oci{B}^i_{jh}-\dfrac14\dot{\pp}_j\dot{\pp}_h F^i,
\ee
where $(\oci{B}^i_{jh},0)$ are the coefficients of Berwald connection of Lagrange space $L^n$.

It follows that $B\Gamma(N)$ is symmetric: $B^i_{jh}-B^i_{hj}=t^i_{jh}=0$.

In the following section we need:

\begin{lemma}
The exterior differential of 1-forms $\de y^i$ are given by
\be
d\de y^i=\dfrac12 R^i_{jh}dx^h\wedge dx^j+ B^i_{jh}\de y^j\wedge dx^h.
\ee
\end{lemma}

Indeed, from (3.4.5') we have:
$$d\de y^i=dN^i_j\wedge dx^j=\de_h N^i_j dx^h\wedge dx^j+\dot{\pp}_h N^i_j\de y^h\wedge dx^j,$$
which are exactly (3.4.9).

By means of the formula (3.4.5') one gets:

\begin{theorem}
The autoparallel curves of the canonical nonlinear connection $N$ are given by the differential system of equations:
\be
y^i=\dfrac{dx^i}{dt},\ \dfrac{\de y^i}{dt}=\dfrac{\oci{\de}y^i}{dt}-\dfrac14\dot{\pp}_j F^i\dfrac{dx^j}{dt}=0.
\ee
\end{theorem}

Evidently, if $Fe=0$, then the canonical nonlinear connection $N$ coincides with the canonical nonlinear connection $\oci{N}$ of the Lagrange space $L^n$.

In particular, if $\sil$ is a Finslerian mechanical system $\Sigma_F$, then the previous theory reduces to that studied in the previous chapter.

\section{Canonical $N-$metrical connection of $\sil$. Structure equations}\index{Connections! $N-$linear}
\setcounter{equation}{0}\setcounter{theorem}{0}\setcounter{definition}{0}\setcounter{proposition}{0}\setcounter{remark}{0}

Taking into account the results from part I, the canonical $N-$metrical connection $C\Gamma(N)$ is characterized by:

\begin{theorem}
The local coefficients $D\Gamma(N)=(L^i_{jk},C^i_{jk})$ of the canonical $N-$metrical connection of Lagrangian mechanical system $\sil$ are given by the generalized Christoffel symbols:
\be
\begin{array}{l}
L^i_{jh}=\dfrac12 g^{is}(\de_j g_{sh}+\de_h g_{sj}-\de_s g_{jk})\\ \\
C^i_{jh}=\dfrac12 g^{is}(\dot{\pp}_j g_{sh}+\dot{\pp}_h g_{sj}-\dot{\pp}_s g_{jh}).
\end{array}
\ee
\end{theorem}

Using the expression (3.4.5) of the operator $\de_k$ and developing the terms $\de_h g_{ij}$ from (3.5.1), we get:

\begin{theorem}
The coefficients $L^i_{jh},C^i_{jh}$ of $D\Gamma(N)$ can be set in the form
\be
\begin{array}{l}
L^i_{jk}=\oci{L}^i_{jk}+\check{C}^i_{jk},\ \ C^i_{jk}=\oci{C}^i_{jk},\\ \\
\check{C}^i_{jk}=\dfrac14 g^{is}\(\oci{C}_{skh}\dot{\pp}F^h+\oci{C}_{jsh}\dot{\pp}F^h-\oci{C}_{jkh}\dot{\pp}F^h\),
\end{array}
\ee
where $C\Gamma(\oci{N})=(\oci{L}^i_{jk},\oci{C}^i_{jk})$ is the canonical $N-$metrical connection of the Lagrange space $L^n=(M,L(x,y))$.
\end{theorem}

Let $\oo^i_j$ be the connection 1-forms of the canonical $N-$metrical connection $D\Gamma(N)$:
\be
\oo^i_j=L^i_{jk}dx^k+C^i_{jk}\de y^k.
\ee

Then, we have:

\begin{theorem}
The structure equations of canonical $N-$metrical connection $D\Gamma(N)$ of Lagrangian mechanical system $\sil$ are given by:
\be
\begin{array}{l}
d(dx^i)-dx^k\wedge \oo^i_k=-\overset{1}\O^i\\ \\
d(\de y^i)-\de y^k\wedge \oo^i_k=-\overset{2}\O^i\\ \\
d\oo^i_j-\oo^k_j\wedge\oo^i_k=-\O^i,
\end{array}
\ee
where the 1-form of torsions $\overset{1}\Omega^i$ and $\overset{2}\Omega^i$ are:
\be
\begin{array}{l}
\overset{1}\O^i=C^i_{jk}dx^j\wedge \de y^k\\ \\
\overset{2}\O^i=\dfrac12 R^i_{jk} dx^j\wedge dx^k+P^i_{jk} dx^j\wedge \de y^k
\end{array}
\ee
and the 2-forms of curvature $\O^i_j$ is as follows:
\be
\O^i_j=\dfrac12 R^i_{jkh}dx^k\wedge dx^h+P^{\ i}_{j\ kh}dx^k\wedge \de y^h+\dfrac12 S^i_{j\ kh}\de y^k\wedge \de y^h
\ee
where the $d-$tensors of torsions are:
\be
\begin{array}{l}
T^i_{jk}=L^i_{jk}-L^i_{kj}=0;\ S^i_{jk}=C^i_{jk}-C^i_{kj}=0,\\ \\
C^i_{jk}=\oci{C}^i_{jk},\ \ R^i_{jk}=\de_k N^i_j-\de_j N^i_k,\ \ P^i_{jk}=B^i_{jk}-L^i_{jk}
\end{array}
\ee
and the $d-$tensors of curvature are:
\be
\begin{array}{l}
R^{\ i}_{j\ kh}=\de_k L^i_{jk}-\de_k L^i_{jh}+L^s_{jk}L^i_{sh}-L^s_{jh}L^i_{sk}+C^i_{js}R^s_{kh}\\ \\
P^{\ i}_{j\ kh}=\dot{\pp}_h L^i_{jk}-C^i_{jh|k}+C^i_{js}P^s_{\ kh}\\ \\
S^{\ i}_{j\ kh}=\dot{\pp}_h C^i_{jk}-\dot{\pp}_k C^i_{jh}+C^s_{jk}C^i_{sh}-C^s_{jh}C^i_{sk}.
\end{array}
\ee
Here $C^i_{jk|h}$ is the $h-$covariant derivative of the tensor $C^i_{jk}$ with respect to $D\Gamma(N)$.
\end{theorem}

So, we have
\be
\begin{array}{l}
g_{ij|h}=\de_h g_{ij}-g_{sj}L^s_{ih}-g_{is} L^s_{jh}=0,\\ \\
g_{ij}|_h=\dot{\pp}_h g_{ij}-g_{sj}C^s_{ih}-g_{is}C^s_{jh}=0.
\end{array}
\ee
Applying the Ricci identities to the fundamental tensor $g_{ij}$, it is not difficult to prove the identities
\be
R_{ijkh}+R_{jikh}=0,\ P_{ijkh}+P_{jikh}=0,\ S_{ijkh}+S_{jikh}=0,
\ee
where $R_{ijkh}=g_{js}R_{i\ kh}^{\ s}$, etc.

Also, $P_{ijk}=g_{is}P^s_{\ kh}$ is totally symmetric.

Taking into account the form (3.5.2) of the coefficients of $D\Gamma(N)$, the calculus of $d-$tensors of torsion (3.5.7) and of $d-$tensors of curvature (3.5.6) can be obtained.

The exterior differentiating the structure equations (3.5.4) modulo the same system of equations (3.5.4) and using Lemma (3.4.1) one obtains the Bianchi identities of the $N-$metrical connection $D\Gamma(N)$.

The $h-$ and $v-$covariant derivative of Liouville vector field $\C=y^i\dfrac{\pp}{\pp y^i}$ lead to the deflection tensors of $D\Gamma(N)$:
\be
D^i_j=y^i_{|j}=\de_j y^i+y^s L^i_{sj},\ d^{i}_{\ j}=y^i |_j=\dot{\pp}_j y^i+y^s C^{\ i}_{s\ j}.
\ee

\begin{theorem}
The expression of $h-$tensor of deflection $D^i_j$ and $v-$ ten\-sor of deflection $d^i_j$ are given by
\be
D^i_{\ j}=y^s L^i_{sj}-N^i_j,\ \ d^i_{\ j}=\de^i_{\ j}+y^s \oci{C}^{\ i}_{s\ j}.
\ee
\end{theorem}

Finally, we determine the horizontal paths, vertical paths of $D\Gamma(N)$ and its autoparallel curves.

The {\it horizontal paths} of metrical connection $D\Gamma(N)=(L^i_{jk},C^i_{jk})$ are given by
\be
\dfrac{d^2 x^i}{dt^2}+L^i_{jk}\(x,\dfrac{dx}{dt}\)\dfrac{dx^j}{dt}\dfrac{dx^k}{dt}=0,\ \ y^i=\dfrac{dx^i}{dt},
\ee
where the coefficients $L^i_{jk}$ are expressed by formula (3.5.2).

In the initial conditions $x^i=x^i_0$, $y^i=\(\dfrac{dx^i}{dt}\)_0$, local, this system of differential equation has a unique solution $x^i=x^i(t)$, $y^i=y^i(t)$, $t\in I$.

The {\it vertical paths}, at a point $(x^i_0)\in M$ are characterized by system of differential equations:
\be
\dfrac{dy^i}{dt}+\oci{C}^i_{jk}(x_0,y)y^j y^k=0,\ \ x^i=x^i_0.
\ee
In the initial conditions $x^i=x^i_0, y^i=y^i_0$ the previous system of differential equations, locally, has a unique solution.

\section{Electromagnetic field}\index{Electromagnetism and gravitational fields}
\setcounter{equation}{0}\setcounter{theorem}{0}\setcounter{definition}{0}\setcounter{proposition}{0}\setcounter{remark}{0}

For Lagrangian mechanical system $\sil=(M,L,Fe)$ the electromagnetic phenomena appear because the deflection tensor $D^i_j$ nonvanishes.

The covariant components $D_{ij}=g_{is}D^s_j$, $d_{ij}=g_{is}d^s_j$ determine $h-$e\-lec\-tro\-magnetic field ${\cal F}_{ij}$ and $v-$electromagnetic field $f_{ij}$ as follows:
\be
{\cal F}_{ij}=\dfrac12(D_{ij}-D_{ji}),\ f_{ij}=\dfrac12(d_{ij}-d_{ji}).
\ee

Taking into account the formula (3.5.11) one obtains:
\be
{\cal F}_{ij}=\oci{\cal F}_{ij}+\dfrac14 F_{ij}+\dfrac14\check{F}_{ij}
\ee
where $\oci{\cal F}_{ij}$ is the electromagnetic tensor field of the Lagrange space $L^n$, $F_{ij}$ is the elicoidal tensor field of mechanical system $\sil$ and $\check{F}_{ij}$ is the skewsymmetric tensor
\be
\check{\cal F}_{ij}=\{\dot{\pp}_j F^s C_{irs}-\dot{\pp}_i F^r C_{jrs}\}y^s.
\ee

In the case of Finslerian mechanical system $\Sigma_F$ the electromagnetic tensor ${\cal F}_{ij}$ is equal to $\dfrac14 F_{ij}$.

Applying the method given in the chapter 2, part I, one gets the Maxwell equations of the Lagrangian mechanical system $\sil$.

\section{The almost Hermitian model of the Lagrangian mechanical system $\sil$}\index{Almost Hermitian model}
\setcounter{equation}{0}\setcounter{theorem}{0}\setcounter{definition}{0}\setcounter{proposition}{0}\setcounter{remark}{0}

Let $N$ be the canonical nonlinear connection of the mechanical system $\sil$ and $(\de_i,\dot{\pp}_i)$ the adapted basis to the distribution $N$ and $V$. Its dual basis is $(dx^i,\de y^i)$. The $N-$lift $\G$ of fundamental tensor $g_{ij}$ on the velocity space $\wt{TM}=TM\setminus\{0\}$ is given by
\be
\G=g_{ij}dx^i\otimes dx^j+g_{ij}\de y^i\otimes \de y^j.
\ee

The almost complex structure $\F$ determined by the nonlinear connection $N$ is expressed by
\be
\F=\de_i\otimes dx^j-\dot{\pp}_i\otimes dx^i.
\ee
Thus, one proves:

\begin{theorem}
We have:
\begin{itemize}
  \item[$1^\circ$] $\G$ is a pseudo-Riemannian structure and $\F$ is an almost complex structure on the manifold $TM$. They depend only on the Lagrangian mechanical system $\sil$.
  \item[$2^\circ$] The pair $(\G,\F)$ is an almost Hermitian structure.
  \item[$3^\circ$] The associated 2-form of $(\G,\F)$ is given by $$\theta=g_{ij}\de y^i\wedge dx^j.$$
  \item[$4^\circ$] $\theta$ is an almost symplectic structure on the velocity space $\wt{TM}$.
  \item[$5^\circ$] The following equality holds: $$\theta=\oci{\theta}-\dfrac14 F_{ij}dx^i\wedge dx^j,$$ $\oci\theta$ being the symplectic structure of the Lagrange space $L^n$, i.e. $\oci{\theta}=g_{ij}\oci\de y^i\wedge dx^j$.
\end{itemize}
\end{theorem}

Since $\oci\theta$ is the symplectic structure of Lagrange space $L^n$, we have $d\oci\theta=0$. So the exterior differential $d\theta$ is as follows:
\be
d\theta=-\dfrac14 dF_{ij}\wedge dx^i\wedge dx^j.
\ee
But $dF_{ij}=\de_k F_{ij}dx^k+\dot{\pp}_k F_{ij}\de y^k$. Consequently one obtain from (2.2):
\be
d\theta=-\dfrac1{12}(F_{ij|k}+F_{jk|i}+F_{ki|j})dx^k\wedge dx^i\wedge dx^j-\dfrac14\dot{\pp}_k F_{ij}\de y^k\wedge dx^i\wedge dx^j.
\ee

Consequently:

$1^\circ$ $\theta=\oci{\theta}$, if and only if the helicoidal tensor $F_{ij}$ vanish.

$2^\circ$ The almost symplectic structure $\theta$ of the Lagrangian mechanical system $\sil$ is integrable, if and only if the helicoidal tensor $F_{ij}$ satisfies the following tensorial equations:
\be
F_{ij|k}+F_{jk|i}+F_{ki|j}=0,\ \ \dot{\pp}_k F_{ij}=0,
\ee
where the operator ``$_|$'' is $h-$covariant derivation with respect to canonical $N-$linear connection $D\Gamma(N)$.

Now we observe that some good applications of this theory can be done for the Lagrangian mechanical systems $\sil=(M,L(x,y),Fe(x,y))$ when the external forces $Fe$ are of Liouville type: $Fe=a(x,y)\C$.

\section{Generalized Lagrangian mechanical systems}
\setcounter{equation}{0}\setcounter{theorem}{0}\setcounter{definition}{0}\setcounter{proposition}{0}\setcounter{remark}{0}

As we know, a generalized Lagrange space $GL^n=(M,g_{ij})$ is given by the configuration space $M$ and by a $d-$tensor field $g_{ij}(x,y)$  on the velocity space $\wt{TM}=TM\setminus\{0\}$, $g_{ij}$ being symmetric, nonsingular and of constant signature.

There are numerous examples of spaces $GL^n$ given by Miron R. \cite{Miron}, Anastasiei M. \cite{AnShi}, R.G. Beil \cite{Beil2}, \cite{Beil3}, T. Kawaguchi \cite{MiKa} etc.

$1^\circ$ The $GL^n$ space with the fundamental tensor
\be
g_{ij}(x,y)=\a_{ij}(x)+y_i y_j,\ \ y_=\a_j y^j
\ee
and $\a_{ij}(x)$ a semidefinite Riemann tensor.

$2^\circ$ $g_{ij}(x,y)=\a_{ij}(x)+\(1-\dfrac1{n(x,y)}\)y_i y_j$, $y_i=\a_{ij}y^j$ where $n(x,y)\geq 1$ is a $C^\9$ function (the refractive index in Relativistic optics), \cite{MiKa}.

$3^\circ$ $g_{ij}(x,y)=\a_{ij}(x,y)+a y_i y_j$, $a\in R_+$, $y_i=\a_{ij}y^j$ and where $\a_{ij}(x,y)$ is the fundamental tensor of a Finsler space.

The space $GL^n=(M,g_{ij}(x,y))$ is called reducible to a Lagrange space if there exists a regular Lagrangian $L(x,y)$ such that
\be
\dfrac12\dfrac{\pp^2 L}{\pp y^i\pp y^j}=g_{ij}(x,y).
\ee

A necessary condition that the space $GL^n=(M,g_{ij}(x,y))$ be reducible to a Lagrange space is that the $d-$tensor field
\be
C_{ijk}=\dfrac12\dfrac{\pp g_{ij}}{\pp y^k}
\ee
be totally symmetric.

The following property holds:

\begin{theorem}
If we have a generalized Lagrange space $GL^n$ for which the fundamental tensor $g_{ij}(x,y)$ is 0-homogeneous with respect to $y^i$ and it is reducible to a Lagrange space $L^n=(M,L)$, then the function $L(x,y)$ is given by:
\be
L(x,y)=g_{ij}(x,y)y^i y^j+2A_i(x)y^i+{\cal U}(x).
\ee
\end{theorem}

The prove does not present some difficulties, \cite{miranabuc}.

Also, we remark that in some conditions the fundamental tensor $g_{ij}(x,y)$ of a space $GL^n$ determine a non linear connection, \cite{Szila}, \cite{Szila2}.

For instance, in the cases $1^\circ, 2^\circ, 3^\circ$ of the previous examples we can take: $1^\circ,2^\circ$, $\oci{N}^i_j=\dfrac12\g^{i}_{jk}(x)y^j y^k$, $\g^i_{jk}(x)$ being the Christoffel symbols of the Riemannian metric $\g_{ij}(x)$. For example $3^\circ$ we take $N^i_j=\dfrac12\g^i_{jk}(x,y)y^j y^k$ where $\g^i_{jk}(x,y)$ are the Christoffel symbols of $\g_{ij}(x,y)$ (see ch. 2, part I).

Let us consider the function ${\cal E}(x,y)$ on $\wt{TM}$:
\be
{\cal E}(x,y)=g_{ij}(x,y)y^i y^j
\ee
called the {\it absolute energy} \cite{mir11} of the space $GL^n$. The fundamental tensor $g_{ij}(x,y)$ is say to be weakly regular if the absolute energy ${\cal E}(x,y)$ is a regular Lagrangian. That means: the tensor $\check{g}_{ij}(x,y)=\dfrac12\dot{\pp}_i\dot{\pp}_j{\cal E}$ is nonsingular. Thus the pair $(M,{\cal E}(x,y))$ is a Lagrange space.

\begin{definition}
A generalized Lagrangian mechanical system is a triple $\Sigma_{GL}=(M,g_{ij},Fe)$, where $GL=(M,g_{ij})$ is a generalized Lagrange space and $Fe$ is the vector field of external forces. $Fe$ being a vertical vector field we can write:
\be
Fe=F^i(x,y)\dot{\pp}.
\ee
\end{definition}

In the following we assume that the fundamental tensor $g_{ij}(x,y)$ of $\Sigma_{GL}$ is weakly regular. Therefore we can give the following Postulate \cite{BuMi}:

\noindent{\bf Postulate.} {\it The evolution equations (or Lagrange equations) of a generalized Lagrange mechanical system $\Sigma_{GL}$ are:
\be
\dfrac{d}{dt}\dfrac{\pp{\cal E}}{\pp y^i}-\dfrac{\pp{\cal E}}{\pp x^i}=F_i(x,y),\ \ y^i=\dfrac{dx^i}{dt}
\ee
with the covariant components of $\F e$:}
\be
F_i(x,y)=g_{ij}(x,y) F^j(x,y).
\ee

The Lagrange equations (3.8.7) are equivalent to the following system of second order differential equations:
\be
\left\{
\begin{array}{l}
\dfrac{d^2 x^i}{dt^2}+2\check{G}^i\(x,\dfrac{dx}{dt}\)=\dfrac12 F^i\(x,\dfrac{dx}{dt}\)\\ \\
2\check{G}^i=\dfrac12 g^{is}\(\dfrac{\pp^2{\cal E}}{\pp y^s\pp x^j}y^j-\dfrac{\pp{\cal E}}{\pp x^s}\).
\end{array}
\right.
\ee
Therefore, one can apply the theory from this chapter for the Lagrangian mechanical system $\sil=(M,{\cal E}(x,y),Fe)$. So we have:

\begin{theorem}
\begin{itemize}
\item[$1^\circ$] The operator
\be
\check{S}=y^i\pp_i-2\(\check{G}^i-\dfrac12 F^i\)\dot{\pp}_i
\ee
is a semispray on the velocity space $\wt{TM}$.
\item[$2^\circ$] The integral curves of $S$ are the evolution curves of $\Sigma_{GL}$.
\item[$3^\circ$] $\check{S}$ is determined only by the generalized mechanical system $\Sigma_{GL}$.
\end{itemize}
\end{theorem}

\begin{theorem}
The variation of energy $E_{\cal E}=y^i\dfrac{\pp{\cal E}}{\pp y^i}-{\cal E}$ are given by
\be
\dfrac{dE_{\cal E}}{dt}=\dfrac{dx^i}{dt}F_i\(x,\dfrac{dx}{dt}\).
\ee
\end{theorem}

$Fe$ are called dissipative if $g_{ij}F^i y^j\leq 0$.

\begin{theorem}
The energy $E_{\cal E}$ is decreasing on the evolution curves $(3.8.9)$ of system $\Sigma_{GL}$ if and only if the external forces $Fe$ are dissipative.
\end{theorem}

Other results:
\begin{itemize}
  \item[$4^\circ$] The canonical nonlinear connection $\check{N}$ of $\Sigma_{G,L}$ has the local coefficients
  \be
  \check{N}^i_j=\dot{\pp}_j\oci{\G}^i-\dfrac14\dot{\pp}_j F^i.
  \ee
  \item[$5^\circ$] The canonical $\check{N}-$linear metric connection of mechanical system $\Sigma_{GL}$, $D\Gamma(\check{N})=(L_{jk},C_{jk})$ has the coefficients:
      \be
      \left\{
      \begin{array}{l}
      \check{L}^i_{jk}=\dfrac12 g^{is}\(\dfrac{\check{\de}g_{js}}{\de x^k}+\dfrac{\check{\de}g_{sk}}{\de x^j}-\dfrac{\check{\de}g_{jk}}{\de x^s}\)\\ \\
      \dot{C}^i_{jk}=\dfrac12 g^{is}\(\dfrac{\pp g_{js}}{\pp y^k}+\dfrac{\pp g_{sk}}{\pp y^j}-\dfrac{\pp g_{jk}}{\pp y^s}\).
      \end{array}
      \right.
      \ee
\end{itemize}

By using the geometrical object fields $g_{ij}$, $\check{S}$, $\check{N}$, $D\Gamma(\check{N})$ we can study the theory of Generalized Lagrangian mechanical systems $\Sigma_{G,L}$.

\newpage

\chapter{Hamiltonian and Cartanian mechanical systems}

The theory of Hamiltonian mechanical systems can be constructed step by step following the theory of Lagrangian mechanical system, using the differential geometry of Hamiltonian spaces expound in part I. The legality of this theory is proved by means of Legendre duality between Lagrange and Hamilton spaces. The Cartanian mechanical system appears as a particular case of the Hamiltonian mechanical systems. We develop here these theories using the author's papers \cite{mir4} and the book of R. Miron, D. Hrimiuc, H. Shimada and S. Sabau \cite{MHS}.

\setcounter{section}{0}
\section{Hamilton spaces. Preliminaries}
\setcounter{equation}{0}
\setcounter{definition}{0}
\setcounter{theorem}{0}
\setcounter{lemma}{0}
\setcounter{remark}{0}
\setcounter{proposition}{0}

Let $M$ be a $C^\9-$real $n-$dimensional manifold, called configuration space and $(T^*M,\pi^*,M)$ be the cotangent bundle, $T^*M$ is called momentum (or phase) space. A point $u^*=(x,p)\in T^*M$, $\pi^*(u^*)=x$, has the local coordinate $(x^i,p_i)$.

As we know from the previous chapter, a changing of local coordinate $(x,p)\to (\wt{x},\wt{p})$ of point $u^*$ is given by
\begin{equation}
\left\{
\begin{array}{l}
\wt{x}^i=\wt{x}^i(x^1,...,x^n),\ \det\(\dfrac{\pp\wt{x}^i}{\pp x^j}\)\neq 0\\ \\
\wt{p}_i=\dfrac{\pp x^j}{\pp\wt{x}^i}p_j
\end{array}
\right.
\end{equation}

The tangent space $T^*_{u^*}T^*M$ has a natural basis $\left(\hspace*{-0.7mm}\dfrac{\pp}{\pp x^i}=\pp_i, \dfrac{\pp}{\pp p_i}=\dot{\pp}^i\hspace*{-0.7mm}\right)$ with respect to (4.1.1) this basis is transformed as follows
\begin{equation}
\begin{array}{l}
\dfrac{\pp}{\pp x^i}=\dfrac{\pp\wt{x}^j}{\pp x^i}\dfrac{\pp}{\pp\wt{x}^j}+\dfrac{\pp\wt{p}_j}{\pp x^i}\dfrac{\pp}{\pp\wt{p}_j}\\ \\
\dot{\pp}^i=\dfrac{\pp x^i}{\pp\wt{x}^j}\wt{\dot{\pp}}^j
\end{array}
\end{equation}
Thus we have:

$1^\circ$ On $T^*M$ there are globally defined the Liouville 1-form:
\be
\check{p}=p_i dx^i
\ee
and the natural symplectic structure
\be
\oci{\theta}=d\check{p}=dp_i\wedge dx^i.
\ee

$2^\circ$ The vertical distribution $V$ has a local basis $(\dot{\pp}^1,...,\dot{\pp}^n)$. It is of dimension $n$ and is integrable.

$3^\circ$ A supplementary distribution $N$ to the distribution $V$ is given by a splitting:
\be
T_{u^*}T^*M=N_{u^*}\oplus V_{u^*},\ \ \forall u^*\in T^*M.
\ee
$N$ is called a horizontal distribution or a {\it nonlinear connection} on the momentum space $T^*M$. The dimension of $N$ is $n$.

$4^\circ$ A local adapted basis to $N$ and $V$ are given by $(\de_i,\dot{\pp}^i)$, $(i=1,...,n)$, with
\be
\de_i=\pp_i+N_{ji}\dot{\pp}^j.
\ee
The system of functions $N_{ji}(x,p)$ are the coefficients of the nonlinear connection $N$.

$5^\circ$
\be
t_{ij}=N_{ij}-N_{ji}
\ee
is a $d-$tensor on $T^*M$ - called the torsion tensor of the nonlinear connection $N$. If $t_{ij}=0$ we say that $N$ is a symmetric nonlinear connection.

$6^\circ$ The dual basis $(dx^i,\de p_i)$ of the adapted basis $(\de_i,\dot{\pp}^i)$ has 1-form $\de p_i$ expressed by
\be
\de p_i=dp_i-N_{ij}dx^j.
\ee

$7^\circ$ \begin{proposition}
If $N$ is a symmetric nonlinear connection then the symplectic structure $\oci{\theta}$ can be written:
\be
\oci\theta=\de p_i\wedge dx^i.
\ee

$8^\circ$ The integrability tensor of the horizontal distribution $N$ is
\be
R_{kij}=\de_i N_{kj}-\de_j N_{ki}
\ee
and the equations $R_{kij}=0$ gives the necessary and sufficient condition for integrability of the distribution $N$.
\end{proposition}

The notion of $N-$linear connection can be taken from the ch. 2, part I.

\begin{definition}
A Hamilton space is a pair $H^n=(M,H(x,p))$ where $H(x,p)$ is a real scalar function on the momentum space $T^*M$ having the following properties:
\begin{itemize}
  \item[$1^\circ$] $H:T^*M\to R$ is differentiable on $\wt{T^*M}=T^*M\setminus\{0\}$ and continuous on the null section of projection $\pi_*$.
  \item[$2^\circ$] The Hessian of $H$, with the elements
  \be
  g^{ij}=\dfrac12\dot{\pp}^i\dot{\pp}^j H
  \ee
  is nonsingular, i.e.
  \be
  \det(g^{ij})\neq 0\ {\rm{on\ }} \wt{T^*M}.
  \ee
  \item[$3^\circ$] The 2-form $g^{ij}(x,p)\eta_i\eta_j$ has a constant signature on $\wt{T^*M}$.
\end{itemize}
\end{definition}

In the ch. 2, part I it is proved the following Miron's result:

\begin{theorem}
$1^\circ$ In a Hamilton space $H^n=(M,H(x,p))$ for which $M$ is a paracompact manifold, there exist nonlinear connections determined only by the fundamental function $H(x,p)$.

$2^\circ$ One of the, $\oci{N}$, has the coefficients
\be
\oci{N}_{ij}=-\dfrac12 g_{jh}\[\dfrac14 g_{ik}\dot{\pp}^k\{H,\dot{\pp}^h H\}+\dot{\pp}^h\pp_i H\].
\ee

$3^\circ$ The nonlinear connection $\oci{N}$ is symmetric.
\end{theorem}

In the formula (4.1.13), $\{,\}$ is the Poisson brackets. $\oci{N}$ is called the canonical nonlinear connection of $H^n$.

The variational problem applied to the integral of action of $H^n$:
\be
I(c)=\dd\int^1_0\[p_i(t)\dfrac{dx^i}{dt}-{\cal H}(x(t),p(t))\]dt,
\ee
where
\be
{\cal H}(x,p)=\dfrac12 H(x,p)
\ee
leads to the following results:

\begin{theorem}
The necessary conditions as the functional $I(c)$ be an extremal value of the functionals $I(\wt{c})$ imply that the curve $c(t)=(x^i(t),p_i(t))$ is a solution of the Hamilton--Jacobi equations:
\be
\dfrac{dx^i}{dt}-\dfrac{\pp{\cal H}}{\pp p_i}=0,\ \ \dfrac{dp_i}{dt}+\dfrac{\pp{\cal H}}{\pp x^i}=0.
\ee

If $\oci{N}$ is the canonical nonlinear connection then the Hamilton--Jacobi equations can be written in the form
\be
\dfrac{dx^i}{dt}-\dfrac{\pp {\cal H}}{\pp p_i}=0,\ \ \dfrac{\de p_i}{dt}+\dfrac{\de {\cal H}}{\de x^i}=0.
\ee
\end{theorem}

Evidently, the equations (4.1.16) have a geometrical meaning.

The curves $c(t)$ which verify the Hamilton--Jacobi equations are called the extremal curves (or geodesics) of $H^n$.

Other important results:

\begin{theorem}
The fundamental function $H(x,p)$ of a Hamilton space $H^n=(M,H)$ is constant along to every extremal curves.
\end{theorem}

\begin{theorem}
The following properties hold:

$1^\circ$ For a Hamilton space $H^n$ there exists a vector field $\oci\xi\in{\cal X}(\wt{T^*M})$ with the property
\be
i_{\oci\xi}\oci\theta=-dH.
\ee

$2^\circ$ $\oci\xi$ is given by
\be
\oci\xi=\dfrac12(\dot{\pp}^i H\pp_i-\pp_i H\dot{\pp}^i).
\ee

$3^\circ$ The integral curve of $\oci\xi$ is given by the Hamilton--Jacobi equations $(4.1.16)$.
\end{theorem}

The vector field $\oci\xi$ from the formula (4.1.19) is called the Hamilton vector of the space $H^n$.

\begin{proposition}
In the adapted basis of the canonical nonlinear connection $\oci{N}$ the Hamilton vector $\oci\xi$ is expressed by
\be
\oci\xi=\dfrac12(\dot{\pp}^i H\de_i-\de_i H\dot{\pp}^i).
\ee
\end{proposition}

By means of the Theorem 4.1.4 we can say that the Hamilton vector field $\xi$ is a dynamical system of the Hamilton space $H^n$.

\section{The Hamiltonian mechanical systems}\index{Analytical Mechanics of! Hamiltonian systems}
\setcounter{equation}{0}
\setcounter{definition}{0}
\setcounter{theorem}{0}
\setcounter{lemma}{0}
\setcounter{remark}{0}
\setcounter{proposition}{0}

Following the ideas from the first part, ch. ..., we can introduce the next definition:
\begin{definition}
A Hamiltonian mechanical system is a triple:
\begin{equation}
\Sigma_H=(M,H(x,p),Fe(x,p)),
\end{equation}
where $H^n=(M,H(x,p))$ is a Hamilton space and
\begin{equation}
Fe(x,p)=F_i(x,p)\dot{\partial}^i
\end{equation}
is a given vertical vector field on the momenta space $T^*M$.

$Fe$ is called the external forces field.
\end{definition}

The evolution equations of $\Sigma_H$ can be defined by means of equations (4.1.16) from the variational problem.

\medskip

\noindent{\bf Postulate 4.2.1.} {\it The evolution equations of the Hamiltonian mechanical system $\Sigma_H$ are the following {\bf Hamilton equations}:
\begin{equation}
\dfrac{dx^i}{dt}-\dot{\partial}^i{
cal H}=0,\ \dfrac{dp_i}{dt}+\partial_i{\cal H}=\dfrac12 F_i(x,p), \ \ {\cal H}=\dfrac12 H.
\end{equation}}

Evidently, for $Fe=0$, the equations (4.2.3) give us the geodesics of the Hamilton space $H^n$.

Using the canonical nonlinear connection $\buildrel{\circ}\over{N}$ we can write in an invariant form the
Hamilton equations, which allow to prove the geometrical meaning of these equations.

\noindent{\bf Examples.} $1^\circ$ Consider $H^n=(M,H(x,p))$ the Hamilton
spaces of electrodynamics, [19]:
$$H(x,p)=\dfrac{1}{mc}\gamma^{ij}(x)p_ip_j-\dfrac{2e}{mc^2}A^i(x)p_i+\dfrac{e^3}{mc^3}A_i(x)A^i(x)$$
and $Fe=p_i\dot{\partial}^i$. Then $\Sigma_H$ is a Hamiltonian
mechanical system determined only by $H^n$.

$2^\circ$ $H^n=(M,K^2(x,p))$ is a Cartan space and $Fe=p_i\dot{\partial}^i$.

$3^\circ$ $H^n=(M,{\cal E}(x,p))$ with ${\cal E}(x,p)=\gamma^{ij}(x)p_ip_j$ and
$Fe=a(x)p_i\dot{\partial}^i$.

Returning to the general theory, we can prove:

\begin{theorem}
The following properties hold:

$1^\circ$ $\xi$ given by
\begin{equation}
\xi=\dfrac12[\dot{\partial}^i H\partial_i-(\partial_i H-F_i)\dot{\partial}^i]
\end{equation}
is a vector field on $\widetilde{T^*M}$.

$2^\circ$ $\xi$ is determined only by the Hamiltonian mechanical system $\Sigma_H$.

$3^\circ$ The integral curves of $\xi$ are given by the Hamilton equation $(4.2.3)$.
\end{theorem}

The previous Theorem is not difficult to prove if we remark the following expression of $\xi$:
\begin{equation}
\xi=\oci{\xi}+\dfrac12 Fe.
\end{equation}
Also we have:
\begin{proposition}
The variation of the Hamiltonian $H(x,p)$ along the evolution curves of $\Sigma_H$
is given by:
\begin{equation}
\dfrac{dH}{dt}=F_i\dfrac{dx^i}{dt}.
\end{equation}
\end{proposition}

As we know, the external forces $Fe$ are dissipative if $\langle Fw,\C\rangle\geq 0$.

Looking at the formula (4.2.6) one can say:

\begin{proposition}
The fundamental function $H(x,p)$ of the Hamiltonian mechanical system $\Sigma_H$ is decreasing on the evolution curves of $\Sigma_H$, if and only if, the external forces $Fe$ are dissipative.
\end{proposition}

The vector field $\xi$ on $\widetilde{T^*M}$ is called the canonical dynamical system of the Hamilton mechanical system
$\Sigma_H$.

Therefore we can say: The geometry of $\Sigma_H$ is the geometry of pair $(H^n,\xi)$.

\section{Canonical nonlinear connection of $\Sigma_H$}\index{Connections! Nonlinear}
\setcounter{equation}{0}
\setcounter{definition}{0}
\setcounter{theorem}{0}
\setcounter{lemma}{0}
\setcounter{remark}{0}
\setcounter{proposition}{0}

The fundamental tensor $g^{ij}(x,p)$ of the Hamilton space $H^n$ is the fundamental
or metric tensor of the mechanical system $\Sigma_H$. But others fundamental geometric
notions, as the canonical nonlinear connection of $\Sigma_H$ cannot be introduced in a straightforward manner. They will be
defined by means of ${\cal L}-$duality between the Lagrangian and the Hamiltonian mechanical systems $\Sigma_L$ and $\Sigma_H$ (see ch. ..., part I).

Let $\Sigma_L=(M,L(x,y),\underset{1}{Fe}(x,y))$, $\underset{1}{Fe}=\underset{1}{F}^i(x,y)\dot{\partial}_i$ be a
Lagrangian mechanical system. The mapping $$\varphi:(x,y)\in TM\to (x,p)\in T^*M,\ p_i=\dfrac12\dot{\partial}_i L$$ is a local
diffeomorphism. It is called the {\it Legendre transformation}.

Let $\psi$ be the inverse of $\varphi$ and
\begin{equation}
H(x,p)=2p_i y^i-L(x,y),\ y=\psi(x,p).
\end{equation}
One prove that $H^n=(M,H(x,p))$ is an Hamilton space. It is the ${\cal L}-$ dual of Lagrange space $L^n=(M,L(x,y))$.

One proves that $\varphi$ transform:

$1^\circ$ The canonical semispray $\buildrel{\circ}\over{S}$ of $L^n$ in the Hamilton vector $\buildrel{\circ}\over{\xi}$ of
$H^n$.

$2^\circ$ The canonical nonlinear connection $\buildrel{\circ}\over{N}_L$ of $L^n$ into the canonical nonlinear
connection $\buildrel{\circ}\over{N}_H$ of $H^n$.

$3^\circ$ The external forces $\underset{1}{Fe}$ of $\Sigma_L$ into external forces $Fe$ of $\Sigma_H$, with
$F_i(x,p)=g_{ij}(x,p)F_1^j(x,\psi(x,p))$.

$4^\circ$ The canonical nonlinear connection $N_L$ of $\Sigma_L$ with coefficients into the canonical nonlinear connection
$N_H$ of $\Sigma_H$ with the coefficients
\begin{equation}
N_{ij}(x,p)=\buildrel{\circ}\over{N}_{ij}(x,p)+\dfrac14 g_{ih}\dot{\partial}^h
F_j.
\end{equation}
$\buildrel{\circ}\over{N}_{ij}$ are given by (...., ch. ...).

Therefore we have:

\begin{proposition}
The canonical nonlinear connection $N$ of the Hamiltonian mechanical system $\Sigma_H$ has the
coefficients $N_{ij}$, $(4.3.2)$.
\end{proposition}

Of course, directly we can prove that $N_{ij}$ are the coefficients of a nonlinear connection. It is canonical for
$\Sigma_H$, since $N$ depend only on the mechanical system $\Sigma_H$.

The torsion of $N$ is
\begin{equation}
t_{ij}=\dfrac14(g_{ih}\dot{\partial}^h
F_j-g_{jh}\dot{\partial}^h F_i).
\end{equation}
Evidently $\dot{\partial}^i F_j=0$, implies that $N$ is symmetric nonlinear connection on the momenta space $T^*M$.

Let $(\delta_i,\dot{\partial}^i)$ be the adapted basis to $N$ and
$V$ and $(dx^i,\delta p_i)$ its adapted cobasis:
\begin{equation}
\begin{array}{l}
\delta_i=\partial_i+N_{ji}\dot{\partial}^j=\de_i+\dfrac14 g_{jh}\dot{\pp}^h F_i\dot{\pp}^j,\\ \\
\delta p_i=dp_i-N_{ij}dx^j=\oci{\de}p_i-\dfrac14 g_{ih}\dot{\pp}^h F_i dx^j.
\end{array}
\end{equation}
The tensor of integrability of the canonical nonlinear connection $N$ is
\begin{equation}
R_{kij}=\delta_i N_{kj}-\delta_j N_{ki}.
\end{equation}
$R_{kij}=0$ characterize the integrability of the horizontal distribution $N$.

The canonical $N-$metrical connection $C\Gamma(N)=(H^i_{jk},C^{jk}_i)$ of the Hamiltonian mechanical
system $\Sigma_H$ is given by the following Theorem:

\begin{theorem}
The following properties hold:

$1)$ There exists only one $N-$linear connection $C\Gamma=(N_{ij},H^i_{jk},C_i^{jk})$ which depend on the
Hamiltonian system $\Sigma_H$ and satisfies the axioms:

$1^\circ$ $N_{ij}$ from $(4.3.2)$ is the canonical nonlinear connection.

$2^\circ$ $C\Gamma$ is $h-$ metric:
\begin{equation}
g^{ij}_{|k}=0.
\end{equation}
$3^\circ$ $C\Gamma$ is $v-$metric:
\begin{equation}
g^{ij}|^k=0.
\tag{4.3.6'}
\end{equation}
$4^\circ$ $C\Gamma$ is $h-$torsion free:
\begin{equation}
T^i_{jk}=H^i_{jk}-H^i_{kj}=0.
\end{equation}
$5^\circ$ $C\Gamma$ is $v-$torsion free
\begin{equation}
S^{jk}_i=C^{jk}_i-C^{kj}_i=0.
\tag{4.3.7'}
\end{equation}

$2)$ The coefficients of $C\Gamma$ are given by the generalized Christoffel symbols:
\begin{equation}
\begin{array}{l}
H^i_{jk}=\dfrac12 g^{is}(\delta_j g_{sk}+\delta_k g_{js}-\delta_s g_{jk})\\ \\
C_i^{jk}=-\dfrac12 g_{is}(\dot{\partial}^j g^{sk}+\dot{\partial}^k
g^{js}-\dot{\partial}^s g^{jk})
\end{array}
\end{equation}
\end{theorem}

Now, we are in possession of all data to construct the geometry of Hamilton mechanical system $\Sigma_H$. Such that we can investigate the electromagnetic and gravitational fields of $\Sigma_H$.

\section{The Cartan mechanical systems}\index{Analytical Mechanics of! Cartanian systems}
\setcounter{equation}{0}
\setcounter{definition}{0}
\setcounter{theorem}{0}
\setcounter{lemma}{0}
\setcounter{remark}{0}
\setcounter{proposition}{0}

An important class of systems $\Sigma_H$ is obtained when the Hamiltonian $H(x,p)$ is 2-homogeneous with respect to momenta
$p_i$.

\begin{definition}
A Cartan mechanical system is a set
\begin{equation}
\Sigma_{\cal C}=(M,K(x,p),Fe(x,p))
\end{equation}
where
${\cal C}^n=(M,K(x,p))$ is a Cartan space and
\begin{equation}
Fe=F_i(x,p)\dot{\partial}^i
\end{equation}
are the external forces.
\end{definition}

The fact that ${\cal C}^n$ is a Cartan spaces implies:

$1^\circ$ $K(x,p)$ is a positive scalar function on $T^*M$.

$2^\circ$ $K(x,p)$ is a positive 1-homogeneous with respect to momenta $p_i$.

$3^\circ$ The pair $H^n=(M,K^2(x,p))$ is a Hamilton space.

Consequently:

a. The fundamental tensor $g^{ij}(x,p)$ of ${\cal C}^n$ is
\begin{equation}
g^{ij}=\dfrac12\dot{\partial}^i\dot{\partial}^j K^2.
\end{equation}

b. We have
\begin{equation}
K^2=g^{ij} p_i p_j.
\end{equation}

c. The Cartan tensor is
$$C^{ijk}=\dfrac14\dot{\partial}^i\dot{\partial}^j\dot{\partial}^k K^2,\ p_i C^{ijk}=0.$$

d. ${\cal C}^n=(M,K(x,p))$ is ${\cal L}-$ dual of a
Finsler space $F^n=(M,$ $F(x,y))$.

e. The canonical nonlinear connection $\buildrel{\circ}\over{N}$,
established by R. Miron \cite{Miron}, has the coefficients
\begin{equation}
\buildrel{\circ}\over{N}_{ij}=\gamma^h_{ij}p_h-\dfrac12(\gamma^h_{sr}p_h
p^r)\dot{\partial}^s g_{ij},
\end{equation}
$\gamma^i_{jk}(x,p)$ being the Christoffel symbols of $g_{ij}(x,p)$.

Clearly, the geometry of $\Sigma_{\cal C}$ is obtained from
the geometry of $\Sigma_H$ taking $H=K^2(x,p)$.

So, we have:

\medskip

\noindent{\bf Postulate 4.4.1.} {\it The evolution equations of the Cartan
mechanical system $\Sigma_{\cal C}$ are the Hamilton
equations:
\begin{equation}
\dfrac{dx^i}{dt}-\dfrac12\dot{\partial}^i K^2=0,\ \
\dfrac{dp_i}{dt}+\dfrac12\partial_i K^2=\dfrac12 F_i(x,p).
\end{equation}}

A first result is given by
\begin{proposition}
$1^\circ$ The energy of the Hamiltonian $K^2$ is given by
${\cal E}_{K^2}=p_i\dot{\partial}^i K^2-K^2=K^2$.

$2^\circ$ The variation of energy ${\cal E}_{K^2}=K^2$ along to
every evolution curve $(4.4.6)$ is
\begin{equation}
\dfrac{dK^2}{dt}=F_i\dfrac{dx^i}{dt}.
\end{equation}
\end{proposition}

\noindent{\bf Example.} $\Sigma_{\cal C}$, with
$K(x,p)=\{\gamma^{ij}(x)p_ip_j\}^{1/2}$, $(M,\gamma_{ij}(x))$
being a Riemann spaces and $Fe=a(x,p)p_i\dot{\partial}^i$.

Theorem ... of part ... can be particularized in:
\begin{theorem}
The following properties hold good:

$1^\circ$ $\xi$ given by
\begin{equation}
\xi=\dfrac12[\dot{\partial}^i K^2\partial_i-(\partial_i K^2-F_i)\dot{\partial}^i]
\end{equation}
is a vector field on $\widetilde{T^*M}$.

$2^\circ$ $\xi$ is determined only by the Cartan mechanical system
$\Sigma_{\cal C}$.

$3^\circ$ The integral curves of $\xi$ are given by the evolution
equations $(4.4.6)$ of $\Sigma_{\cal C}$.
\end{theorem}

The vector $\xi$ is Hamiltonian vector field on $T^*M$ or of the
Cartan mechanical system $\Sigma_{\cal C}$.

Therefore, the geometry of $\Sigma_{\cal C}$ is the differential geometry
of the pair $({\cal C}^n,\xi)$.

The fundamental object fields of this geometry are ${\cal C}^n,\xi,Fe$, the canonical nonlinear connection $N$ with the
coefficients
\begin{equation}
N_{ij}=\buildrel{\circ}\over{N}_{ij}+\dfrac14 g_{ih}\dot{\partial}^h F_j,
\end{equation}
Taking into account that the vector fields $\delta_i=\partial_i-N_{ji}\dot{\partial}^j$ determine an adapted
basis to the nonlinear connection $N$, we can get the canonical $N-$ metrical connection
$C\Gamma(N)$ of $\Sigma_{\cal C}$.

\begin{theorem}
The canonical $N-$ metrical connection $C\Gamma(N)$ of the Cartan
mechanical system $\Sigma_{\cal C}$ has the coefficients
\begin{equation}
H^i_{jk}=\dfrac12 g^{is}(\delta_j g_{sk}+\delta_k
g_{js}-\delta_s g_{jk}),\ C^{jk}_i=g_{is}C^{sjk}.
\end{equation}
\end{theorem}

Using the canonical connections $N$ and $C\Gamma(N)$ one can study
the electromagnetic and gravitational fields on the momenta space
$T^*M$ of the Cartan mechanical systems $\Sigma_{\cal C}$, as
well as the dynamical system $\xi$ of $\Sigma_{\cal C}$.

\newpage

\chapter{Lagrangian, Finslerian and Hamiltonian mechanical systems of order $k\geq 1$}

The notion of Lagrange or Finsler spaces of higher order allows us to introduce the Lagrangian and Finslerian mechanical systems of order $k\geq 1$. They will be defined as an natural extension of the Lagrangian and Finslerian mechanical systems. Now, the geometrical theory of these analytical mechanics is constructed by means of the differential geometry of the manifold $T^kM-$space of accelerations of order $k$ and of the notion of $k-$semispray, presented in the part II.

\setcounter{section}{0}
\section{Lagrangian Mechanical systems of order $k\geq 1$}
\setcounter{equation}{0}
\setcounter{definition}{0}
\setcounter{theorem}{0}
\setcounter{lemma}{0}
\setcounter{remark}{0}
\setcounter{proposition}{0}

We repeat shortly some fundamental notion about the differential geometry of the acceleration space $T^kM$.

Let $T^kM$ be the acceleration space of order $k$ (see ch. 1, part II) and $u\in T^kM$, $u=(x,y^{(1)},...,y^{(k)})$ a point with the local coordinates $(x^i,y^{(1)i},...,y^{(k)i})$.

A smooth curve $c:I\to M$ represented on a local chart ${\cal U}\subset M$ by $x^i=x^i(t)$, $t\in I$ can be extended to $(\pi^k)^{-1}({\cal U})\subset T^kM$ by $\wt{c}:I\to T^kM$, given as follows:
\be
x^i=x^i(t),\ \ y^{(1)i}=\dfrac{1}{1!}\dfrac{dx^i(t)}{dt},...,y^{(k)i}=\dfrac{1}{k!}\dfrac{d^k x^i}{dt^k}(t),\ \ t\in I.
\ee

Consider the vertical distribution $V_1,V_2,...,V_k$. They are integrable and have the properties $V_1\supset V_2\supset ...\supset V_k$. The dimension of $V_k$ is $n$, dim$V_{k-1}=2n$, ..., dim$V_1=kn$.

The Liouville vector fields on $T^kM$ are:
\be
\begin{array}{l}
\overset{1}\Gamma=y^{(1)i}\dfrac{\pp}{\pp y^{(k)i}}\\ \\
\overset{2}\Gamma=y^{(1)i}\dfrac{\pp}{\pp y^{(k-1)i}}+2y^{(2)i}\dfrac{\pp}{\pp y^{(k)i}}\\
.............\\
\overset{k}\Gamma=y^{(1)i}\dfrac{\pp}{\pp y^{(1)i}}+2y^{(2)i}\dfrac{\pp}{\pp y^{(2)i}}+...+ky^{(k)i}\dfrac{\pp}{\pp y^{(k)i}}.
\end{array}
\ee
Thus $\overset{1}\Gamma\in V_k$, $\overset{2}\Gamma\in V_{k-1}$, ..., $\overset{k}\Gamma\in V_1$.

The following operator:
\be
\Gamma=y^{(1)i}\dfrac{\pp}{\pp x^i}+2y^{(2)i}\dfrac{\pp}{\pp y^{(1)i}}+...+ky^{(k)i}\dfrac{\pp}{\pp y^{(k-1)i}}
\ee
is not a vector field on $T^k<$, but it is frequently used. On $T^kM$ there exists a $k$ tangent structure $J$ defined in ch. 2. $J$ is integrable and has the properties
\be
{\rm{Im}}J=V,\ \ \ker J=V_k,\ \ {\rm{rank}}J=k^n
\ee
and $J(\overset{1}\Gamma)=0$, $J(\overdoi\Gamma)=\overunu\Gamma$, ..., $J(\overk\Gamma)=\overset{k-1}\Gamma$, $J\circ J\circ ...\circ J=0$ ($k$ times).

A $k-$semispray is a vector field $S$ on the acceleration space $T^kM$ with the property
\be
JS=\overk\Gamma.
\ee
Also, $S$ is called a {\it dynamical system} on $T^kM$.

The local expression of a $k-$semispray $S$ is (ch. 1, part II)
\be
S=\Gamma-(k+1)G^i\dfrac{\pp}{\pp y^{(k)i}}.
\ee
The system of functions $G^i(x,y^{(1)},...,y^{(k)})$ are the {\it coefficients} of $S$.

The integral curves of the $k-$semispray $S$ are given by the following system of differential equations
\be
\left\{
\begin{array}{l}
\dfrac{d^{k-1}x^i}{dt^{k-1}}+(k+1)G^i(x,y^{(1)},...,y^{(k)})=0\\ \\
y^{(1)i}=\dfrac{dx^i}{dt},...,y^{(k)i}=\dfrac{1}{k!}\dfrac{d^k x^i}{dt^k}.
\end{array}
\right.
\ee

If $Fe=F^i(x,y^{(1)},...,y^{(k)})\dfrac{\pp}{\pp y^{(k)i}}$ is an arbitrary vertical vector field into the vertical distribution $V_k$ and $S$ is a $k-$semispray, thus
\be
S'=S+Fe
\ee
is a $k-$semispray, because $JS'=JS=\overk\Gamma$.

Recall the notion of nonlinear connection (ch. 1, part II):

A nonlinear connection on the manifold $T^kM$ is a distribution $N$ supplementary to the vertical distribution $V=V_1$:
\be
T_u T^kM=N(u)+V(u),\ \ \forall u\in T^kM.
\ee

An adapted basis to the distribution $N$ is (ch. 1, part II)
\be
\dfrac{\de}{\de x^i}=\dfrac{\pp}{\pp x^i}-\underset{1}{N}^j_i\dfrac{\pp}{\pp y^{(1)j}}-...-\underset{k}{N}^j_i\dfrac{\pp}{\pp y^{(k)j}}.
\ee
The systems of functions $\underset{1}{N}^j_i,...,\underset{k}{N}^j_i$ are the coefficients of nonlinear connection $N$.

Let $N_1,...,N_{k-1},V_k$ the vertical distributions defined by
\be
N_1=J(N),\ N_2=J(N_1),..., N_{k-1}=J(N_{k-2}),\ V_k=J(N_{k-1})\ {\rm{and\ }}N_0=N.
\ee
Thus, dim$N_0=$dim$N_1=...=$dim$N_k=n$ and we have
\be
T_uT^kN=N_0(u)\oplus N_1(u)\oplus...\oplus N_{k-1}(u)\oplus V_k,\ \ \forall u\in T^kM.
\ee
The adapted basis to the direct decomposition (5.1.12) is as follows:
\be
\(\dfrac{\de}{\de x^i},\dfrac{\de}{\de y^{(1)i}},...,\dfrac{\de}{\de y^{(k)i}}\),
\ee
where $\dfrac{\de}{\de x^i}$ is given by (5.1.10) and
\be
\begin{array}{c}
\dfrac{\de}{\de y^{(1)i}}=J\(\dfrac{\de}{\de x^i}\), \dfrac{\de}{\de y^{(2)i}}=J\(\dfrac{\de}{\de y^{(1)i}}\),...,\dfrac{\de}{\de y^{(k)i}}=\\ \\ =J\(\dfrac{\de}{\de y^{(k-1)i}}\)=\dfrac{\pp}{\pp y^{(k)i}}.
\end{array}
\ee
The expressions of the basis (5.1.13) can be easily obtained from the formulas (5.1.10), (5.1.14) (ch. 1, part II).

The dual basis of the adapted basis (5.1.13) is:
\be
(\de x^i,\de y^{(1)i},...,\de y^{(k)i}),
\ee
where
\be
\begin{array}{l}
\de x^i=dx^i\\ \\
\de y^{(1)i}=dy^{(1)i}+\underset{1}{M}^i_j dx^j\\ \\
\de y^{(2)i}=dy^{(2)i}+\underset{1}{M}^i_j dy^{(1)j}+\underset{2}{M}^i_j dx^j\\
.......................\\
\de y^{(k)i}=dy^{(k)i}+\underset{1}{M}^i_j dy^{(k-1)j}+...+\underset{k}{M}^i_j dx^j
\end{array}
\ee
and the systems of functions $\underset{1}{M}^i_j,...,\underset{k}{M}^i_j$ are the dual coefficients of the nonlinear connection $N$. The relations between dual coefficients $\underset{1}{M}^i_j,...,\underset{k}{M}^i_j$ and (primal) coefficients $\underset{1}{N}^i_j,...,\underset{k}{N}^i_j$ are (ch. 1, part II):
\be
\begin{array}{l}
\underset{1}{M}^i_j=\underset{1}{N}^i_j,\\ \\
\underset{2}{M}^i_j=\underset{2}{N}^i_j+\underset{1}{N}^m_j\underset{1}{M}^i_m,\\
..................................\\
\underset{k}{M}^i_j=\underset{k}{N}^i_j+\underset{k-1}{N}^m_j\underset{1}{M}^i_m+...+\underset{1}{N}^m_j\underset{k-1}{N}^i_m.
\end{array}
\ee
This recurrence formulas allow us to calculate the primal coefficients $\underset{1}{N}^i_j,...,\underset{k}{N}^i_j$ functions by the dual coefficients $\underset{1}{M}^i_j,...,\underset{k}{N}^i_j$.

In the adapted basis (5.1.13) the Liouville vector fields $\overunu\Gamma,...,\overk\Gamma$ can be expressed:
\begin{equation}
\begin{array}{l}
\overunu\Gamma=z^{(1)i}\dfrac{\de}{\de y^{(k)i}}\\ \\
\overdoi\Gamma=z^{(1)i}\dfrac{\de}{\de y^{(k-1)i}}+2z^{(2)i}\dfrac{\de}{\de y^{(k)i}}\\
................................\\
\overk\Gamma=z^{(1)i}\dfrac{\de}{\de y^{(1)i}}+2z^{(2)i}\dfrac{\de}{\de y^{(2)i}}+...+kz^{(k)i}\dfrac{\de}{\de y^{(k)i}},
\end{array}
\end{equation}
where $z^{(1)i},...,z^{(k)i}$ are {\it the d-Liouville vector fields}:
\be
\begin{array}{l}
z^{(1)i}=y^{(1)i}\\ \\
2z^{(2)i}=2y^{(2)i}+\underset{1}{M}^i_m y^{(1)m}\\
................................\\
kz^{(k)i}=ky^{(k)i}+\overset{(k-1)}{\underset{(1)}{M}}^i_m y^{(k-1)m}+...+\overset{(k-1)}{\underset{(1)}{M}}^i_m y^{(1)m}.
\end{array}
\ee
To a change of local coordinates (ch. 1, part II) on the acceleration manifold $T^kM$, every one of the Liouville $d-$vectors transform on the rule:
\be
\wt{z}^{(\a)i}=\dfrac{\pp\wt{x}^i}{\pp x^j}z^{(\a)j},\ \ (\a=1,...,k).
\ee
So, they have a geometric meaning.

\section{Lagrangian mechanical system of order $k$, $\Sigma_{L^k}$}\index{Analytical Mechanics of! Lagrangian of order $k$}
\setcounter{equation}{0}
\setcounter{definition}{0}
\setcounter{theorem}{0}
\setcounter{lemma}{0}
\setcounter{remark}{0}
\setcounter{proposition}{0}

\begin{definition}
A Lagrangian mechanical system of order $k\geq 1$ is a triple:
\be
\Sigma_{L^k}=(M,L(x,y^{(1)},...,y^{(k)}),Fe(x,y^{(1)},...,y^{(k)})),
\ee
where
\be
L^{(k)n}=(M,L(x,y^{(1)},...,y^{(k)}))
\ee
is a Lagrange space of order $k\geq 1$, and
\be
Fe(x,y^{(1)},...,y^{(k)})=F^i(x,y^{(1)},...,y^{(k)})\dfrac{\pp}{\pp y^{(k)i}}
\ee
is an a priori given vertical vector field.
\end{definition}

Clearly, $Fe$ is a vector field belonging to the vertical distribution $V_k$ and its contravariant components $F^i$ is a distinguished vector field. $Fe$ or $F^i$ are called the external forces of $\Sigma_{L^n}$.

Applying the theory of Lagrange spaces of order $k\geq 1$ to the space $L^{(k)n}$ we have the fundamental (or metric) tensor field of $\Sigma_{L^k}$:
\be
g_{ij}=\dfrac12\dfrac{\pp^2 L}{\pp y^{(k)i}\pp y^{(k)j}},
\ee
and the Euler--Lagrange equations
\be
\begin{array}{l}
\oci{E}_i(L)\overset{det}{=}\dfrac{\pp L}{\pp x^i}-\dfrac{d}{dt}\dfrac{\pp L}{\pp y^{(1)i}}+...+(-1)^k\dfrac{d^k}{dt^k}\dfrac{\pp L}{\pp y^{(k)i}}=0\\ \\
y^{(1)i}=\dfrac{dx^i}{dt},...,y^{(k)1}=\dfrac{1}{k!}\dfrac{d^k x^i}{dt^k}.
\end{array}
\ee
This is a system of differential equations of order $2k$, autoadjoint. It allows us to determine a canonical semispray of order $k$ on $T^kM$.

Indeed, since$\oci{E}_i(L)$ is a $d-$covector field on $T^kM$, we can calculate $\oci{E}(\phi(t)L(x,y^{(1),...,y^{(k)}}))$ and obtain (ch. 1, part II):
\be
\oci{E}_i(\phi L)=\phi\oci{E}_i+\dfrac{d\phi}{dt}\overset{1}{E}_i(L)+...+\dfrac{d^k\phi}{dt^k}\overk{E}_i(L).
\ee
The coefficients $\overunu{E}_i(L),...,\overk{E}_i(L)$ are $d-$covector fields discovered by Craig and Synge (we refer to books \cite{mir11}).

So, we have:
\be
\overset{k-1}{E}_i(L)=(-1)^{k-1}\dfrac{1}{(k-1)!}\[\dfrac{\pp L}{\pp y^{(k-1)i}}-\dfrac{d}{dt}\dfrac{\pp L}{\pp y^{(k)i}}\].
\ee

\begin{theorem}
For any Lagrange spaces $L^{(k)n}=(M,L)$ the following properties hold:

$1^\circ$ the system of differential equations $g^{ij}\overset{k-1}{E}_i(L)=0$ has the form
\be
\dfrac{d^{k+1}x^i}{dt^{k+1}}+(k+1)!\oci{G}^i\(x,\dfrac{dx}{dt},...,\dfrac{1}{k!}\dfrac{d^k x}{dt^k}\)=0
\ee
with
\be
(k+1)\oci{G}^i=\dfrac12 g^{ij}\[\Gamma\dfrac{\pp L}{\pp y^{(k)j}}-\dfrac{\pp L}{\pp y^{(k-1)j}}\].
\ee

$2^\circ$ The system of functions $\oci{G}^i$ from (5.2.9) are the coefficients of $k-$semispray $\oci{S}$ which depends only on the Lagrangian $L(x,y^{(1)}$, ..., $y^{(k)})$.
\end{theorem}

So, $\oci{S}$ is the {\bf canonical semispray} of $L^{(k)n}$. But it leads to a canonical nonlinear connection $\oci{N}$ for $L^{(k)n}$.

Indeed, we have

\begin{theorem} [Miron \cite{mir11}]
$\oci{S}$ being the canonical $k-$semispray of Lagrange space $L^{(k)n}$, there exists a nonlinear connection $\oci{N}$ of the space $L^{(k)n}$ which depend only on the fundamental function $L$. The non-connection $\oci{N}$ has the dual coefficients:
\be
\begin{array}{l}
\underset{1}{M}^i_j=\dfrac{\pp G^i}{\pp y^{(k)j}},\\ \\
\underset{2}{M}^i_j=\dfrac12(\oci{S}\underset{1}{M}^i_j+\underset{1}{M}^i_m\underset{1}{M}^m_j), \\
........................................\\
\underset{k}{M}^i_j=\dfrac{1}{k}(\oci{S}\underset{k-1}{M}^i_j+\underset{k-1}{M}^i_m\underset{k-1}{M}^m_j).
\end{array}
\ee
\end{theorem}

Now, we can develop the geometry of Lagrange space $L^{(k)n}$ by means of the geometrical object fields $L,\oci{S},\oci{N}$ and $g_{ij}$ (ch. 2, part II).
We apply these considerations to the geometrization of Lagrange mechanical system $\Sigma_{L^k}$.

\section{Canonical $k-$semispray of mechanical system $\Sigma_{L^k}$}\index{$k-$semisprays! of mechanical system $\Sigma_{L^k}$}
\setcounter{equation}{0}
\setcounter{definition}{0}
\setcounter{theorem}{0}
\setcounter{lemma}{0}
\setcounter{remark}{0}
\setcounter{proposition}{0}

The fundamental equations of mechanical system $\Sigma_k$ are given by:

\noindent{\bf Postulate 5.3.1.} {\it The evolution equations of the Lagrangian mechanical system of order $k$, $\Sigma_{L^k}=(M,L,Fe)$ are the following Lagrange equations:
\be
\left\{
\begin{array}{l}
\dfrac{d}{dt}\dfrac{\pp L}{\pp y^{(k)i}}-\dfrac{\pp L}{\pp y^{(k-1)i}}=F_i,\\ \\
y^{(1)i}=\dfrac{dx^i}{dt},...,y^{(k)i}=\dfrac{1}{k!}\dfrac{d^k x^i}{dt^k},
\end{array}
\right.
\ee
where
\be
F_i(x,y^{(1)},...,y^{(k)})=g_{ij}(x,y^{(1)},...,y^{(k)})F^i(x,y^{(1)},...,y^{(k)})
\ee
are the covariant components of external forces $Fe$.}

Evidently, in the case $k=1$, the previous equations are the evolution equations of Lagrangian mechanical system $\Sigma_L=(M,L(x,y),Fe(x,y))$. Also, the equation (5.3.1) for $k=1$ are the Lagrange equations for Riemannian or Finslerian mechanical systems. These arguments justify the introduction of previous Postulate.

\bigskip

\noindent{\bf Examples.} $1^\circ$ Let $L^{(k)n}=(M,L(x,y^{(1)},...,y^{(k)}))$ be a Lagrange space of order $k\geq 1$, and $\oci{z}^{(k)i}$ be its Liouville $d-$vector field of order $k$. Consider the external forces
\be
Fe=h\oci{z}^{(k)i}\dfrac{\pp}{\pp y^{(k)i}},\ \ h\in\R^*.
\ee
Thus, $\Sigma_{L^k}=(M,L(x,y^{(1)},...,y^{(k)}))$, $Fe(x,y^{(1)},...,y^{(k)})$ is a Lagrangian mechanical system of order $k$.

The covariant component $F_i$ of $Fe$ is $F_i=hg_{ij}\oci{z}^{(k)j}$, which substitute in (5.3.1) give us the Lagrange equations of order $k$ for the considered mechanical system.

In the general case of mechanical systems $\Sigma_{L^k}$  is important to prove the following property:

\begin{theorem}
The Lagrange equations $(5.3.1)$ are equivalent to the following system of differential equations of order $k+1$:
\be
\dfrac{d^{k+1}x^i}{dt^{k+1}}+(k+1)!\oci{G}^i\(x,\dfrac{dx}{dt},...,\dfrac{1}{k}\dfrac{d^k_x}{dt^k}\)=\dfrac{k!}{2}F^i
\ee
with the coefficients $\oci{G}^i(x,y^{(1)},...,y^{(k)})$ from $(5.2.9)$.
\end{theorem}

\begin{proof}
We have $$\dfrac{d}{dt}\dfrac{\pp L}{\pp y^{(k)i}}-\dfrac{\pp L}{\pp y^{(k-1)i}}=\Gamma\dfrac{\pp L}{\pp y^{(k)i}}-\dfrac{\pp L}{\pp y^{(k-1)i}}+\dfrac{2}{k!}g_{ij}\dfrac{d^{k+1}x^j}{dt^{k+1}}-F_i.$$ It follows
\be
\dfrac{d^{k+1}x^i}{dt^{k+1}}+\dfrac{k!}{2}g^{ij}\(\Gamma\dfrac{\pp L}{\pp y^{(k)j}}-\dfrac{\pp L}{\pp y^{(k-1)j}}\)=\dfrac{k!}{2}F^i.\tag{5.3.4'}
\ee
But the previous equations is exactly (5.3.4), with the coefficients $(k+1)\oci{G}^i$ given by (5.2.9).
\end{proof}

In initial conditions: $(x^i)_0=x^i_0$, $\(\dfrac{dx^i}{dt}\)_0=y^{(1)i}_0,...,\dfrac{1}{k!}\(\dfrac{d^k x^i}{dt^k}\)_0=y^{(k)i}_0$, locally there exists a unique solution of evolution equations (5.3.4): $x^i=x^i(t)$, $y^{(1)i}=y^{(1)i}(t),...,y^{(k)i}=y^{(k)i}(t)$, $t\in (a,b)$ which express the moving of the mechanical system $\Sigma_{L^k}$.

But, it is convenient to write the $k-$Lagrange equations (5.3.4) in the form
\be
\begin{array}{l}
y^{(1)i}=\dfrac{dx^i}{dt},\ y^{(2)i}=\dfrac12\dfrac{d^2 x^i}{dt^2},...,y^{(k)i}=\dfrac{1}{k!}\dfrac{d^k x^i}{dt^k}\\ \\
\dfrac{dy^{(k)i}}{dt}+(k+1)\oci{G}^i(x,y^{(1)},...,y^{(k)})=\dfrac12 F^i(x,y^{(1)},...y^{(k)}).
\end{array}
\ee
These equations determine the trajectories of the vector field $S$ on the acceleration space $T^kM$:
\be
\begin{array}{l}
S=y^{(1)i}\dfrac{\pp}{\pp x^i}+2y^{(2)i}\dfrac{\pp}{\pp y^{(1)i}}+...+ky^{(k)i}\dfrac{\pp}{\pp y^{(k-1)i}}-\\ \\ -(k+1)\oci{G}^i\dfrac{\pp}{\pp y^{(k)i}}+\dfrac12 F^i\dfrac{\pp}{\pp y^{(k)i}}
\end{array}
\ee
or
\be
S=\oci{S}+\dfrac12 Fe\tag{5.3.6'}
\ee
where $\oci{S}$ is the canonical $k-$semispray of the Lagrange space of order $k$, $L^{(k)n}=(M,L)$.
Consequently, we have:

\begin{theorem}
$1^\circ$ For the Lagrangian mechanical system of order $k$, $\Sigma_{L^k}$ the operator $S$ from $(5.3.6)$ is a $k-$semispray which depend only on the mechanical system $\Sigma_{L^k}$.

$2^\circ$ The integral curves of $S$ are the evolution curves given by Lagrange equations $(5.3.5)$.
\end{theorem}

$S$ is called {\it the canonical semispray of} $\Sigma_{L^k}$ or it is called the {\it dynamical system of} $\Sigma_{L^k}$, too.

The coefficients $G^i$ of the canonical $k-$semispray $S$ are:
\be
(k+1)G^i=(k+1)\oci{G}^i-\dfrac12 F^i.
\ee

In the example (5.3.3) the canonical $k-$semispray has the coefficients
\be
(k+1)G^i=(k+1)\oci{G}^i-\dfrac{h}{2}z^{(k)i}.
\ee
Of course, the vector field $S$ on the manifold of the acceleration of order $k\geq 1$, $T^kM$ being a dynamical system it can be used for investigate the qualitative problems, as stability of solutions, equilibrium points etc.

Now, we can study the canonical nonlinear connection of the mechanical system $\Sigma_{L^k}$.

\section{Canonical nonlinear connection of mechanical system $\Sigma_{L^k}$}\index{Connections! Nonlinear}
\setcounter{equation}{0}
\setcounter{definition}{0}
\setcounter{theorem}{0}
\setcounter{lemma}{0}
\setcounter{remark}{0}
\setcounter{proposition}{0}

Let $G^i$ be the coefficients of the canonical semispray $S$ of mechanical system $\Sigma_{L^{k}}$. Thus
\be
G^i=\oci{G}^i-\dfrac{1}{2(k+1)}F^i,
\ee
where
\be
\oci{G}^i=\dfrac{1}{2(k+1)}g^{ij}\[\Gamma\dfrac{\pp L}{\pp y^{(k)j}}-\dfrac{\pp L}{\pp y^{(k-1)j}}\].
\ee
The dual coefficients $\underset{1}{\oci{M}}^i_j,...,\underset{k}{\oci{M}}^i_j$ of the canonical nonlinear connection $\oci{N}$ of Lagrange space $L^{(k)n}=(M,L)$ are given by (5.2.10):
\be
\begin{array}{l}
\underset{1}{\oci{M}}^i_j=\dfrac{\pp\oci{G}^i}{\pp y^{(k)j}}, \underset{2}{\oci{M}}^i_j=\dfrac12(\oci{S}\underset{1}{\oci{M}}^i_j+\underset{1}{\oci{M}}^i_m\underset{1}{\oci{M}}^m_j),...,\\ \\
\underset{k}{\oci{M}}^i_j=\dfrac{1}{k}(\oci{S}\underset{k-1}{\oci{M}}^i_j+\underset{1}{\oci{M}}^i_m\underset{k-1}{\oci{M}}^m_j).
\end{array}
\ee
By virtue of (5.3.6'):
\be
S=\oci{S}+\dfrac12 Fe
\ee
the dual coefficients $\underset{1}{M}^i_j,...,\underset{k}M^i_j$ of canonical nonlinear connection $N$ of $\Sigma_{L^k}$ can be written:
\be
\begin{array}{l}
\underset{1}{M}^i_j=\underset{1}{\oci{M}}^i_j-\dfrac{1}{2(k+1)}\dfrac{\pp F^i}{\pp y^{(k+1)j}}\\ \\
\underset{2}{M}^i_j=\dfrac12(S\underset{1}{M}^i_j+\underset{1}{M}^i_m\underset{1}{M}^m_j)\\
........................\\
\underset{(k)}{M}^i_j=\dfrac{1}{k}(S\underset{k-1}{M}^i_j+\underset{1}{M}^i_m\underset{k-1}{M}^m_j)
\end{array}
\ee
where
\be
S\underset{\a}{M}^i_j=\oci{S}\underset{\a}{M}^i_j+\dfrac12 Fe(\underset{\a}{M}^i_j),\ \ \a=1,2,...,k-1.
\ee

The coefficients $\underset{1}{N}^i_j,...,\underset{(k)}{N}^i_j$ of $N$ are given by the formulas
\be
\underset{1}{N}^i_j=\underset{1}{M}^i_j,\ \underset{(\a)}{N}^i_j=\underset{\a}{M}^i_j
-\underset{(\a-1)}{N}^m_j\underset{(1)}{M}^i_m-...-\underset{(1)}{N}^m_j\underset{(\a-1)}{M}^i_m,\ (\a=2,...,k).
\ee
The adapted basis $\dfrac{\de}{\de x^i}$ ($i=1,...,n$) to the distribution $N$ and the adapted basis to the vertical distributions $V_1,...,V_k$:
\be
\(\dfrac{\de}{\de x^i},\dfrac{\de}{\de y^{(1)i}},...,\dfrac{\de}{\de y^{(k)i}}\)
\ee
can be explicitly written.

Its dual adapted basis
\be
(\de x^i,\de y^{(1)i},...,\de y^{(k)i})
\ee
has the following expression:
\be
\begin{array}{c}
\de x^i=dx^i,\ \de y^{(1)i}=dy^{(1)i}+\underset{(1)}{M}^i_j dx^j,...,\de y^{(k)i}=dy^{(k)i}+\\ \\ +\underset{(1)}{M}^i_jdy^{(k-1)j}+...+\underset{(k)}{M}^i_j dx^j.\tag{5.4.9'}
\end{array}
\ee
By using the Definition from part II, equations (1.5.9), we have

\begin{theorem}
The autoparallel curves of the canonical nonlinear connection $N$ of the mechanical system $\Sigma_{L^k}$ are given by the following system of differential equation
\be
\begin{array}{l}
\dfrac{\de y^{(1)i}}{dt}=0,...,\dfrac{\de y^{(k)i}}{dt}=0\\ \\
\dfrac{dx^i}{dt}=y^{(1)i},...,\dfrac{1}{k!}\dfrac{d^k x^i}{dt^k}=y^{(k)i}.
\end{array}
\ee
\end{theorem}

\begin{theorem}
The variation of the Lagrangian $L$ along the autoparallel curves of the canonical nonlinear connection $N$ are given by
\be
\dfrac{dL}{dt}=\dfrac{\de L}{\de x^i}\dfrac{dx^i}{dt}.
\ee
\end{theorem}

\begin{corollary}
The Lagrangian $L$ of the space $L^{(k)n}=(M,L)$ of a mechanical system $\Sigma_{L^k}$ is conserved along the autoparallel curves of the canonical nonlinear connection $N$ iff the scalar product $\dfrac{\de L}{\de x^i}\dfrac{dx^i}{dt}$ vanishes.
\end{corollary}

\begin{remark}
The Bucataru's nonlinear connection $\overset{b}{N}$ of the mechanical system $\Sigma_{L^k}$ has the coefficients
\be
\begin{array}{l}
\overset{b}{\underset{(1)}{N}}^i_j=\dfrac{\pp\oci{G}^i}{\pp y^{(k)j}}+\dfrac{1}{2(k+1)}\dfrac{\pp F^i}{\pp y^{(k)j}}\\ \\
\underset{(2)}{N}^i_j=\dfrac{\pp\oci{G}^i}{\pp y^{(k-1)j}}+\dfrac{1}{2(k+1)}\dfrac{\pp F^i}{\pp y^{(k-1)j}}\\
.........................\\
\underset{(k)}{N}^i_j=\dfrac{\pp\oci{G}^i}{\pp y^{(1)j}}+\dfrac{1}{2(k+1)}\dfrac{\pp F^i}{\pp y^{(1)j}}.
\end{array}
\ee
\end{remark}
This nonlinear connection depend only on the mechanical system $\Sigma_{L^k}$, too. Sometime $\underset{b}{N}$ is more convenient to solve the problems of the geometrization of mechanical system $\Sigma_{L^k}$.

\section{Canonical metrical $N-$connection}\index{Connections! $N-$linear}
\setcounter{equation}{0}
\setcounter{definition}{0}
\setcounter{theorem}{0}
\setcounter{lemma}{0}
\setcounter{remark}{0}
\setcounter{proposition}{0}

Let $N$ be the canonical nonlinear connection with the coefficients (5.4.7) or $N$ being the Bucataru's connection $\overset{b}{N}$ with the coefficients (5.4.12). An $N-$linear connection $D$ with the coefficients $D\Gamma(N)=(L^i_{jh},\underset{(1)}{C}^i_{jh},...,\underset{(k)}{C}^i_{jh})$ is metrical with respect to the metric tensor $g_{ij}$ of the mechanical system $\slk$ (cf. ch. 2, part II) if the following equations hold:
\be
g_{ij|h}=0,\ \ g_{ij}\overset{(\a)}{|}_h=0,\ \ (\a=1,...,k).
\ee
Thus, the theorem ... implies:

\begin{theorem}
The we have:

\noindent $({\rm{I}})$ There exists a unique $N-$linear connection $D$ on $\wt{T^kM}$ satisfying the axioms:

$(1)$ $N$ is the canonical nonlinear connection of the Lagrange space $L^{(k)n}$ (or $N$ is Bucataru nonlinear connection);

$(2)$ $g_{ij|h}=0$, ($D$ is $h-$metrical);

$(3)$ $g_{ij}\overset{(\a)}{|}_h=0$, $(\a=1,...,k)$ ($D$ is $v_\a-$metrical);

$(4)$ $T^i_{jh}=0$\ \ ($D$ is $h-$torsion free);

$(5)$ $\underset{(\a)}{S}^i_{jh}=0$\ \ ($D$ is $v_\a-$torsion free).

\noindent $({\rm{II}})$ The coefficients $C\Gamma(N)=(L^h_{ij},\underset{(1)}{C}^h_{ij},...,\underset{(k)}{C}^h_{ij})$ of $D$ are given by the generalized Christoffel symbols $$\begin{array}{l}L^h_{ij}=\dfrac12 g^{hs}\(\dfrac{\de g_{is}}{\de x^j}+\dfrac{\de g_{sj}}{\de x^i}-\dfrac{\de g_{ij}}{\de x^s}\)\\ \\
\underset{(\a)}{C^h_{ij}}=\dfrac12 g^{hs}\(\dfrac{\de g_{is}}{\de y^{(\a)j}}+\dfrac{\de g_{sj}}{\de y^{(\a)i}}-\dfrac{\de g_{ij}}{\de y^{(\a)s}}\),\ (\a=1,...,k).\end{array}$$

\noindent $({\rm{III}})$ $D$ depends only on the mechanical system $\slk$.
\end{theorem}

The $N-$metrical connection $D$ is called the {\it canonical metrical connection} of the Lagrangian mechanical system of order $k\geq 1$ $\slk$.

Now, one can affirm: {\it the geometrization of $\slk$ can be developed by means of the canonical semispray $S$ and by the geometrical object fields: $L,g_{ij},N$, and $C\Gamma(N)$.}

\section{The Riemannian $(k-1)n$ almost contact model of the Lagrangian mechanical system of order $k$, $\slk$}
\setcounter{equation}{0}
\setcounter{definition}{0}
\setcounter{theorem}{0}
\setcounter{lemma}{0}
\setcounter{remark}{0}
\setcounter{proposition}{0}

$N$ being the nonlinear connection with the coefficients  (5.4.7) (or with the coefficients (5.4.12)) and $(\de x^i,\de y^{(1)i},...,\de y^{(k)i})$ be the adapted cobasis to the direct decomposition (5.1.12).

Consider the $N-$lift $\G$ of the metric tensor $g_{ij}$:
\be
\G=g_{ij}dx^i\otimes dx^j+g_{ij}\de y^{(1)i}\otimes\de y^{(1)j}+...+g_{ij}\de y^{(k)j}\otimes\de y^{(k)i}.
\ee
Thus, $\G$ is a semidefinite Riemannian structure on the manifold $\wt{T^kM}$ depending only on $\slk$.

The distributions $\{\underset{0}{N},\underset{1}{N},...,\underset{k-1}{N},V_k\}$ are two by two orthogonal with respect to $\G$.

The nonlinear connection $N$ uniquely determines the $(k-1)n-$almost contact structure $\F$ by:
\be
\F\(\dfrac{\de}{\de x^i}\)=-\dfrac{\pp}{\pp y^{(k)i}};\ \F\(\dfrac{\de}{\de y^{(\a)i}}\)=0,\ \a=1,...,k-1;\ \F\(\dfrac{\pp}{\pp y^{(k)i}}\)=\dfrac{\de}{\de x^i}.
\ee

Indeed, it follows:

\begin{theorem}
\begin{itemize}
  \item[$(1)$] \ $\F$ is globally defined on the manifold $\wt{T^kM}$ and depend on $\slk$, only.
  \item[$(2)$] \ $\ker\F=N_1\oplus N_2\oplus ...\oplus N_{k-2}$, ${\rm{Im}}\F=N_0\oplus V_k$.
  \item[$(3)$] \ ${\rm{rank}}\|\F\|=2n$.
  \item[$(4)$] \ $\F^3+\F=0$.
\end{itemize}
So, $\F$ is an almost $(k-1)n-$contact structure on the bundle of acceleration of order $k$, $T^kM$.
\end{theorem}

Let $\(\underset{1a}{\xi},\underset{2a}{\xi},...,\underset{(k-1)a}{\xi}\)$, $a=1,...,n$, be a local basis adapted to the direct decomposition $\underset{1}{N}\oplus ...\oplus\underset{k-1}{N}$ and $\overset{1a}{\eta},\overset{2a}{\eta},...,\overset{(k-1)a}{\eta}$ its dual basis.

Thus the set
\be
\(\F,\underset{1a}{\xi},...,\underset{(k-1)a}{\xi},\overset{1a}{\eta},...,\overset{(k-1)a}{\eta}\), (a=1,...,n)
\ee
is a $(k-1)n-$almost contact structure.

Indeed, we have
\be
\F\(\underset{\a a}{\xi}\)=0, \overset{a\ a}{\eta}\(\underset{\b b}{\xi}\)=\left\{\begin{array}{ll}\de^a_b, & {\rm{for\ }}\a=\b\\ 0,& {\rm{for\ }}\a\neq\b,\ (\a,\b=1,...,n-1)\end{array}\right.
\ee
and
\be
\left\{\begin{array}{l}\F^2(X)=-X+\dd\sum^{n}_{a=1}\dd\sum^{k-1}_{\a=1}\overset{a\ a}{\eta}(X)\underset{\a a}{\xi},\ \forall X\in{\cal X}(T^kM)\\ \\
\overset{a\ a}{\eta}\circ\F=0.
\end{array}\right.
\ee

The Nijenhuis tensor $N_\F$ is expressed by
\be
N_\F(X,Y)=[\F X,\F Y]+\F^2(X,Y)-\F[\F X,Y]-\F[X,\F Y].
\ee
The structure (5.6.5) is {\it normal} if:
\be
N_{\F}(X,Y)+\dd\sum^{n}_{a=1}\dd\sum^{k-1}_{\a=1}d\overset{a\ a}{\eta}(X,Y)=0,\ \forall X,Y\in{\cal X}(\wt{T^kM}).
\ee

A characterization of the normality of structure (4.6.7) is as follows:

\begin{theorem}
The almost $(k-1)n-$contact structure $\(\F,\underset{\a a}{\xi},\overset{a\ b}{\eta}\)$ is normal, if and only if, for any vector fields $X,Y$ on $\wt{T^kM}$ we have
\be
N_\F(X,Y)+\dd\sum^{n}_{a=1}\dd\sum^{k-1}_{\a=1}d(\dd y^{(\a)a})(X,Y)=0.
\ee
\end{theorem}

Thus, we can prove:

\begin{theorem}
The pair $(\G,\F)$ is a Riemannian $(k-1)n-$almost contact structure on $\wt{T^kM}$, depending only on the Lagrangian mechanical system of order $k\geq 1$, $\slk$.
\end{theorem}

Indeed, the following condition holds:
\be
\G(\F X,Y)=-\G(\F Y,X),\ \ \forall X,Y\in{\cal X}(\wt{T^kM}).
\ee
Therefore, the triple $(T^kM,\G,\F)$ is called the Riemannian $(k-1)n-$ almost contact model of the mechanical system $\slk$. It can be used to solve the problem of geometrization of mechanical system $\slk$.

\section{Classical Riemannian mechanical system with external forces depending on higher order accelerations}
\setcounter{equation}{0}
\setcounter{definition}{0}
\setcounter{theorem}{0}
\setcounter{lemma}{0}
\setcounter{remark}{0}
\setcounter{proposition}{0}

Let us consider the classical Riemannian mechanical systems
\be
\Sigma_{\cal R}=(M,T,Fe),
\ee
with $T=\dfrac12\g_{ij}(x)\dfrac{dx^i}{dt}\dfrac{dx^j}{dt}$ as kinetic energy and the external forces depending on the acceleration of order $1,2,...,k$, i.e. $Fe(x,y^{(1)},...,y^{(k)})$. The problem is {\it what kind of evolution equations can be postulate in order to study the movings of these systems}?

Of course, in nature we have such kind of mechanical systems. For instance, the friction forces, which action on a rocket in the time of its moving are the external forces depending by higher order accelerations. But, it is clear that the classical Lagrange equations:
$$\dfrac{d}{dt}\(\dfrac{\pp(2T(x,y^{(1)}))}{\pp y^{(1)i}}\)-\dfrac{\pp(2T(x,y^{(1)}))}{\pp x^i}=F_i(x,y^{((1))},...,y^{(1)})$$ have the form
$$\dfrac{d^2 x^i}{dt^2}+\g^i_{jh}(z)\dfrac{dx^j}{dt}\dfrac{dx^h}{dt}=\dfrac12\g^{is}(x)F_s\(x,\dfrac{dx}{dt},...,\dfrac{1}{k!}\dfrac{d^k x}{dt^k}\)$$ and for $k>2$ it is not a differential system of equations of order 2, like in the classical Lagrange equations.

We solve this problem by means of the prolongation of order $k$ of the Riemannian space ${\cal R}^n=(M,\g_{ij}(x))$.

Let $\Sigma_{\cal R}=(M,2T,Fe)$ a mechanical system of form (5.7.1) with $T=\dfrac12\g_{ij}(x)y^{(1)i}y^{(1)j}$, $\(y^{(1)i}=\dfrac{dx^i}{dt}\)$ as kinetic energy. Thus the Riemannian space ${\cal R}^n=(M,\g_{ij}(x))$ can be prolonged to the Riemannian space of order $k\geq 1$ Prol$^k{\cal R}^n=(T^kM,\G)$, where
\be
\G=\g_{ij}(x)\de x^i\otimes \de x^j+\g_{ij}(x)\de y^{(1)i}\otimes\de y^{(1)j}+...+\g_{ij}(x)\de y^{(k)i}\otimes\de y^{(k)j}.
\ee
The 1-forms $\de x^i,\de y^{(1)i},...,\de y^{(k)i}$ being determined as follows:
\be
\begin{array}{l}
\de x^i=dx^i,\ \de y^{(1)i}=dy^{(1)i}+\underset{(1)}{M}^i_j dx^j,...,\de y^{(k)i}=dy^{(k)i}+\\ \\
+\underset{(1)}{M}^i_j dy^{(k-1)i}+...+\underset{k}{M}^i_j dx^j.
\end{array}
\ee
In these formulas $\underset{(\a)}{M}^i_j$ ($\a=1,...,k$) are the dual coefficients of a nonlinear connection $\oci{N}$ determined only on the Riemannian structure $\g_{ij}(x)$. That are
\be
\begin{array}{l}
\underset{1}{M}^i_j(x,y^{(1)})=\g_{jh}^i(x)y^{(1)m},\\ \\
\underset{2}{M}^i_j(x,y^{(1)},y^{(2)})=\dfrac12\(\Gamma\underset{(1)}{M}^i_j+\underset{(1)}{M}^i_m\underset{(1)}{M}^m_j\)\\
..................................\\
\underset{k}{M}^i_j(x,y^{(1)},...,y^{(k)})=\dfrac12\(\Gamma\underset{(k-1)}{M}^i_j+\underset{(1)}{M}^i_m\underset{(k-1)}{M}^m_j\),
\end{array}
\ee
$\Gamma$ being the nonlinear operator:
\be
\Gamma=y^{(1)i}\dfrac{\pp}{\pp x^i}+2y^{(2)i}\dfrac{\pp}{\pp y^{(1)i}}+...+ky^{(k)i}\dfrac{\pp}{\pp y^{(k-1)i}}.
\ee
Thus, (cf. ch. ...) the dual coefficients $\underset{(\a)}{M}$ (5.7.4) and the primal coefficients $\underset{(\a)}{N}$ (5.4.7) depend on the Riemannian structure $\g_{ij}(x)$, only.

Now the Lagrangian mechanical system of order $k$:
\be
\begin{array}{l}
\Sigma_{Prol^k{\cal R}^n}=(M,L(x,y^{(1)},...,y^{(k)})),\ Fe(x,y^{(1)},y^{(k)}),\\ \\
L=\g_{ij}(x)z^{(k)i}z^{(k)j},\\ \\
kz^{(k)i}=ky^{(k)i}+\underset{(1)}{M}^i_m y^{(k-1)m}+...+\underset{(k-1)}{M}^i_m y^{(1)m},
\end{array}
\ee
depend only on $\Sigma_{{\cal R}^n}$. It is the prolongation of order $k$ of the Riemannian mechanical system (4.7.1), $\Sigma_{{\cal R}^n}$.

We can postulate:

{\it The Lagrange equations of $\Sigma_{{\cal R}^n}$ are the Lagrange equations of the prolonged Lagrange equations of order $k$, $\Sigma_{Prol^k{\cal R}^n}$:}
\be
\begin{array}{l}
\dfrac{d}{dt}\(\dfrac{\pp L(x,y^{(1)},...,y^{(k)})}{\pp y^{(k)i}}\)-\dfrac{\pp L(x,y^{(1)},...,y^{(k)})}{\pp y^{(k-1)i}}=F_i(x,y^{(1)},...,y^{(k)})\\ \\
y^{(1)i}=\dfrac{dx^i}{dt},...,y^{(k)i}=\dfrac{1}{k!}\dfrac{d^k x^i}{dt^k}.
\end{array}
\ee

Applying the general theory from this chapter, we have

\begin{theorem}
The Lagrange equations of the classical Riemannian systems $\Sigma_{\cal R}$ with external forces depending on higher order accelerations are given by
\be
\begin{array}{l}
\dfrac{d^{k+1}x^i}{dt^{k+1}}+(k+1)!\oci{G}^i(x,y^{(1)},...,y^{(k)})=\dfrac12 k! F^i(x,y^{(1)},...,y^{(k)})\\ \\
(k+1)\oci{G}^i=\dfrac12 \g^{ij}\[\Gamma\dfrac{\pp L}{\pp y^{(k)j}}-\dfrac{\pp L}{\pp y^{(k-1)}}\]\\ \\
L=\g_{ij}(x)z^{(k)i}z^{(k)j}.
\end{array}
\ee
\end{theorem}

The canonical semispray of $\Sigma_{\cal R}$ is expressed by
\be
S=y^{(1)i}\dfrac{\pp}{\pp x^i}+...+ky^{(k)i}\dfrac{\pp}{\pp y^{(k-1)i}}-(k+1)\oci{G}^i\dfrac{\pp}{\pp y^{(k)i}}+\dfrac12 F^i\dfrac{\pp}{\pp y^{(k)i}}.
\ee
The integral curves of $S$ are given by the Lagrange equations (5.7.15).

The canonical semispray $S$ is the dynamical system of the classical Lagrangian mechanical system $\Sigma_{\cal R}$ with the external forces depending on the higher order accelerations.

\section{Finslerian mechanical systems of order $k$, $\sfk$}\index{Analytical Mechanics of! Finslerian of order $k$}
\setcounter{equation}{0}
\setcounter{definition}{0}
\setcounter{theorem}{0}
\setcounter{lemma}{0}
\setcounter{remark}{0}
\setcounter{proposition}{0}

The theory of the Finslerian mechanical systems of order $k$ is a natural particularization of that of Lagrangian mechanical systems of order $k\geq 1$, described in the previous sections of the present chapter.

\begin{definition}
A Finslerian mechanical system of order $k\geq 1$ is a triple:
\be
\sfk=(M,F(x,y^{(1)},...,y^{(k)})), Fe(x,y^{(1)},...,y^{(k)}),
\ee
where
\be
F^{(k)n}=(M,F(x,y^{(1)},...,y^{(k)}))
\ee
is a Finsler space of order $k\geq 1$, and
\be
Fe(x,y^{(1)},...,y^{(k)})=F^i(x,y^{(1)},...,y^{(k)})\dfrac{\pp}{\pp y^{(k)i}}
\ee
is an a priori given vertical vector field.
\end{definition}

$F(x,y^{(1)},...y^{(k)})$ is the fundamental function of $\sfk$,
\be
g_{ij}=\dfrac12\dfrac{\dot{\pp}}{\pp y^{(k)i}}\dfrac{\pp}{\pp y^{(k)j}}F^2
\ee
is the fundamental tensor of $\sfk$ and $Fe$ are the external forces.

The Euler--Lagrange equations of $F^{(k)n}$ are given by (5.2.5) $\oci{E}_i(F^2)=0$, $y^{(1)i}=\dfrac{dx^i}{dt},...,y^{(k)i}=\dfrac{1}{k!}\dfrac{d^k x^i}{dt^k}$ and the Craig--Synge covector field is $E^{k-1}(F^2)$ from (5.2.7). Theorem 5.2.1 for $L=F^2$ can be applied.

Thus, the canonical semispray $\oci{S}$ of $F^{(k)n}$ is:
\be
\oci{S}=y^{(1)i}\dfrac{\pp}{\pp x^i}+2y^{(2)i}\dfrac{\pp}{\pp y^{(1)i}}+...+ky^{(k)i}\dfrac{\pp}{\pp y^{(k-1)i}}-(k+1)\oci{G}\dfrac{\pp}{\pp y^{(k)i}},
\ee
with the coefficients:
\be
(k+1)\oci{G}^i=\dfrac12 g^{ij}\[\Gamma\dfrac{\pp F^2}{\pp y^{(k)i}}-\dfrac{\pp F^2}{\pp y^{(k-1)i}}\]
\ee
and with nonlinear operator
\be
\Gamma=y^{(1)i}\dfrac{\pp}{\pp x^i}+2y^{(2)i}\dfrac{\pp}{\pp y^{(1)i}}+...+ky^{(k)i}\dfrac{\pp}{\pp y^{(k-1)i}}.
\ee
The canonical nonlinear connection $\oci{N}$ of $F^{(k)n}$ has dual coefficients (5.2.10) particularized in the following form:
\be
\begin{array}{l}
\oci{\underset{1}{M}}^i_j=\dfrac{\pp\oci{G}^i}{\pp y^{(k)j}}\\ \\
\oci{\underset{2}{M}}^i_j=\dfrac12(\oci{S}\oci{\underset{1}{M}}^i_j+\oci{\underset{1}{M}}^i_m\oci{\underset{1}{M}}^m_j),\\
.................................\\
\oci{\underset{k}{M}}^i_j=\dfrac{1}{k}(\oci{S}\oci{\underset{k-1}{M}}^i_j+\oci{\underset{k-1}{M}}^i_m\oci{\underset{k-1}{M}}^m_j).
\end{array}
\ee

The fundamental equations of the Finslerian mechanical systems of order $n\geq 1$, $\Sigma_{F^n}$ are given by

\medskip

\noindent{\bf Postulate 4.8.1.} {\it The evolution equations of the Finslerian mechanical system $\sfk=(M,F,Fe)$ are the following Lagrange equations}
\be
\begin{array}{l}
\dfrac{d}{dt}\dfrac{\pp F^2}{\pp y^{(k)i}}-\dfrac{\pp F^2}{\pp y^{(k-1)i}}=F_i,\ \ F_i=g_{ij}F^j,\\ \\
y^{(1)i}=\dfrac{dx^i}{dt},...,y^{(k)i}=\dfrac{1}{k!}\dfrac{d^k x^i}{dt^k}.
\end{array}
\ee

\begin{remark}
For $k=1$, (5.8.8) are the Lagrange equations of a Finslerian mechanical system $\Sigma_F=(M,F,Fe)$.
\end{remark}

The Lagrange equations are equivalent to the following system of differential equations of order $k+1$:
\be
\dfrac{d^{k+1}x^i}{dt^{k+1}}+\dfrac{k!}{2}g^{ij}\[\Gamma\dfrac{\pp F^2}{\pp y^{(k)i}}-\dfrac{\pp F^2}{\pp y^{(k-1)}}\]=\dfrac{k!}{2}F^i.
\ee
Integrating this system in initial conditions we obtain an unique solution $x^i=x^i(t)$, which express the moving of Finslerian mechanical system $\sfk$.

But it is convenient to write the system (5.8.9) in the form (5.3.5):
\be
\begin{array}{l}
y^{(1)i}=\dfrac{dx^i}{dt},\ y^{(2)i}=\dfrac12\dfrac{d^2 x^i}{dt^2},...,y^{(k)i}=\dfrac{1}{k!}\dfrac{d^k x^i}{dt^k}\\ \\
\dfrac{dy^{(k)i}}{dt}+(k+1)\oci{G}^i(x,y^{(1)},...,y^{(k)})=\dfrac12 F^i(x,y^{(1)},...,y^{(k)}).
\end{array}
\ee
These equations determine the trajectories of the vector field
\be
S=\oci{S}+\dfrac12 Fe.
\ee
Consequently, we have:

\begin{theorem}
$1^\circ$ For the Finslerian mechanical system of order $k\geq 1$, $\sfk$ the operator $S$ from $(5.8.12)$ is a $k-$semispray which depend only on $\sfk$.

$2^\circ$ The integral curves of $S$ are the evolution curves of $\sfk$.
\end{theorem}

$S$ is the canonical semispray of the mechanical system $\sfk$. It is named the dynamical system of $\Sigma_{F^n}$, too.

The geometrical theory of the Finslerian mechanical system $\sfk$ can be developed by means of the canonical semispray $S$, applying the theory from the sections 5.4, 5.5, 5.6 from this chapter.

\begin{remark}
Let $\Sigma_F=(M,F,Fe)$ be a Finslerian mechanical system, where $Fe$ depend on the higher order accelerations. What kind of Lagrange equations we have for $\Sigma_F$?
\end{remark}

The problem can be solved by means of the prolongation of Finsler space $F^n=(M,F)$ to the manifold $T^kM$, following the method used in the section 5.7 of the present chapter.

\section{Hamiltonian mechanical systems of order $k\geq 1$}\index{Analytical Mechanics of! Hamiltonian of order $k$}
\setcounter{equation}{0}
\setcounter{definition}{0}
\setcounter{theorem}{0}
\setcounter{lemma}{0}
\setcounter{remark}{0}
\setcounter{proposition}{0}

The Hamiltonian mechanical systems of order $k\geq 1$ can be studied as a natural extension of that of Hamiltonian mechanical systems of order $k=1$, given by $\Sigma_H=(M,H(x,p),Fe(x,p))$, ch. 4, part II.

So, a Hamiltonian mechanical system of order $k\geq 1$ is a triple
\be
\shk=(M,H(x,y^{(1)},...,y^{(k-1)}),p), Fe(x,y^{(1)},...,y^{(k-1)},p)
\ee
where
\be
H^{(k)n}=(M,H(x,y^{(1)},...,y^{(k-1)}),p)
\ee
is a Hamilton space of order $k\geq 1$ (see definition 5.1.1) and
\be
Fe(,y^{(1)},...,y^{(k-1)},p)=F_i(x,y^{(1)},...,y^{(k-1)},p)\dot{\pp}^i
\ee
are the external forces.

Of course, $\dot{\pp}^i=\dfrac{\pp}{\pp p_i}$.

Looking to the Hamilton--Jacobi equation of the space $H^{(k)n}$ we can formulate

\medskip

\noindent{\bf Postulate 5.9.1.} {\it The Hamilton equations of the Hamilton mechanical system $\shk$ are as follows:}
\be
\left\{
\begin{array}{l}
{\cal H}=\dfrac12 H\\ \\
\dfrac{dx^i}{dt}=\dfrac{\pp{\cal H}}{\pp p_i}\\ \\
\dfrac{dp_i}{dt}=-\dfrac{\pp{\cal H}}{\pp x_i}+\dfrac{d}{dt}\dfrac{\pp{\cal H}}{\pp y^{(1)i}}+...+(-1)\dfrac{1}{(k-1)!}\dfrac{d}{dt}\dfrac{d^{k-1}}{dt^{k-1}}\dfrac{\pp{\cal H}}{\pp y^{k-1}}=\dfrac12 F_i\\ \\
y^{(1)i}=\dfrac{dx^i}{dt},\ y^{(2)i}=\dfrac12\dfrac{d^2 x^i}{dt^2},...,y^{(k-1)}=\dfrac{1}{(k-1)!}\dfrac{d^{k-1}x^i}{dt^{k-1}}.
\end{array}
\right.
\ee

It is not difficult to see that these equations have a geometric meaning.

For $k=1$ the equations (5.9.4) are the Hamilton equations of a Hamiltonian mechanical system $\Sigma_H=(M,H(x,p),Fe(x,p))$ .

These reasons allow us to say that the Hamilton equations (5.9.4) are {\it the fundamental equations} for the evolution of the Hamiltonian mechanical systems of order $k\geq 1$.

\medskip

\noindent{\bf Example 5.9.1.} The mechanical system $\shk$ with
\be
H(x,y^{(1)},...,y^{(k-1)},p)=\a\g^{ij}(x,y^{(1)},...,y^{(k-1)})p_i p_j,
\ee
$\g^{ij}$ being a Riemannian metric tensor on the manifold $T^{(k-1)n}$ and with
\be
Fe=\b_i(x,y^{(1)},...,y^{(k-1)})\dot{\pp}^i\tag{5.9.5'}
\ee
$\b_i$ is a $d-$covector field on $T^{k-1}M$ and $\a>0$ a constant.

Let us consider the energy of order $k-1$ of
\be
\begin{array}{c}
{\cal E}^{k-1}(H)=I^{k-1}(H)-\dfrac{1}{2!}\dfrac{d}{dt}I^{k-2}(H)+...+\\ \\ +(-1)^{k-2}\dfrac{1}{(k-1)!}\dfrac{d^{k-2}}{dt^{k-2}}(H)-H,
\end{array}
\ee
where $$I^{(1)}H={\cal L}_{\overset{1}\Gamma}H,...,I^{k-1}H={\cal L}_{\overset{{k-1}}\Gamma}.$$

\begin{theorem}
The variation of energy ${\cal E}^{k-1}(H)$ along the evolution curves $(5.9.4)$ can be calculate without difficulties.
\end{theorem}

\bigskip

\bigskip

Thus, the canonical nonlinear connection $N^*$ of $\shk$ is given by the direct decomposition (5.1.11), and has the coefficients $\underset{(1)}{N}^i_j$,...,$\underset{(k-1)}{N}^i_j$, $N_{ij}$ from (5.1.12') and its dual coefficients $\underset{(1)}{M}$,...,$\underset{(k-1)}{M}$ from (5.1.14). We can calculate the Liouville vector fields $z^{(1)},...,z^{(k-1)}$ from (5.1.17).

The metrical $N^*-$linear connection has the coefficients $D\Gamma(N^*)=(H^s_{ij},\underset{(k)}{C}^i_{jk},C^{js}_i)$ given by (5.5.2).

Now we can say: The geometries of the Hamiltonian mechanical systems of order $k\geq 1$, $\shk$ can be constructed only by the fundamental equations (5.9.4), by the nonlinear connection of $\Sigma_{H^k}$ and by the $N^*-$metrical connection $C\Gamma(N^*)$.

\begin{remark}
The homogeneous case of Cartanian mechanical systems of order $k$, $\Sigma_{{\cal C}^k}=(M,K(x,y^{(1)},...,y^{(k-1)}),p)$, $Fe(x,y^{(1)}$, ...,\break $y^{(k-1)}p)$ is a direct particularization of the previous theory for $H(x,y^{(1)}$, ..., $y^{(k-1)},p)=K^2(x,y^{(1)},$ ..., $y^{(k-1)},p)$, where $K(x,y^{(1)},...,y^{(k-1)},p)$ is $k-$homogeneous with respect to variables $(y^{(1)},...,y^{(k-1)},p)$.
\end{remark}

\backmatter
\printindex


\end{document}